\documentclass[a4paper, 10pt]{amsart}

\usepackage{amsmath,amsthm,amssymb,graphicx,rotating,
pinlabel,yhmath,tikz,tikz-cd,
listings,xcolor,enumerate,braket,amsfonts,subfiles}

\usepackage[normalem]{ulem}
\usepackage[all]{xy}

\usetikzlibrary{arrows.meta, decorations.pathmorphing, positioning}
\usetikzlibrary{decorations.pathreplacing}

\usepackage{hyperref}
\hypersetup{colorlinks=true,citecolor=brown,
linkcolor=brown,urlcolor=black,filecolor=black}

\allowdisplaybreaks
\numberwithin{equation}{section} 
\SelectTips{cm}{}

\theoremstyle{definition}
\newtheorem{definition}{Definition}[section]
\newtheorem{remark}[definition]{Remark}
\newtheorem{example}[definition]{Example}

\theoremstyle{plain}
\newtheorem{theorem}[definition]{Theorem}
\newtheorem{proposition}[definition]{Proposition}
\newtheorem{lemma}[definition]{Lemma}
\newtheorem{corollary}[definition]{Corollary}

\newcommand{\up}{\vspace{-0.5cm}}

\newcommand{\by}[1]{\stackrel{\eqref{#1}}{=}}

\newcommand{\Q}{\mathbb{Q}}
\newcommand{\Z}{\mathbb{Z}}

\newcommand{\frakm}{\mathfrak{m}}

\newcommand{\frakL}{\mathfrak{L}}

\newcommand{\calA}{\mathcal{A}}
\newcommand{\calB}{\mathcal{B}}
\newcommand{\calI}{\mathcal{I}}
\newcommand{\calIC}{\mathcal{IC}}
\newcommand{\calM}{\mathcal{M}}

\newcommand{\calP}{\mathcal{P}}
\newcommand{\calQ}{\mathcal{Q}}
\newcommand{\calR}{\mathcal{R}}
\newcommand{\calS}{\mathcal{S}}

\newcommand{\calT}{\mathcal{T}}

\newcommand{\fr}{\operatorname{fr}}
\newcommand{\Gr}{\operatorname{Gr}}

\newcommand{\Hom}{\operatorname{Hom}}
\newcommand{\Lie}{\operatorname{Lie}}

\newcommand{\Sp}{\operatorname{Sp}}

\newcommand{\id}{\operatorname{id}}
\newcommand{\Ker}{\operatorname{Ker}}
\newcommand{\Tr}{\operatorname{Tr}}
\newcommand{\Tors}{\operatorname{Tors}}
\newcommand{\Der}{\operatorname{Der}}
\renewcommand{\Im}{\operatorname{Im}}

\newcommand{\Aut}{\mathrm{Aut}}

\newcommand{\bqm}{\operatorname{``}\!}
\newcommand{\eqm}{\!\operatorname{''}}

\definecolor{myblue}{RGB}{0, 136, 170}
\definecolor{myred}{RGB}{191, 59, 59}

\newcommand{\Yrev}{\rotatebox{180}{$Y$}}

\newcommand{\bead}{\tikz{
\draw[black, fill=lightgray] (0,0) circle (0.5ex);} }

\newcommand{\threetree}[3]{\tikz[baseline=8pt ,scale = 0.5]{
\draw [thick,color=black] (1,1.5)-- (1,0.75);
\draw [thick,color=black] (1,0.75)-- (1-0.866*0.75,0.75-.37);
\draw [thick,color=black] (1,0.75)-- (1+0.866*0.75,0.75-.37);
\draw [black, fill = black] (1,0.75) circle (.5ex);
\draw[color=black] (1,1.5+0.2) node {\small $#1$};
\draw[color=black] (1-0.866*0.75-0.2,0.75-.37-0.2) node {\small $#2$};
\draw[color=black] (1+0.866*0.75+0.2,0.75-.37-0.2) node {\small $#3$};}}

\newcommand{\fourtree}[4]{\tikz[baseline=8pt, scale = 0.5]{
\draw [thick,color=black] (0,2-1.25)-- (2,2-1.25);
\draw [thick,color=black] (0,1.5)-- (0,0);
\draw [thick,color=black] (2,1.5)-- (2,0);
\draw [black, fill = black] (0,0.75) circle (.5ex);
\draw [black, fill = black] (2,0.75) circle (.5ex);
\draw[color=black] (0,3-1.25) node[yshift=2pt] {\small $#1$};
\draw[color=black] (0,1.1-1.25) node[yshift=-2pt] {\small $#2$};
\draw[color=black] (2,1.1-1.25) node[yshift=-2pt] {\small $#3$};
\draw[color=black] (2,3-1.25) node[yshift=2pt] {\small $#4$};}}

\newcommand{\fivetree}[5]{\tikz[baseline=8pt ,scale = 0.5]{
\draw [thick,color=black] (0,1.25)-- (0,0.25);
\draw [thick,color=black] (0,0.75)-- (2,0.75);
\draw [thick,color=black] (2,1.25)-- (2,0.25);
\draw [thick,color=black] (1,0.75)-- (1,1.5);
\draw [black, fill = black] (0,0.75) circle (.5ex);
\draw [black, fill = black] (1,0.75) circle (.5ex);
\draw [black, fill = black] (2,0.75) circle (.5ex);
\draw[color=black] (0,0) node {\small $#1$};
\draw[color=black] (0,1.6) node {\small $#2$};
\draw[color=black] (1,1.75) node {\small $#3$};
\draw[color=black] (2,1.6) node {\small $#4$};
\draw[color=black] (2,0) node {\small $#5$};}}

\newcommand{\loopedtwo}[2]{\tikz[baseline=8pt ,scale = 0.5]{
\draw [thick,color=black] (1,0.75) circle (0.5);
\draw [thick,color=black] (0.5,0.75)-- (0,0.75);
\draw [thick,color=black] (1.5,0.75)-- (2,0.75);
\draw [black, fill = black] (1.5,0.75) circle (.5ex);
\draw [black, fill = black] (0.5,0.75) circle (.5ex);
\draw[color=black] (-0.3,0.75) node {\small $#1$};
\draw[color=black] (2.5,0.75) node {\small $#2$};
}}

\newcommand{\Odiag}[3]{\tikz[baseline=8pt ,scale = 0.5]{
\draw [thick,color=black] (1,0.75) circle (0.5);
\draw [thick,color=black] (1,0.75+0.5)-- (1,0.75+1);
\draw [thick,color=black] (1+0.707*0.5,0.75 + 0.707*0.5)-- (1+0.707,0.75 + 0.707);
\draw [thick,color=black] (1-0.707*0.5,0.75 + 0.707*0.5)-- (1-0.707,0.75 + 0.707);
\draw [black, fill = black] (1,0.75+0.5) circle (.5ex);
\draw [black, fill = black] (1-0.707*0.5,0.75 + 0.707*0.5) circle (.5ex);
\draw [black, fill = black] (1+0.707*0.5,0.75 + 0.707*0.5) circle (.5ex);
\draw[color=black] (1-0.707*1.3,0.75 + 0.707*1.3) node {\small $#1$};
\draw[color=black] (1,0.75+1.3) node {\small $#2$};
\draw[color=black] (1+0.707*1.3,0.75 + 0.707*1.3) node {\small $#3$};
\node at (1+0.707*0.65,0.75 + 0.707*1.55) [rotate=-25] {\scriptsize $<$};
\node at (1-0.707*0.65,0.75 + 0.707*1.4) [rotate=30] {\scriptsize $<$};
}}

\newcommand{\loopedfour}[4]{\tikz[baseline=8pt ,scale = 0.5]{
\draw [thick,color=black] (1,0.75) circle (0.5);
\draw [thick,color=black](1+0.707*0.5,0.75+0.707*0.5)-- (1+0.707*0.85,0.75+0.707*0.85);
\draw [thick,color=black](1+0.707*0.5,0.75-0.707*0.5)-- (1+0.707*0.85,0.75-0.707*0.85);
\draw [thick,color=black](1-0.707*0.5,0.75-0.707*0.5)-- (1-0.707*0.85,0.75-0.707*0.85);
\draw [thick,color=black](1-0.707*0.5,0.75+0.707*0.5)-- (1-0.707*0.85,0.75+0.707*0.85);
\draw [black, fill = black] 
(1+0.707*0.5,0.75+0.707*0.5) circle (.5ex);
\draw [black, fill = black] (1+0.707*0.5,0.75-0.707*0.5) circle (.5ex);
\draw [black, fill = black] (1-0.707*0.5,0.75+0.707*0.5) circle (.5ex);
\draw [black, fill = black] (1-0.707*0.5,0.75-0.707*0.5) circle (.5ex);
\draw[color=black] (1+0.707*1.25,0.75+0.707*1.25) node {\small $#1$};
\draw[color=black] (1-0.707*1.25,0.75+0.707*1.25) node {\small $#2$};
\draw[color=black] (1-0.707*1.25,0.75-0.707*1.25) node {\small $#3$};
\draw[color=black] (1+0.707*1.25,0.75-0.707*1.25) node {\small $#4$};
}}

\newcommand{\ltritree}[3]{\tikz[baseline=8pt, scale = 0.8]{
\draw[thick,color=black] (0.5,0)--(0.5,0.25);
\draw[thick,color=black] (0.5,0.25)--(0,0.75);
\draw[thick,color=black] (0.25,0.5)--(0.5,0.75);
\draw[thick,color=black] (0.5,0.25)--(1,0.75);
\draw [black, fill = black] (0.5,0.25) circle (.3ex);
\draw [black, fill = black] (0.25,0.5) circle (.3ex);
\draw[color=black] (0,1) node {\small $#1$};
\draw[color=black] (0.5,1) node {\small $#2$};
\draw[color=black] (1,1) node {\small $#3$};}}

\newcommand{\rooteddottritree}[1]{\tikz[baseline=8pt, scale = 0.8, thick]{
\draw (0.5,0)--(0.5,0.25);
\draw (0.5,0.25)--(0,0.75);
\draw (0.25,0.5)--(0.5,0.75);
\draw (0.5,0.25)--(1,0.75);
\begin{scope}[dotted]
\draw (-0.25,1)--(0,0.75);
\draw (0.5,0.75)--(0.78,1.03);
\draw (1.25,1)--(1,0.75);
\end{scope}
\draw [fill] (0.5,0.25) circle (.3ex);
\draw [fill] (0.25,0.5) circle (.3ex);
\draw[color=black] (0.5,-0.25) node {\small $#1$};
}}

\newcommand{\revltritree}{\tikz[baseline=8pt, scale = 0.8]{
\begin{scope}[yscale=1,xscale=-1]
\draw[thick,color=black] (0.5,0)--(0.5,0.25);
\draw[thick,color=black] (0.5,0.25)--(0,0.75);
\draw[thick,color=black] (0.25,0.5)--(0.5,0.75);
\draw[thick,color=black] (0.5,0.25)--(1,0.75);
\draw [black, fill = black] (0.5,0.25) circle (.3ex);
\draw [black, fill = black] (0.25,0.5) circle (.3ex);
\end{scope}}}

\newcommand{\ltritreebis}{\tikz[baseline=8pt, scale = 0.8]{
\draw[thick,color=black] (0.5,0)--(0.5,0.25);
\draw[thick,color=black] (0.5,0.25)--(0,0.75);
\draw[thick, color=black]
  (0.25,0.5) to[out=30, in=-110] (1,0.75);
\fill [white] (0.58,0.55) circle [radius=0.12];
\draw[thick,color=black] (0.5,0.25)--(0.65,0.75);
\draw [black, fill = black] (0.5,0.25) circle (.3ex);
\draw [black, fill = black] (0.25,0.5) circle (.3ex);}}

\newcommand{\forktree}[5]{ \tikz[baseline = 1.75cm]{
\draw[thick, color = black](4.5,1.75) .. controls (4.5,2) .. (5,2);
\draw[thick, color = black](4.5,1.75) .. controls (4.5,1.5) .. (5,1.5);
\draw[thick, color = black](5,2) -- (5.4,2);
\draw[thick, color = black](5,1.5) -- (5.5,1.5);
\draw[thick, color = black](5,2) -- (5,2.25);
\draw[thick, color = black](5.5,1.5) -- (5.5,2.25);
\draw[thick, color = black](5.5,1.5) .. controls (6.5,1.5) .. (6.5,2.25);
\draw[thick, color = black](5.6,2) .. controls (6,2) .. (6,2.25);
\draw[thick, color = black](4.5,1.75) .. controls (4.25,1.75) .. (4.25,2.25);
\draw [black, fill = black] (4.5,1.75) circle (.4ex);
\draw [black, fill = black] (5,2) circle (.4ex);
\draw [black, fill = black] (5.5,1.5) circle (.4ex);
\draw[color=black] (4.25,2.45) node {\small $#1$};
\draw[color=black] (5,2.45) node {\small $#2$};
\draw[color=black] (5.5,2.45) node {\small $#3$};
\draw[color=black] (6,2.45) node {\small $#4$};
\draw[color=black] (6.5,2.45) node {\small $#5$};
\draw[color=black] (5.25,2.425) node {\scriptsize $<$};
\draw[color=black] (5.75,2.425) node {\scriptsize $<$};
\draw[color=black] (6.25,2.425) node {\scriptsize $<$};
\draw[color=black] (4.6,2.425) node {\scriptsize $<$};
}}

\newcommand{\thetagraph}{\tikz[baseline=-2pt, scale=0.5]{
  \draw[thick, color=black] (0,0) circle (0.5);
  \draw[thick, color=black] (-0.5,0) -- (0.5,0);
  \draw[black, fill=black] (-0.5,0) circle (0.5ex);
  \draw[black, fill=black] (0.5,0) circle (0.5ex);
}}

\def\clasper{
\begin{tikzpicture}[xscale=0.6, yscale=0.55, baseline={(0,0.8)}]
\def\dx{0.5}
\def\dy{0.9}
\def\W{13.8}
\pgfmathsetmacro{\xA}{\W/2-5.4}
\draw (0,0) -- (\W,0) -- (\W+\dx,\dy) -- (\dx,\dy) -- cycle;
\draw (0,2) -- (\W,2) -- (\W+\dx,2+\dy) -- (\dx,2+\dy) -- cycle;
\draw (0,0) -- (0,2);
\draw (\dx,\dy) -- (\dx,2+\dy);
\draw (\W,0) -- (\W,2);
\draw (\W+\dx,\dy) -- (\W+\dx,2+\dy);
\foreach\aa in {0,1.6,3.2,4.8,6.4,8,9.6,11.2} {
\draw[line width=2pt,white] (\xA-1+\dx+\aa,2+\dy) -- (\xA+\dx+\aa,2+\dy);
\draw[line width=2pt,white] ( \xA-0.7+\dx+\aa, \dy) -- ( \xA-0.3+\dx+\aa, \dy);
  \draw (\xA+\aa-1+\dx,2.4) arc[start angle=180,end angle=360,x radius=0.5,y radius=0.15];
  \draw [thick,dotted] (\xA+\aa+\dx,2.4) arc[start angle=0,end angle=180,x radius=0.5,y radius=0.15];
  \draw (\xA+\aa-1+\dx,2.4) -- (\xA+\aa-1+\dx,3.1);
  \draw (\xA+\aa+\dx,2.4) -- (\xA+\aa+\dx,3.1);
  \draw (\xA+\aa-0.7+\dx,0.5) arc[start angle=180,end angle=360,x radius=0.2,y radius=0.1];
  \draw[thick,dotted] (\xA+\aa-0.3+\dx,0.5) arc[start angle=0,end angle=180,x radius=0.2,y radius=0.1];
  \draw (\xA+\aa-0.7+\dx,0.5) -- (\xA+\aa-0.7+\dx,3.1);
  \draw (\xA+\aa-0.3+\dx,0.5) -- (\xA+\aa-0.3+\dx,3.1);
}
\foreach\aa in {0,3.2,6.4,9.6} {
\draw (\xA-1+\dx+\aa,3.1) arc[start angle=180,end angle=0,radius=1.3];
\draw (\xA+\dx+\aa,3.1) arc[start angle=180,end angle=0,radius=0.3];
\draw (\xA-0.7+\dx+\aa,3.1) arc[start angle=180,end angle=0,radius=1.0];
\draw (\xA-0.3+\dx+\aa,3.1) arc[start angle=180,end angle=0,radius=0.6];
}
\begin{scope}[myred, very thick]
\draw (\xA+0.3,3.05) arc[start angle=180,end angle=0,radius=0.5];
\draw (\xA+0.3,2.55) arc[start angle=-180,end angle=0,radius=0.5];
\draw (\xA+0.3,3.05) -- (\xA+0.3,2.55);
\draw (\xA+1.3,3.05) -- (\xA+1.3,2.55);
\draw (\xA+3.5,3.05) arc[start angle=180,end angle=0,radius=0.5];
\draw (\xA+3.5,1.05) arc[start angle=-180,end angle=0,radius=0.5];
\draw (\xA+3.5,3.05) -- (\xA+3.5,1.05);
\draw (\xA+4.5,3.05) -- (\xA+4.5,1.05);
\draw[line width=2pt,white] ( \xA+4.5, 0.95) -- ( \xA+4.5, 1.20);
\draw (\xA+0.8,2.05) -- (\xA+0.8,0.5);
\draw (\xA+0.3,1.1) -- (\xA+0.8,0.5);
\draw (\xA+3.5,1.1) -- (\xA+0.8,0.5);
\fill (\xA+0.8,0.5) circle [x radius=0.25,y radius=0.2];
\end{scope}
\draw[line width=2pt,white] ( \xA+1.9, 0.73) -- ( \xA+2.2, 0.8);
\begin{scope}[shift={(1.2,1.1)},myred, very thick, scale=0.3]
 \draw (45:0.5) .. controls +(1,1) and +(0,1) .. (2,0);
 \draw (2,0) .. controls +(0,-1) and +(1,-1) .. (-45:0.5);
 \draw (-45:0.5) -- (135:0.5);
 \draw (135:0.5) .. controls +(-1,1) and +(0,1) .. (-2,0);
 \draw (-2,0) .. controls +(0,-1) and +(-1,-1) .. (-135:0.5);
\end{scope}
\begin{scope}[myblue, very thick]
\foreach\aa in {0,6.4,9.6} {
\draw (\xA+1.86+\aa,1.4) arc[start angle=50,end angle=-230,x radius=0.4,y radius=0.2];
}
\draw (\xA+5.16,1.4) arc[start angle=50,end angle=-230,x radius=0.6,y radius=0.2];
\fill (\xA+8.9,0.3) circle [x radius=0.25,y radius=0.2];
\fill (\xA+5.6,0.3) circle [x radius=0.25,y radius=0.2];
\draw (\xA+8.9,0.3) -- (\xA+5.6,0.3);
\draw (\xA+8.9,0.3) -- (\xA+10.9,1.1);
\draw (\xA+8.9,0.3) -- (\xA+8.1,1.05);
\draw (\xA+5.6,0.3) -- (\xA+4.9,1.05);
\draw (\xA+5.6,0.3) parabola bend (\xA+3.5, 0.3) (\xA+1.7,1.05);
\end{scope}
\end{tikzpicture}
}

\begin{document}

\title[On the Sp-structure   
of the Lie algebra of homology cylinders]{
On the Sp-structure of the torsion of\\
the Lie algebra of homology cylinders}

\author{Quentin Faes}
\address{Institut de Recherche en Mathématique et Physique, Université catholique de
Louvain, Chemin du Cyclotron 2, B-1348 Louvain-la-Neuve, Belgium}
\email{quentin.faes@uclouvain.be}

\author{Gw\'ena\"el Massuyeau}
\address{Université Bourgogne Europe, 
CNRS, IMB (UMR 5584), F-21000 Dijon, France}
\email{gwenael.massuyeau@ube.fr}

\author{Masatoshi Sato}
\address{
Department of Mathematics and Data Science,
Tokyo Denki University,
5 Senjuasahi-cho, Adachi-ku, Tokyo 120-8551,
Japan}
\email{msato@mail.dendai.ac.jp}

\date{}

\thanks{}

\begin{abstract}
Let $\Sigma$ be a compact oriented surface. 
We investigate the graded Lie algebra of homology cylinders over $\Sigma$,
as introduced by M.~Goussarov and K.~Habiro. 
To analyze its torsion, we refine a strategy initiated
by Y.~Nozaki, M.~Suzuki, and the third author, 
which relies on the reduction modulo $1$ of the LMO functor 
and on clasper calculus. 
Specifically, we develop general tools for studying, 
under the standard action of the symplectic group,
the $\Sp$-module structure of the torsion 
in the odd-degree component of this Lie algebra. 
We show that  this torsion part 
surjects onto an $2$-torsion $\Sp$-module,
which is explicitly described in terms of Jacobi diagrams.
As an application, we provide an intrinsic description
of the $\Sp$-module given  by the degree-three component
of the Lie algebra of homology cylinders.
A further motivation is to understand torsion phenomena 
in the associated graded  of the lower central series 
of the Torelli group of $\Sigma$:
in this direction, we exhibit an explicit Sp-module 
onto which the torsion of the Torelli Lie algebra surjects in degree~three.
\end{abstract}

\maketitle

\vspace{-1.1cm}
{\small \tableofcontents}

\section{Introduction}

\label{sec:intro}

Let $\Sigma$ be a compact, connected, oriented surface of genus $g\geq 0$.
For simplicity, we  assume here
that~$\Sigma$ has exactly one boundary component.
A \emph{homology cylinder} over $\Sigma$ is a compact, oriented
$3$-manifold $M$ with ``the same'' boundary 
and ``the same'' homology type as the unit cylinder 
$U:=\Sigma \times [-1,+1]$:
to be more specific, we assume that $M$ comes 
with an orientation-preserving diffeomorphism 
$m:\partial U \to \partial M$
(the \emph{boundary parametrization}) 
and there exists an isomorphism 
$H_*(U; \Z) \simeq H_*(M;\Z)$
compatible with the map $m_*$ induced by $m$ in homology.
Two homology cylinders $M$ and $M'$ are \emph{diffeomorphic} 
if there is an orientation-preserving diffeomorphism
$h: M \to M'$ such that $m'= h\vert_{\partial M} \circ m$.

Let $\mathcal{IC}:= \mathcal{IC}(\Sigma)$
be the set of diffeomorphism classes of homology cylinders over~$\Sigma$.
It has a natural structure of a monoid,
where the multiplication $M' \circ M$ of two elements $M,M' \in \calIC$
is defined by gluing the \emph{bottom boundary} $m(\Sigma \times \{-1\})$ of $M$
to the \emph{top boundary} $m'(\Sigma \times \{+1\})$ of $M'$; 
the unit element of $\calIC$ is the unit cylinder $U$
(with its obvious boundary parametrization).

The study of the monoid $\calIC$ by surgery techniques 
and  finite-type invariants was  initiated 
by M. Goussarov and K. Habiro  \cite{Gou00,Hab00}.
The reader is referred to the survey paper \cite{HM12} 
for an introduction to these topics and further references.
There are several reasons to be interested 
in the theory of  finite-type invariants for homology cylinders.
First of all, when $g = 0$, the surface $\Sigma$ is a disk, and  
$\mathcal{IC}$ consists of homology $3$-balls.  
In this case,  $\mathcal{IC}$ 
is essentially the monoid of homology $3$-spheres:
the class of $3$-manifolds for which the theory of finite-type invariants  
has been most thoroughly developed.  
In contrast, for $g > 0$, 
the monoid~$\mathcal{IC}$ provides an archetypal class  
of homologically non-trivial $3$-manifolds,  
revealing to what extent the theory of finite-type invariants  
is sensitive to the homology type.  

In addition, the monoid $\mathcal{IC}$ 
offers a natural framework for applying $3$-manifold invariants 
to the study of the \emph{Torelli group}  
$\mathcal{I} := \mathcal{I}(\Sigma)$,  
which is the subgroup of the mapping class group  
$\mathcal{M} := \mathcal{M}(\Sigma)$ acting trivially  
on the homology $H := H_1(\Sigma; \Z)$.  
Indeed, the ``mapping cylinder'' construction defines a monoid homomorphism  
\begin{equation}
\label{eq:mcc}
\mathbf{c} : \mathcal{I} \longrightarrow \mathcal{IC}    
\end{equation}
which is injective, and whose image  consists 
of the invertible elements of~$\mathcal{IC}$.
The conjugation action of $\calM$ on $\calI$ extends to an action of $\calM$ on $\calIC$: indeed, an $f\in \calM$ acts on $M\in \calIC$ by changing its boundary parametrization $m$ to $m\circ (f^{-1} \times \id)\vert_{\partial U}$.

An important tool in the study of the Torelli group $\mathcal{I}$  
is its lower central series  
$
\mathcal{I} = \Gamma_1 \mathcal{I}
\supset \Gamma_2 \mathcal{I}
\supset \Gamma_3 \mathcal{I} \supset \cdots
$
defined inductively by  
$
\Gamma_{k+1} \mathcal{I} := [\Gamma_k \mathcal{I}, \mathcal{I}].
$
It is well known that $\mathcal{I}$ is residually nilpotent,  
i.e.\ $\bigcap_{k \ge 1} \Gamma_k \mathcal{I} = \{1\}$.  
The \emph{Torelli Lie algebra} is the associated graded object  
\begin{equation} \label{eq:TLA}
\operatorname{Gr}^\Gamma \mathcal{I} := \bigoplus_{k \ge 1}
 \frac{\Gamma_k \mathcal{I}}{\Gamma_{k+1} \mathcal{I}}    
\end{equation}
of this filtration.
It has the structure of a graded Lie $\mathbb{Z}$-algebra
and is endowed with the action  of the symplectic group  
$\Sp(H) \simeq \mathcal{M}/\mathcal{I}$  
induced by the conjugation action of $\mathcal{M}$ on~$\mathcal{I}$.  
(Here the symplectic form 
is the intersection form $\omega:H \times H \to \Z$ of the surface.) 
The degree-one component of $\operatorname{Gr}^\Gamma \mathcal{I}$
 is the abelianization $\mathcal{I}_{\mathrm{ab}}$ of $\mathcal{I}$,  
which, as an $\Sp(H)$-module, 
was computed by D.~Johnson  (in genus $g \ge 3$)
in one of his seminal works on the Torelli group~\cite{Joh85}.  
In particular, his result revealed the presence of $2$-torsion,  
detected by the Birman--Craggs homomorphism  
(the $2$-dimensional incarnation of the  
Rokhlin invariant for homology $3$-spheres).  
All torsion vanishes after tensoring with $\mathbb{Q}$  
and, in his fundamental work~\cite{Hai97},  
R.~Hain obtained (in genus $g \ge 4$) 
an $\Sp(H \otimes \mathbb{Q})$-presentation  
of the graded Lie $\mathbb{Q}$-algebra  
$\big(\!\operatorname{Gr}^\Gamma \mathcal{I}\big) \otimes \mathbb{Q}$.

Of course, since it is not a group,  
the monoid $\mathcal{IC}$ does not possess a lower central series.  
Nevertheless, it carries a filtration by submonoids  
which plays a similar role:
this is the \emph{$Y$-filtration} 
\begin{equation} \label{eq:Y-filtration}
\mathcal{IC} = Y_1 \mathcal{IC} \supset Y_2 \mathcal{IC} 
\supset Y_3 \mathcal{IC} \supset \cdots
\end{equation}
which is defined by  
$Y_k \mathcal{IC} :=  \big\{\, M \in \mathcal{IC} :
M \text{ is $Y_k$-equivalent to } U \big\}$.
Here, the \emph{$Y_k$-equivalence} relation  
is generated by a specific type of surgery along
embedded graphs in $3$-manifolds, 
which are called \emph{graph clasper} of degree $k$;
the $Y_k$-equivalence relation becomes finer and finer as $k$ increases.
(See~\cite{HM12} and the references therein for the precise definition;  
a brief overview is also given in \S\ref{subsec:claspers}.)  
Conjecturally, it is expected that two homology cylinders are diffeomorphic
if and only if they are $Y_k$-equivalent for every $k \ge 1$,  
which would parallel the residual nilpotency of~$\mathcal{I}$.  
This conjecture is equivalent to asserting that  
finite-type invariants classify homology cylinders up to diffeomorphism~\cite{Mas07}.

Similarly to the Torelli Lie algebra $\Gr^\Gamma \calI$
with respect to the lower central series of $\calI$,
the \emph{Lie algebra of homology cylinders}  
is defined as the associated graded of the $Y$-filtration of $\calIC$:
\begin{equation} \label{eq:LAHC}
\operatorname{Gr}^Y\! \mathcal{IC} := 
\bigoplus_{k \ge 1} \frac{Y_k \mathcal{IC}}{Y_{k+1}}.
\end{equation}
It has the structure of a graded $\mathbb{Z}$-Lie algebra,
and it also carries a natural action of $\Sp(H)$  
via the canonical action of $\mathcal{M}$ on~$\mathcal{IC}$.
Since the mapping cylinder construction~$\mathbf{c}$ 
preserves the filtrations, 
it induces an $\Sp(H)$-equivariant homomorphism
$$
\Gr \mathbf{c} : \operatorname{Gr}^\Gamma \mathcal{I} 
\longrightarrow \operatorname{Gr}^Y\! \mathcal{IC} 
$$
of graded $\Z$-Lie algebras.
The degree-one component $\mathcal{IC}/Y_2$ 
of $\operatorname{Gr}^Y\! \mathcal{IC}$,
as an $\Sp(H)$-module, has been determined in~\cite{MM03}:
it turns out to be isomorphic 
to~$\mathcal{I}_{\mathrm{ab}}$  (in genus $g \ge 3$).  
Furthermore, a combinatorial description of the $\mathbb{Q}$-Lie algebra  
$\big(\operatorname{Gr}^Y\! \mathcal{IC}\big) \otimes \mathbb{Q}$,  
including its $\Sp(H\otimes\mathbb{Q})$-action,  
was obtained in~\cite{HM09} in terms of \emph{Jacobi diagrams},
and it was also related to Hain’s description of the  
rational Torelli Lie algebra.

To be more specific about this diagrammatic description 
of the Lie algebra of homology cylinders, let us recall that it needs
the clasper calculus developed by M.~Goussarov and K.~Habiro,
and the LMO functor~$\widetilde{Z}$ constructed in \cite{CHM08}.  
The construction of this TQFT-like invariant of cobordisms involves
the choice of a system of meridians and parallels
$(\alpha_1,\dots,\alpha_g,\beta_1,\dots,\beta_g)$
on the surface~$\Sigma$
\begin{equation}
\label{eq:surface}
\begin{array}{c}
\includegraphics[scale=0.45]{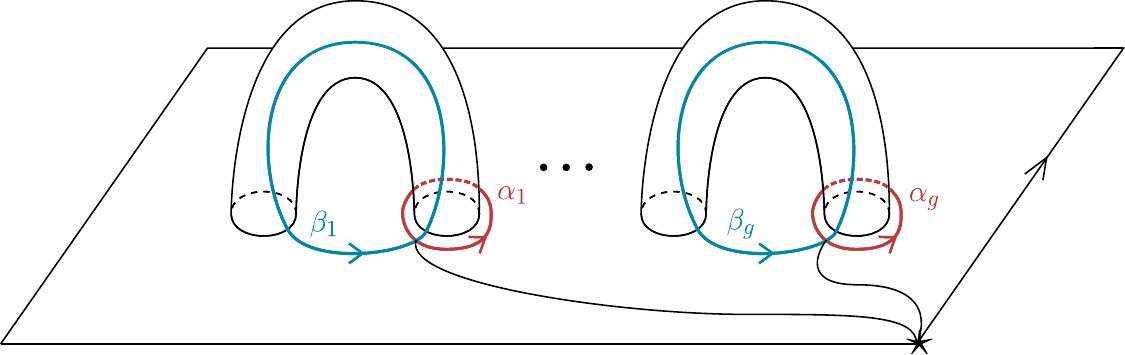}
\end{array}
\end{equation}

\noindent
as well as the choice of a Drinfeld associator. 
Note that this system of based curves induces a basis of $\pi:=\pi_1(\Sigma,\star)$,
and a basis $S=(a_1,\dots,a_g,b_1,\dots,b_g)$ of $H=H_1(\Sigma;\Z)$.
Thus, the representation
\begin{equation} \label{eq:raw_LMO}
\widetilde{Z}^Y: \calIC \longrightarrow \calA(S)\otimes \Q    
\end{equation}
that is \emph{directly} derived from the LMO functor 
for homology cylinders,
turns out to be highly dependent on these two choices;
it takes values in the algebra of Jacobi diagrams colored 
by the finite set $S$.
To partly control this double dependency,
it is better to work with two variations of $\widetilde{Z}^Y$,
which are introduced in  \cite{HM09}
by post-composing $\widetilde{Z}^Y$ with appropriate isomorphisms:
\begin{itemize}
\item the \emph{LMO homomorphism} $Z^<:  \calIC \to\calA^<(H)\otimes \Q$
with values in a diagrammatic analogue of the Weyl algebra,
where Jacobi diagrams are colored by~$H$
and  equipped with a certain order;
\item the \emph{symmetrized LMO homomorphism} $Z:  \calIC \to \calA(H)\otimes \Q$
with values in a diagrammatic analogue of the Moyal--Weyl deformation of the symmetric algebra, 
where Jacobi diagrams are only colored by $H$;
\end{itemize}
see \S \ref{subsec:LMO} and \S \ref{subsec:sym_LMO} 
for a very brief review of those constructions. 
While $Z^<$ is well-suited to clasper calculus, 
the homomorphism $Z$ seems to be more adapted for comparison with classical invariants.
Both homomorphisms are filtration preserving, 
and they induce $\Sp(H\otimes\Q)$-isomorphisms of $\Q$-Lie algebras
\begin{equation} \label{eq:isos}
\big(\Gr^Y\!\calIC\big)\otimes \Q \simeq \calA^{<,c}(H)\otimes \Q
\simeq \calA^{c}(H)\otimes \Q
\end{equation}
 at the graded level. (Here, the superscript ``c'' indicates modules
of \emph{connected} Jacobi diagrams.)
We observe that, 
as a consequence of the results of \cite{KRW23,NW26}, the map 
$(\Gr \mathbf{c})\otimes \Q: (\operatorname{Gr}^\Gamma \mathcal{I})\otimes \Q 
\to (\operatorname{Gr}^Y \mathcal{IC})\otimes \Q$
is injective in the stable range; its image corresponds 
to the Lie subalgebra of $\calA^{c}(H)\otimes \Q$ 
generated by its degree-one component
$\calA^{c}_1(H)\otimes \Q\simeq \Lambda^3 H\otimes \Q$ through the isomorphisms \eqref{eq:isos}.

With integral coefficients, the principal unresolved issues 
regarding the Torelli Lie algebra 
and the Lie algebra of homology cylinders can be formulated as follows, in each degree $k \ge 1$:\\[-0.3cm]
\begin{enumerate}
\item Describe  both
$\Gamma_k \mathcal{I} / \Gamma_{k+1} \mathcal{I}$ and $Y_k \mathcal{IC} / Y_{k+1}$  
as  $\Sp(H)$-modules.\\[-0.3cm]
\item Decide whether the $\Sp(H)$-homomorphism
$\Gr_k \mathbf{c}:\Gamma_k \mathcal{I} / \Gamma_{k+1} \mathcal{I} 
\to Y_k \mathcal{IC} / Y_{k+1}$
induced by~\eqref{eq:mcc}  is injective, and determine its image.\\[-0.3cm]
\end{enumerate}
As noted earlier, the degree $k=1$ case is completely understood \cite{Joh85, MM03}, 
and the same holds in degree $k=2$ \cite{Mor91,Hai97,MM13, FMS26}. 
Observe, moreover, that these two situations are markedly different: in degree $1$ there is $2$-torsion, whereas in degree $2$ no torsion occurs.

Recently, the LMO functor $\widetilde{Z}$ has also been shown 
to be effective in detecting the presence of previously unknown torsion 
in both $\Gr^\Gamma \calI$ and $\Gr^Y\!\calIC$, 
and thus determining
the isomorphism type of the abelian group $Y_k \mathcal{IC} / Y_{k+1}$ 
for new low values of~$k$~\cite{NSS22a, NSS24}. 
In these works, the central tool is the group homomorphism
\begin{equation} \label{eq:zYmod1}
\big(\widetilde{Z}^Y_{k+1}\!\!\! \mod 1\big): Y_k \mathcal{IC} / Y_{k+1} 
\longrightarrow \calA^c_{k+1}(S)\otimes \Q/\Z
\end{equation}
and other variants induced by the ``raw'' LMO homomorphism \eqref{eq:raw_LMO}. 
Thus, this approach differs in two essential respects 
from the use of the functor $\widetilde{Z}$ 
in proving the rational isomorphism \eqref{eq:isos}: 
(i) the functor is not applied at the  graded level 
(but with the degree shifted by $1$), 
and (ii) only the reduction modulo $1$ of $\widetilde{Z}$ is used.
It turns out that the descriptions obtained in \cite{NSS22a, NSS24}
depend on the choices inherent to the construction of $\widetilde{Z}$.\\

In this article, we refine  the approach 
of the aforementioned works
by \emph{explicitly incorporating the $\Sp(H)$-actions}. 
A principal aim is to \emph{derive intrinsic characterizations} 
of the $\Sp(H)$-modules 
$\Gamma_k \mathcal{I} / \Gamma_{k+1} \mathcal{I}$ 
and $Y_k \mathcal{IC} / Y_{k+1}$, 
with particular emphasis on their torsion submodules.

Thus, we begin
by setting in \S \ref{sec:surgery} and \S \ref{sec:Z^<}
the general constructions to reach these goals.
First of all, we review in \S \ref{sec:surgery}
the clasper surgery maps considered in \cite{Hab00,MM03,HM09,HM12},
which combine here into a surjective Lie map
$$
\psi:\calA^{<,c}(P) \longrightarrow \Gr^Y\!\calIC
$$
where $P:=H_1(\hbox{U}(\Sigma);\Z)$ 
is the first homology group of the unit tangent bundle of~$\Sigma$,
and $\calA^{<,c}(P)$ is a refinement of the above-mentioned
$\Z$-Lie algebra $\calA^{<,c}(H)$.
In fact, we give a more general treatment
by considering the refinement $\calA^{<}(P)$
of the entire $\Z$-algebra $\calA^<(H)$,  mapping it
by clasper surgery
to the larger \emph{algebra  of homology cylinders},
which needed slight additions with respect to what can already be
found in the literature 
(see Theorems \ref{th:psi} and \ref{th:psi_connected}).
Next, in place of the map \eqref{eq:zYmod1} which was the main tool of \cite{NSS22a},
we consider in \S \ref{sec:Z^<} the homomorphism
\begin{equation} \label{eq:z<mod1}
\zeta_{k+1}^< := \big({Z}^<_{k+1}\!\!\! \mod 1\big): Y_k \mathcal{IC} / Y_{k+1} 
\longrightarrow \calA_{k+1}^<(H)\otimes \Q/\Z .
\end{equation}
We introduce a degree $1$ homomorphism
$\delta^<:\calA^<(P) \to \calA^<(H)\otimes \Q/\Z$
(Proposition~\ref{prop:delta}) 
and, by combining its definition with the main result of \cite{NSS22a}, we prove that
$\zeta^<_{k+1}\circ \psi_k=\delta^<_k$
(Theorem~\ref{th:psi_delta_Z}):
said equivalently,
we compute the variation of $\zeta^<_{k+1}$
under \emph{arbitrary} clasper surgery.
We emphasize that, although explicit,
the definition of the map $\delta^<$ \emph{is not intrinsic}, 
in that it involves  the truncation to degrees $\leq 2$ 
of an appropriate expansion $\theta$ of $\pi$ 
into the (degree-completed) tensor algebra $\hbox{T}(H)\otimes \Q$: 
this truncation agrees with the corresponding truncation of the \emph{symplectic expansion}  considered in \cite{Mas12}, 
which clarifies the manner in which $\zeta^<_{k+1}$ 
depends on the two  choices involved in the construction of $\widetilde{Z}$.
As an application, we investigate structural properties of $\delta^<$ and, in particular, 
control the lack of  $\Sp(H)$-equivariance of $\zeta_{k+1}^<$ (Corollary \ref{cor:equivariance}).

Starting from \S \ref{sec:odd_degrees}, 
we mainly focus on the $2$-torsion
of the odd-degree component of $\Gr^Y\!\calIC$
that is produced by surgery along graph claspers 
with  order 2 symmetry.
For this purpose, we introduce a new module $\calB^<(P^{(2)})$
of Jacobi diagrams \emph{with beads} 
colored by $P^{(2)}:=H_1(\hbox{U}(\Sigma);\Z/2\Z)$: 
the ``beads'' on such a Jacobi diagram  are meant 
to be counted modulo 2, so that they encode a 2-fold cover of the Jacobi diagram,
which can thereafter be realized as a graph clasper in $U$.
In this way, 
we construct a new surgery map 
$$
\Psi: \calB^<_n\big(P^{(2)}\big) 
\longrightarrow 
\Tors_2\Big(\frac{Y_{2n+1} \calIC}{Y_{2n+2}}\Big),
$$
with values in the 2-torsion subgroup of $\Gr_{2n+1}^Y\calIC$ 
(see Theorem \ref{th:Psi}),
which is related to the ``usual'' surgery map $\psi$
through an ``expanding'' homomorphism
$
E:\calB^<_n\big(P^{(2)}\big)
\to \Tors_2\big(\calA^{<,c}_{2n+1}(P)\big)
$
(see Proposition \ref{prop:Psi_psi}).
Then, as an application of the constructions and results of 
\S\ref{sec:surgery}--\S \ref{sec:Z^<}, we deduce the following:\\

\noindent
\textbf{Theorem A.} (See Theorem \ref{th:Delta_zeta} and Corollary \ref{cor:lower_bound}.)
\emph{Let $n>0$.
There is an explicit  degree $1$ homomorphism 
$\Delta$ which fits into the following commutative diagram
of $\Sp(H)$-modules}:
$$
\xymatrix{
 \calB^<_n\big(P^{(2)}\big) 
 \ar[rrd]^-\Delta \ar[d]_-\Psi 
 &&  \\
\Tors_2\Big({\displaystyle\frac{Y_{2n+1} \calIC}{Y_{2n+2}}}\Big) \ar[rr]^-{\zeta^<_{2n+2}} 
&&  \Tors_2\big( \calA^{<,c}_{2n+2}(H) \otimes \Q/\Z\big).
}    
$$
\emph{In particular, the $\Sp(H)$-module 
$\Tors_2\big({Y_{2n+1}\calIC}/{Y_{2n+2}}\big)$
surjects onto the $\Sp(H)$-submodule  
$\Delta\big(\calB^<_n\big(P^{(2)}\big)\big)$
of $\calA^{<,c}_{2n+2}(H) \otimes \Q/\Z$.}\\[-0.2cm]

\noindent
The problem of determining 
the structure of the $\Sp(H)$-module $\Delta\big(\calB^<_n\big(P^{(2)}\big)\big)$
 seems to be very challenging in general.
However, the subsequent sections address this issue 
by either establishing lower bounds on its size,
or, by restricting attention to small values of~$n$.

Thus, we consider in \S \ref{sec:loop} the \emph{loop filtration}
on modules of Jacobi diagrams (and its variant, the
\emph{$s$-loop filtration}, for Jacobi diagrams with beads).
We compute the homomorphism $\Gr^L\! \Delta$ 
induced by $\Delta$ at the graded level
(Proposition \ref{prop:Gr_Delta}), 
and obtain the following lower bound.\\[-0.2cm]

\noindent
\textbf{Theorem B.}
(See Propositions \ref{prop:gr_lower_bound}, \ref{prop:t_0}
and \ref{prop:t_1}.)
\emph{For every $n>0$, we have} 
$$
\left\vert 
\Tors_2\Big(\frac{Y_{2n+1}\calIC}{Y_{2n+2}}\Big)\right\vert \;\geq\; 
\sum_{k\geq 0} t_k(n)
$$
\emph{where $t_k(n)$ is the cardinality
of the image of $\Gr^L\! \Delta$  
on $\Gr^L_k \calB^<_n\big(P^{(2)}\big)$,
and is  computed in loop degrees $k=0$ and $k=1$.}\\[-0.2cm]

In \S \ref{sec:symLMO}, we study the symmetrized LMO homomorphism
$Z:\calIC \to \calA(H)\otimes \Q$
together with its  ``modulo~$1$'' reduction
$
\zeta: \calIC \to {\big(\calA(H)\otimes \Q\big)}/{\calR(H)}.
$
Here, $\calR(H)$ denotes a suitable lattice in $\calA(H)\otimes \Q$, for which we provide several equivalent descriptions. 
In particular, we define inductively 
a family $\Tr^{(r)}$ (for $r\geq 0$)  
of   \emph{trace homomorphisms}:
defined on torsion-free quotients of $\calA^c(H)$,
these maps  may be regarded as ``multi-loop'' analogues 
of Morita’s trace operators \cite{Mor93b}, 
taking values in dyadic torsion abelian groups. 
We  use those operators $\Tr^{(r)}$ 
to characterize the $\Sp(H)$-submodule
$\calR(H)$ in $\calA(H)\otimes \Q$ (see Proposition \ref{prop:P2}),
and this characterization  has two types of applications. 
First, they yield congruence relations among the loop-degree components of $Z$ (Proposition~\ref{prop:ZZ}). 
In low loop degrees, we thus obtain congruence relations 
involving two classical invariants: 
the \emph{Johnson homomorphisms $\tau_k$} (for $k\geq 1$)
and the \emph{non-commutative Reidemeister torsion~$\widetilde{\alpha}$} 
introduced in \cite{NSS23} 
(see Corollaries~\ref{cor:0_1} and~\ref{cor:0_1_2}). 
Second, returning to the odd-degree $2$-torsion subgroup $\Tors_2\big(Y_{2n+1}\calIC/Y_{2n+2}\big)$, 
we construct from the loop-degree components of $Z$ a sequence of homomorphisms $R_{2n+2}^{(r)}$ (for $r\geq 0$) 
which compute~$\zeta_{2n+2}^<$ on the associated graded 
 of the loop filtration  (Proposition~\ref{prop:R_2n+2}). 
For small values of $r$, the homomorphisms $R_{2n+2}^{(r)}$
can be written explicitly 
in relation with the above-mentioned classical invariants
$\tau_{2n+2}$ and $\widetilde{\alpha}_{2n+2}$
(see Corollaries~\ref{cor:R'_2n+2} and~\ref{cor:R''_2n+2}).

The low-degree cases provide testing grounds for the general framework developed in the paper. As recalled above from the literature, 
the Torelli Lie algebra 
and the Lie algebra of homology cylinders 
are already thoroughly understood in degree one. 
Nevertheless, we apply in \S \ref{sec:deg_1}
the  methods of the preceding sections 
for the study of $\Tors_2\big(Y_{2n+1}\calIC/Y_{2n+2}\big)$ 
by considering the special case $n:=0$. This yields 
the following description  of $\Tors\big(\calIC/Y_{2}\big)$
in terms of $P^{(2)}=H_1(\hbox{U}(\Sigma);\Z/2\Z)$ 
and $H^{(2)}:=H_1(\Sigma;\Z/2\Z)$:\\[-0.2cm]

\noindent
\textbf{Theorem C.}
(See Theorem~\ref{th:Tors(IC/Y_2)}.)
{\it There is an explicit degree $1$ homomorphism $\Delta$ fitting into the following commutative diagram of $\Sp(H)$-modules}:
$$
\xymatrix{
 S^2 P^{(2)}/H^{(2)} \ar[d]_-\Psi
 \ar[rrd]^-\Delta  &&
 \\
\Tors\Big({\displaystyle\frac{ \calIC}{Y_{2}}}\Big) 
\ar[rr]_-{\zeta^<_{2}}  
&&   \calA^{<,c}_{2}(H) \otimes \Q/\Z,
}
$$
{\it Moreover, $\Psi$ is an isomorphism, 
while $\zeta_2^<$ and $\Delta$ are injective.}\\[-0.2cm]

\noindent
Thus, the LMO homomorphism induces an isomorphism of $\Sp(H)$-modules 
between $\Tors(\calIC/Y_2)$ and a copy of $S^2 P^{(2)}/H^{(2)}$ 
inside the module of Jacobi diagrams $\calA^{<,c}_{2}(H)\otimes \Q/\Z$. 
The ``classical'' description of $\Tors(\calIC/Y_2)$ 
involves the \emph{Birman--Craggs homomorphism~$\beta$},
with values in the space of quadratic 
functions on the space of spin structures of $\Sigma$ \cite{MM03}:
this can be recovered 
from Theorem C (see Proposition~\ref{prop:CvsQ}). 
From this perspective, the maps $R_2^{(r)}$ for $r\in\{0,1,2\}$ 
introduced in \S \ref{sec:symLMO} correspond precisely 
to the successive approximations of $\beta$ 
defined by formal differentiation (see Remark \ref{rem:R_BC}). 
Finally, a description 
of the full $\Sp(H)$-module $\calIC/Y_2$ 
is obtained by additionally considering the degree $1$ part of the LMO homomorphism, or equivalently, in the ``classical'' framework, the first Johnson homomorphism $\tau_1$ (see Proposition \ref{prop:deg_1}).

In \S \ref{sec:deg_3}, we develop an analogous approach for the degree-three components of the Torelli Lie algebra and of the Lie algebra of homology cylinders.\\[-0.2cm]

\noindent
\textbf{Theorem D.} (See Theorem~\ref{th:Tors(Y_3IC/Y_4)}.)
\emph{We have the commutative diagram of $\Sp(H)$-modules}
$$
\xymatrix{
 \big(H^{(2)} \otimes \Lambda^2 P^{(2)}\big)\big/\Lambda^3 P^{(2)}  \ar[d]_-\Psi \ar[drr]^-\Delta
 &&
 \\
\Tors\Big({\displaystyle\frac{ Y_3\calIC}{Y_{4}}}\Big)
\ar[rr]_-{\zeta^<_{4}} 
&&   \calA^{<,c}_{4}(H) \otimes \Q/\Z 
}
$$
\emph{in which $\Psi$ is an isomorphism, 
and $\zeta_4^<,\Delta$ are injective.}\\[-0.2cm]

\noindent
Consequently, for $n:=1$, the torsion subgroup of $Y_{2n+1}\calI/Y_{2n+2}$
is purely $2$-torsion 
and coincides with the $\Sp(H)$-submodule $\Delta\big(\calB^<_n\big(P^{(2)}\big)\big)$ 
of $\calA^{<,c}_{2n+2}(H) \otimes \Q/\Z$ singled out by Theorem A. 
Moreover, a description of the full $\Sp(H)$-module $Y_3\calIC/Y_4$ 
is obtained by incorporating the degree $3$ component of 
the LMO homomorphism (Corollary \ref{cor:Y_3(IC)/Y_4}). 
By means of  the maps $R_4^{(r)}$ for $r\in\{0,1\}$, 
introduced in \S \ref{sec:symLMO}, 
we further connect this description with classical invariants (Propositions \ref{prop:tf3} and \ref{prop:tors3}). 
As an application, we derive a criterion 
for the triviality of the $Y_4$-equivalence relation in $\calIC$:
this criterion involves the variation of the Casson invariant,
(a portion of)  the action of $\calIC$ on $\pi/\Gamma_6\pi$, 
and (a portion of) the non-commutative Reidemeister torsion 
truncated to degrees $\leq 4$ (see Corollary~\ref{cor:Y_4}). 
We note that the isomorphism class 
of the underlying abelian group $Y_3\calIC/Y_4$ 
was already determined in \cite{NSS22a} 
using the ``raw'' LMO homomorphism \eqref{eq:raw_LMO}; 
however, the $\Sp(H)$-module structure and the connection 
with classical invariants for the torsion subgroup 
were not addressed there.

The article concludes with an application to the Torelli group. Using the above description of $\Tors\big({\displaystyle{ Y_3\calIC}/{Y_{4}}}\big)$, we establish the following statement.\\[-0.2cm]

\noindent
\textbf{Theorem E.} (See Theorem \ref{th:onto}.)
{\it Assume that $g\geq 4$.
The mapping cylinder construction induces, in degree $3$, a surjective homomorphism}
$$
\operatorname{Gr}_3 \mathbf{c}: 
\Tors(\Gamma_3 \calI /\Gamma_4 \calI)
\longrightarrow \Tors(Y_3 \calIC / Y_4).
$$
{\it In particular, the $\Sp(H)$-module $\Tors(\Gamma_3 \calI /\Gamma_4 \calI)$ surjects onto ${\displaystyle
\frac{H^{(2)} \otimes \Lambda^2 P^{(2)}}{\Lambda^3 P^{(2)}}}$.}\\

\noindent
We do not address here the injectivity of $\operatorname{Gr}_3 \mathbf{c}$; nevertheless, we provide several equivalent reformulations (Proposition \ref{prop:injectivity}). 
By analogy with the degree $1$ situation, 
a further natural problem in degree $3$ would be
to relate the homomorphisms~$R_4^{(r)}$ 
to the variations of an appropriate reduction 
of the degree $4$ part of the LMO invariant 
of integral homology $3$-spheres or, equivalently, 
of the second Ohtsuki invariant 
(in the same way the Birman--Craggs homomorphism 
is defined by variations of the Rokhlin invariant).\\

\noindent
\textbf{Acknowledgment.} 
The authors are grateful to Takuya Sakasai
for suggesting the strategy of proof for Lemma \ref{lem:injectivity_tau_3}.
The first author was supported by the FNRS grant 1.B.176.24F.
The second author's institute receives support 
from the EIPHI Graduate School (ANR-17-EURE-0002).
The work of the third author was partly supported by JSPS KAKENHI Grant Numbers JP22K03298, JP26K06789, and JP26H01994.
\hfill $\blacksquare$\\

\noindent
\textbf{Use of AI.} 
While preparing this manuscript, the authors used Overleaf’s ``AI writing tools'' to rephrase certain portions of the text and enhance clarity. 
(In addition, ChatGPT was employed to assist in revising some TikZ figures.) 
Following each of these steps, 
the authors carefully reviewed and edited the text (and figures) and therefore assume full intellectual responsibility for the final manuscript.
\hfill $\blacksquare$\\

\noindent
\textbf{Conventions.} 
The equivalence class in a quotient set $S/\!\!\sim$
represented by an $s\in S$ is denoted by $\{s\}$.
The cyclic group of order $r\geq 1$ is denoted by $\Z_r:= \Z/r\Z$.

If not specified, the ground ring for linear algebra 
and homological algebra is $\Z$. 
Thus, a ``module'' means a $\Z$-module 
(i.e., an abelian group),
and (co)homology groups are taken  with coefficients in $\Z$. 

The torsion submodule of a module $M$ 
is denoted by $\Tors(M)$ and, for $r\geq 1$, 
its $r$-torsion part is denoted by $\Tors_r(M)$.
We also denote by $\bqm M \eqm$ the natural image of $M$ 
in $M \otimes \Q$,
which is canonically isomorphic to  $M/\Tors(M)$;
thus, we have $M\otimes \Q/\Z \simeq (M\otimes \Q)/\!\bqm M \eqm$.
\hfill $\blacksquare$

\section{The surgery map}
\label{sec:surgery}

In this section, we review the surgery map
which approximates the Lie algebra
of homology cylinders in terms of Jacobi diagrams.
We also extend the definition of this map
on a ``larger'' algebra.

\subsection{The $Y_k$-equivalence relations}
\label{subsec:claspers}

The so-called ``clasper calculus'' 
has been introduced independently 
by Goussarov \cite{Gou00} and Habiro \cite{Hab00}:
we follow here the terminology and conventions of \cite{Hab00}.
A brief overview specialized to the setting of homology cylinders 
can  be found  in \cite[\S 5]{HM12}.
Let us simply recall here a few definitions
for the reader's convenience.

A \emph{graph clasper} $G$ 
in a homology cylinder  $M$ is a compact surface
 in the interior of $M$, 
which is decomposed into 
\emph{leaves} (diffeomorphic to $D^1\times S^1$), 
\emph{edges} (diffeomorphic to $D^1\times D^1$) 
and \emph{nodes} (diffeomorphic to $D^2$):
edges (viewed as $1$-handles) 
connect leaves and nodes, in such a way
that every leaf is incident to a single edge,
and every node (viewed as a $0$-handle) 
is incident to three half-edges.

The graph clasper $G$ is said to be \emph{allowable}
if every connected component of $G$ has at least one node.
A leaf of $G$ is \emph{special}
if it is $(-1)$-framed and unknotted in~$M$.

The \emph{shape} of a graph clasper $G$
is the unitrivalent abstract graph 
that one obtains by deleting
the leaves and by collapsing edges and nodes
(viewed as handles) to their cores.
The \emph{(internal) degree} of~$G$ is defined 
as the number of nodes, i.e.,
the number of trivalent vertices of its shape.

\begin{example}
A \emph{$Y$-graph} is a $Y$-shaped graph clasper, 
and an \emph{$H$-graph} is an $H$-shaped graph clasper.
For instance, 
the bottom figure of Example \ref{ex:Jacobi_to_clasper}
shows a graph clasper consisting of a $Y$-graph 
(with one special leaf) and an $H$-graph. 
\hfill $\blacksquare$
\end{example}

A graph clasper $G$ in $M$ carries instructions for \emph{surgery}:
the result of the surgery 
is denoted by $M_G$ and, 
if $G$ is allowable,
then $M_G$ is also a homology cylinder.
Conversely, any homology cylinder can be obtained from the
unit cylinder $U=\Sigma\times [-1,+1]$
by surgery along an allowable graph clasper.

The \emph{$Y_k$-equivalence} is the relation on $\calIC$
generated by surgeries along connected graph claspers of degree $k\geq 1$.
This  definition is
especially suited to analysis of the $Y_k$-equivalence via clasper calculus:
see \cite[\S 5]{HM12} and references therein.
It turns out that
two homology cylinders $M$ and $M'$ are $Y_k$-equivalent if,
and only if,
there exists a compact, oriented, connected surface
with one boundary component $S\subset M$ and an element
$s\in \Gamma_k \calI(S)$ such that $M'$ is diffeomorphic to
the homology cylinder obtained from $M$ by cutting along $S$
and regluing with $s$.

\subsection{The algebra of homology cylinders}

As recalled in the introduction, 
the $Y$-filtration \eqref{eq:Y-filtration} of $\calIC$ 
is defined by considering, for all $k\geq 1$,
the submonoid $Y_k \calIC$
of homology cylinders that are $Y_k$-equivalent
to the unit cylinder $U$.
Clasper calculus shows that 
the quotient monoid $\mathcal{IC}/Y_{k+1}$ 
is in fact a group for every $k \ge 1$,  
and that the inclusion  of commutators
\begin{equation} \label{eq:centrality}
  [\, Y_i \mathcal{IC}/Y_{k+1},\, Y_j \mathcal{IC}/Y_{k+1} \,] 
\subset Y_{i+j} \mathcal{IC}/Y_{k+1}  
\end{equation}
holds in this group for all $i, j \ge 1$.
Then, the associated graded 
$$
\operatorname{Gr}^Y\!\!\mathcal{IC} = 
\bigoplus_{k \ge 1} \frac{Y_k \mathcal{IC}}{Y_{k+1}}
$$
of the  $Y$-filtration has the structure of a graded Lie algebra over $\Z$, and we call it the \emph{Lie algebra of homology cylinders}.
Furthermore, the natural  action 
of the mapping class group $\calM$ on $\calIC$ 
preserves the $Y$-filtration;
as a consequence of \eqref{eq:centrality},
this action factorizes through $\Sp(H)\simeq \calM/\calI$ 
at the graded level:  so $\operatorname{Gr}^Y \mathcal{IC}$ is an $\Sp(H)$-module.

Similarly, we can consider the monoid algebra $\Z[\calIC]$
endowed with the \emph{Goussarov--Habiro filtration}
$$
\Z[\calIC] = F_0 \Z[\calIC] 
\supset F_1 \Z[\calIC] \supset F_2 \Z[\calIC] 
\supset \cdots 
$$
which is defined as follows:
for every $k\geq 0$,
the submodule $F_k \Z[\calIC]$ is generated 
by the linear combinations of the form
$$
[M;G]
:= \sum_{G'\subset G}
(-1)^{\vert G\setminus G'\vert} \cdot  M_{G'}
$$
for all $M\in \calIC$
and any allowable graph clasper $G$ in $M$ of degree $k$,
where the sum is over all ways $G'$
of selecting some connected components  of $G$
and $\vert G\setminus G'\vert$ 
is the number of remaining components.
Recall that an invariant $I:\calIC \to A$ with values
in an abelian group $A$ is said to be of
\emph{finite-type} of \emph{degree} at most $d$ 
if $\Z[I]$ vanishes on $F_{d+1} \Z[\calIC]$.

Clearly, we have
$F_i \Z[\calIC] \cdot F_j \Z[\calIC]  \subset F_{i+j} \Z[\calIC]$
for all $i,j \geq 0$: hence the associated graded
$$
\Gr^F\!\Z[\calIC] := 
\bigoplus_{k\geq 0} \frac{F_k \Z[\calIC]}{F_{k+1} \Z[\calIC]}
$$
of the Goussarov--Habiro filtration
is a graded (associative, unital) algebra,
called the \emph{algebra of homology cylinders}.
Furthermore, the canonical action 
of $\calM$ on $\Z[\calIC]$ 
preserves the Goussarov--Habiro filtration, and it
factorizes through $\Sp(H)\simeq \calM/\calI$ at the graded level
(by the same arguments of clasper calculus that lead 
to \eqref{eq:centrality}):  
so $\Gr^F \Z[\calIC]$ is an $\Sp(H)$-module.

The relationship between the 
algebra of homology cylinders and its Lie version
is given by the following result, 
which is essentially contained in  \cite{Mas07} 
and based on a result of \cite{Hab00}.

\begin{theorem} \label{th:iota}
The  map $\iota: \Gr^Y\!\!\calIC \to \Gr^F \Z[\calIC]$ 
defined by 
$$
Y_k \calIC/Y_{k+1} \ni  \{M\}   \ \stackrel{\iota}{\longmapsto} 
\ \{M-U\} \in {F_k \Z[\calIC]}/{F_{k+1} \Z[\calIC]}
$$
induces an $\Sp(H)$-equivariant algebra homomorphism
$$
\operatorname{U}(\iota): 
\operatorname{U}\big(\Gr^Y\!\!\calIC\big) 
\longrightarrow \Gr^F \Z[\calIC]
$$
on the universal enveloping algebra.
Furthermore, $\operatorname{U}(\iota)$ is surjective.
\end{theorem}

\begin{proof}
The map $\iota$ is linear since we have
\begin{align*}
\{MN-U\} &=\{(M-U)(N-U)+ (M-U)+(N-U)\} \\&= \{M-U\} +\{N-U\} \quad \in \calIC/Y_{k+1}
\end{align*}
for any $M,N\in Y_k \calIC$, 
and  $\iota$ is a Lie map since we have
\begin{eqnarray*}
&& \{MNM'N'-U\} \\
&=& \{(M-U)(N-U)M'N' - (N-U)(M-U)M'N' + NMM'N' -U\} \\
&=& \{(M-U)(N-U) - (N-U)(M-U) + N(MM'-U)N' + (NN'-U)  \} \\
&=& \{(M-U)(N-U) - (N-U)(M-U) \} 
\quad \in \calIC/Y_{i+j+1}
\end{eqnarray*}
for any  $M,M'\in Y_i \calIC$ and for any  $N,N'\in Y_j \calIC$ 
such that $M\circ M'\sim_{Y_{i+j+1}} U$ 
and $N\circ N'\sim_{Y_{i+j+1}} U$.
Therefore, $\iota$ induces an algebra homomorphism
$\operatorname{U}(\iota): 
\operatorname{U}\big(\Gr^Y\!\!\calIC\big) 
\to \Gr^F \Z[\calIC]$.
Clearly, $\iota$ is $\Sp(H)$-equivariant,
and so is $\operatorname{U}(\iota)$.

The map $\operatorname{U}(\iota)$ is surjective, 
because the Goussarov--Habiro filtration on the monoid algebra
$\Z[\calIC]$ is induced by the $Y$-filtration 
on the monoid $\calIC$, meaning that
\begin{equation} \label{eq:Habiro}
F_k \Z[\calIC] = \sum_{l=1}^k
\sum_{\substack{k_1,\dots, k_l\geq 1 \\ k_1+\cdots + k_l=k}}
(Y_{k_1}\calIC -U) \cdots (Y_{k_l}\calIC -U).
\end{equation}
This identity is due to Habiro
in the analogous case of string-links 
\cite[Prop$.$ 6.10]{Hab00}:
see  \cite[Lemma 3.8]{Mas07}.
\end{proof}

\begin{remark}
The injectivity of the map $\operatorname{U}(\iota)$
does not seem to be known.
According to \cite[\S 5.2]{Mas07}
and as an analogue of Quillen's result \cite{Qui68},
$\operatorname{U}(\iota \otimes \mathbb{F})$ is an isomorphism 
if coefficients are taken in a field $\mathbb{F}$ instead of $\Z$
(provided the universal enveloping algebra functor
$\operatorname{U}(-)$ is understood in the restricted sense
when the characteristic of $\mathbb{F}$ is positive).

The injectivity of the map $\iota$ itself
is equivalent to requiring that, for every $k\geq 1$,
we have $Y_k \calIC = \calIC \cap \big(U + F_k \Z[\calIC]\big)$
in $\Z[\calIC]$. This is another way to state 
the \emph{Goussarov--Habiro conjecture}:
for every $d\geq 0$, two homology cylinders 
are not distinguished by finite-type invariants 
of degree at most $d$ if, and only if,
they are $Y_{d+1}$-equivalent. \hfill $\blacksquare$
\end{remark}

Recall that the monoid algebra $\Z[G]$ of a monoid $G$
has a  structure of cocommutative bialgebra,
whose coproduct $\Delta$ (resp$.$ counit $\varepsilon$)
is given by $\Delta(g)= g \otimes g $
(resp$.$ $\varepsilon(g)=1$) for every $g\in G$.
Besides, the universal enveloping algebra 
$\operatorname{U}(\mathfrak{g})$ of a Lie algebra $\mathfrak{g}$
also has 
a structure of cocommutative bialgebra,
whose coproduct $\Delta$ (resp$.$ counit $\varepsilon$)
is given by $\Delta(g)= g \otimes 1 + 1 \otimes g $
(resp$.$ $\varepsilon(g)=0$) for every $g$
in the image of $\mathfrak{g}$
in  $\operatorname{U}(\mathfrak{g})$.

\begin{corollary}
The bialgebra structure on $\Z[\calIC]$ induces 
a bialgebra structure on $\Gr^F \Z[\calIC]$.
Furthermore, 
$\operatorname{U}(\iota): 
\operatorname{U}\big(\Gr^Y\!\!\calIC\big) 
\to \Gr^F \Z[\calIC]$
preserves the bialgebra structures.
\end{corollary}

\begin{proof}
The counit $\varepsilon: \Z[\calIC]\to \Z$
is filtration-preserving, since we clearly have 
$\varepsilon(F_k \Z[\calIC])=0  $ for every $k\geq 1$.
Moreover, the coproduct 
$\Delta:  \Z[\calIC] \to \Z[\calIC] \otimes \Z[\calIC] $
is filtration-preserving too, since we have
$$
\Delta(F_k \Z[\calIC]) \subset 
\bigoplus_{i+j=k} F_i \Z[\calIC] \otimes F_j \Z[\calIC]
$$
 for any  $k\geq 1$, as can be deduced
 from \eqref{eq:Habiro} and the following fact:
\begin{equation} \label{eq:Delta}
\Delta(M-U)= (M-U)\otimes U  + U  \otimes (M-U)
+ (M-U) \otimes (M-U)
\hbox{ for any $M\in \calIC$.}
\end{equation}
Thus, $\Gr^F \Z[\calIC]$ inherits from $\Z[\calIC]$
a bialgebra structure.

The identity \eqref{eq:Delta}  shows
that the image of $\iota: \Gr^Y\!\!\calIC \to \Gr^F \Z[\calIC]$ 
consists of primitive elements.
Since the algebra $\operatorname{U}\big(\Gr^Y\!\!\calIC\big)$
is generated by the image of $\Gr^Y\!\!\calIC$,
we deduce that the algebra map $\operatorname{U}(\iota)$
is also a coalgebra map.
\end{proof}

\subsection{The abelian group $P$}
\label{subsec:P}

We consider the oriented frame bundle
$p:\operatorname{F}(U) \to U$ of the unit cylinder $U= \Sigma \times [-1,+1]$, 
which is a principal $\operatorname{GL}_+(3;\mathbb{R})$-bundle.
As a refinement of the surface homology group
$$
H =H_1(\Sigma) \simeq H_1(U),
$$
we shall need
$$
P:= H_1(\operatorname{F}(U))
$$
in the next subsections. Denoting by  $s$
the generator of 
$H_1(\operatorname{GL}_+(3;\mathbb{R})) \simeq \Z_2$,
we deduce the following 
from the Serre exact sequence in homology:
\begin{equation} \label{eq:Serre}
\xymatrix{
0\ar[r] &\Z_2 \, s \ar[r]&P\ar[r]^-{p_*}& H\ar[r]&0.
}    
\end{equation}

The next lemma reviews a few well-known facts 
about the abelian group $P$.

\begin{lemma} \label{lem:P}
(a) There is a group isomorphism
\[
t :  
\big(\{\text{framed oriented knots in }U\} /\!\sim \,, \sharp \big)
\longrightarrow P
\]
which adds to any framed oriented knot an extra $(+1)$-twist
and takes the homology class of its lift to the frame bundle;
here,
two oriented framed knots $K_1$ and $K_2$ 
are $\sim$-\emph{equivalent}
if $(-K_1)\sqcup K_2$ bounds a compact oriented surface
with respect to which the framings of $K_1$ and $K_2$
differ by an even integer.\\[0.2cm]
(b) $P$ is isomorphic to $\Z_2 \rtimes_\omega H$, 
which denotes the set $\Z_2 \times H$ with the internal law
$(z,h)\cdot(z',h') = \big(z+z'+(\omega(h,h')\!\!\mod 2), h+h'\big)$;
specifically, the group extension~\eqref{eq:Serre}
has a setwise section $\sigma:H \to P$ whose associated $2$-cocycle 
is the mod $2$ reduction of $\omega:H \times H \to \Z$.\\[0.2cm]
(c) The canonical action of $\calM$ on $P$ factorizes through $\Sp(H)$.
\end{lemma}

\begin{proof}
(a) is proved in  \cite[Lemma 2.7.b]{MM03}.
To prove (b), recall that there is a notion
of \emph{framing number} $\operatorname{Fr}(K)$
for framed knots $K$ in $U$
(see, for instance, \cite[\S B]{MM13}).
Then, according to (a), 
there is a map $\sigma:H \to P$ 
defined by 
\begin{equation} \label{eq:sigma}
\sigma(h) := t(\{K_h\})+ \operatorname{Fr}(K_h)\, s
\end{equation}
where $K_h$ is any framed oriented knot in $U$
whose homology class is $h\in H_1(U)$.
Clearly, $\sigma$ is a right section of \eqref{eq:Serre}. 
(Its mod 2 reduction
$H_1(\Sigma;\Z_2) \to H_1(\operatorname{F}(U);\Z_2)$ 
coincides with Johnson's construction \cite{Joh80}.)
Moreover, 
the associated {$2$-cocycle} $H \times H \to \Z_2\, s\,$ of $\sigma$
maps any $(h,h')$ to
\begin{eqnarray*}
&& \sigma(h+h')-\sigma(h)-\sigma(h') \\
&=&     t(\{K_h \sharp K_{h'} \})
+ \operatorname{Fr}(K_h\sharp K_{h'})\, s 
- t(\{K_h\}) - \operatorname{Fr}(K_h)\, s 
-t(\{K_{h'}\})- \operatorname{Fr}(K_{h'})\, s \\
&=& 2\operatorname{Lk}(K_h,K_{h'}) s \ = \ \omega(h,h')\, s.
\end{eqnarray*}  
(Here we used \cite[Lemma B.4]{MM13}
and the notion of {linking number} 
$\operatorname{Lk}(-,-)$ in $U=\Sigma \times [-1,+1]$.)
Consequently, there is an  isomorphism 
$\rho:P \to \Z_2\, s \rtimes_\omega H$
defined by 
$\rho(x) := \big(x-\sigma(p_*(x)),p_*(x)\big)$.

We now prove (c). Let $f\in \calM$
inducing $f_*:P \to P$.
For any $h\in H$, we have
\begin{eqnarray}
\label{eq:action_sigma} f_* \sigma(h) &= & 
f_* t(\{K_h\})+ \operatorname{Fr}(K_h)\, s\\ 
\notag &=& t(\{(f\times \id)(K_h)\})
+ \operatorname{Fr}((f\times \id) (K_h))\, s    \\
\notag &=&  t(\{K_{f_*(h)}\})+ \operatorname{Fr}(K_{f_*(h)}) 
\ = \ \sigma f_*(h).
\end{eqnarray}
It follows that, for any $x\in P$, we have 
\begin{eqnarray*}
\rho(f_*(x)) &=& \big(f_*(x)-\sigma(p_*f_*(x)),p_*f_*(x)\big)\\
&=& \big(f_*(x)-\sigma(f_*p_*(x)),f_*p_*(x)\big)\\
&=& \big(f_*(x-\sigma(p_*(x))),f_*p_*(x)\big) 
\ = \ (\id \times f_*) (\rho(x)).
\end{eqnarray*}
This computation shows that the action of $f$ on $P$
corresponds to $\id \times f_*$ on $\Z_2 s \times H$ 
via the isomorphism $\rho$: 
hence, the action of $f$ on $P$ only depends on $f_* \in \Sp(H)$.
\end{proof}

We call the quantity $t(\{K\})\in P$, defined in Lemma \ref{lem:P}.(a),
the \emph{type} of a framed oriented knot $K$ in $U$. 
For instance, for a simple oriented closed curve $\delta \subset \Sigma$ 
representing $d\in H$,
we deduce from \eqref{eq:sigma} that the oriented knot $D:=\delta \times \{0\}$ 
(with the surface framing) has type $t(\{D\})=\sigma(d)$.

Besides, observe from Lemma \ref{lem:P}.(b) that 
the abelian group $P$ has the primary decomposition
\begin{equation} \label{eq:primary}
P=
\Z_2\, s \oplus 
\Z\, \sigma(a_1) \oplus \cdots \oplus \Z\, \sigma(a_g) \oplus
\Z\, \sigma(b_1) \oplus \cdots \oplus \Z\, \sigma(b_g),
\end{equation}
where $(a_1,\dots,a_g,b_1,\dots,b_g)$
is any symplectic basis of $H$
(for instance, arising from the choice of 
a system of curves as in \eqref{eq:surface}).

\begin{remark} \label{rem:spin}
There is yet another description of $P$,
which involves the set    
$\Omega:= \operatorname{Spin}(\Sigma)$ 
of spin structures on  $\Sigma$.
This set can be identified to $\operatorname{Spin}(U)$, hence
$$
\Omega = 
\big\{ y\in H^1(\operatorname{F}(U);\Z_2): 
\langle y, s \rangle \neq 0 \big\}.
$$
Recall that $\Omega$
is a $\Z_2$-affine space over the $\Z_2$-vector space
$H^1(\Sigma;\Z_2)$.
Then, $P$ is isomorphic to the fibered product
$H \times_{H^{(2)}} \operatorname{Aff}(\Omega)$
where $\operatorname{Aff}(\Omega)$
is the space of affine functions 
on  $\Omega$ and $H^{(2)} := H_1(\Sigma;\Z_2)$;
specifically, according to \cite[Lemma~2.7.a]{MM03},
we have the pull-back diagram
\[
\begin{tikzcd}
P \arrow[r,"e"] \arrow[d,"p_*"'] 
\arrow[dr, phantom, "\lrcorner", very near start] &  \operatorname{Aff}(\Omega) 
\arrow[d, "\kappa"] \\
H \arrow[r, "\bmod\! 2"'] & H^{(2)}
\end{tikzcd}
\]
where $e$ maps any $x\in P$ 
to the Kronecker evaluation 
${\langle -,x\rangle\colon H^1(F(U);\Z_2)\to\Z_2}$ at $x$,
and $\kappa$ maps any element of
$\operatorname{Aff}(\Omega)$
to the corresponding $\Z_2$-linear form in
$\Hom(H^1(\Sigma;\Z_2),\Z_2) \simeq H^{(2)}$.
\hfill $\blacksquare$
\end{remark}

\subsection{The module $\calA^<(P)$ of Jacobi diagrams}
\label{subsec:Jacobi_diagrams}

We now define a module of Jacobi diagrams, which simultaneously generalizes two earlier constructions 
and involves the group~$P$ of the previous subsection.

Recall that a \emph{Jacobi diagram} is a finite
graph whose vertices are either univalent (and called \emph{external}), 
or trivalent (and called \emph{internal}).
Furthermore, every internal vertex is \emph{oriented}
in that its incident half-edges are cyclically ordered.
(When Jacobi diagrams are drawn,
the convention is that internal vertices 
are oriented counterclockwise.)
A Jacobi diagram is \emph{allowable}
if each connected component has at least one internal vertex.
The \emph{(internal) degree} of a {Jacobi diagram}
is the number of its internal vertices.

\begin{example}\hphantom{.}\\
\begin{minipage}{0.23\textwidth}
\begin{tikzpicture}[thick,color=black,
baseline=-0.5em, x=1em, y=1em]
\footnotesize
\draw (-1,1) -- (0,0) -- (1,1);
\draw (0,0) -- (0,-1.5);
\draw [fill = black] (0,0) circle (.3ex);
\end{tikzpicture} \qquad 
\begin{tikzpicture}[thick,color=black,baseline=-0.2em, x=1.5em, y=1.5em]
\draw (-1/2,0) -- (1/2,0);
\draw (1/2,0)--({sqrt(3)/2},1);
\draw (-1/2,0)--({-sqrt(3)/2},1);
\draw (1/2,0)--({sqrt(3)/2},-1);
\draw (-1/2,0)--({-sqrt(3)/2},-1);
\draw [fill = black] (1/2,0) circle (.3ex);
\draw [fill = black] (-1/2,0) circle (.3ex);
\end{tikzpicture}
\end{minipage}
\begin{minipage}{0.75\textwidth}
A \emph{$Y$-diagram} is a Jacobi diagram  
which is $Y$-shaped (it has degree $1$),
and a \emph{$H$-diagram} is a Jacobi diagram  
which is $H$-shaped (it has degree $2$).
\hfill $\blacksquare$
\end{minipage}
\end{example}

Then,
we define the following module of Jacobi diagrams:
\[
\calA^<(P) := \frac{\Z\left\{\parbox{2.6in}{\centering 
\text{allowable Jacobi diagrams  whose external} 
\text{ vertices are $P$-colored and totally ordered}}\right\}}{\text{AS, IHX, multilinearity, special, STU-like, slide}}
\]
Here the relations are 
\vspace{4mm}

{\centering
AS:
$\begin{tikzpicture}[thick,color=black,baseline=-0.5em, x=1.5em, y=1.5em]
\draw (0,-0.2)--(0,-0.8);
\draw (0,-0.2) to[out=150, in=210, out looseness=2](0.866,1/2);
\fill [white] (0,0.25) circle [radius=0.12];
\draw (0,-0.2) to[out=30, in=-30, out looseness=2](-0.866,1/2);
\begin{scope}[dotted]
\draw ({sqrt(3)/2},1/2) -- ({sqrt(3)/2*1.5},1.5/2);
\draw (-{sqrt(3)/2},1/2) -- (-{sqrt(3)/2*1.5},1.5/2);
\draw (0,-0.8) -- (0,-1.3);
\end{scope}
\draw [fill = black] (0,-0.2) circle (.3ex);
\end{tikzpicture}
= -\begin{tikzpicture}[thick,color=black,baseline=-0.5em, x=1.5em, y=1.5em]
\draw (0,-0.1)--(0,-0.8);
\draw (0,-0.1)--({sqrt(3)/2},1/2);
\draw (0,-0.1)--({-sqrt(3)/2},1/2);
\begin{scope}[dotted]
\draw ({sqrt(3)/2},1/2) -- ({sqrt(3)/2*1.5},1.5/2);
\draw (-{sqrt(3)/2},1/2) -- (-{sqrt(3)/2*1.5},1.5/2);
\draw (0,-0.8) -- (0,-1.3);
\end{scope}
\draw [fill = black] (0,-0.1) circle (.3ex);
\end{tikzpicture}$, \quad 
IHX: $\begin{tikzpicture}[thick,color=black,baseline=-0.2em, x=1.5em, y=1.5em]
\draw (0,-1/2) -- (0,1/2);
\draw [fill = black] (0,1/2) circle (.3ex);
\draw [fill = black] (0,-1/2) circle (.3ex);
\begin{scope}[shift={(0,1/2)}]
\draw (0,0)--(30:1);
\draw (0,0)--(150:1);
\draw [dotted] (30:1)--(30:3/2);
\draw [dotted] (150:1)--(150:3/2);
\end{scope}
\begin{scope}[shift={(0,-1/2)}]
\draw (0,0)--(-150:1);
\draw (0,0)--(-30:1);
\draw [dotted] (-150:1)--(-150:3/2);
\draw [dotted] (-30:1)--(-30:3/2);
\end{scope}
\end{tikzpicture}
\quad - \quad 
\begin{tikzpicture}[thick,color=black,baseline=-0.2em, x=1.5em, y=1.5em]
\draw (-1/2,0) -- (1/2,0);
\draw [fill = black] (1/2,0) circle (.3ex);
\draw [fill = black] (-1/2,0) circle (.3ex);
\begin{scope}[shift={(1/2,0)}]
\draw (0,0)--(60:1);
\draw (0,0)--(-60:1);
\draw [dotted] (60:1)--(60:3/2);
\draw [dotted] (-60:1)--(-60:3/2);
\end{scope}
\begin{scope}[shift={(-1/2,0)}]
\draw (0,0)--(120:1);
\draw (0,0)--(-120:1);
\draw [dotted] (120:1)--(120:3/2);
\draw [dotted] (-120:1)--(-120:3/2);
\end{scope}
\end{tikzpicture}
\quad + \quad 
\begin{tikzpicture}[thick,color=black,baseline=-0.2em, x=1.5em, y=1.5em]
\draw (-1/2,0) -- (1/2,0);
\draw [fill = black] (1/2,0) circle (.3ex);
\draw [fill = black] (-1/2,0) circle (.3ex);
\begin{scope}[shift={(1/2,0)}]
\draw (0,0)--(145:1.7);
\draw (0,0)--(-60:1);
\draw[dotted] (145:2.2)--(145:1.7);
\draw[dotted] (-60:1.5)--(-60:1);
\end{scope}
\draw [white] (0,0.35) circle (.4ex);
\begin{scope}[shift={(-1/2,0)}]
\draw (0,0)--(35:1.7);
\draw (0,0)--(240:1);
\draw[dotted] (35:2.2)--(35:1.7);
\draw[dotted] (240:1.5)--(240:1);
\end{scope}
\end{tikzpicture}=0$,\\
 \vspace{2mm}multilinearity:
$\begin{tikzpicture}[thick,color=black,baseline=-0.9em, x=2em, y=2em]
\footnotesize
\draw[dotted] (0,0.5) -- (0,0);
\draw (0,0) -- (0,-1) node[below=1pt]{$x+y$};
\end{tikzpicture}
=
\begin{tikzpicture}[thick,color=black,baseline=-0.9em, x=2em, y=2em]
\footnotesize
\draw[dotted] (0,0.5) -- (0,0);
\draw (0,0) -- (0,-1) node[below=1pt]{$x$};
\end{tikzpicture}
+
\begin{tikzpicture}[thick,color=black,baseline=-0.9em, x=2em, y=2em]
\footnotesize
\draw[dotted] (0,0.5) -- (0,0);
\draw (0,0) -- (0,-1) node[below=1pt]{$y$};
\end{tikzpicture}$, \quad 
special:\!\!\!
$\rooteddottritree{\cdots\! < s <\! \cdots}\!\!=0$.\\[0.2cm]
STU-like: 
$\begin{tikzpicture}[thick,color=black,baseline=-0.9em, x=2em, y=2em]
\footnotesize
\draw[dotted] (0,0.5) -- (0,0);
\draw (0,0) -- (0,-1);
\draw[dotted] (1,0.5) -- (1,0);
\draw (1,0) -- (1,-1);
\node[below=1pt] at (0.5,-1) {$\cdots<x<y<\cdots$};
\end{tikzpicture}
-
\begin{tikzpicture}[thick,color=black,baseline=-0.9em, x=2em, y=2em]
\footnotesize
\draw[dotted] (0,0.5) -- (0,0);
\draw[dotted] (1,0.5) -- (1,0);
\draw (0,0) to [out=-90,in=90] (1,-1);
\fill [white] (0.5,-0.5) circle [radius=0.2];
\draw (1,0) to [out=-90,in=90] (0,-1);
\node[below=1pt] at (0.5,-1) {$\cdots<y<x<\cdots$};
\end{tikzpicture}
=\omega(p_*(x),p_*(y))
\begin{tikzpicture}[thick,color=black,baseline=-0.9em, x=2em, y=2em]
\draw[dotted] (0,0.5) -- (0,0);
\draw[dotted] (1,0.5) -- (1,0);
\draw (0,0) -- (0,-0.5);
\draw (1,0) -- (1,-0.5);
\draw (0,-0.5) to [out=-90,in=-90] (1,-0.5) ;
\node[below=1pt] at (-0.3,-0.8) {$\cdots$};
\node[below=1pt] at (1.4,-0.8) {$\cdots$};
\end{tikzpicture}$,\\
\qquad \qquad slide: \!\!\!
$
\begin{tikzpicture}[thick,color=black,baseline=-0.5em, x=1.5em, y=1.5em]
\footnotesize
\draw (-0.5,-1) -- (0,0) -- (0.5,-1);
\draw (0,0) -- (0,0.7);
\draw [dotted] (0,1.3) -- (0,0.7);
\draw [fill = black] (0,0) circle (.3ex);
\node[below=1pt] at (0,-1) {$\cdots\!<x<s<\!\cdots$};
\end{tikzpicture}
\! = \! 
\begin{tikzpicture}[thick,color=black,baseline=-0.5em, x=1.5em, y=1.5em]
\footnotesize
\draw (-0.5,-1) -- (0,0) -- (0.5,-1);
\draw (0,0) -- (0,0.7);
\draw [dotted] (0,1.3) -- (0,0.7);
\draw [fill = black] (0,0) circle (.3ex);
\node[below=1pt] at (0,-1) {$\cdots\!<x<x<\!\cdots$};
\end{tikzpicture}
\! = \! \begin{tikzpicture}[thick,color=black,baseline=-0.5em, x=1.5em, y=1.5em]
\footnotesize
\draw (-0.5,-1) -- (0,0) -- (0.5,-1);
\draw (0,0) -- (0,0.7);
\draw [dotted] (0,1.3) -- (0,0.7);
\draw [fill = black] (0,0) circle (.3ex);
\node[below=1pt] at (0,-1) {$\cdots\!<s<x<\!\cdots$};
\end{tikzpicture}$\!\!,}\\[0.2cm]

\noindent
where $x,y\in P$ are arbitrary
and $s\in P$ is the special element defined by the fiber.

\begin{remark} \label{rem:few}
The defining relations of $\calA^<(P)$ have the following consequences:
\begin{itemize}
\item We have the extra relation
stating that any Jacobi diagram containing a looped edge 
must be zero:
$$
\mathrm{self\text{-}loop}: \qquad
\begin{tikzpicture}[thick,color=black,baseline=-0.2em, x=2em, y=2em]
\draw[dotted] (-2.5,0) -- (-2,0);
\draw (-2,0) -- (-1,0);
\draw (-0.5,0) circle [radius=0.5];
\draw [fill = black] (-1,0) circle (.3ex);
\end{tikzpicture}
=0
$$
For connected components with at least two internal vertices, this follows from ``IHX'', 
while for the remaining components it is a consequence of 
``STU-like''  with ``AS'', ``multilinearity'', and ``slide''.
\item The special and multilinear relations imply
\rooteddottritree{\dots < x < \dots} $=$  \rooteddottritree{\dots < y < \dots} whenever $p_*(x) = p_*(y)\in H$.
So, the relevance of $P$-coloring
with respect to $H$-coloring 
is only justified by the possibility
of $Y$-diagram components.
\item It follows from the STU-like relation that
the order of external vertices is not relevant for $s$-colored vertices.
\hfill $\blacksquare$
\end{itemize}
\end{remark}

Since all the relations are homogeneous
with respect to the  (internal) 
degree of Jacobi diagrams, 
the module $\calA^<(P)$ is graded:
for $n\geq 1$, the degree $n$ submodule 
is denoted by  $\calA^<_n(P)$.
Besides, let $\calA^{<,c}_n(P)$ be the submodule of $\calA^<_n(P)$
generated by connected Jacobi diagrams.

As explained in the next proposition,
the module of Jacobi diagrams $\calA^<(P)$ should be regarded
as the synthesis of two earlier constructions:
\begin{itemize}
\item[(i)] As in  \cite[\S 2]{MM03}, 
observing that ``AS'' is implied by the other relations in degree $1$, we define
$$
\calA_1(P) := \frac{\Z\left\{\parbox{3.1in}{\centering 
\text{$Y$-diagrams  whose external vertices are $P$-colored}}\right\}}{\text{multilinearity, slide}}.
$$
\item[(ii)] As in \cite[\S 8.5]{Hab00} 
and \cite[\S 1.8]{HM09}, we define
$$
\qquad \calA^<(H) := \frac{\Z\left\{\parbox{2.7in}{\centering 
\text{allowable Jacobi diagrams  whose external }
\text{vertices are $H$-colored  and totally ordered}}\right\}}{\text{AS, IHX, multilinearity, STU-like, slide}},
$$
where the first three relations are the same as before, 
and the last two relations are written as follows:\\[-0.2cm]

STU-like: \!\!\!\!\!\!
$\begin{tikzpicture}[thick,color=black,baseline=-0.9em, x=2em, y=2em]
\footnotesize
\draw[dotted] (0,0.5) -- (0,0);
\draw (0,0) -- (0,-1);
\draw[dotted] (1,0.5) -- (1,0);
\draw (1,0) -- (1,-1);
\node[below=1pt] at (0.5,-1) {$\cdots<x<y<\cdots$};
\end{tikzpicture}
\!\!\!\!-\!\!\!\!
\begin{tikzpicture}[thick,color=black,baseline=-0.9em, x=2em, y=2em]
\footnotesize
\draw[dotted] (0,0.5) -- (0,0);
\draw[dotted] (1,0.5) -- (1,0);
\draw (0,0) to [out=-90,in=90] (1,-1);
\fill [white] (0.5,-0.5) circle [radius=0.2];
\draw (1,0) to [out=-90,in=90] (0,-1);
\node[below=1pt] at (0.5,-1) {$\cdots<y<x<\cdots$};
\end{tikzpicture}
\!\!\!\!=\omega(x,y) \!\!
\begin{tikzpicture}[thick,color=black,baseline=-0.9em, x=2em, y=2em]
\draw[dotted] (0,0.5) -- (0,0);
\draw[dotted] (1,0.5) -- (1,0);
\draw (0,0) -- (0,-0.5);
\draw (1,0) -- (1,-0.5);
\draw (0,-0.5) to [out=-90,in=-90] (1,-0.5) ;
\node[below=1pt] at (-0.3,-0.8) {$\cdots$};
\node[below=1pt] at (1.4,-0.8) {$\cdots$};
\end{tikzpicture}\!\!$ 
($x,y\in  H$),

\qquad \qquad \qquad slide: \!\!\!
$
\begin{tikzpicture}[thick,color=black,baseline=-0.5em, x=1.5em, y=1.5em]
\footnotesize
\draw (-0.5,-1) -- (0,0) -- (0.5,-1);
\draw (0,0) -- (0,0.7);
\draw [dotted] (0,1.3) -- (0,0.7);
\draw [fill = black] (0,0) circle (.3ex);
\node[below=1pt] at (0,-1) {$\cdots\!<x<x<\!\cdots$};
\end{tikzpicture}
\! = \! 0$ \quad ($x\in  H$),

\noindent
Note that  ``slide'' can be removed 
if one takes coefficients in~$\Q$, or, 
if one restricts oneself to $\calA_n^{<,c}(H)$ with $n\geq 2$.
Besides,  $\calA^<_1(H)$ can be identified to $ \Lambda^3 H$
through the map 
$\begin{tikzpicture}[thick,color=black,baseline=-1.0em, x=1em, y=1em]
\footnotesize
\draw (0,0) -- (1.5,-1);
\draw (0,0) -- (0,-1);
\draw (0,0) -- (-1.5,-1);
\draw [fill = black] (0,0) circle (.3ex);
\node[below=1pt] at (0,-1) {$x < y < z $};
\end{tikzpicture}\longmapsto x\wedge y \wedge z$.
\end{itemize}       

\begin{proposition} \label{prop:before}
We have the following $\Sp(H)$-isomorphisms:
\begin{itemize}
\item[(i)] \!\!\! The map $\calA_1^<(P) = \calA_1^{<,c}(P)
\!\to\!\calA_1(P)$ that
forgets the order of external vertices is an isomorphism.
\item[(ii)]  The  map $p_*:\calA^<(P)\to \calA^<(H)$
that applies $p_*:P \to H$ to all external vertices 
induces an isomorphism 
$\calA^{<}(P)\otimes \Q\simeq  \calA^{<}(H)\otimes \Q$,
as well as an isomorphism 
$\calA^{<,c}_n(P)\simeq  \calA^{<,c}_n(H)$ for $n\geq 2$.
\end{itemize}
\end{proposition}

\begin{proof}
Statements (i) and (ii) are easily checked from the definitions, and the above observations 
on the defining relations of $\calA^<(P)$. 
\end{proof}
    
Let us also mention
some further structures on the module $\calA^<(P)$.
On the one hand, it is an $\Sp(H)$-module
thanks to the canonical action of $\Sp(H)$ on $P$
(see Lemma \ref{lem:P}.(c)).
On the other hand,  $\calA^<(P)$
is a cocommutative bialgebra:
the product is given by 
the ordered disjoint union $\stackrel{<}{\sqcup}$ of diagrams,
the unit is the empty diagram $\varnothing$, 
the coproduct $\Delta$ is defined on any diagram $D$ by 
$$
\Delta(D) = \sum_{D=D' \sqcup D''} D' \otimes D''
$$
where the sum is over all the decompositions of $D$ into two families of connected components $D', D''$
(their external vertices being ordered  as they were in $D$), 
and the counit $\varepsilon$  is defined by
$\varepsilon(D) = \delta_{D,\varnothing}$.

Finally, we draw the reader’s attention to some upcoming abuses of notation that will be used in handling Jacobi diagrams in $\mathcal{A}^<(P)$ and in related modules:
\begin{equation} \label{eq:lifts}
\parbox{0.8\textwidth}{
\it Given $x \in P^{(2)}:=H_1(\operatorname{F}(U);\Z_2)$ 
(resp$.$, $x\in H^{(2)}=H_1(\Sigma;\Z_2)$), 
the choice of a lift to $P$ (resp$.$, to $H$) 
is denoted by the same symbol~$x$
if the choice of this lift is irrelevant in our formula.
}
\end{equation}
\begin{equation} \label{eq:colors}
\parbox{0.8\textwidth}{
\it We color the external vertices of a Jacobi diagram
indifferently  with $P$ or~$H$ if the diagram is connected of degree at least 2,
or, if we are working with rational coefficients
(see Proposition \ref{prop:before}).
}
\end{equation}

\subsection{The surgery map $\psi$}
\label{subsec:surgery}

By merging the constructions of \cite{MM03}
(which deal with the degree-one case)
and those of \cite{Hab00,HM09,HM12}
(which deal with the higher-degree connected case),
we shall now define a \emph{surgery map}
$$
\psi :\calA^<(P) \longrightarrow \operatorname{Gr}^{F}\Z[\mathcal{IC}].
$$

For this, we describe a recipe to associate to
every Jacobi diagram $D$ among the generators of $\calA^<(P)$,
a graph clasper \( C(D) \) in $U=\Sigma \times [-1,+1]$
which is a  “topological realization”  of  \( D \).
(In particular, $C(D)$ has the same shape as $D$.)
This recipe  is divided into two steps:
\begin{enumerate}
  \item[\bf Step 1:] We thicken \( D \) to a compact oriented surface (using its vertex-orientation), so that vertices of $D$ are thickened to disks, and edges of $D$ to bands. For each disk produced in this way from an external vertex, we cut a smaller disk in the interior, so as to produce an oriented compact surface \( S(D) \), decomposed into disks, bands and annuli. The orientation of \( S(D) \) also induces an orientation on the core of each annulus 
  (namely, the orientation parallel to the outer oriented boundary of the annulus).
  \item[\bf Step 2:] We embed \( S(D) \) into the interior of \( U \) in such a way that each annulus $A$ of \( S(D) \), viewed as a framed oriented knot $A\subset U$, has the type $t(A)\in P$
   specified by the color of the corresponding external vertex of \( D \).
  Furthermore, the   annuli of $S(D)$ should be embedded 
  in disjoint ``horizontal slices'' of \( U \), 
  in such a way that  
  their ``vertical height'' along \( [-1,+1] \) 
  respect the total ordering 
  of the corresponding external vertices of~\( D \). 
  The result is a graph clasper \( C(D) \) in \( U \).
\end{enumerate}
Note that the surface $S(D)$ produced in ``Step 1''
is univocal, and it does not depend on the coloring 
nor on the ordering of the external vertices of $D$.
On the contrary, the graph clasper $G(D)$
produced in ``Step 2'' depends on the coloring and the ordering,
and it is equivocal in the sense that one has to make
choices for the embedding of $S(D)$ into $U$.

\begin{example} \label{ex:Jacobi_to_clasper}
The following is an example of a Jacobi diagram $D$ 
of degree $3$,
together with the corresponding ``abstract'' surface $S(D)$
and a graph clasper $C(D) \subset U$ realizing $D$:
$$
D=
\begin{tikzpicture}[thick,color=black,baseline=-1.0em, x=1.5em, y=1.5em]
\footnotesize
\begin{scope}[myred]
\coordinate (R) at (4.5,1);
\draw (R) -- (-1.2,-1);
\draw (R) -- (0.2,-1);
\draw (R) -- (10.5,-1);
\draw [fill = myred] (R) circle (.3ex);
\end{scope}
\begin{scope}[myblue]
\coordinate (Q) at (6.7,0);
\coordinate (P) at (4,0);
\draw (P) -- (2.1,-1);
\draw (P) -- (4,-1);
\draw (Q) -- (6.7,-1);
\draw (Q) -- (8.5,-1);
\draw (P) -- (Q);
\draw [fill = myblue] (P) circle (.3ex);
\draw [fill = myblue] (Q) circle (.3ex);
\node[black, below=1pt] at (5,-1) {$s<\sigma(b_2)<\sigma(a_4)<\sigma(a_3)<\sigma(a_2)<\sigma(a_1)<\sigma(b_1)$};
\end{scope}
\end{tikzpicture}
$$
$$
S(D)=
\begin{tikzpicture}[myred, very thick, scale=0.45, line width=1pt, baseline={(0,-0.3)}]
  \draw (0,0) circle (0.6);
\foreach \angle in {190,270,350}{
  \begin{scope}[myred, rotate=\angle]  
    \draw [very thick]
    (0.57,0.2) -- (1.1,0.2);
    \draw [very thick]
    (0.57,-0.2) -- (1.1,-0.2); 
  \draw (1.7,0) circle (0.6);
  \draw (1.7,0) circle (0.4);
  \end{scope}
}
\end{tikzpicture} \ 
\begin{tikzpicture}[myblue,very thick, scale=0.45, line width=1pt, baseline={(0,-0.3)}]
\draw (0,0) circle (0.6);
\draw (2,0) circle (0.6);
    \draw (0.57,0.2) -- (1.4,0.2);
    \draw (0.57,-0.2) -- (1.4,-0.2); 
\foreach \angle in {190,270}{
  \begin{scope}[rotate=\angle]  
    \draw (0.57,0.2) -- (1.1,0.2);
    \draw (0.57,-0.2) -- (1.1,-0.2); 
    \draw (1.7,0) circle (0.6);
    \draw (1.7,0) circle (0.4);
  \end{scope}
}
  \foreach \angle in {270,350}{
  \begin{scope}[shift={(2,0)},rotate=\angle]  
    \draw 
    (0.57,0.2) -- (1.1,0.2);
    \draw 
    (0.57,-0.2) -- (1.1,-0.2); 
    \draw (1.7,0) circle (0.6);
     \draw (1.7,0) circle (0.4);
  \end{scope}
}
\end{tikzpicture} \qquad
$$
$$
\clasper
$$
(Here we use the
elements $s,\sigma(a_i),\sigma(b_i)$ 
of a basis  \eqref{eq:primary} of $P$.)
\hfill $\blacksquare$
\end{example}

The following is a natural generalization
of \cite[Th$.$ 2.11]{MM03} 
(which deals with the degree $1$ case)
and \cite[Th$.$ 6.5]{HM12}
(which deals with the higher-degree connected case).

\begin{theorem} \label{th:psi}
Using the above recipe, we define a graded homomorphism
$\psi :\calA^<(P) \to 
\operatorname{Gr}^{F}\Z[\mathcal{IC}]$
by setting
$$
\psi(D) := \big\{ [U;C(D)] \big\} \in F_k \Z[\mathcal{IC}]/
F_{k+1}\Z[\mathcal{IC}]
$$
for any Jacobi diagram $D \in \calA^<(P)$ of degree $k$.
Furthermore, 
$\psi$ is surjective, $\Sp(H)$-equivariant
and it respects the bialgebra structures.
\end{theorem}

\begin{proof}
That $\psi$ is well-defined and surjective
is proved by combining the arguments of
\cite[\S 2.3]{MM03} 
as well as \cite[\S 2.2]{HM09} and \cite[\S 6.2]{HM12},
which are  based  on clasper calculus.
In particular, note that the special relation 
is implied by \cite[Lemma 4.9]{GGP01},
and the STU-like relation is verified
following  Remark \ref{rem:STU-like}  below.

The $\Sp(H)$-equivariance of $\psi$
is verified as follows.
Let $f\in \calM$ inducing $f_*\in \Sp(H)\simeq \calM/\calI$.
Using the canonical action of $\calM$ on 
$\calIC$, we obtain
$$
\psi(f_*\cdot D)=
\big\{ \big[U; (f\times [-1,+1])(C(D))\big] \big\} 
= \big\{ f \cdot  [U;C(D)]  \big\} 
= f_* \cdot \psi(D).
$$
That $\psi$ preserves the bialgebra structures
is easily checked 
from the definitions of the product and coproduct.
\end{proof}

\begin{example} \label{ex:genus_0}
In genus $g=0$ (i.e., for $\Sigma = D^2$), 
$P= \Z_2\, s$ is the order $2$ group generated by $s$.
Hence we have a canonical isomorphism of graded modules
\begin{equation} \label{eq:genus_0}
    \mathcal{A}(\varnothing) \otimes \Z[Y]/\langle 2Y \rangle
    \stackrel{\simeq}{\longrightarrow} \calA^<(P) 
\end{equation}
where $\mathcal{A}(\varnothing)$ denotes the graded module 
of purely-trivalent Jacobi diagrams (modulo ``AS'' and ``IHX''), 
and $\Z[Y]/\langle 2Y \rangle$ is the quotient of the polynomial
algebra by the ideal generated by $2Y$.
This isomorphism maps the generator $Y$ 
to the $Y$-diagram with  vertices colored by $s$.
Then Theorem \ref{th:psi} specializes 
to \cite[Th$.$ 4.13]{GGP01}. \hfill $\blacksquare$
\end{example}

There is also a variant of Theorem \ref{th:psi}
for the Lie version of the algebra of homology cylinders.
To state this, we consider the submodule $\calA^{<,c}(P)$ 
of $\calA^<(P)$ generated by connected Jacobi diagrams.
The commutator of the product in the algebra 
$\calA^{<}(P)$  restricts to a Lie bracket on $\calA^{<,c}(P)$:
specifically, for any two connected Jacobi diagrams $D$ and $E$,
we have
\begin{equation} \label{eq:bracket}
[D,E] = \sum_{u,v} \omega(u,v) \cdot D\cup_{u=v} E,
\end{equation}
where the sum is  over all external vertices $u$ and $v$ of $D$ and $E$, respectively,
we denote by $\omega(u,v)$ the evaluation of 
$\omega\circ(p_* \times p_*)$ 
on the pair of colors of $(u,v)$,
and  $D\cup_{u=v} E$ is obtained 
by gluing  $D$ and $E$  at $u=v$ and
by ordering the external vertices as follows:
first, the vertices of $E$ lower than $v$ 
(in their given order);
second, the vertices of $D$ other than $u$ 
(in their given order);
third, the vertices of $E$ upper than $v$
(in their given order).

\begin{theorem} \label{th:psi_connected}
We define a surjective $\Sp(H)$-equivariant graded
homomorphism
$\psi:\calA^{<,c}(P) \to \Gr^Y\!\!\calIC$ of Lie algebras by setting
$$
\psi(D) := \big\{ U_{C(D)} \big\} 
\in Y_k \calIC/Y_{k+1}
$$
for any connected Jacobi diagram 
$D \in \calA^{<,c}_k(P)$ of degree $k\geq 1$.
Furthermore, we have the following commutative diagram:
\begin{equation} \label{eq:psi_psi}
\xymatrix{
\calA^{<,c}(P) \ar@{^{(}->}[d] \ar[r]^-\psi & \Gr^Y\calIC \ar[d]^-\iota \\
\calA^{<}(P) \ar[r]^-\psi & \Gr^F \Z[\calIC]
}    
\end{equation}
\end{theorem}

\begin{proof}
The well-definedness and  surjectivity of $\psi$
is justified by clasper calculus,  
exactly as for the proof of Theorem \ref{th:psi},
appealing to \cite{MM03,HM09}. 
(Again, the most noteworthy relation to verify is ``STU-like’’, 
which follows from Remark \ref{rem:STU-like}  below.)
The $\Sp(H)$-equivariance is also justified 
as in the proof of Theorem \ref{th:psi},
and the commutativity of \eqref{eq:psi_psi} is obvious.

That $\psi$ is a Lie map is also well-known
in some analogous situations but, since the argument 
does not seem to be detailed anywhere, we include it here.
Let $D\in \calA_k^{<,c}(P)$ and $E\in \calA_l^{<,c}(P)$ be connected Jacobi diagrams:
then, $[\psi_k(D),\psi_l(E)]\in Y_{k+l}\calIC/Y_{k+l+1}$ 
is represented by 
$$
U_{C(D)\cup C(E) \cup C(D)' \cup G \cup C(E)' \cup H}
\in \calIC
$$
where the graph claspers $C(D), C(E), C(D)', G, C(E)', H$ are located
in disjoint horizontal slices of $U = \Sigma \times [-1,+1]$ in this order from bottom to top,
$C(D)'$ differs from $C(D)$  by a half-twist on an edge,
$G$ is a disjoint union of graph claspers of degrees $\geq k+1$
 such that $U_{C(D') \cup G}$ is an inverse of $U_{C(D)}$ 
in $\calIC/Y_{k+l+1}$, and $C(E)',H$ 
play similar roles with respect to $E$.
By clasper calculus, we can swap $C(E)$ and $C(D)'$
at the price of surgeries along the graph claspers that realize $[D,E]$ in \eqref{eq:bracket}:
since any of these graph claspers has degree $k+l$,
they can be displaced to the top slice 
up to $Y_{k+l+1}$-equivalence.
Hence we have
$$
[\psi_k(D),\psi_l(E)] = 
\big\{ U_{C(D)\cup  C(D)' \cup C(E)  \cup G \cup C(E)' \cup H} \big\} 
\circ \psi_{k+l}([D,E]) \in {\calIC}/{Y_{k+l+1}};
$$
since $U_{C(E)}$ and $U_{G}$ commute
in ${\calIC}/{Y_{k+l+1}}$ (for $E$ being of degree $l$ 
and $G$ being of degree $\geq k+1$),
we deduce that $[\psi_k(D),\psi_l(E)]= \psi_{k+l}([D,E])$.
\end{proof}

We now recall how the surgery map $\psi$
of Theorem \ref{th:psi_connected} relates
to an approximation of the graded Lie algebra 
associated to the lower central series 
of the Torelli group $\calI$.
The determination of 
the abelianized Torelli group
$\calI_{\operatorname{ab}}$
by Johnson \cite{Joh85} is formulated in \cite{MM03} 
as a group isomorphism $J: \calA_1(P) \to \calI_{\operatorname{ab}}$. 
This induces a surjective  Lie algebra homomorphism
$$
J:\Lie(\calA_1(P))\longrightarrow \Gr^\Gamma \calI
$$
from the Lie algebra freely generated 
by the module $\calA_1(P)$ to the associated graded 
of the lower central series of $\calI$.
Then, it follows from \cite[Th$.$ A]{FMS26} that
$J$ factorizes in every degree $k\geq 2$
through $p_*: \Lie_k(\calA_1(P)) \to \Lie_k(\Lambda^3 H)$
to a homomorphism
$$
J_k: \Lie_k\big(\Lambda^3 H\big) \longrightarrow 
\Gr^\Gamma_k \calI.
$$
Besides, by Proposition \ref{prop:before},
the map $\psi_k$ 
can be viewed as a homomorphism
$$
\psi_k: \calA^{<,c}_k(H)  \longrightarrow 
\Gr^Y_k \calIC.
$$

\begin{proposition} \label{prop:bracketing}
For every $k\geq 2$, we have the 
commutative diagram
$$
\xymatrix{
\Lie_k\big(\Lambda^3 H\big)\ar[d]_{B_k}  \ar@{->>}[r]^-{J_k} &
\Gamma_k \calI/\Gamma_{k+1} \calI \ar[d]^-{\Gr_k\!\mathbf{c}}\\
\calA_k^{<,c}(H) \ar@{->>}[r]_-{\psi_k} &
Y_k \calIC/ Y_{k+1}
}
$$
where the left vertical map is induced by the isomorphism
$\Lambda^3 H\simeq \calA_1^{<,c}(H)$
and the right vertical map
is given by the 
mapping cylinder construction \eqref{eq:mcc}.
\end{proposition}

\begin{proof}
We have the commutative triangle
$$
\xymatrix{
\calA_1(P) \ar[r]^-J \ar[rd]_-{\psi_1} & 
\calI_{\operatorname{ab}}\ar[d]^-{\Gr_1\!\mathbf{c}} \\
& \calIC/Y_2,
}
$$
which implies the commutative square of Lie algebras
$$
\xymatrix{
\Lie\big(\calA_1(P)\big)\ar[d]  \ar@{->>}[r]^-{J} &
\Gr^\Gamma \calI \ar[d]^-{\Gr\mathbf{c}}\\
\calA^{<,c}(P) \ar@{->>}[r]_-{\psi} &
\Gr^Y\!\!\calIC
}
$$
where the left vertical map is induced by the isomorphism
$\calA_1(P) \simeq \calA_1^{<,c}(P)$.
The proposition follows immediately.
\end{proof}

\subsection{Another description of $\calA^<(P)$}
\label{subsec:another_description}

We conclude this section with a simpler, 
\emph{but} non-intrinsic, description  of  $\calA^<(P)$. 
It depends on the choice of 
a symplectic basis 
$$S:=(a_1,\dots,a_g,b_1,\dots,b_g)$$ 
of $H$.
Let 
\begin{equation} \label{eq:A(S)}
\calA(S):=
 \frac{\Z\left\{\parbox{2.1in}{\centering 
\text{allowable Jacobi diagrams whose} 
\text{external vertices are $S$-colored}}\right\}}{\text{AS, IHX,  self-loop, swap}}
\end{equation}
where ``AS, ``IHX'' and ``self-loop'' are as before, 
and the last relation is  
$$
\text{swap: \quad}
\begin{tikzpicture}[thick,color=black,
baseline=-0.5em, x=1em, y=1em]
\footnotesize
\draw (-1,1) -- (0,0) -- (1,1);
\draw (0,0) -- (0,-1.5);
\draw [fill = black] (0,0) circle (.3ex);
\node[below=1pt] at (0,-1.5) {$y$};
\node[right=0.5pt] at (1,1) {$x$};
\node[left=0.5pt] at (-1,1) {$x$};
\end{tikzpicture}
=\begin{tikzpicture}[thick,color=black,
baseline=-0.5em, x=1em, y=1em]
\footnotesize
\draw (-1,1) -- (0,0) -- (1,1);
\draw (0,0) -- (0,-1.5);
\draw [fill = black] (0,0) circle (.3ex);
\node[below=1pt] at (0,-1.5) {$x$};
\node[right=0.5pt] at (1,1) {$y$};
\node[left=0.5pt] at (-1,1) {$y$};
\end{tikzpicture}
\quad \text{for all $x,y\in S$}.
$$

\begin{remark} \label{rem:before}
With rational coefficients, 
the self-loop and swap relations are consequences of ``AS''.
Hence, $\calA(S)\otimes \Q$ corresponds 
to the  space that is denoted by 
$\calA^Y(\lfloor g \rceil^+\cup \lfloor g \rceil^-) $ 
in \cite{CHM08}. \hfill $\blacksquare$
\end{remark}

Any Jacobi diagram $D$ with $S$-colored external vertices
defines a  Jacobi diagram $D^<$ 
with totally-ordered and $P$-colored vertices,
by transforming every color $u\in S\subset H$ 
to  $\sigma(u) \in P$ and declaring 
that all the vertices colored 
by $\{\sigma(a_1),\dots, \sigma(a_g)\}$
are lower than the vertices colored 
by $\{\sigma(b_1),\dots, \sigma(b_g)\}$.
Then, the following proposition generalizes \eqref{eq:genus_0}
to the genus $g>0$ case, 
and, given Remark \ref{rem:before},
it is a refinement of  \cite[Lemma~8.4]{CHM08}.

\begin{proposition} \label{prop:varphi}
The homomorphism 
$\varphi: \calA(S) \otimes \Z[Y]/\langle 2Y \rangle 
\to \calA^<(P)$ 
that is defined, for every Jacobi diagram $D$
and every $k\geq 0$, by 
\begin{equation} \label{eq:varphi}
\varphi(D \otimes Y^k) :=  (-1)^{\chi(D)} \cdot
D^< \ \underbrace{\stackrel{< }{\sqcup}  
\begin{tikzpicture}[thick,color=black,baseline=-1.0em, x=1em, y=1em]
\footnotesize
\draw (0,0) -- (1,-1);
\draw (0,0) -- (0,-1);
\draw (0,0) -- (-1,-1);
\draw [fill = black] (0,0) circle (.3ex);
\node[below=1pt] at (0,-1) {$s < s < s$};
\end{tikzpicture}
\stackrel{< }{\sqcup}   \cdots \stackrel{< }{\sqcup}  
\begin{tikzpicture}[thick,color=black,baseline=-1.0em, x=1em, y=1em]
\footnotesize
\draw (0,0) -- (1,-1);
\draw (0,0) -- (0,-1);
\draw (0,0) -- (-1,-1);
\draw [fill = black] (0,0) circle (.3ex);
\node[below=1pt] at (0,-1) {$s < s < s$};
\end{tikzpicture}}_{\hbox{\scriptsize $k$ times}}
\end{equation}
is an isomorphism of graded abelian groups.
\end{proposition}
    
\begin{proof} 
The right-hand side of \eqref{eq:varphi}
is well-defined in $\calA^<(P)$
as a consequence of the STU-like relation.
Hence we get a group homomorphism 
$$ 
{\Z\left\{\parbox{2in}{\centering 
\text{allowable Jacobi diagrams} 
\text{with $S$-colored external vertices}}\right\}} \otimes \Z[Y]/\langle 2Y \rangle
\longrightarrow  \calA^<(P).
$$
Clearly, this maps  ``AS'',  ``IHX'', ``self-loop'' and ``swap''
to  ``AS'', ``IHX'', ``self-loop'' and ``slide'', respectively.
So, we get  a map 
$\varphi: \calA(S) \otimes \Z[Y]/\langle 2Y \rangle \to \calA^<(P)$.

It is easily checked from the STU-like, slide and special relations 
in  $\calA^<(P)$ that the map $\varphi$ is surjective.
To prove the injectivity, we shall construct 
a homomorphism $\rho: \calA^<(P) \to \calA(S) \otimes \Z[Y]/\langle 2Y \rangle $
such that $\rho  \circ \varphi=\id$.
Since $P$ has the primary decomposition \eqref{eq:primary},
the abelian group $\calA^<(P)$ can be identified to
$$
\calA^<(S \sqcup \{s\}) := \frac{\Z\left\{\parbox{2.8in}{\centering 
\text{allowable Jacobi diagrams}  \text{whose external vertices are $S\cup\{s\}$-colored} 
\text{and $S$-colored vertices are totally ordered}}\right\}}{\text{AS, IHX, special, STU-like, slide}}
$$
where ``AS'', ``IHX''  and ``special'' are as before, 
the STU-like relation involves only the $S$-colored vertices:\\[-0.1cm] 

\noindent
STU-like: 
$\begin{tikzpicture}[thick,color=black,baseline=-0.9em, x=2em, y=2em]
\footnotesize
\draw[dotted] (0,0.5) -- (0,0);
\draw (0,0) -- (0,-1);
\draw[dotted] (1,0.5) -- (1,0);
\draw (1,0) -- (1,-1);
\node[below=1pt] at (0.5,-1) {$\cdots<x<y<\cdots$};
\end{tikzpicture}
-
\begin{tikzpicture}[thick,color=black,baseline=-0.9em, x=2em, y=2em]
\footnotesize
\draw[dotted] (0,0.5) -- (0,0);
\draw[dotted] (1,0.5) -- (1,0);
\draw (0,0) to [out=-90,in=90] (1,-1);
\fill [white] (0.5,-0.5) circle [radius=0.2];
\draw (1,0) to [out=-90,in=90] (0,-1);
\node[below=1pt] at (0.5,-1) {$\cdots<y<x<\cdots$};
\end{tikzpicture}
=\omega(x,y)
\begin{tikzpicture}[thick,color=black,baseline=-0.9em, x=2em, y=2em]
\draw[dotted] (0,0.5) -- (0,0);
\draw[dotted] (1,0.5) -- (1,0);
\draw (0,0) -- (0,-0.5);
\draw (1,0) -- (1,-0.5);
\draw (0,-0.5) to [out=-90,in=-90] (1,-0.5) ;
\node[below=1pt] at (-0.3,-0.8) {$\cdots$};
\node[below=1pt] at (1.4,-0.8) {$\cdots$};
\end{tikzpicture}$ for $x,y\in S$,

\noindent
and the slide relation writes as follows:\\[0.2cm] 

\noindent
\qquad \qquad \qquad slide:
\begin{tikzpicture}[thick,color=black,baseline=-0.5em, x=1.5em, y=1.5em]
\footnotesize
\draw (0,-1) -- (0,0) -- (1,0);
\draw (0,0) -- (0,0.7);
\draw[dotted] (0,0.7) -- (0,1.3);
\draw [fill = black] (0,0) circle (.3ex);
\node[right=1pt] at (1,0) {$s$};
\node[below=1pt] at (0,-1) {$\cdots <x < \cdots$};
\end{tikzpicture}
\! = \! 
\begin{tikzpicture}[thick,color=black,baseline=-0.5em, x=1.5em, y=1.5em]
\footnotesize
\draw (-0.5,-1) -- (0,0) -- (0.5,-1);
\draw (0,0) -- (0,0.7);
\draw[dotted] (0,0.7) -- (0,1.3);
\draw [fill = black] (0,0) circle (.3ex);
\node[below=1pt] at (0,-1) {$\cdots\!<x<x<\!\cdots$};
\end{tikzpicture}
\! = \! \begin{tikzpicture}[thick,color=black,baseline=-0.5em, x=1.5em, y=1.5em]
\footnotesize
\draw (0,-1) -- (0,0) -- (-1,0);
\draw (0,0) -- (0,0.7);
\draw[dotted] (0,0.7) -- (0,1.3);
\draw [fill = black] (0,0) circle (.3ex);
\node[left=1pt] at (-1,0) {$s$};
\node[below=1pt] at (0,-1) {$\cdots <x < \cdots$};
\end{tikzpicture}  for $x\in S$.

\noindent
(Note that ``AS'' and ``slide'' imply that  any 
diagram with an $s$-colored vertex has order $2$ 
in $\calA^<(S \sqcup \{s\})$.)
Then, for every Jacobi diagram $E$
among the generators of $\calA^<(S \sqcup \{s\})$.
we define
\[
\rho(E):=(-1)^{\chi(E)} \cdot 
\left(\begin{array}{c}
\text{sum of all ways of connecting  } \\ 
\text{some $a_i$-colored vertices of $E'$} \\
\text{ to some lower $b_i$-colored vertices}
\end{array}\right)_{
\!\!\!\hbox{\& forget the order}} \otimes Y^e
\]
where $E'$ is the diagram without $s$-colored vertex
obtained from $E$ as follows:
every component of $E$ of degree $\geq 2$
with an $s$-colored vertex is set to $0$;
any $Y$-diagram component
with all three vertices colored by $s$ 
is erased from $E$,
and we denote by $e$ the number of such components 
originally in $E$;
every $Y$-diagram component  of $E$ 
with one or two $s$-colored vertices is transformed as follows:
$$
\begin{tikzpicture}[thick,color=black,baseline=-1.0em, x=1em, y=1em]
\footnotesize
\draw (0,0.5) -- (0,1.5);
\draw (0,0.5) -- (2,-1);
\draw (0,0.5) -- (-2,-1);
\draw [fill = black] (0,0.5) circle (.3ex);
\node[above=1pt] at (0,1.5) {$s$} ;
\node[below=1pt] at (0,-1) {$\cdots  < x < \cdots< y < \cdots $};
\end{tikzpicture} \leadsto
\begin{tikzpicture}[thick,color=black,baseline=-1.0em, x=1em, y=1em]
\footnotesize
\draw (0,0.5) -- (2.5,-1);
\draw (0,0.5) -- (-2.5,-1);
\draw (0,0.5) -- (-1,-1);
\draw [fill = black] (0,0.5) circle (.3ex);
\node[below=1pt] at (0,-1) {$\cdots  < x < x < \cdots< y < \cdots $};
\end{tikzpicture},
$$
$$
\begin{tikzpicture}[thick,color=black,baseline=-1.0em, x=1em, y=1em]
\footnotesize
\draw (0,1) -- (0,-0);
\draw (0,1) -- (2,-1);
\draw (0,1) -- (-2,-1);
\draw [fill = black] (0,1) circle (.3ex);
\node[below=1pt] at (0,-0) {$s$} ;
\node[below=1pt] at (0,-1) {$\cdots  < x < \cdots< y < \cdots $};
\end{tikzpicture} \leadsto
\begin{tikzpicture}[thick,color=black,baseline=-1.0em, x=1em, y=1em]
\footnotesize
\draw (0,0.5) -- (2.5,-1);
\draw (0,0.5) -- (-2.5,-1);
\draw (0,0.5) -- (1,-1);
\draw [fill = black] (0,0.5) circle (.3ex);
\node[below=1pt] at (0,-1){$\cdots  < x < \cdots< y < y<\cdots $};
\end{tikzpicture}
,
$$
$$
\begin{tikzpicture}[thick,color=black,baseline=-1.0em, x=1em, y=1em]
\footnotesize
\draw (0,0) -- (0,-1);
\draw (0,0) -- (1,1);
\draw (0,0) -- (-1,1);
\draw [fill = black] (0,0) circle (.3ex);
\node[left=1pt] at (-1,1) {$s$} ;
\node[right=1pt] at (1,1) {$s$} ;
\node[below=1pt] at (0,-1) {$\cdots  < x < \cdots$};
\end{tikzpicture}
\leadsto 
\begin{tikzpicture}[thick,color=black,baseline=-1.0em, x=1em, y=1em]
\footnotesize
\draw (0,0) -- (1.5,-1);
\draw (0,0) -- (0,-1);
\draw (0,0) -- (-1.5,-1);
\draw [fill = black] (0,0) circle (.3ex);
\node[below=1pt] at (0,-1) {$\cdots <x < x < x< \cdots $};
\end{tikzpicture}
$$
It is straightforward to verify 
that each of the  defining relations of 
$\calA^<(S \sqcup \{s\})$ is mapped by $\rho$ to zero;
hence we get a  map 
$\rho:\calA^<(S \sqcup \{s\}) \to \calA(S) \otimes \Z[Y]/\langle 2Y \rangle$. Since the computation of $\rho$ on $\varphi(D \otimes Y^k)$ does not involve any connection of vertices, it follows that $\rho  \circ \varphi=\id$.
\end{proof}

\begin{remark} \label{rem:bialgebra}
As in \cite[\S 8]{CHM08}, we define on $\calA(S)$ a bialgebra structure, with product $\star$ given by 
\begin{equation} \label{eq:star_product}
D \star E 
= \left(\begin{array}{c}
\text{sum of all ways of connecting
some $b_i$-colored vertices of $D$}\\
\text{to some $a_i$-colored vertices of $E$,
for all $i\in\{1,\dots,g\}$}
\end{array} \right).
\end{equation}
Combining this with the usual bialgebra structure on $\Z[Y]/\langle 2Y \rangle$,
we obtain  that the map
$\varphi: \calA(S) \otimes \Z[Y]/\langle 2Y \rangle
\to \calA^<(P)$
is a bialgebra map. 

Restricting to the submodules generated by connected diagrams, we obtain a Lie algebra isomorphism
$\varphi: \calA^c(S) \oplus \Z_2 Y \to \calA^{<,c}(P)$.
The Lie bracket on the source is given 
by the commutator of the product $\star$,
and is such that $Y$ is central.
\hfill $\blacksquare$
\end{remark}

\section{The LMO homomorphism and the ``delta'' maps}

\label{sec:Z^<}

After a very brief review of the LMO homomorphism \cite{CHM08,HM09,HM12},
and a study of symplectic expansions up to degree $2$ \cite{Mas12},
we define the map $\delta^<$
which is an alternative version of the map $\delta$
introduced in \cite{NSS22a} to compute 
the ``mod 1 reduction'' of the LMO homomorphism.

\subsection{The LMO homomorphism $Z^<$ and its mod $1$ reduction $\zeta^<$}

\label{subsec:LMO}

The LMO functor can be restricted to the setting of homology cylinders, thereby yielding a monoid homomorphism
$$
\widetilde{Z}^Y: 
\calIC \longrightarrow 
{\calA(S) \otimes \Q}
$$
called the \emph{LMO homomorphism}. (Strictly speaking, $\widetilde{Z}^Y$ actually takes values in the degree-completion of $\calA(S) \otimes \Q$, but we shall keep the same notation for the degree-completed modules.) We refer the reader to \cite{CHM08,HM09,HM12}, and recall that the construction of $\widetilde{Z}^Y$ involves two choices:
\begin{itemize}
\item[(i)] a Drinfeld associator, used to define the Kontsevich integral of tangles, on which the LMO functor is built via surgery presentations of cobordisms;
\item[(ii)] a system of curves
$(\alpha_1,\dots,\alpha_g,\beta_1,\dots,\beta_g)$
as in \eqref{eq:surface}, giving rise to the symplectic basis $S:=(a_1,\dots,a_g,b_1,\dots,b_g)$ of $H$.
\end{itemize}

\begin{remark}
The target $\calA(S) \otimes \Q$ of $\widetilde{Z}^Y$
is denoted by 
$\calA^Y(\lfloor g \rceil^+\cup \lfloor g \rceil^-) $  in \cite{CHM08}, 
where rational coefficients are taken by default.
\hfill $\blacksquare$
\end{remark}

To partly “control’’ how 
$\widetilde{Z}^Y$ depends on the above choices (i) and (ii),
we look at the following compositions:
$$
\xymatrix{
\calIC \ar[r]^-{\widetilde{Z}^Y} 
\ar@/_2pc/@{-->}[rr]_(.8){{Z^{<}}}
\ar@/_3pc/@{-->}[rrr]_-{\zeta^<\, :=\, (Z^<\!\!\!\!\mod 1)}
& \calA(S) \otimes \Q
\ar[r]^-{\varphi^\Q}_-\simeq
& \calA^<(H)\otimes \Q \ar@{->>}[r]
& \calA^<(H)\otimes \Q/\Z.
}
$$
In this diagram, $\varphi^\Q$ denotes the rational version 
of the map $\varphi$ from Proposition~\ref{prop:varphi};
more precisely, for any Jacobi diagram $D$ among the generators of $\calA(S)$, we have
$
\varphi^\Q(D) = (-1)^{\chi(D)} \cdot D^<,
$
where $D^< \in \calA^<(H)$ is obtained by imposing an order 
on the external vertices of $D$ such that 
every vertex colored by some $a_i$
lies below every vertex colored by some $b_j$.

\begin{remark}
This variant $Z^<$ of the LMO homomorphism 
is equal to the map that is simply denoted by ``LMO''
in \cite[(2.3)]{HM09}.
$\hfill$ $\blacksquare$
\end{remark}

For every $k\geq 1$, the degree $k$  part $\widetilde{Z}^Y_k$ 
of $\widetilde{Z}^Y$ is a finite-type invariant of degree $k$, 
and \cite[Cor$.$ 8.6]{CHM08} computes the induced homomorphism 
$
\widetilde{Z}^Y_k: F_k \Q[\calIC] / F_{k+1} \Q[\calIC] 
\to \calA_k(S) \otimes \Q
$
in terms of clasper surgery. As a result,
we obtain the commutative diagram
\begin{equation} \label{eq:psi_Z}
\xymatrix{
\calA_k^<(H) \otimes \Q \ar[r]^{\psi^\Q} \ar@{=}[rd]
& {\displaystyle \frac{F_k\Q[\calIC]}{F_{k+1}\Q[\calIC]} } 
\ar[d]^-{Z^<_k}  \\
& \calA_k^<(H) \otimes \Q.
}
\end{equation}
Here, the map $\psi^\Q$ is the rationalization of 
the map $\psi$ in Theorem~\ref{th:psi}; 
specifically,  we have
$
\psi^\Q(D) = \big\{ [U;C(D)] \big\} 
$
for any Jacobi diagram $D$ among the generators of $\calA^<(H)$,
where $C(D)$ is a graph clasper in $U$ ``realizing'' $D$
by the same  recipe  as in \S \ref{subsec:surgery}, 
except that we only have to care here about  the  homology classes 
of the leaves of $C(D)$ (instead of their types).
As a consequence,  
$Z_k^<: \Gr_k^F \Q[\calIC] \to \calA_k^<(H) \otimes \Q$  
does not depend on the above choices (i) and (ii);
moreover,  it is $\Sp(H)$-equivariant. 

More recently, it has been shown that $(\widetilde{Z}^Y_{k+1} \bmod 1)$ is itself a finite-type invariant of degree $k$, and \cite[Th$.$ 1.1]{NSS22a} describes the induced homomorphism
$$
\big(\widetilde{Z}^Y_{k+1}\!\!\bmod 1\big): 
{\displaystyle F_k \Z[\calIC] / F_{k+1} \Z[\calIC] }
\to \calA_{k+1}(S) \otimes \Q/\Z
$$
in terms of clasper surgery.
The purpose of the rest of this section is to give an alternative formulation of this result by explicitly computing the composition
\begin{equation} \label{eq:tbc}
\xymatrix{
\calA_k^<(P) \ar[r]^-{\psi}
& {\displaystyle \frac{F_k\Z[\calIC]}{F_{k+1}\Z[\calIC]} }
\ar[r]^-{\zeta^<_{k+1}}  
& \calA_{k+1}^<(H) \otimes \Q/\Z.
}
\end{equation}
In particular, our calculation will make explicit to what extent the map
$\zeta_{k+1}^<$  
depends on the above choices (i) and (ii), and will measure its failure to be $\Sp(H)$-equivariant.

\begin{remark} \label{rem:Q/Z}
Since $p_*: \calA^<(P) \to \calA^<(H)$ is surjective and, moreover, yields an isomorphism
$p_*: \calA^<(P)\otimes \Q \to \calA^<(H)\otimes \Q$
(see Proposition~\ref{prop:before}), it follows that it also induces an isomorphism
$p_*: \calA^<(P)\otimes \Q/\Z \to \calA^<(H)\otimes \Q/\Z$. \hfill $\blacksquare$
\end{remark}

\subsection{Symplectic expansions up to degree 2}

In order to compute the composition \eqref{eq:tbc},
we need the notion of 
``symplectic expansions'' and their ``truncations''.

Let $\pi:=\pi_1(\Sigma,\star)$ be the fundamental group
of $\Sigma$ based at $\star \in \partial \Sigma$.
Recall from \cite{Mas12} that  a \emph{symplectic expansion} 
is a monoid map $\theta \colon \pi \to \widehat{T}^\Q$,
with group-like values in the complete tensor algebra $\widehat{T}^\Q:=\widehat{T}(H^\Q)$
generated by  $H^\Q$,  satisfying:
\begin{itemize}
\item $\theta(x) = 1+[x] +(\deg \geq 2)$
for all $x\in \pi$ representing $[x]\in H$,
\item $\theta(\partial \Sigma)=\exp(-\omega)$
where $\omega \in \Lambda^2 H \subset H^{\otimes 2}$
is the intersection form.
\end{itemize} 
Sometimes, it is more convenient to work 
with the corresponding \emph{symplectic logansion} 
$\ell^\theta := \log \circ\, \theta : \pi \to \widehat{\mathfrak{L}}^\Q$,
where $\widehat{\mathfrak{L}}^\Q 
:= \widehat{\mathfrak{L}}(H^\Q)$ is the complete free Lie algebra generated by $H^\Q$.

The \emph{truncation} (to degrees at most $2$)
of a symplectic expansion $\theta$ 
is the   map 
$$
\theta_{\leq 2}:\pi \longrightarrow 
\widehat{T}^\Q/\widehat{T}^\Q_{\geq 3}\simeq 
\big(\Z \oplus H \oplus H^{\otimes 2}\big)\otimes \Q.
$$
Observe that it is determined 
by the map $\theta_2:\pi\to H^{\otimes 2}\otimes \Q$.
We say that $\theta$ is 
\emph{even} (up to degrees at most $2$)
if $\theta_2(\pi) \subset \frac{1}{2}  H^{\otimes 2}$
or, equivalently, 
if $\ell^\theta_2(\pi) \subset \frac{1}{2}  \mathfrak{L}_2(H)$.

\begin{lemma} \label{lem:expansion} 
(a) The set of truncated symplectic expansions is
a torsor over $\Lambda^3 H^\Q$.
Specifically, 
for any symplectic expansions $\theta$ and $\tilde \theta$, 
there is a unique
\begin{equation} \label{eq:symp_cond}
d\in \Ker\Big(\Hom_\Q(H^\Q,\mathfrak{L}^\Q_2) \simeq H^\Q \otimes_\Q  \mathfrak{L}_2^\Q
\xrightarrow{[-,-]}\mathfrak{L}_3^\Q\Big) \simeq \Lambda^3 H^\Q
\end{equation}
such that 
$\tilde \theta_{\leq 2}(x) 
= \theta_{\leq 2}(x) + d([x])
\ \hbox{for all x $\in \pi$}$.\\[0.2cm]
(b) For any symplectic expansion $\theta$, 
the map $\theta_2: \pi \to H^{\otimes 2} \otimes \Q$ 
induces a homomorphism 
$\vartheta: H \to H^{\otimes 2} \otimes \Q/\Z$.
Equivalently, $\ell^\theta_2: \pi 
\to \mathfrak{L}_2(H) \otimes \Q$ induces a quadratic map 
$\lambda: H \to \mathfrak{L}_2(H) \otimes \Q/\Z$.\\[0.2cm]
(c) The set of truncated even symplectic expansions is
a torsor over $\frac{1}{2} \Lambda^3H$.
Moreover, for any even symplectic expansion $\theta$, 
 we have
\begin{itemize}
\item[(i)]  $u' \otimes u'' = u'' \otimes u' 
\in H^{\otimes 2} \otimes \Q/\Z,$ 
for all $u \in H$, hence $\omega(u',u'')=0\in \Q/\Z$,
\item[(ii)]  
$\omega(u',v)\, u'' = \omega(v',u)\, v'' \in H\otimes \Q/\Z$, 
for all $u,v \in H$,
\item[(iii)]   $\omega(u',u)\, u''  =  u/2 \in H\otimes \Q/\Z$, for all $u \in H$,
\end{itemize}
where we use Sweedler's notation 
$\vartheta(h)= \sum_i  h'_i\otimes h''_i 
=: h'\otimes h''$ for all $h\in H$.
\end{lemma}

\begin{proof}
(a) That the set  of truncated symplectic expansions
is non-empty follows from the existence of symplectic expansions.
For any two symplectic expansions $\theta$ and $\tilde\theta$, 
there exists a unique automorphism $\psi$ of the complete Hopf algebra $\widehat{T}^\Q$,
which induces the identity at the graded level 
and fixes $\omega\in \mathfrak{L}_2(H)$,
such that $\theta = \psi \circ \tilde\theta$. 
Then, 
$\delta := \log \psi$ is a (co-)derivation 
of $\widehat{T}^\Q$,
which strictly increases degrees and vanishes on $\omega$.
Denoting by $d$ the restriction of $\delta$ to $H^\Q$, truncated to degrees  $\leq 2$, we obtain
\begin{eqnarray*}
\tilde\theta(x) 
&=& (\id+ \delta + \delta^2/2 + \cdots )(\theta(x))\\
&=& \theta(x) + \delta(\theta(x)) + (\deg\geq 3)
\ = \ \theta(x) + d([x]) + (\deg\geq 3)    
\end{eqnarray*}
for every $x\in \pi$.
That $d$ belongs to the kernel of the bracketing map written at \eqref{eq:symp_cond} follows 
from the  condition $\delta(\omega)=0$.
This kernel is well-known to be isomorphic to $\Lambda^3 H^\Q$
by the map
\begin{equation} \label{eq:expand}
\Lambda^3 H^\Q \longrightarrow 
H^\Q \otimes \mathfrak{L}_2(H^\Q), \quad
a\wedge b \wedge c \longmapsto 
a \otimes [b,c] + c \otimes [a,b] + b\otimes [c,a].
\end{equation}
(b) For all $x, y \in \pi$, we have 
\begin{align*}
\theta (xy) = \theta (x) \cdot \theta (y)
   & = \big(1 + [x] + \theta_2 (x) + (\deg \geq 3)\big) \cdot
   \big(1 + [y] + \theta_2 (y) + (\deg \geq 3) \big) \\
&= 1 + [x] + [y] + {[x][y] + \theta_2 (x) 
+ \theta_2 (y)} + (\deg \geq 3).
\end{align*}
which shows that
$\theta_2(xy) =  \theta_2 (x)  + \theta_2 (y)+ [x][y] $.
Hence, 
$\theta_2: \pi \to H^\Q \otimes_\Q H^\Q$ induces 
a group homomorphism
$\vartheta: H  \to (H^\Q \otimes_\Q H^\Q)/ (H \otimes H)$.

The previous computation also shows  that 
$\ell_2^\theta(xy) = \ell_2^\theta(x)+\ell_2^\theta(y) 
+\big[[x],[y]\big]/2$, for all $x,y \in \pi$.
Thus, 
$\ell^\theta_2: \pi \to \mathfrak{L}_2(H^\Q)$ induces a map 
$\lambda: H \to \mathfrak{L}_2(H^\Q)/ \mathfrak{L}_2(H)$
such that
\begin{equation} \label{eq:l2}
\lambda(u+v) = \lambda(u)+\lambda(v)+ \frac{[u,v]}{2}, 
\quad \hbox{for all $u,v\in H$.}
\end{equation}
(c) 
According to \cite[Ex$.$~2.19]{Mas12},
for every choice of a  basis $(\alpha,\beta)$ of $\pi$
inducing the basis $(a,b)$ of $H$,
there exists a symplectic  
expansion $\tilde\theta$  such that
\[
\begin{cases}
 \tilde\theta(\alpha_i) = 1 + a_i + \frac{1}{2} a_i^2 -\frac{1}{2}[a_i,b_i]
+ (\deg \geq 3),   \\
\tilde\theta(\beta_i) = 1 + b_i + \frac{1}{2} b_i^2 -\frac{1}{2}[a_i,b_i]
+ (\deg \geq 3).
\end{cases}
\]
Hence,  
the set  of truncated even symplectic expansions is non-empty:
it is easily deduced from (a) that
it is a torsor over $\frac{1}{2}\Lambda^3 H$.

Let $\theta$ be an even symplectic expansion.
Since we have 
\begin{equation} \label{eq:vartheta}
\vartheta([x])=\frac{[x]^2}{2}+\ell_2^\theta(x)\in H^{\otimes 2} \otimes \Q/\Z, 
\quad \hbox{for all $x\in \pi$,}
\end{equation}
and since $\ell_2^\theta(x)$ is
in $\frac{1}{2} \mathfrak{L}_2(H)$,
the condition (i) is clearly satisfied
(using the fact that the form $\omega$ is alternate). 
We now prove~(ii).

A straightforward computation on $u,v\in \{a_1,\dots,a_g, b_1,\dots, b_g\}$
shows that (ii) is satisfied by the particular expansion $\tilde \theta$
that has been mentioned above. 
Let  $d\in \Hom(H^\Q,\mathfrak{L}_2^\Q)$ 
be given by the difference $\tilde\theta-\theta$.
Then $\tilde \vartheta = \vartheta + (d \mod 1)$.
Thus, to deduce that $\vartheta$ satisfies  (ii) 
just as $\tilde \vartheta$ does, it suffices to prove that
\begin{equation} \label{eq:4_omegas}
\omega(\dot u,v)\ddot u - \omega(\ddot u,v)\dot u 
\equiv \omega(\dot v,u)\ddot v - \omega(\ddot v,u)\dot v \mod H,
\quad \hbox{for all $u,v\in H$},
\end{equation}
where we write here $d(u)$ in Sweedler's notation:
$$
d(u) = \sum_i [\dot u_i, \ddot u_i ] =: 
[\dot u, \ddot u ] = \dot u \otimes \ddot u - \ddot u \otimes \dot u \in H^\Q \otimes_\Q H^\Q.
$$
Since $d(H) \subset \frac{1}{2}\mathfrak{L}_2(H)$,
 \eqref{eq:4_omegas} can be equivalently written as 
\begin{equation} \label{eq:4_omegas_bis}
\omega(\dot u,v) \ddot u 
+\omega(\ddot u,v) \dot u + \omega(\dot v,u) \ddot v
+  \omega(\ddot v,u)\dot v  \equiv 0 \!\!\mod H,
\  \hbox{for all $u,v\in H$}.
\end{equation}
Since $d$ is in the image of $\frac{1}{2}\Lambda^3 H$
by the map \eqref{eq:expand}, we can assume that 
there exist $a,b,c \in H$ such that 
\begin{equation} \label{eq:d_form}
    d= \frac{1}{2} \omega(a,-)[b,c] +\frac{1}{2} \omega(b,-)[c,a] 
+ \frac{1}{2} \omega(c,-)[a,b].
\end{equation}
Hence, \eqref{eq:4_omegas_bis}
easily follows for any $u,v\in H$.

To prove (iii), we first deduce from (ii) that
the map $y: H \to H^\Q/H$ defined by $y(u):=\omega(u',u)\,u''$ is linear.
Then, a straightforward computation with $u\in \{a_1,\dots,a_g, b_1,\dots, b_g\}$
shows that (iii) is satisfied by the particular expansion $\tilde \theta$.
Thus, as in the previous paragraph,  letting $d:= \tilde\theta-\theta$, 
we are reduced to showing that
\begin{equation}
\label{eq:oo}
\omega(\dot u, u)\, \ddot u - \omega(\ddot u,u)\, \dot u \equiv 0 \mod H.
\end{equation}
Again, we can assume without loss of generality that $d$ is of the form \eqref{eq:d_form},
and \eqref{eq:oo} easily follows.
\end{proof}

We shall also need the following technical lemma 
about symplectic expansions.
Let $\fr:P \to \Z_2$ be the projection 
given by Lemma \ref{lem:P}.(b), which is defined by the identity 
\begin{equation} \label{eq:fr}
\fr(x)\, s = x-\sigma(p_*(x))\in P, \quad \hbox{for all $x\in P$}.
\end{equation}
(Thus, by \eqref{eq:sigma},
$\fr(t(K))$ is the mod 2 reduction of the 
framing number $\operatorname{Fr}(K)\in \Z$
for any framed oriented knot $K$ in $U$.)
Besides, let
$\varpi: \mathfrak{L}_2(H) \otimes \Q/\Z \to \Q/\Z$
be the homomorphism defined by $\varpi([u,v]):=\omega(u,v)$
for all $u,v\in H^\Q$. In the sequel, for any $l=\{L\}\in \Z_2$,
we denote 
${l}/{2} := \{L/2\} \in \Q/\Z$.

\begin{lemma}\label{lem:upsilon} 
For any symplectic expansion $\theta$, the map
\[
\upsilon:=\frac{\fr}{2}
+ \varpi \circ \lambda \circ p_* \colon P\longrightarrow \Q/\Z
\]
is a homomorphism.
If $\theta$ is even, then $\upsilon$ is valued in
$\big(\frac{1}{2} \Z\big)/\Z$.
\end{lemma}

\begin{proof}
The second statement is obvious.
To prove the first statement, let $x,y\in P$ and 
set  $\bar{x}:=p_*(x),\bar{y}:=p_*(y)\in H$.
Since the associated $2$-cocycle of the section $\sigma$ 
of $p_*:P \to H$ is $(\omega$ mod 2), 
it follows from \eqref{eq:fr} that 
\begin{equation} \label{eq:fr2}
\fr(x+y)-\fr(x)-\fr(y)=\omega(\bar{x},\bar{y})\in \Z_2.
\end{equation}
On the other hand, we deduce from \eqref{eq:l2} that
\[
\varpi\lambda(\bar x+\bar y)-\varpi\lambda(\bar x)
-\varpi\lambda(\bar y)
=\frac{\omega(\bar{x},\bar{y})}{2}\in \Q/\Z.
\]
We conclude that $\upsilon(x+y)-\upsilon(x)-\upsilon(y)=0$.
\end{proof}

\subsection{The map $\delta^<$}
\label{subsec:delta^<}

We shall define a degree 1 linear map
$$\delta^<=\delta^<_\theta: \calA^<(P) \longrightarrow \calA^<(P)\otimes \Q/\Z$$
for every choice of an even symplectic expansion $\theta$.
(Recall from  Remark \ref{rem:Q/Z} 
that $\calA^<(P)\otimes \Q/\Z$ is isomorphic to $\calA^<(H)\otimes \Q/\Z $: so, at the target of $\delta^<$, considering colors in $P$ is equivalent to considering colors in $H$.)

Let $n\geq 1$ and let $D$ be a Jacobi diagram 
among the generators of $\mathcal A_n^<(P)$.
In order to define
$
\delta^<(D) \in \mathcal A_{n+1}^<(P)\otimes \Q/\Z,
$
we first introduce the following Jacobi diagrams, 
constructed for each external vertex $i$ of $D$
with color $c_i\in P$:

\begin{itemize}
\item For $u,v\in H$, we obtain $D_i^{II}(u,v)$ from $D$
by duplicating the edge attached to $i$ 
and assigning colors as indicated:
\[
D= \begin{array}{c} \begin{tikzpicture}[thick,baseline=-3.0em, x=1em, y=1em]
\footnotesize
 \draw ([shift={(0,5)}]-80:6) arc [radius=6, start angle = -80, end angle=-100];
 \draw [dotted]([shift={(0,5)}]-80:6) arc [radius=6, start angle = -80, end angle= -70];
 \draw [dotted]([shift={(0,5)}]-100:6) arc [radius=6, start angle = -100, end angle= -110];
 \draw [fill = black] (0,-1) circle (.3ex);
 \draw (0,-1)-- (0,-2.5) node [below] {$\cdots<c_i<\cdots$};
\end{tikzpicture} \end{array} 
\leadsto
\begin{array}{c}  \begin{tikzpicture}[thick,baseline=-3.0em, x=1em, y=1em]
\footnotesize
 \draw ([shift={(0,5)}]-80:6) arc [radius=6, start angle = -80, end angle=-100];
 \draw [dotted]([shift={(0,5)}]-80:6) arc [radius=6, start angle = -80, end angle= -70];
 \draw [dotted]([shift={(0,5)}]-100:6) arc [radius=6, start angle = -100, end angle= -110];
 \draw [fill = black] (-0.8,-1) circle (.3ex);
 \draw [fill = black] (0.8,-1) circle (.3ex);
 \draw (-0.8,-1)-- (-0.8,-2.4);
 \node [below] at (0,-2.4) {$\,\cdots<u<v<\cdots$};
 \draw (0.8,-1)-- (0.8,-2.4);
\end{tikzpicture} \end{array} =: D_i^{II}(u,v)
\]
\item We define $D_i^U$ from $D$ by again doubling the edge incident to $i$
, but this time joining the two new external vertices by an additional edge:
\[
D= \begin{array}{c} \begin{tikzpicture}[thick,baseline=-3.0em, x=1em, y=1em]
\footnotesize
 \draw ([shift={(0,5)}]-80:6) arc [radius=6, start angle = -80, end angle=-100];
 \draw [dotted]([shift={(0,5)}]-80:6) arc [radius=6, start angle = -80, end angle= -70];
 \draw [dotted]([shift={(0,5)}]-100:6) arc [radius=6, start angle = -100, end angle= -110];
 \draw [fill = black] (0,-1) circle (.3ex);
 \draw (0,-1)-- (0,-2.5) node [below] {$\cdots<c_i<\cdots$};
\end{tikzpicture} \end{array}
\leadsto
\begin{array}{c}  \begin{tikzpicture}[thick,baseline=-3.0em, x=1em, y=1em]
\footnotesize
 \draw ([shift={(0,5)}]-80:6) arc [radius=6, start angle = -80, end angle=-100];
 \draw [dotted]([shift={(0,5)}]-80:6) arc [radius=6, start angle = -80, end angle= -70];
 \draw [dotted]([shift={(0,5)}]-100:6) arc [radius=6, start angle = -100, end angle= -110];
 \draw [fill = black] (-0.5,-1) circle (.3ex);
 \draw [fill = black] (0.5,-1) circle (.3ex);
 \draw (-0.5,-1)-- (-0.5,-2.);
 \draw (0.5,-2.) -- (-0.5,-2.);
 \draw (0.5,-1)-- (0.5,-2.);
 \node[below] at (-1.3,-2.0) {$\cdots$};
 \node[below] at (1.4,-2.0) {$\cdots$};
\end{tikzpicture} \end{array} =: D_i^U
\]
\end{itemize}
With these preparations, we set
\begin{equation} \label{eq:delta_i}
\delta_i(D) := D_i^{II}(c'_i,c_i'') +
\upsilon(c_i)\, D_i^U\in \calA^{<}_{n+1}(P)\otimes \Q/\Z    
\end{equation}
where $\vartheta(p_*(c_i))=c'_i \otimes c''_i
\in H^{\otimes 2}\otimes \Q/\Z$
is defined in  Lemma \ref{lem:expansion}.(b)
and is written here using Sweedler's notation,
and  $\upsilon(c_i) \in\Q/\Z$ is 
introduced in Lemma \ref{lem:upsilon}.

Besides, for any $Y$-diagram component $C$ of $D$,
with external vertices colored by $x,y,z\in P$,
we define $\delta_C(D)\in \calA^{<}_{n+1}(P)\otimes \Z_2$ 
 by ``duplicating''  $C$ in $D$ as follows:
\[
D= \begin{tikzpicture}[thick,baseline=-1.0em, x=2em, y=2em]
\hspace{-4mm}
\footnotesize
\node at (0.5,0.2){$C$};
\node[below=1pt] at (0,-1){$\cdots x<\cdots<y<\cdots <z\cdots$};
 \draw [fill = black] (0,0) circle (.3ex);
\draw (0,0)-- (-1.8,-1);
\draw (0,0)-- (0,-1);
\draw (0,0)-- (1.8,-1);
\end{tikzpicture}  \leadsto
\begin{tikzpicture}[thick,baseline=-1.0em, x=2em, y=2em]
\hspace{-4mm}
\footnotesize
\node[below=1pt] at (0,-1){\qquad$\cdots x\parallel x<\cdots <y\parallel y<\cdots <z \parallel z\cdots$};
 \draw [fill = black] (0,0) circle (.3ex);
 \draw (0,0)-- (-2.6,-1);
\draw (0,0)-- (0,-1);
\draw (0,0)-- (2.6,-1);
\fill [white] (0.4,-0.12) circle [radius=0.1];
\fill [white] (0.0,-0.3) circle [radius=0.1];
\fill [white] (0.8,-0.3) circle [radius=0.1];
\draw [fill = black] (0.8,0) circle (.3ex);
\draw (0.8,0)-- (-1.8,-1);
\draw (0.8,0)-- (0.8,-1);
\draw (0.8,0)-- (3.4,-1);
\end{tikzpicture} =: \delta_C(D)
\]
or \[D= \begin{tikzpicture}[thick,baseline=-1.0em, x=2em, y=2em]
\footnotesize
\node[below=1pt] at (0,-1){$\cdots x<\cdots<y<\cdots <z\cdots$};
 \draw [fill = black] (0,0) circle (.3ex);
		\node  (2) at (0, 0) {};
		\node  (3) at (0, -1) {};
		\node  (4) at (-1.75, -1) {};
		\node  (5) at (1.75, -1) {};
		\node  (6) at (0, 0) {};
		\node  (7) at (0, 0.25) {};
		\node  (8) at (-0.5, 0.75) {};
		\node  (9) at (-1, 0.25) {};
		\draw (2.center) to (5.center);
		\draw (6.center) to (7.center);
        \draw (2.center) to (4.center);
        \fill [white] (-0.7,-0.4) circle [radius=0.1];
        \draw [in=0, out=90] (7.center) to (8.center);
		\draw [in=90, out=-180] (8.center) to (9.center);
		\draw [in=150, out=-90, looseness=0.75] (9.center) to (3.center);
\end{tikzpicture} \leadsto
\begin{tikzpicture}[thick,baseline=-1.0em, x=2em, y=2em]
\hspace{-4mm}
\footnotesize
\node[below=1pt] at (-0.35,-1){\qquad$\cdots x\parallel x<\cdots <y\parallel y<\cdots <z \parallel z\cdots$};
		\node (2) at (-0.25, 0) {};
		\node (3) at (0.3, -1) {};
		\node (4) at (-2.8, -1) {};
		\node (5) at (2.25, -1) {};
		\node (6) at (-0.25, 0) {};
		\node (7) at (-0.25, 0.25) {};
		\node (8) at (-0.75, 0.75) {};
		\node (9) at (-1.25, 0.25) {};
		\node (13) at (-0.25, -1) {};
		\node (14) at (0.25, 0) {};
		\node (15) at (2.9, -1) {};
		\node (16) at (-2.2, -1) {};
		\node (17) at (0.25, 0.5) {};
		\node (18) at (-0.75, 1.25) {};
		\node (19) at (-1.75, 0.25) {};
        \draw [fill = black] (-0.25,0) circle (.3ex);
        \draw [fill = black] (0.25,0) circle (.3ex);
		\draw (2.center) to (5.center);
		\draw  (2.center) to (4.center);
		\draw (6.center) to (7.center);
		\draw [in=0, out=90] (7.center) to (8.center);
		\draw [in=90, out=-180] (8.center) to (9.center);
        \fill [white] (0,-0.1) circle [radius=0.1];
        \draw (16.center) to (14.center);
        \fill [white] (-0.85,-0.25) circle [radius=0.1];
        \fill [white] (-0.95,-0.5) circle [radius=0.1];
        \fill [white] (-1.2,-0.36) circle [radius=0.1];
        \fill [white] (-0.65,-0.36) circle [radius=0.1];
		\draw [in=120, out=-90, looseness=0.75] (9.center) to (3.center);
		\draw (14.center) to (15.center);
		\draw [in=270, out=90] (14.center) to (17.center);
		\draw [in=0, out=90] (17.center) to (18.center);
		\draw [in=90, out=-180] (18.center) to (19.center);
		\draw [in=120, out=-90, looseness=0.75] (19.center) to (13.center);
\end{tikzpicture} =: \delta_C(D) \]
Here, the expression on the right-hand side represents a sum of $2^3$ Jacobi diagrams, written according to the 
notational convention
\begin{equation} \label{eq:parallel_notation}
\qquad 
\begin{tikzpicture}[thick,baseline=-1.5em, x=1em, y=1em]
\footnotesize
\draw[dotted] (0.7,-0.8)-- (0.7,0);
\draw[dotted] (-0.7,-0.8)-- (-0.7,0);
\draw (0.7,-2)-- (0.7,-0.8);
\draw (-0.7,-2)-- (-0.7,-0.8);
\node[below] at (0,-2) {$\cdots<c \parallel c<\cdots$};
\end{tikzpicture}
:= 
\begin{tikzpicture}[thick,baseline=-1.5em, x=1em, y=1em]
\footnotesize
\draw[dotted] (0.7,-0.8)-- (0.7,0);
\draw[dotted] (-0.7,-0.8)-- (-0.7,0);
\draw (0.7,-2)-- (0.7,-0.8);
\draw (-0.7,-2)-- (-0.7,-0.8);
\node[below] at (0,-2) {$\cdots<c<c<\cdots$};
\end{tikzpicture}
+ {\fr(c)}\quad
\begin{tikzpicture}[thick,baseline=-1.5em, x=1em, y=1em]
\footnotesize]
\draw[dotted] (0.7,-0.8)-- (0.7,0);
\draw[dotted] (-0.7,-0.8)-- (-0.7,0);
 \draw (0.7,-0.8) to [out=-90, in=-90, looseness=3] (-0.7,-0.8);
\end{tikzpicture}\, \ \in \calA^<(P) \otimes \Z_2.
\end{equation}

\begin{remark}
This notation is inspired from an identity in clasper calculus,
when considering the surgery along  a graph clasper 
with two parallel leaves of type $c$ in~$U$.
(See the last part of the proof of Proposition \ref{prop:Psi_psi}.)
\hfill $\blacksquare$
\end{remark}

We now proceed with the construction of the map $\delta^<$.
This map is a variation (and a mild generalization)
of the map $\delta$ introduced in \cite[\S 3]{NSS22a},
and it is better suited for the analysis of certain properties.

\begin{proposition}
\label{prop:delta}
Let $\theta$ be an even symplectic expansion of $\pi$.
There is a well-defined  
homomorphism of degree $1$
\[
\delta^<  = \delta^<_\theta :\calA^< (P) \longrightarrow \calA^<(P) \otimes \Q/\Z
\]
which maps every Jacobi diagram with a looped edge to zero,
and maps any Jacobi diagram $D$ without looped edge to
\begin{equation} \label{eq:delta}
\delta^<(D) := 
\sum_{i} \delta_i(D)
+ \sum_{C} \frac{\delta_C (D)}{2},  
\end{equation} 
where the first sum ranges 
over external vertices $i$ of $D$,
and the second sum  over  $Y$-diagram  components $C$ of $D$.
\end{proposition}

\noindent
Before establishing the proposition,  
we further discuss formula \eqref{eq:delta}:
\begin{itemize}
\item Since $\theta$ is assumed to be even, 
the map $\delta^<$ actually lands in the 2-torsion subgroup of $ \calA^<(P) \otimes \Q/\Z$. 
More precisely, for any Jacobi diagram $D$, each summand 
$\delta_i(D)$ and $\delta_C(D)/2$ appearing in \eqref{eq:delta} is an element of order $2$; furthermore, every term in the formula \eqref{eq:delta_i} for $\delta_i(D)$ has order $2$. 
\item Assume that the Jacobi diagram $D$ is connected and of degree $n\geq 2$. Then we have
\begin{equation} \label{eq:connected}
\delta^<(D) = \sum_i D_i^{II}(c'_i,c_i'') \in \calA^<_{n+1}(P) \otimes \Q/\Z
\end{equation}
where $\vartheta(p_*(c_i))=c'_i \otimes c''_i \in H^{\otimes 2}\otimes \Q/\Z$,  
and the sum runs over all external vertices $i$ colored by $c_i\in P$.
This simpler expression is a consequence of the fact that any connected Jacobi diagram of degree $\geq 3$ containing a
\begin{tikzpicture}[thick,color=black,baseline=-0.2em, x=0.8em, y=0.8em]
\draw [dotted] (-2.0,0) -- (-1.6,0);
\draw (-1.6,0) -- (-1,0);
\draw (-0.5,0) circle [radius=0.5];
\draw (0,0) -- (0.6,0);
\draw [dotted] (1,0) -- (0.6,0);
\draw [fill = black] (-1,0) circle (.2ex);
\draw [fill = black] (0,0) circle (.2ex);
\end{tikzpicture}
subdiagram is divisible by $2$ in $\calA^<(P)$:
\begin{equation} \label{eq:bubble}
\begin{tikzpicture}[thick,color=black,baseline=-0.2em, x=1.5em, y=1.5em]
\draw[dotted] (-2.5,0) -- (-2,0);
\draw (-2,0) -- (-1,0);
\draw (-1.5,0) -- (-1.5,0.9);
\draw [dotted] (-1.5,0.9) -- (-1.5,1.4);
\draw (-0.5,0) circle [radius=0.5];
\draw (0,0) -- (0.6,0);
\draw [dotted] (0.6,0) -- (1.1,0);
\draw [fill = black] (-1,0) circle (.3ex);
\draw [fill = black] (0,0) circle (.3ex);
\draw [fill = black] (-1.5,0) circle (.3ex);
\end{tikzpicture}
\ = \
\begin{tikzpicture}[thick,color=black,baseline=-0.2em, x=1.5em, y=1.5em]
\draw (-1.5,1) to[out=-90,in=90] (-0.5,0.5);
\draw [dotted] (-1.5,1.4) -- (-1.5,1);
\draw[dotted] (-2.5,0) -- (-2,0);
\draw (-2,0) -- (-1,0);
\draw (-0.5,0) circle [radius=0.5];
\draw (0,0) -- (0.6,0);
\draw [dotted] (0.6,0) -- (1.1,0);
\draw [fill = black] (-1,0) circle (.3ex);
\draw [fill = black] (0,0) circle (.3ex);
\draw [fill = black] (-0.5,0.5) circle (.3ex);
\end{tikzpicture}
\ + \
\begin{tikzpicture}[thick,color=black,baseline=-0.2em, x=1.5em, y=1.5em]
\draw (-1.5,1) to[out=-90,in=90,looseness=2] (-0.5,-0.5);
\draw [dotted] (-1.5,1.4) -- (-1.5,1);
\fill [white] (-0.95,0.2) circle [radius=0.14];
\draw [dotted] (-2.5,0) -- (-2,0);
\draw (-2,0) -- (-1,0);
\draw (-0.5,0) circle [radius=0.5];
\draw (0,0) -- (0.6,0);
\draw [dotted] (0.6,0) -- (1.1,0);
\draw [fill = black] (-1,0) circle (.3ex);
\draw [fill = black] (0,0) circle (.3ex);
\draw [fill = black] (-0.5,-0.5) circle (.3ex);
\end{tikzpicture}
\ = \ 2 \ 
\begin{tikzpicture}[thick,color=black,baseline=-0.2em, x=1.5em, y=1.5em]
\draw (-1.5,1) to[out=-90,in=90] (-0.5,0.5);
\draw [dotted] (-1.5,1.4) -- (-1.5,1);
\draw[dotted] (-2.5,0) -- (-2,0);
\draw (-2,0) -- (-1,0);
\draw (-0.5,0) circle [radius=0.5];
\draw (0,0) -- (0.6,0);
\draw [dotted] (0.6,0) -- (1.1,0);
\draw [fill = black] (-1,0) circle (.3ex);
\draw [fill = black] (0,0) circle (.3ex);
\draw [fill = black] (-0.5,0.5) circle (.3ex);
\end{tikzpicture}
\end{equation}

\item For an arbitrary Jacobi diagram $D$,  
the term $\delta_i(D)$ appearing in \eqref{eq:delta} can be rewritten as
\begin{equation} \label{eq:delta_i_bis}
\delta_i(D) = 
\frac{1}{2} D_i^{II}(c_i,c_i) + D_i^{\Yrev}(\dot c_i, \ddot c_i) +
\frac{\fr(c_i)}{2}\, D_i^U\in \calA^{<}_{n+1}(P)\otimes \Q/\Z,
\end{equation}
where $\lambda(p_*(c_i))=[\dot c_i , \ddot c_i ]
\in \mathfrak{L}_2(H)\otimes \Q/\Z$ 
is defined in Lemma \ref{lem:expansion}.(b) 
and is written here in Sweedler's notation, and,
for any $x,y\in H$, the diagram $D_i^{\Yrev}(x,y)$ is obtained from $D$ by the following local modification:
\begin{equation} \label{eq:Yrev}
\qquad 
D= \begin{array}{c} \begin{tikzpicture}[thick,baseline=-3.0em, x=1em, y=1em]
\footnotesize
 \draw ([shift={(0,5)}]-80:6) arc [radius=6, start angle = -80, end angle=-100];
 \draw [dotted]([shift={(0,5)}]-80:6) arc [radius=6, start angle = -80, end angle= -70];
 \draw [dotted]([shift={(0,5)}]-100:6) arc [radius=6, start angle = -100, end angle= -110];
 \draw [fill = black] (0,-1) circle (.3ex);
 \draw (0,-1)-- (0,-2.5) node [below] {$\cdots<c_i<\cdots$};
\end{tikzpicture} \end{array}
 \ \leadsto \
\begin{array}{c}
\begin{tikzpicture}[thick,baseline=-2.5em, x=1em, y=1em]
\footnotesize]
 \draw ([shift={(0,5)}]-80:6) arc [radius=6, start angle = -80, end angle=-100];
 \draw [dotted]([shift={(0,5)}]-80:6) arc [radius=6, start angle = -80, end angle= -70];
 \draw [dotted]([shift={(0,5)}]-100:6) arc [radius=6, start angle = -100, end angle= -110];
 \draw [fill = black] (0,-1) circle (.3ex);
 \draw [fill = black] (0,-1.5) circle (.3ex);
 \draw (0,-1)-- (0,-1.5);
 \draw (0.4,-2) -- (0,-1.5) -- (-0.4,-2);
\node[below] at (0,-2) {$\cdots<x<y<\cdots$};
\end{tikzpicture} \end{array} =: D_i^{\Yrev}(x,y).
\end{equation}
(Using the IHX and STU-like relations,
the identity \eqref{eq:delta_i_bis}
follows from (\ref{eq:vartheta})
which, in Sweedler's notation, writes
$c'_i \otimes c''_i = \frac{1}{2} c_i \otimes c_i+ [\dot c_i,\ddot{c}_i]$.)

\item It is immediate that the definition of $\delta^<_\theta$
depends only on the truncation of~$\theta$ in degrees $\leq 2$.
If $\tilde{\theta}$ is another even symplectic expansion, we obtain a corresponding map $\delta^<_{\tilde \theta}$. In that case,  
$\delta^<_{\tilde\theta}-\delta^<_{\theta}$ can be explicitly determined from the element of $\frac{1}{2} \Lambda^3 H$ that encodes the change from $\theta$ to $\tilde\theta$.
\end{itemize}

\begin{proof}[Proof of Proposition \ref{prop:delta}]
We need to check that $\delta^<$ is compatible 
with each of the defining relations of $\calA^<(P)$.  
As already observed from the assumption on $\theta$, every term that occurs in the formulas for $\delta^<(D)$, for a Jacobi diagram $D$, has order $2$ in $\calA^<(P)\otimes \Q/\Z$.  
Hence, in the computations that follow, we can disregard all sign considerations.

We begin with  \emph{the AS relation}
applied to a Jacobi diagram $D$.  
If this relation occurs away from the external vertices, there is nothing to verify. Otherwise,  
we have to examine the two contributions $\delta_i(D)$ and $\delta_C(D)$
appearing in the expression \eqref{eq:delta} for $\delta^<(D)$.
The identity
\[
\begin{tikzpicture}[baseline=-0.5em, x=1.85em, y=1.85em,thick]
\draw (-0.5,-0.2) to[out=180, in=210, out looseness=2](0.366,1/2);
\fill [white] (0,0.3) circle [radius=0.15];
\draw (0.5,-0.2) to[out=0, in=-30, out looseness=2](-0.366,1/2);
\draw (0.5,-0.2) -- (-0.5,-0.2);
\draw[dotted] ({-0.5+sqrt(3)/2},1/2) -- ({-0.5+sqrt(3)/2*1.5},1.5/2);
\draw[dotted] ({0.5-sqrt(3)/2},1/2) -- ({0.5-sqrt(3)/2*1.5},1.5/2);
\draw (-0.4,-0.2)--(-0.4,-0.8);
\draw (0.4,-0.2)--(0.4,-0.8);
\draw [fill = black] (-0.4,-0.2) circle (.3ex);
\draw [fill = black] (0.4,-0.2) circle (.3ex);
\node at (0,-1.1) {\footnotesize$\cdots<x<x<\cdots$};
\end{tikzpicture}
=
\begin{tikzpicture}[baseline=-0.5em, x=1.85em, y=1.85em,thick]
\draw (0.5,-0.2) to[out=0, in=210, out looseness=2](0.366,1/2);
\fill [white] (0,0.3) circle [radius=0.15];
\draw (-0.5,-0.2) to[out=180, in=-30, out looseness=2](-0.366,1/2);
\draw (0.5,-0.2) -- (-0.5,-0.2);
\draw[dotted] ({-0.5+sqrt(3)/2},1/2) -- ({-0.5+sqrt(3)/2*1.5},1.5/2);
\draw[dotted] ({0.5-sqrt(3)/2},1/2) -- ({0.5-sqrt(3)/2*1.5},1.5/2);
\draw (0.4,-0.2)--(-0.4,-0.8);
\fill [white] (0,-0.5) circle [radius=0.15];
\draw (-0.4,-0.2)--(0.4,-0.8);
\draw [fill = black] (-0.4,-0.2) circle (.3ex);
\draw [fill = black] (0.4,-0.2) circle (.3ex);
\node at (0,-1.1) {\footnotesize$\cdots<x<x<\cdots$};
\end{tikzpicture}
=
\begin{tikzpicture}[baseline=-0.5em, x=1.85em, y=1.85em,thick]
\draw (0.5,-0.2) to[out=0, in=210, out looseness=2](0.366,1/2);
\fill [white] (0,0.3) circle [radius=0.15];
\draw (-0.5,-0.2) to[out=180, in=-30, out looseness=2](-0.366,1/2);
\draw (0.5,-0.2) -- (-0.5,-0.2);
\draw[dotted] ({-0.5+sqrt(3)/2},1/2) -- ({-0.5+sqrt(3)/2*1.5},1.5/2);
\draw[dotted] ({0.5-sqrt(3)/2},1/2) -- ({0.5-sqrt(3)/2*1.5},1.5/2);
\draw (-0.4,-0.2)--(-0.4,-0.8);
\draw (0.4,-0.2)--(0.4,-0.8);
\draw [fill = black] (-0.4,-0.2) circle (.3ex);
\draw [fill = black] (0.4,-0.2) circle (.3ex);
\node at (0,-1.1) {\footnotesize$\cdots<x<x<\cdots$};
\end{tikzpicture}\in \calA^<(P)
\]
shows that the first summand of 
$\delta_i(D)\in \calA^<(P)\otimes \Q/\Z$ in \eqref{eq:delta_i_bis} respects the AS relation. (The first equality is an application of ``AS”, and the second arises from ``STU-like''.)  
The second and third summands of $\delta_i(D)$ in \eqref{eq:delta_i_bis} are handled analogously
(by using the AS relation in the target).
Moreover, from its  definition,
it is immediate that 
the term $\delta_C(D)\in \calA^<(P)\otimes \Z_2$ 
also respects the AS relation.

To check \emph{the IHX relation}, 
we only have to consider the terms $\delta_i(D)$
in the formula \eqref{eq:delta} for $\delta^<(D)$. Note that we have
\[
\begin{tikzpicture}[baseline=-0.5em, x=1.5em, y=1.5em,thick]
\draw[dotted] (-2.5,0) -- (-2,0);
\draw (-2,0) -- (-0.7,0);
\draw [fill = black] (-0.7,0) circle (.3ex);
\draw (-0.5,0.3) to [out =180, in = 180] (-0.5,-0.3);
\draw (-0.5,0.3) -- (0.8,0.3);
\draw (-0.5,-0.3) -- (0.8,-0.3);
\draw[dotted] (1.3,0.3) -- (0.8,0.3);
\draw[dotted] (1.3,-0.3) -- (0.8,-0.3);
\draw [fill = black] (-1.6,0) circle (.3ex);
\draw [fill = black] (-1.2,0) circle (.3ex);
\draw (-1.9,-1) to [out=90,in=-90] (-1.6,0);
\draw (-0.9,-1) to [out=90,in=-90] (-1.2,0);
\node at (-1.3,-1.3) {\footnotesize$\cdots<x<x<\cdots$};
\end{tikzpicture}
=\hspace{1mm}\begin{tikzpicture}[baseline=-0.5em, x=1.5em, y=1.5em,thick]
\draw[dotted] (-2.5,0) -- (-2,0);
\draw (-2,0) -- (-0.7,0);
\draw [fill = black] (-0.7,0) circle (.3ex);
\draw (-0.5,0.3) to [out =180, in = 180] (-0.5,-0.3);
\draw (-0.5,0.3) -- (0.8,0.3);
\draw (-0.5,-0.3) -- (0.8,-0.3);
\fill [white] (-0.2,-0.3) circle [radius=0.15];
\fill [white] (0.4,-0.3) circle [radius=0.15];
\draw[dotted] (1.3,0.3) -- (0.8,0.3);
\draw[dotted] (1.3,-0.3) -- (0.8,-0.3);
\draw [fill = black] (-0.1,0.3) circle (.3ex);
\draw [fill = black] (0.3,0.3) circle (.3ex);
\draw (-0.4,-1) to [out=90,in=-90] (-0.1,0.3);
\draw (0.6,-1) to [out=90,in=-90] (0.3,0.3);
\node at (0.2,-1.3) {\footnotesize$\cdots<x<x<\cdots$};
\end{tikzpicture}+\hspace{2mm}
\begin{tikzpicture}[thick, baseline=-0.5em, x=1.5em, y=1.5em]
\draw[dotted] (-2.5,0) -- (-2,0);
\draw (-2,0) -- (-0.7,0);
\draw [fill = black] (-0.7,0) circle (.3ex);
\draw (-0.5,0.3) to [out =180, in = 180] (-0.5,-0.3);
\draw (-0.5,0.3) -- (0.8,0.3);
\draw (-0.5,-0.3) -- (0.8,-0.3);
\draw[dotted] (1.3,0.3) -- (0.8,0.3);
\draw[dotted] (1.3,-0.3) -- (0.8,-0.3);
\draw [fill = black] (-0.1,-0.3) circle (.3ex);
\draw [fill = black] (0.3,-0.3) circle (.3ex);
\draw (-0.4,-1) to [out=90,in=-90] (-0.1,-0.3);
\draw (0.6,-1) to [out=90,in=-90] (0.3,-0.3);
\node at (0.2,-1.3) {\footnotesize$\cdots<x<x<\cdots$};
\end{tikzpicture}
\in \calA^<(P)\otimes(\Q/2\Z)
\]
by ``pushing'' the vertical edges to the right via the IHX relation and then applying the STU-like relation to discard the other two terms modulo $2\Z$.
This verifies that the first term of $\delta_i(D)$
 in \eqref{eq:delta_i_bis} is compatible 
 with the IHX relation in the Kirchhoff form of \cite[\S 5.2.7]{CDM12}.
 The second and third terms of $\delta_i(D)$
 in \eqref{eq:delta_i_bis} are treated 
 using ``IHX'' in the target and \eqref{eq:bubble}, respectively.

Next, we examine \emph{the multilinearity relation} 
for a Jacobi diagram $D(x)$ with respect 
to a fixed external vertex $i$ colored by $x\in P$.
Let $c_i,d_i \in P$; our goal is to prove that
\begin{equation} \label{eq:delta_ml}
\delta^<(D(c_i+d_i))=\delta^<(D(c_i))+\delta^<(D(d_i)).
\end{equation}
Let $C$ be a $Y$-diagram component of $D$: 
if $i$ belongs to $C$,
then by using the multilinearity and STU-like relations, 
together with (\ref{eq:fr2}), we obtain
\[
\delta_C(D(c_i+d_i))=\delta_C(D(c_i))+\delta_C(D(d_i))
\in \calA^<(P)\otimes\Z_2;
\]
if $i$ does not lie in $C$, the same identity follows immediately
from  ``multilinearity'' in $\calA^<(P)\otimes\Z_2$.
Now let $j$ be an external vertex of $D$: 
if $j\neq i$, then for the same reason, we clearly have
\[
\delta_j(D(c_i+d_i))=\delta_j(D(c_i))+\delta_j(D(d_i)) 
\in \calA^<(P)\otimes\Q/\Z;
\]
if instead $j=i$, the same equality
holds true by \eqref{eq:delta_i}
and the fact that both $\vartheta$ and $v$ are homomorphisms.
This establishes \eqref{eq:delta_ml}.

The compatibility of $\delta^<$ with \emph{the special relation} follows directly from the expression of $\delta_i$ in \eqref{eq:delta_i_bis} for an external vertex $i$ using the observation \eqref{eq:bubble}, and from the fact that $\delta_C$ leaves all components of degree $\geq 2$ unchanged for any $Y$-diagram component $C$.

We now address \emph{the STU-like relation}.
Let $D$ be a Jacobi diagram 
with two consecutive external vertices $i<j$.
Define $E$ (respectively, $G$) 
to be the diagram obtained by exchanging the order of $i$ and $j$ 
(respectively, by gluing the vertices $i$ and $j$).
Our goal is to show that
\begin{equation} \label{eq:aim_STU-like}
\delta^<(D)-\delta^<(E) = \omega(c_i,c_j) \, \delta^<(G)
\end{equation}
where $c_i$ and $c_j$ are the colors of $i$ and $j$, respectively.
This equality holds immediately if $D$ contains a looped edge,
so from now on we assume that $D$ has no looped edge.
First of all, we claim that
\begin{equation} \label{eq:ij}
\delta_i(D)- \delta_i(E) = \delta_j(D) -\delta_j(E).
\end{equation}
The vertex $i$ (respectively $j$)
is connected to an internal vertex by an edge $I$ (respectively $J$).
Let us assume that $I\cap J=\varnothing$, 
since the case $I\cap J\neq \varnothing$ can be done similarly.
We can depict $D$ in a neighborhood of $I \cup J$ as follows:
\[
D=
\begin{tikzpicture}[thick,baseline=-1em, x=1.5em, y=1.5em]
\footnotesize
\draw [dotted] (-1.5,0.5) -- (-1.5,0);
\draw [dotted] (-0.5,0.5) -- (-0.5,0);
\draw (-1.5,0) to [out=-90,in=-90] (-0.5,0);
\draw [fill = black] (-1,-0.3) circle (.3ex);
\draw [dotted] (0.5,0.5) -- (0.5,0);
\draw [dotted] (1.5,0.5) -- (1.5,0);
\draw (0.5,0) to [out=-90,in=-90] (1.5,0);
\draw [fill = black] (1,-0.3) circle (.3ex);
\draw (-1,-0.3) -- (-1,-1.5);
\draw (1,-0.3) -- (1,-1.5);
\node[below=1pt] at (-0.3,-1.5){\quad$\cdots <\hspace{1mm}c_i \hspace{1mm} \ < \ \hspace{1mm}c_j \hspace{1mm}<\cdots$};
\end{tikzpicture}.
\]
In this local picture, we compute
\begin{align*}
\delta_i(D)-\delta_i(E)&=D^{II}_i(c'_i,c''_i)-E^{II}_i(c'_i,c''_i)\\
&=
\hspace*{-3mm}
\begin{tikzpicture}[thick, baseline=-1em, x=1.5em, y=1.5em]
\footnotesize
\draw [dotted] (-1.5,0.5) -- (-1.5,0);
\draw [dotted] (-0.2,0.5) -- (-0.2,0);
\draw (-1.5,0) to [out=-90,in=-90] (-0.2,0);
\draw (-1.2,-0.35) to (-1.4,-1.5);
\draw (-0.5,-0.35) to (-0.3,-1.5);
\draw [fill = black] (-1.2,-0.35) circle (.3ex);
\draw [fill = black] (-0.5,-0.35) circle (.3ex);
\draw (1,-0.3) -- (1,-1.5);
\draw [dotted] (0.5,0.5) -- (0.5,0);
\draw [dotted] (1.5,0.5) -- (1.5,0);
\draw (0.5,0) to [out=-90,in=-90] (1.5,0);
\draw [fill = black] (1,-0.3) circle (.3ex);
\node[below=1pt] at (-0.2,-1.4) {$\scalebox{0.8}{$\cdots \hspace{1mm} c'_i \hspace{0.5mm}<\hspace{0.5mm}c''_i \ < \ c_j\cdots$}%
$};
\end{tikzpicture}
-
\begin{tikzpicture}[ thick, baseline=-1em, x=1.5em, y=1.5em]
\footnotesize
\draw [dotted] (-1.5,0.5) -- (-1.5,0);
\draw [dotted] (-0.2,0.5) -- (-0.2,0);
\draw [dotted] (0.5,0.5) -- (0.5,0);
\draw [dotted] (1.5,0.5) -- (1.5,0);
\draw (0.5,0) to [out=-90,in=-90] (1.5,0);
\draw [fill = black] (1,-0.3) circle (.3ex);
\draw (-1.5,0) to [out=-90,in=-90] (-0.2,0);
\draw (-1.2,-0.35) to (-0.2,-1.5);
\draw (-0.5,-0.35) to (1,-1.5);
\draw [fill = black] (-1.2,-0.35) circle (.3ex);
\draw [fill = black] (-0.5,-0.35) circle (.3ex);
\fill [white] (-0.65,-1) circle [radius=0.18];
\fill [white] (0.2,-0.85) circle [radius=0.18];
\draw (1,-0.3) to [out=-90,in=90] (-1.2,-1.5);
\node[below=1pt] at (-0.1,-1.4) {$\scalebox{0.8}{$\cdots\hspace{1mm} c_j\hspace{0.5mm}<\hspace{0.5mm}c'_i\hspace{0.5mm}<\hspace{0.5mm}c''_i\cdots$}$};
\end{tikzpicture}
\hspace*{-3mm}\\
&=\omega(c''_i,c_j)
\begin{tikzpicture}[thick, baseline=-1em, x=1.5em, y=1.5em]
\footnotesize
\draw [dotted] (-1.5,0.5) -- (-1.5,0);
\draw [dotted] (-0.2,0.5) -- (-0.2,0);
\draw [dotted] (0.5,0.5) -- (0.5,0);
\draw [dotted] (1.5,0.5) -- (1.5,0);
\draw (0.5,0) to [out=-90,in=-90] (1.5,0);
\draw [fill = black] (1,-0.3) circle (.3ex);
\draw (-1.5,0) to [out=-90,in=-90] (-0.2,0);
\draw [fill = black] (-1.2,-0.35) circle (.3ex);
\draw [fill = black] (-0.5,-0.35) circle (.3ex);
\draw (-1.5,0) to [out=-90,in=-90] (-0.2,0);
\draw (-0.5,-0.35) to [out=-90,in=-90] (1,-0.3);
\draw (-1.2,-0.35) to [out=-90,in=90] (0,-1.5);
\node[below=1pt] at (0,-1.4) {$\scalebox{0.8}{$\cdots <c'_i<\cdots$}$};
\end{tikzpicture}
+\omega(c'_i,c_j)
\begin{tikzpicture}[thick, baseline=-1em, x=1.5em, y=1.5em]
\footnotesize
\draw [dotted] (-1.5,0.5) -- (-1.5,0);
\draw [dotted] (-0.2,0.5) -- (-0.2,0);
\draw [dotted] (0.5,0.5) -- (0.5,0);
\draw [dotted] (1.5,0.5) -- (1.5,0);
\draw (0.5,0) to [out=-90,in=-90] (1.5,0);
\draw [fill = black] (1,-0.3) circle (.3ex);
\draw (-1.5,0) to [out=-90,in=-90] (-0.2,0);
\draw [fill = black] (-1.2,-0.35) circle (.3ex);
\draw [fill = black] (-0.5,-0.35) circle (.3ex);
\draw (-0.5,-0.35) to [out=-90,in=90] (0,-1.5);
\fill [white] (-0.2,-0.9) circle [radius=0.18];
\draw (-1.2,-0.35) to [out=-90,in=-90] (1,-0.3);
\node[below=1pt] at (0.1,-1.4) {$\scalebox{0.8}{$\cdots <c''_i<\cdots$}$};
\end{tikzpicture}\\
&=\omega(c'_i,c_j)
\begin{tikzpicture}[thick, baseline=-1em, x=1.5em, y=1.5em]
\footnotesize
\draw [dotted] (-1.5,0.5) -- (-1.5,0);
\draw [dotted] (-0.2,0.5) -- (-0.2,0);
\draw [dotted] (0.5,0.5) -- (0.5,0);
\draw [dotted] (1.5,0.5) -- (1.5,0);
\draw (0.5,0) to [out=-90,in=-90] (1.5,0);
\draw [fill = black] (1,-0.3) circle (.3ex);
\draw (-1.5,0) to [out=-90,in=-90] (-0.2,0);
\draw [fill = black] (-1.2,-0.35) circle (.3ex);
\draw [fill = black] (-0.5,-0.35) circle (.3ex);
\draw (-0.5,-0.35) to [out=-90,in=-90] (1,-0.3);
\draw (-1.2,-0.35) to [out=-90,in=90] (0,-1.5);
\node[below=1pt] at (0.1,-1.4) {$\scalebox{0.8}{$\cdots <c''_i<\cdots$}$};
\end{tikzpicture}
+\omega(c'_i,c_j)
\begin{tikzpicture}[thick, baseline=-1em, x=1.5em, y=1.5em]
\footnotesize
\draw [dotted] (-1.5,0.5) -- (-1.5,0);
\draw [dotted] (-0.2,0.5) -- (-0.2,0);
\draw [dotted] (0.5,0.5) -- (0.5,0);
\draw [dotted] (1.5,0.5) -- (1.5,0);
\draw (0.5,0) to [out=-90,in=-90] (1.5,0);
\draw [fill = black] (1,-0.3) circle (.3ex);
\draw (-1.5,0) to [out=-90,in=-90] (-0.2,0);
\draw [fill = black] (-1.2,-0.35) circle (.3ex);
\draw [fill = black] (-0.5,-0.35) circle (.3ex);
\draw (-0.5,-0.35) to [out=-90,in=90] (0,-1.5);
\fill [white] (-0.2,-0.9) circle [radius=0.18];
\draw (-1.2,-0.35) to [out=-90,in=-90] (1,-0.3);
\node[below=1pt] at (0.1,-1.4) {$\scalebox{0.8}{$\cdots <c''_i<\cdots$}$};
\end{tikzpicture}
\ = \ \omega(c'_i,c_j)\, \begin{tikzpicture}[thick, baseline=-0.5em, x=1.5em, y=1.5em]
\footnotesize
\draw [dotted] (-1.5,0.5) -- (-1.5,0);
\draw [dotted] (-0.2,0.5) -- (-0.2,0);
\draw [dotted] (0.5,0.5) -- (0.5,0);
\draw [dotted] (1.5,0.5) -- (1.5,0);
\draw (0.5,0) to [out=-90,in=-90] (1.5,0);
\draw [fill = black] (1,-0.3) circle (.3ex);
\draw (-1.5,0) to [out=-90,in=-90] (-0.2,0);
\draw [fill = black] (-0.85,-0.4) circle (.3ex);
\draw (-0.85,-0.4) to [out=-90,in=-90] (1,-0.3);
\draw [fill = black] (0,-0.9) circle (.3ex);
\draw (0,-0.9) to [out=-90,in=90] (0,-1.5);
\node[below=1pt] at (0.1,-1.4) {$\scalebox{0.8}{$\cdots <c''_i<\cdots$}$};
\end{tikzpicture},
\end{align*}
where the first identity follows from \eqref{eq:delta_i}
using that $D_i^U=E_i^U$, and the fourth identity
uses Lemma \ref{lem:expansion}.(c.i).
In the same way, we obtain
$$
\delta_j(D)-\delta_j(E)=
\omega(c'_j,c_i)\, \begin{tikzpicture}[thick, baseline=-0.5em, x=1.5em, y=1.5em]
\footnotesize
\draw [dotted] (-1.5,0.5) -- (-1.5,0);
\draw [dotted] (-0.2,0.5) -- (-0.2,0);
\draw [dotted] (0.5,0.5) -- (0.5,0);
\draw [dotted] (1.5,0.5) -- (1.5,0);
\draw (0.5,0) to [out=-90,in=-90] (1.5,0);
\draw [fill = black] (1,-0.3) circle (.3ex);
\draw (-1.5,0) to [out=-90,in=-90] (-0.2,0);
\draw [fill = black] (-0.85,-0.4) circle (.3ex);
\draw (-0.85,-0.4) to [out=-90,in=-90] (1,-0.3);
\draw [fill = black] (0,-0.9) circle (.3ex);
\draw (0,-0.9) to [out=-90,in=90] (0,-1.5);
\node[below=1pt] at (0.1,-1.4) {$\scalebox{0.8}{$\cdots <c''_j<\cdots$}$};
\end{tikzpicture}.
$$
Thus, claim \eqref{eq:ij} is deduced  
from  Lemma~\ref{lem:expansion}.(c.ii).
We now distinguish two cases:
\begin{itemize}
\item \emph{Assume that $i,j$ do not belong to a same $Y$-diagram component of $D$.}
Hence $G$ has no looped edge neither. On the one hand, 
the STU-like relation in the target implies that
$\delta_k(D)-\delta_k(E)=\omega(c_i,c_j)\, \delta_k(G)$
for any external vertex $k$ of $G$. Then, \eqref{eq:ij} implies that
\begin{eqnarray} \label{eq:delta_1/3}
&&\Big( \delta_i(D)+ \delta_j(D) + \sum_k\delta_k(D)\Big)
- \Big( \delta_i(E)+ \delta_j(E) + \sum_k\delta_k(E)\Big) \\
\notag 
&=&\omega(c_i,c_j) \sum_k\delta_k(G) \in \calA^<(P)\otimes \Q/\Z
\end{eqnarray}
where the sums are indexed by external vertices $k$ of $G$.
On the other hand, 
for any $Y$-diagram component $C$ of $D$ 
that contains neither $i$ nor $j$,
we obtain
\begin{equation} \label{eq:delta_2/3}
\delta_C(D)-\delta_C(E)= \omega(c_i,c_j)\, \delta_C(G)
\in \calA^<(P) \otimes \Z_2
\end{equation}
by application of the STU-like relation;
besides, if only one  of $i$ or $j$ belongs to $C$, say $i$, 
then $C$ does not count anymore as a $Y$-diagram of $G$ 
and we get
\begin{align} \label{eq:delta_3/3}
\delta_C(D)-\delta_C(E)&=
\begin{tikzpicture}[thick, baseline=-1.6em, x=1.5em, y=1.5em]
\footnotesize
\draw [dotted] (-1.8,0.5) -- (-1.8,0);
\draw [dotted] (-1.1,0.5) -- (-1.1,0);
\draw [dotted] (0.8,0.5) -- (0.8,0);
\draw [dotted] (0.1,0.5) -- (0.1,0);
\draw [dotted] (1.4,0.5) -- (1.4,0);
\draw (-1.8,0) to [out=-90,in=-90] (0.1,0);
\fill [white] (-0.55,-0.55) circle [radius=0.3];
\draw (-1.1,0) to [out=-90,in=-90] (0.8,0);
\draw [fill = black] (-1.4,-0.5) circle (.3ex);
\draw [fill = black] (0.4,-0.5) circle (.3ex);
\draw (-1.4,-0.5) -- (-1.4,-1.5);
\draw (0.4,-0.5) to [out=-90,in=90, looseness =1.5]  (0,-1.5);
\draw (1.4,0) -- (1.4,-1.5);
\node[below=1pt] at (0,-1.4) {$\scalebox{0.9}{$\cdots \hspace{1mm} c_i \ < \ c_i \ < c_j \cdots$}$};
\end{tikzpicture}
-\begin{tikzpicture}[thick, baseline=-2.2em, x=1.5em, y=1.5em]
\footnotesize
\draw [dotted] (-1.8,0.5) -- (-1.8,0);
\draw [dotted] (-1.1,0.5) -- (-1.1,0);
\draw [dotted] (0.8,0.5) -- (0.8,0);
\draw [dotted] (0.1,0.5) -- (0.1,0);
\draw [dotted] (1.4,0.5) -- (1.4,0);
\draw (-1.8,0) to [out=-90,in=-90] (0.1,0);
\fill [white] (-0.55,-0.55) circle [radius=0.3];
\draw (-1.1,0) to [out=-90,in=-90] (0.8,0);
\draw [fill = black] (-1.4,-0.5) circle (.3ex);
\draw [fill = black] (0.4,-0.5) circle (.3ex);
\draw (-1.4,-0.6) -- (-1.4,-0.8);
\draw (0.4,-0.6) -- (0.4,-0.8);
\draw (1.4,0) -- (1.4,-0.6);
\draw (0.4,-0.8) to [out=-90,in=90]  (1.4,-1.8);
\draw (-1.4,-0.8) to [out=-90,in=90]  (0,-1.8);
\fill [white] (-0.65,-1.28) circle [radius=0.28];
\fill [white] (0.65,-1.22) circle [radius=0.24];
\draw (1.4,-0.6) to [out=-90,in=90]  (-1.2,-1.8);
\node[below=1pt] at (0,-1.8) {$\scalebox{0.9}{$\cdots \hspace{1mm} c_j \ < \ c_i \ < \ c_i \cdots$}$};
\end{tikzpicture}\\
\notag &=
\omega(c_i,c_j)\left(
\begin{tikzpicture}[thick, baseline=-1.6em, x=1.5em, y=1.5em]
\footnotesize
\draw [dotted] (-1.8,0.5) -- (-1.8,0);
\draw [dotted] (-1.1,0.5) -- (-1.1,0);
\draw [dotted] (0.8,0.5) -- (0.8,0);
\draw [dotted] (0.1,0.5) -- (0.1,0);
\draw [dotted] (1.4,0.5) -- (1.4,0);
\draw (-1.8,0) to [out=-90,in=-90] (0.1,0);
\fill [white] (-0.55,-0.55) circle [radius=0.3];
\draw (-1.1,0) to [out=-90,in=-90] (0.8,0);
\draw [fill = black] (-1.4,-0.5) circle (.3ex);
\draw [fill = black] (0.4,-0.5) circle (.3ex);
\draw (1.4,-0.5) -- (1.4,0);
\draw (0.4,-0.5) to [out=-90,in=90]  (0,-1.9);
\fill [white] (0.1,-1.35) circle [radius=0.22];
\draw (-1.4,-0.5) to [out=-90,in=-90] (1.4,-0.5);
\node[below=1pt] at (0.1,-1.8) {$\scalebox{0.9}{$\cdots< c_i <\cdots$}$};
\end{tikzpicture}
+
\begin{tikzpicture}[thick, baseline=-1.6em, x=1.5em, y=1.5em]
\footnotesize
\draw [dotted] (-1.8,0.5) -- (-1.8,0);
\draw [dotted] (-1.1,0.5) -- (-1.1,0);
\draw [dotted] (0.8,0.5) -- (0.8,0);
\draw [dotted] (0.1,0.5) -- (0.1,0);
\draw [dotted] (1.4,0.5) -- (1.4,0);
\draw (-1.8,0) to [out=-90,in=-90] (0.1,0);
\fill [white] (-0.55,-0.55) circle [radius=0.3];
\draw (-1.1,0) to [out=-90,in=-90] (0.8,0);
\draw [fill = black] (-1.4,-0.5) circle (.3ex);
\draw [fill = black] (0.4,-0.5) circle (.3ex);
\draw (1.4,0) -- (1.4,-0.5);
\draw (1.4,-0.5) to [out=-90,in=-90] (0.4,-0.5);
\draw (-1.4,-0.5) to [out=-90,in=90]  (0,-1.8);\node[below=1pt] at (0.1,-1.8) {$\scalebox{0.9}{$\cdots< c_i <\cdots$}$};
\end{tikzpicture}
\right) \ = \ 0 ,
\end{align} 
where the first identity is obtained by canceling the two terms involving $\fr(c_i)$, 
the second identity uses the STU-like relation, and the third one
is obtained by canceling pairwise the $2\cdot 2^2=8$ terms arising from 
the other two  external vertices of $C$ (again, thanks to ``STU-like'').
Thus, \eqref{eq:aim_STU-like} 
is deduced from \eqref{eq:delta_1/3}--\eqref{eq:delta_3/3} 
using the formula \eqref{eq:delta}.
\item \emph{Assume that $i,j$  belong 
to a same $Y$-diagram component $T$  of $D$}, 
and denote by $t$ the third external vertex of $T$.
Hence $G$ has a looped edge, so that 
the identity \eqref{eq:aim_STU-like} to be shown reduces
to $\delta^<(D)=\delta^<(E)$.
For any external vertex $k$ of $D\setminus T$, we have
$\delta_k(D)=\delta_k(E)$ by applying ``STU-like'' and ``self-loop'';
similarly, for any $Y$-diagram component $C$ of $D\setminus T$, 
we have $\delta_C(D)=\delta_C(E)$. So, using the formula \eqref{eq:delta},
we deduce from \eqref{eq:ij} that 
\begin{equation} \label{eq:D_E}
\delta^<(D)-\delta^<(E) = \big(\delta_t(D) -\delta_t(E)\big)+
\frac{\delta_T(D) -\delta_T(E)}{2}
\ \in \calA^<(P)\otimes \Q/\Z.
\end{equation}
Without loss of generality, we can assume that $j<t$
so that, in a neighborhood of $T$, we have 
\[
\qquad \quad \quad
D=\begin{tikzpicture}[thick,baseline=-1.5em, x=1em, y=1em]
\footnotesize
\draw (-2.2,-0.5) to [out=30,in=90] (2.7,-2);
 \draw [fill = black] (-2.2,-0.5) circle (.3ex);
 \draw (-3,-2) -- (-2.2,-0.5) -- (-1.3,-2);
\node[below] at (0,-2){$\cdots<c_i<c_j<\cdots<c_t<\cdots$};
\end{tikzpicture}
\quad \hbox{and} \quad
E=\begin{tikzpicture}[thick,baseline=-1.5em, x=1em, y=1em]
\footnotesize
\draw (-2.2,-0.5) to [out=30,in=90] (2.7,-2);
 \draw [fill = black] (-2.2,-0.5) circle (.3ex);
 \draw (-3.1,-2) to [out=60,in=-45,looseness=2] (-2.2,-0.5);
 \fill [white] (-2.2,-1.25) circle [radius=0.2];
  \draw (-1.3,-2) to [out=120,in=225,looseness=2] (-2.2,-0.5);
\node[below] at (0.1,-2){$\cdots<c_j<c_i<\cdots<c_t<\cdots$};
\end{tikzpicture}.
\]
Using \eqref{eq:delta_i_bis} and, next, 
``STU-like'' and ``self-loop'', we get
\begin{eqnarray*}
&& \delta_t(D)-\delta_t(E) \\
&=&
\frac{1}{2}\begin{tikzpicture}[thick,baseline=-1.5em, x=1em, y=1em]
\footnotesize
\draw (-2.0,-0.5) to [out=60,in=90] (2.5,-2);
\draw (-2.8,-0.5) to [out=80,in=90,looseness=1] (3.7,-2);
\draw [fill = black] (-2.0,-0.5) circle (.3ex);
\draw [fill = black] (-2.8,-0.5) circle (.3ex);
\draw (-3.3,-2) to [out=90,in=90,looseness=3] (-1.5,-2);
\node[below] at (0.3,-2){$\cdots<c_i<c_j<\cdots<c_t||c_t<\cdots$};
\end{tikzpicture}
-
\frac{1}{2}\begin{tikzpicture}[thick,baseline=-1.5em, x=1em, y=1em]
\footnotesize
\draw (-2.0,-0.5) to [out=60,in=90] (2.5,-2);
\draw (-2.8,-0.5) to [out=80,in=90,looseness=1] (3.7,-2);
\draw [fill = black] (-2.0,-0.5) circle (.3ex);
\draw [fill = black] (-2.8,-0.5) circle (.3ex);
\draw (-2.0,-0.5) to [out=155,in=25,looseness=1.3] (-2.8,-0.5);
\draw (-2.0,-0.5) to [out=-25,in=70,looseness=1.6] (-3.3,-2);
\fill [white] (-2.4,-1.15) circle [radius=0.2];
\draw (-2.8,-0.5) to [out=205,in=110,looseness=1.6] (-1.5,-2);
\node[below] at (0.3,-2){$\cdots<c_j<c_i<\cdots<c_t||c_t<\cdots$};
\end{tikzpicture}\\
&=&
\frac{\omega(c_i,c_j)}{2}\begin{tikzpicture}[thick,baseline=-1.5em, x=1em, y=1em]
\footnotesize
\draw (0,0) circle (1 and 0.6);
\draw (0.9,-1.6) to (0.4,-0.55);
\draw (-0.9,-1.6) to (-0.4,-0.55);
\draw [fill = black] (0.4,-0.55) circle (.3ex);
\draw [fill = black] (-0.4,-0.55) circle (.3ex);
\node[below] at (-1,-1.5){$\cdots<\cdots<c_t\,\,||\,\, c_t<\cdots$};
\end{tikzpicture};
\end{eqnarray*}
similarly, we obtain
\begin{eqnarray*}
&& \delta_T(D)-\delta_T(E) \\
&=&
\begin{tikzpicture}[thick,baseline=-1.5em, x=0.8em, y=0.8em]
\scriptsize
\node[below] at (0.1,-2){$\cdots<c_i<c_i<c_j<c_j<\cdots<c_t||c_t<\cdots$};
\draw (-1.8,-0.5) to [out=30,in=110] (6.4,-2);
\draw [fill = black] (-1.8,-0.5) circle (.3ex);
\draw (-4.2,-2) -- (-1.8,-0.5) -- (0,-2);
\fill [white] (-2.9,-1.3) circle [radius=0.2];
\fill [white] (1.7,1.0) circle [radius=1.0];
\draw (-3.8,-0.5) to [out=30,in=110] (5,-2);
\draw [fill = black] (-3.8,-0.5) circle (.3ex);
\draw (-6.2,-2) -- (-3.8,-0.5) -- (-2,-2);
\end{tikzpicture}
-
\begin{tikzpicture}[thick,baseline=-1.5em, x=0.8em, y=0.8em]
\scriptsize
\node[below] at (0.1,-2){$\cdots<c_j<c_j<c_i<c_i<\cdots<c_t||c_t<\cdots$};
\draw (-2.3,0.5) to [out=30,in=110] (6.3,-2);
\draw [fill = black] (-2.3,0.5) circle (.3ex);
\draw (-4.2,-2) to [out=30,in=-20,looseness=1.8] (-2.3,0.5);
\draw (-6.4,-2) to [out=30,in=-20,looseness=1.8] (-3.5,0.5);
\fill [white] (0.0,1.0) circle [radius=1.0];
\fill [white] (-2.9,0.2) circle [radius=0.2];
\fill [white] (-2.9,-0.5) circle [radius=0.2];
\fill [white] (-2.2,-0.9) circle [radius=0.2];
\fill [white] (-3.8,-0.9) circle [radius=0.2];
\fill [white] (-2.7,-1.5) circle [radius=0.2];
\fill [white] (-3.1,-1.4) circle [radius=0.2];
\draw (-3.5,0.5) to [out=30,in=110] (4.8,-2);
\draw [fill = black] (-3.5,0.5) circle (.3ex);
\draw (0,-2) to [out=150,in=200,looseness=1.8] (-2.3,0.5);
\draw (-2,-2) to [out=150,in=200,looseness=1.8] (-3.5,0.5);
\end{tikzpicture}\\
&=&\omega(c_i,c_j)\!\!\!
\begin{tikzpicture}[thick,baseline=-1.5em, x=1em, y=1em]
\footnotesize
\draw (-2.0,-0.5) to [out=60,in=90] (2.5,-2);
\draw (-2.8,-0.5) to [out=80,in=90,looseness=1] (3.7,-2);
\draw [fill = black] (-2.0,-0.5) circle (.3ex);
\draw [fill = black] (-2.8,-0.5) circle (.3ex);
\draw (-3.3,-2) to [out=90,in=90,looseness=3] (-1.5,-2);
\node[below] at (0.3,-2){$\cdots<c_i<c_j<\cdots<c_t||c_t<\cdots$};
\end{tikzpicture}
-\omega(c_i,c_j)\!\!\!
\begin{tikzpicture}[thick,baseline=-1.5em, x=1em, y=1em]
\footnotesize
\draw (-2.0,-0.5) to [out=60,in=90] (2.5,-2);
\draw (-2.8,-0.5) to [out=80,in=90,looseness=1] (3.7,-2);
\draw [fill = black] (-2.0,-0.5) circle (.3ex);
\draw [fill = black] (-2.8,-0.5) circle (.3ex);
\draw (-3.3,-2) to [out=90,in=90,looseness=3] (-1.5,-2);
\node[below] at (0.3,-2){$\cdots<c_j<c_i<\cdots<c_t||c_t<\cdots$};
\end{tikzpicture}\\
&=&
\omega(c_i,c_j)\begin{tikzpicture}[thick,baseline=-1.5em, x=1em, y=1em]
\footnotesize
\draw (0,0) circle (1 and 0.6);
\draw (0.9,-1.6) to (0.4,-0.55);
\draw (-0.9,-1.6) to (-0.4,-0.55);
\draw [fill = black] (0.4,-0.55) circle (.3ex);
\draw [fill = black] (-0.4,-0.55) circle (.3ex);
\node[below] at (-1,-1.5){$\cdots<\cdots<c_t\,\,||\,\, c_t<\cdots$};
\end{tikzpicture}.
\end{eqnarray*}
Therefore,  we deduce from \eqref{eq:D_E} that $\delta^<(D)=\delta^<(E)$.
\end{itemize}

To conclude the proof of the proposition, 
and deal with \emph{the slide relation},
we consider a Jacobi diagram $D$ with two consecutive external vertices $i<j$,
with colors $s$ and $x$ respectively, 
which are adjacent to the same internal vertex.
Let $E$ be the same Jacobi diagram as $D$, 
but with the vertex $i$ colored by $x$ instead of $s$.
We need to show that $\delta^<(D)= \delta^<(E)$.
(The other version of  {the slide relation} with $i>j$ can be proved
in the same way, so that we shall omit it.)
We can assume that the connected component $V$ of $D$ 
containing $i$ and $j$ is a $Y$-diagram since, otherwise,
the slide relation can be written
as a combination of the special, IHX and STU-like relations
(which we have treated in the previous paragraphs).
Furthermore, having dealt with ``STU-like''
in the previous paragraph,
we can assume that  the third external vertex $k$ of $V$ 
(colored by $y\in P$)
is immediately next to $j$; in other words, we have 
$$
D  = D' \stackrel{< }{\sqcup}   
\underbrace{\begin{tikzpicture}[thick,color=black,baseline=-1.0em, x=1em, y=1em]
\footnotesize
\draw (0,0) -- (1.5,-1);
\draw (0,0) -- (0,-1);
\draw (0,0) -- (-1.5,-1);
\draw [fill = black] (0,0) circle (.3ex);
\node[below=1pt] at (0,-1) {$s < x < y $};
\end{tikzpicture} }_V \stackrel{< }{\sqcup}  D''
\quad \hbox{and} \quad
E  = D' \stackrel{< }{\sqcup}   
\underbrace{\begin{tikzpicture}[thick,color=black,baseline=-1.0em, x=1em, y=1em]
\footnotesize
\draw (0,0) -- (1.5,-1);
\draw (0,0) -- (0,-1);
\draw (0,0) -- (-1.5,-1);
\draw [fill = black] (0,0) circle (.3ex);
\node[below=1pt] at (0,-1) {$x < x < y $};
\end{tikzpicture} }_W \stackrel{< }{\sqcup}  D''.
$$
Of course, the slide relation in the target gives
\begin{equation} \label{eq:DE1}
\delta_l(D)=\delta_l(E) \in \calA^<(P) \otimes \Q/\Z
\quad \hbox{and} \quad
\delta_C(D)=\delta_C(E) \in \calA^<(P) \otimes \Z_2
\end{equation}
for any external vertex $l$ of $D'\sqcup D''$
 and any $Y$-diagram component $C$ of $D'\sqcup D''$.
Using the special and STU-like relations, we obtain
$$
\delta_V(V) = 
\tikz[baseline=5pt]{\footnotesize
  \draw[thick, color=black] (0,0)--(0.5,0.5);
  \draw[thick, color=black] (0.5,0)--(1,0.5);
  \fill [white] (0.75,0.25) circle [radius=0.1];
  \draw[thick, color=black] (0.5,0.5)--(1,0);
  \draw[thick, color=black] (1,0.5)--(1.5,0);
  \draw[black, fill=black] (0.5,0.5) circle (0.25ex);
   \draw[black, fill=black] (1,0.5) circle (0.25ex);
  \draw[thick, color=black] (0.5,0.5) to[out=90, in=90] (1,0.5);
  \draw[color=black] (0.75,-0.25) node {\small $x \parallel x  <y\parallel y$}
}+
\begin{tikzpicture}[thick,color=black,baseline=-1.0em, x=1em, y=1em]
\footnotesize
\draw (0,0) -- (1,-1);
\draw (0,0) -- (0,-1);
\draw (0,0) -- (-1,-1);
\draw [fill = black] (0,0) circle (.3ex);
\node[below=1pt] at (0,-1) {$s < x < y$};
\end{tikzpicture}
\stackrel{< }{\sqcup}   
\begin{tikzpicture}[thick,color=black,baseline=-1.0em, x=1em, y=1em]
\footnotesize
\draw (0,0) -- (1,-1);
\draw (0,0) -- (0,-1);
\draw (0,0) -- (-1,-1);
\draw [fill = black] (0,0) circle (.3ex);
\node[below=1pt] at (0,-1) {$s < x < y$};
\end{tikzpicture}
$$
and, using furthermore the slide relation, we get
$$
\delta_W(W) = \fr(x) \tikz[baseline=-3pt, scale=0.35]{\footnotesize
  \draw[thick] (0,0) circle (0.6);
  \draw[fill=black] (-0.6,0) circle (0.8ex);
  \draw[fill=black] (0.6,0) circle (0.8ex);
  \draw[thick] (-0.6,0) to[out=180,in=90] (-1.2,-1);
  \node at (-1.3,-1.4) {$y$};
  \draw[thick] (0.6,0) to[out=0,in=90] (1.2,-1);
  \node at (1.2,-1.4) {$y$};
  \node at (0,-1.4) {$\parallel$};
}+ 
\begin{tikzpicture}[thick,color=black,baseline=-1.0em, x=1em, y=1em]
\footnotesize
\draw (0,0) -- (1,-1);
\draw (0,0) -- (0,-1);
\draw (0,0) -- (-1,-1);
\draw [fill = black] (0,0) circle (.3ex);
\node[below=1pt] at (0,-1) {$s < x < y$};
\end{tikzpicture}
\stackrel{< }{\sqcup}   
\begin{tikzpicture}[thick,color=black,baseline=-1.0em, x=1em, y=1em]
\footnotesize
\draw (0,0) -- (1,-1);
\draw (0,0) -- (0,-1);
\draw (0,0) -- (-1,-1);
\draw [fill = black] (0,0) circle (.3ex);
\node[below=1pt] at (0,-1) {$s < x < y$};
\end{tikzpicture};
$$
therefore, we have
\begin{eqnarray}
\label{eq:DE2} \delta_{W}(E)-\delta_{V}(D)&=& D'\stackrel{< }{\sqcup}  \Bigg(
\tikz[baseline=5pt]{\footnotesize
  \draw[thick, color=black] (0,0)--(0.5,0.5);
  \draw[thick, color=black] (0.5,0)--(1,0.5);
  \fill [white] (0.75,0.25) circle [radius=0.1];
  \draw[thick, color=black] (0.5,0.5)--(1,0);
  \draw[thick, color=black] (1,0.5)--(1.5,0);
  \draw[black, fill=black] (0.5,0.5) circle (0.25ex);
   \draw[black, fill=black] (1,0.5) circle (0.25ex);
  \draw[thick, color=black] (0.5,0.5) to[out=90, in=90] (1,0.5);
  \draw[color=black] (0.75,-0.25) node {\small $x < x  <y < y$}}+
\fr(y) \tikz[baseline=-3pt, scale=0.35]{\footnotesize
  \draw[thick] (0,0) circle (0.6);
  \draw[fill=black] (-0.6,0) circle (0.8ex);
  \draw[fill=black] (0.6,0) circle (0.8ex);
  \draw[thick] (-0.6,0) to[out=180,in=90] (-1.2,-1);
  \node at (-1.3,-1.4) {$x$};
  \draw[thick] (0.6,0) to[out=0,in=90] (1.2,-1);
  \node at (1.2,-1.4) {$x$};
  \node at (0,-1.4) {$<$};
} \Bigg)\stackrel{< }{\sqcup}  D''\\
\notag &=& E_k^{II}(y,y)+ \fr(y)\, E_k^U 
\in \calA^<(P) \otimes \Z_2.
\end{eqnarray}
On the other hand, we deduce from \eqref{eq:delta_i_bis} 
and the special relation that
\[
\delta_i(D)=\frac{1}{2}D_i^U,\qquad
\delta_j(D)=\delta_k(D)=0.
\]
Furthermore, we  deduce from \eqref{eq:delta_i_bis} 
and the IHX relation that
$$
\delta_k(E)
=\frac{1}{2}E_k^{II}(y,y)+\frac{\fr(y)}{2}E_k^U.
$$
Moreover, a double application of  \eqref{eq:delta_i} gives 
\begin{align*}
\delta_i(E)+\delta_j(E)
&=\big( E_i^{II}(x',x'') +\upsilon(x)E_i^U\big)+\big( E_j^{II}(x',x'')+ \upsilon(x)E_j^U\big)\\
&= E_i^{II}(x',x'') + E_j^{II}(x',x'')
\ = \ \frac{1}{2}D_i^U;
\end{align*}
here, the second equality uses the facts that $E_i^U=E_j^U$ 
and $\upsilon(x)\in \big(\frac{1}{2} \Z\big)/\Z$,
while the third equality is justified by  ``STU-like'' and Lemma~\ref{lem:expansion}.(c).
It follows from the last three equalities that 
\begin{eqnarray} \label{eq:DE3}
&& \big(\delta_i(E)+\delta_j(E)+\delta_k(E)\big) - 
\big((\delta_i(D)+\delta_j(D)+\delta_k(D)\big) \\
\notag & = &\frac{1}{2}E_k^{II}(y,y)+\frac{\fr(y)}{2}E_k^U \ \in \calA^<(P)\otimes \Q/\Z.
\end{eqnarray}
By combining \eqref{eq:DE1}--\eqref{eq:DE3}, we conclude that $\delta^<(D)=\delta^<(E)$.
\end{proof}

The following theorem provides an alternative version and a mild generalization of \cite[Th$.$ 1.1]{NSS22a}. Throughout, we impose the standing assumption that \emph{the associator chosen for the construction of the LMO functor is even}. 
Recall moreover that this construction requires the choice of an arbitrary system of meridians and parallels on the surface $\Sigma$ (as described in \eqref{eq:surface}); such a choice determines a basis {$(\alpha,\beta)$} of the fundamental group $\pi$, and consequently induces a basis {$(a,b)$} of the homology group $H$. 
Let $\theta$ be an even symplectic expansion whose truncation up to degree $\leq 2$ is given by
\begin{equation} \label{eq:even_theta}
\begin{cases}
 \theta(\alpha_i) = 1 + a_i + \frac{1}{2} a_i^2 -\frac{1}{2}[a_i,b_i]
+ (\deg \geq 3),   \\
\theta(\beta_i) = 1 + b_i + \frac{1}{2} b_i^2 -\frac{1}{2}[a_i,b_i]
+ (\deg \geq 3).
\end{cases}
\end{equation}

\begin{theorem}
\label{th:psi_delta_Z}
For every $n\geq 1$,
we have the following commutative diagram:
\begin{equation} \label{eq:AGAA}
\begin{tikzcd}
\mathcal{A}^{<}_{n}(P) \arrow[rr, "\psi"] \arrow[d, "\delta^<"'] 
& &\mathrm{Gr}^{{F}}_n \, \mathbb{Z}[\mathcal{IC}] \arrow[d, "\zeta^{<}_{n+1}"] \\
\mathcal{A}^{<}_{n+1}(P) \otimes \Q/\Z
\arrow[rr,"\simeq","p_*" swap]
& & \mathcal{A}^{<}_{n+1}(H) \otimes \Q/\Z 
\end{tikzcd}
\end{equation}
\end{theorem}

\begin{proof}
We shall deduce the commutativity of \eqref{eq:AGAA}
from \cite[Th$.$ 1.1]{NSS22a} 
by relating our map $\delta^<$
to the map that is denoted by $\delta+ \mathbb{Y}$ in \cite{NSS22a}.
Let us first recall the definition of the latter.

The module $\calA(S)$ of Jacobi diagrams 
colored by $S=\{a_1,\dots,a_g,b_1\dots,b_g\}$
is defined at \eqref{eq:A(S)}, 
and it corresponds 
to  $\calA^Y (\lfloor g^+\rceil \cup \lfloor g^-\rceil)$
in \cite{CHM08,NSS22a} 
(or, to be exact, it corresponds 
to its quotient by the swap relation).
We denote by $(x\mapsto x^*)$ the involution of $S$
that exchanges $a_k$ with $b_k$ for every $k\in \{1,\dots g\}$.
Then, following \cite{NSS22a}, we consider
the map $\delta:\calA(S) \to \calA(S) \otimes \Z_2$ defined by 
$$
\delta(D) := 
\sum_i \big(  D_i^{II} +  D_i^{\Yrev} \big)
+ \sum_{\{i,j\}}  D_{ij}^Y
$$
on every Jacobi diagram $D$.
Here, the  first sum ranges over external vertices $i$ of $D$
with color $c_i$,
the second sum ranges over pairs $\{i,j\}$ 
of external vertices  sharing the same color $c_i=c_j$,
and we have used the following notations
(which are ``argument-free'' versions of some notations 
from \S \ref{subsec:delta^<}):
\begin{align*}
D = \begin{array}{c} \begin{tikzpicture}[thick,baseline=-3.0em, x=1em, y=1em]
\footnotesize
 \draw ([shift={(0,5)}]-80:6) arc [radius=6, start angle = -80, end angle=-100];
 \draw [dotted]([shift={(0,5)}]-80:6) arc [radius=6, start angle = -80, end angle= -70];
 \draw [dotted]([shift={(0,5)}]-100:6) arc [radius=6, start angle = -100, end angle= -110];
 \draw [fill = black] (0,-1) circle (.3ex);
 \draw (0,-1)-- (0,-2.5) node [below] {$c_i$};
\end{tikzpicture} \end{array}
&\leadsto
\begin{array}{c}  \begin{tikzpicture}[thick,baseline=-3.0em, x=1em, y=1em]
\footnotesize
 \draw ([shift={(0,5)}]-80:6) arc [radius=6, start angle = -80, end angle=-100];
 \draw [dotted]([shift={(0,5)}]-80:6) arc [radius=6, start angle = -80, end angle= -70];
 \draw [dotted]([shift={(0,5)}]-100:6) arc [radius=6, start angle = -100, end angle= -110];
 \draw [fill = black] (-0.8,-1) circle (.3ex);
 \draw [fill = black] (0.8,-1) circle (.3ex);
 \draw (-0.8,-1)-- (-0.8,-2.4);
 \node [below] at (0,-2.4) {$c_i\quad c_i$};
 \draw (0.8,-1)-- (0.8,-2.4);
\end{tikzpicture} \end{array} =:  D_{i}^{II}
\end{align*}
\begin{align*}
D=
\begin{array}{c} \begin{tikzpicture}[thick,baseline=-3.0em, x=1em, y=1em]
\footnotesize
 \draw ([shift={(0,5)}]-80:6) arc [radius=6, start angle = -80, end angle=-100];
 \draw [dotted]([shift={(0,5)}]-80:6) arc [radius=6, start angle = -80, end angle= -70];
 \draw [dotted]([shift={(0,5)}]-100:6) arc [radius=6, start angle = -100, end angle= -110];
 \draw [fill = black] (0,-1) circle (.3ex);
 \draw (0,-1)-- (0,-2.5) node [below] {$c_i$};
\end{tikzpicture} \end{array}
&\leadsto
\begin{array}{c}
\begin{tikzpicture}[thick,baseline=-2.5em, x=1em, y=1em]
\footnotesize]
 \draw ([shift={(0,5)}]-80:6) arc [radius=6, start angle = -80, end angle=-100];
 \draw [dotted]([shift={(0,5)}]-80:6) arc [radius=6, start angle = -80, end angle= -70];
 \draw [dotted]([shift={(0,5)}]-100:6) arc [radius=6, start angle = -100, end angle= -110];
 \draw [fill = black] (0,-1) circle (.3ex);
 \draw [fill = black] (0,-1.5) circle (.3ex);
 \draw (0,-1)-- (0,-1.5);
 \draw (0.4,-2) -- (0,-1.5) -- (-0.4,-2);
\node[below] at (0,-2) {$\quad c_i \quad (c_i)^*$};
\end{tikzpicture} \end{array}  =:   D_i^{\Yrev}
\end{align*}
\begin{align*}
  D= 
\begin{array}{c}
\begin{tikzpicture}[thick,baseline=-3.0em, x=1em, y=1em]
\footnotesize
 \draw ([shift={(0,5)}]-80:6) arc [radius=6, start angle = -80, end angle=-100];
 \draw [dotted]([shift={(0,5)}]-80:6) arc [radius=6, start angle = -80, end angle= -70];
 \draw [dotted]([shift={(0,5)}]-100:6) arc [radius=6, start angle = -100, end angle= -110];
 \draw [fill = black] (0,-1) circle (.3ex);
 \draw (0,-1)-- (0,-2.5) node [below] {$c_i$};
 \begin{scope}[shift={(5,0)}]
 \draw ([shift={(0,5)}]-80:6) arc [radius=6, start angle = -80, end angle=-100];
 \draw [dotted]([shift={(0,5)}]-80:6) arc [radius=6, start angle = -80, end angle= -70];
 \draw [dotted]([shift={(0,5)}]-100:6) arc [radius=6, start angle = -100, end angle= -110];
 \draw [fill = black] (0,-1) circle (.3ex);
  \draw (0,-1)-- (0,-2.5) node [below] {$c_j$};
 \end{scope}
\end{tikzpicture}
\end{array}
&\leadsto
\begin{array}{c}
\begin{tikzpicture}[thick,baseline=-3.0em, x=1em, y=1em]
\footnotesize
 \draw ([shift={(0,5)}]-80:6) arc [radius=6, start angle = -80, end angle=-100];
 \draw [dotted]([shift={(0,5)}]-80:6) arc [radius=6, start angle = -80, end angle= -70];
 \draw [dotted]([shift={(0,5)}]-100:6) arc [radius=6, start angle = -100, end angle= -110];
 \draw [fill = black] (0,-1) circle (.3ex);
 \begin{scope}[shift={(5,0)}]
 \draw ([shift={(0,5)}]-80:6) arc [radius=6, start angle = -80, end angle=-100];
 \draw [dotted]([shift={(0,5)}]-80:6) arc [radius=6, start angle = -80, end angle= -70];
 \draw [dotted]([shift={(0,5)}]-100:6) arc [radius=6, start angle = -100, end angle= -110];
 \draw [fill = black] (0,-1) circle (.3ex);
 \end{scope}
 \draw (0,-1) to [out=-50,in=-130, looseness=0.5] (5,-1);
\draw (2.5,-1.6)-- (2.5,-2.5) node [below] {$c_i$};
\draw [fill = black] (2.5,-1.6) circle (.3ex);
\end{tikzpicture}
\end{array}
=: D_{ij}^Y
\end{align*}
As for the map $\mathbb{Y}:\calA(S) \to \calA(S) \otimes \Z_2$,
it is defined in \cite{NSS22a} 
by transforming any Jacobi diagram $D$
to the sum of all ways of duplicating
a $Y$-diagram in $D$.
(In fact, the swap relation is not considered in \cite{NSS22a}
but it is easily verified that $\delta$ and $\mathbb{Y}$
are compatible with this relation.)

First of all, we claim that, for every $m\geq 1$,
the following diagram is commutative:
\begin{equation} \label{eq:delta_delta}
\xymatrix{
\calA_m(S) \ar[d]_-{\delta + \mathbb{Y}}
\ar[r]^-\varphi & \calA^<_m(P) \ar[d]^-{\delta^<} \\
\calA_{m+1}(S) \otimes \Z_2 \ar[r]_-{\varphi \otimes \frac{1}{2}} 
& \calA^<_{m+1}(P) \otimes \Q/\Z
}
\end{equation}
Here $\varphi$ denotes the restriction
of the map $\varphi: \calA(S) \otimes \Z[Y]/\langle 2Y \rangle \to \calA^<(P)$
given by  Proposition~\ref{prop:varphi}.
Indeed, let $D \in \calA_m(S)$ be a Jacobi diagram:
recall that $E:=\varphi(D)$ is obtained from $D$
by transforming every color $u\in S$ 
to  $\sigma(u) \in P$ and declaring 
that all the vertices colored 
by $\{\sigma(a_1),\dots, \sigma(a_g)\}$
are lower than those colored 
by $\{\sigma(b_1),\dots, \sigma(b_g)\}$.
The identity
$$
\frac{1}{2} \varphi\big(\mathbb{Y}(D)\big) = \frac{1}{2} \sum_C \delta_C(E)
\quad \in \calA_{m+1}^<(P) \otimes {\Q}/{\Z}
$$
(where the sum ranges over 
the $Y$-diagram components $C$ of $E$) 
follows easily from the definitions, 
using that $\fr:P\to \Z_2$
vanishes on the color of any external vertex  of $E$.
(Indeed, 
according to \eqref{eq:fr}, we have $\fr(\sigma(u))=0$
for every $u\in S$.)
Hence, according to \eqref{eq:delta_i_bis},
the commutativity of \eqref{eq:delta_delta}
is equivalent to the identity
\begin{eqnarray} 
\notag && \frac{1}{2} \sum_i E_i^{II}(c_i,c_i) + \sum_i E_i^{\Yrev}(\dot c_i, \ddot c_i) \\
\label{eq:delta_delta_bis} &=&
\frac{1}{2}  \sum_i \big( \varphi(  D_i^{II}) + \varphi(  D_i^{\Yrev}) \big)
+ \frac{1}{2}  \sum_{\{i,j\}} \varphi(   D_{ij}^{Y}) 
\quad \in \calA_{m+1}^<(P) \otimes {\Q}/{\Z}.
\end{eqnarray}
Of course, in  \eqref{eq:delta_delta_bis}, every external vertex $i$ of $D$ (colored by $c_i\in S$)
corresponds to an external vertex $i$ of $E$
(colored by $\sigma(c_i)\in P$ 
with projection $p_*(\sigma(c_i))=c_i\in H$).
Clearly, we  have 
$$
\frac{1}{2} E_i^{II}(c_i,c_i) = \frac{1}{2} \varphi(  D_i^{II}).
$$
Besides, it follows from \eqref{eq:even_theta} 
and the STU-like relation that 
$$
E_i^{\Yrev}(\dot c_i, \ddot c_i)  = 
\frac{1}{2} E_i^{\Yrev}(c_i, (c_i)^*) =   
\frac{1}{2} \varphi\big(  D_i^{\Yrev} \big)
+ \frac{1}{2}  \left\{\begin{array}{ll}
 \sum_{j>i} \varphi\big(   D_{ij}^{Y}\big) & \hbox{if $c_i\in \{a_1,\dots, a_g\}$}, \\[1mm]
 \sum_{j<i} \varphi\big(   D_{ij}^{Y}\big)  & \hbox{if $c_i\in \{b_1,\dots, b_g\}$},
\end{array}\right.
$$
where we use the total order on the set of external vertices of $E=\varphi(D)$.
Thus, we have proved the commutativity of \eqref{eq:delta_delta}.

Since the map $\varphi: \calA(S) \otimes \Z[Y]/\langle 2Y \rangle \to \calA^<(P)$ is surjective,
the  commutativity of \eqref{eq:AGAA} 
will follow from the identity 
\begin{equation} \label{eq:goal}
\zeta^<_{n+1} \psi \varphi(D \otimes Y^k)   
=  \delta^<  \varphi(D \otimes Y^k)
\end{equation}
for any $m,k\geq 0$ with $n=m+k$  and any $D\in \calA_m(S)$, where we have made implicit the isomorphism $p_*$
between $\calA^<(P)\otimes \Q/\Z$ and  $\calA^<(H)\otimes \Q/\Z$.
We compute the left-hand side of \eqref{eq:goal}: we obtain 
\begin{eqnarray*}
&& \zeta^<_{n+1} \psi \varphi(D \otimes Y^k)   \\
&=&  \pm \ \zeta^<_{n+1} \psi \Big(\varphi(D) \stackrel{< }{\sqcup}  
\begin{tikzpicture}[thick,color=black,baseline=-1.0em, x=1em, y=1em]
\footnotesize
\draw (0,0) -- (1,-1);
\draw (0,0) -- (0,-1);
\draw (0,0) -- (-1,-1);
\draw [fill = black] (0,0) circle (.3ex);
\node[below=1pt] at (0,-1) {$s < s < s$};
\end{tikzpicture}
\stackrel{< }{\sqcup}   \cdots \stackrel{< }{\sqcup}  
\begin{tikzpicture}[thick,color=black,baseline=-1.0em, x=1em, y=1em]
\footnotesize
\draw (0,0) -- (1,-1);
\draw (0,0) -- (0,-1);
\draw (0,0) -- (-1,-1);
\draw [fill = black] (0,0) circle (.3ex);
\node[below=1pt] at (0,-1) {$s < s < s$};
\end{tikzpicture}\Big) \\
&=&  \pm \ Z^<_{n+1} \big(
\psi\varphi(D) \circ (U_G-U)^k \big) \!\!\mod 1 \\
&=&  \pm \ \hbox{deg$.$ ($n$+1) part of }
\Big(\varphi^\Q \widetilde{Z}^Y\big(
\psi\varphi(D)\big) \stackrel{< }{\sqcup}  \big(\varphi^\Q \widetilde{Z}^Y (U_G-U) \big)^k\Big) \!\!\mod 1 
\end{eqnarray*}
where $G$ denotes a $Y$-graph with three special leaves,
and the last identity uses Remark \ref{rem:bialgebra}.
Note that $U_G$ is the connected sum of $U$
and the Poincar\'e $3$-sphere, hence
$$
\varphi^\Q \widetilde{Z}^Y (U_G-U)  =   \pm \frac{1}{2} \thetagraph + (\deg \geq 3).
$$
So, we deduce from \cite[Th$.$ 7.11]{CHM08} 
and \cite[Th$.$ 1.1]{NSS22a} that
\begin{equation} \label{eq:lll}
\zeta^<_{n+1} \psi \varphi(D \otimes Y^k)   =
\left\{\begin{array}{ll}
0 & \hbox{if $k\geq 2$,} \\
\pm \frac{1}{2}\varphi (D) \stackrel{<}{\sqcup}   \thetagraph &  \hbox{if $k=1$,}\\
 \pm \frac{1}{2} \varphi \big((\delta+ \mathbb{Y}) (D)\big) &   \hbox{if $k=0$.}
\end{array}\right.
\end{equation}
In particular, since the above values in $\calA^<(P)\otimes \Q$ belong to $\calA^<(P)\otimes \frac{1}{2} \Z$,
the above sign ambiguities are not relevant, and we can remove them. 

We now compute  the right-hand side of \eqref{eq:goal}, 
observing firstly that
$$
\frac{1}{2}
\begin{tikzpicture}[thick,baseline=-1.0em, x=2em, y=2em]
\hspace{-4mm}
\footnotesize
\node[below=1pt] at (0,-1){\quad \quad   $ s \parallel s \quad  < \qquad s \parallel s \qquad  < \quad s \parallel s $};
 \draw [fill = black] (0,0) circle (.3ex);
 \draw (0,0)-- (-2.6,-1);
\draw (0,0)-- (0,-1);
\draw (0,0)-- (2.6,-1);
\fill [white] (0.4,-0.12) circle [radius=0.1];
\fill [white] (0.0,-0.3) circle [radius=0.1];
\fill [white] (0.8,-0.3) circle [radius=0.1];
\draw [fill = black] (0.8,0) circle (.3ex);
\draw (0.8,0)-- (-1.8,-1);
\draw (0.8,0)-- (0.8,-1);
\draw (0.8,0)-- (3.4,-1);
\end{tikzpicture} = \frac{1}{2} \thetagraph
\quad \in \calA^<(P) \otimes \Q/\Z.
$$
Hence, we obtain
\begin{eqnarray}
 \notag  \delta^< \varphi(D \otimes Y^k)   
&=&   \delta^< \Big(\varphi(D) \stackrel{< }{\sqcup}  
\begin{tikzpicture}[thick,color=black,baseline=-1.0em, x=1em, y=1em]
\footnotesize
\draw (0,0) -- (1,-1);
\draw (0,0) -- (0,-1);
\draw (0,0) -- (-1,-1);
\draw [fill = black] (0,0) circle (.3ex);
\node[below=1pt] at (0,-1) {$s < s < s$};
\end{tikzpicture}
\stackrel{< }{\sqcup}   \cdots \stackrel{< }{\sqcup}  
\begin{tikzpicture}[thick,color=black,baseline=-1.0em, x=1em, y=1em]
\footnotesize
\draw (0,0) -- (1,-1);
\draw (0,0) -- (0,-1);
\draw (0,0) -- (-1,-1);
\draw [fill = black] (0,0) circle (.3ex);
\node[below=1pt] at (0,-1) {$s < s < s$};
\end{tikzpicture}\Big) \\
\label{eq:rrr} &=& \left\{\begin{array}{ll}
0 & \hbox{if $k\geq 2$} \\
 \frac{1}{2}  \varphi (D) \stackrel{<}{\sqcup}   \thetagraph  & \hbox{if $k=1$} \\
 \delta^< \varphi(D)  & \hbox{if $k=0$}.
\end{array}\right.
\end{eqnarray}
By comparing \eqref{eq:lll} and \eqref{eq:rrr} using \eqref{eq:delta_delta}, we obtain \eqref{eq:goal}.
\end{proof}

\begin{remark}
The methods developed in \cite{NSS22a}, which are based on explicit computations of the LMO functor, can be extended to determine
$$
\zeta^<_{n+1}([U;C]) \in \calA_{n+1}^<(H)\otimes \Q/\Z
$$
for any graph clasper $C\subset U$.
Using this strategy, one obtains a direct proof of Theorem~\ref{th:psi_delta_Z} and Proposition~\ref{prop:delta}.
In addition, this viewpoint makes the role of the truncated even symplectic expansion \eqref{eq:even_theta} more transparent:
this actually coincides with the truncation of the symplectic expansion produced by the LMO functor \cite{Mas12}, 
once an even associator has been fixed.
(It is likely that a general even symplectic expansion
cannot always be expressed in the specific form \eqref{eq:even_theta}
for an appropriate choice of curves $(\alpha, \beta)$.)
\hfill $\blacksquare$
\end{remark}

\subsection{Properties of the map $\delta^<$}

The first property of $\delta^<$ is easily deduced
from its definition:

\begin{proposition} \label{prop:derivation}
The map $\delta^<$ is a derivation
of the algebra $\calA^<(P)$ 
in the $\calA^<(P)$-module $\calA^<(P)\otimes \Q/\Z$.
\end{proposition}

\begin{proof}
Let $D,E\in \calA^<(P)$ be Jacobi diagrams.
By Proposition \ref{prop:delta}, we have
$$
\delta^<\big(D\stackrel{<}{\sqcup} E\big)
= \sum_{i} \delta_i(D\stackrel{<}{\sqcup}E)
+ \sum_{C} \frac{\delta_C (D\stackrel{<}{\sqcup}E)}{2}   
$$
where $i$ ranges over external vertices of 
$D\stackrel{<}{\sqcup}E$ and 
$C$ ranges over $Y$-components of $D\stackrel{<}{\sqcup}E$.
Clearly, we have
$$
\delta_i(D\stackrel{<}{\sqcup}E)
= \left\{\begin{array}{ll}
\delta_i(D)\stackrel{<}{\sqcup}E  & \hbox{if $i\in D$,}\\
D\stackrel{<}{\sqcup}\delta_i(E)  & \hbox{if $i\in E$,}
\end{array}\right.
$$
and, similarly, we have 
$$
\delta_C(D\stackrel{<}{\sqcup}E)
= \left\{\begin{array}{ll}
\delta_C(D)\stackrel{<}{\sqcup}E  & \hbox{if $C\subset D$,}\\
D\stackrel{<}{\sqcup}\delta_C(E)  & \hbox{if $C\subset E$.}
\end{array}\right.
$$
We deduce that 
$\delta^<\big(D\stackrel{<}{\sqcup} E\big)=
\delta^<(D) \stackrel{<}{\sqcup} E+
D\stackrel{<}{\sqcup} \delta^<(E).
$
\end{proof}

Recall that the map $\delta^<=\delta^<_\theta$
depends on the degree $\leq 2$ truncation 
of an even symplectic expansion $\theta$.
Consequently, $\delta^<: \calA^<(P) \to \calA^<(P) \otimes \Q/\Z$ \emph{is not} $\Sp(H)$-equivariant 
with respect to the natural action of $\Sp(H)$ on $P$ 
(see Lemma~\ref{lem:P}.(c)).
Yet, we can control the defect of $\Sp(H)$-equivariance
in terms of the truncated symplectic expansion.
To state this second property of $\delta^<$,
we  need a homomorphism
$$
\tau=\tau_\theta: \calM \longrightarrow 
 \Lambda^3(H^\Q) \rtimes  \Sp(H)
$$
which depends on (the degree $\leq 2$ truncation of) $\theta$
and encodes the action of the mapping class group $\calM=\calM(\Sigma)$ 
on the second nilpotent quotient $\pi/\Gamma_3 \pi$ of $\pi$.

Indeed recall from \cite{Mas12} that, 
using any symplectic expansion $\theta$,
one can turn the Dehn--Nielsen representation of $\calM$
(i.e., its canonical action on $\pi$)
 into a homomorphism
$$
\varrho^\theta:
\calM \longrightarrow \Aut_\omega(\widehat{\mathfrak{L}}^\Q)
$$
with values in the group of automorphisms
(which fix $\omega\in \mathfrak{L}_2^\Q$) 
of the degree-completed  free Lie algebra 
$\widehat{\mathfrak{L}}^\Q=\widehat{\mathfrak{L}}(H^\Q)$.
Then, for any $f\in \calM$, we set
$$ 
\tau(f) := \big(f_\circledast, f_* \big)
\in  \Lambda^3(H^\Q) \rtimes  \Sp(H)
$$
where $f_*\in \Sp(H)$ is the action induced by $f$
in homology, and $f_\circledast$ is the degree $1$
part of the symplectic derivation
$$
\log\big(\varrho^\theta(f)\circ f_*^{-1}\big)  
\in \Der_\omega\big(\widehat{\mathfrak{L}}^\Q\big),
$$
which is viewed as a trivector 
through the isomorphism \eqref{eq:symp_cond}.

\begin{remark}
The map $\tau$ is a variant of Morita's 
``extension of the Johnson homomorphism''
\cite{Mor93a}, and it  appears in \cite{MS20} 
under the notation $\widetilde{\mathsf{Ab}}^\theta_1$.
The restriction of $\tau$ to $\calI$ is
the first Johnson homomorphism $\tau_1$ recalled at \eqref{eq:tau_k}.
\hfill $\blacksquare$
\end{remark}
In order to formulate the equivariance defect of $\delta^<$, we also introduce the bilinear map
\[
\varsigma : \frac{1}{2}\Lambda^3 H \times \mathcal{A}^<(H) \longrightarrow
\mathcal{A}^<(H) \otimes \mathbb{Q}/\mathbb{Z} \simeq \mathcal{A}^<(P) \otimes \mathbb{Q}/\mathbb{Z}
\]
defined, for arbitrary $x,y,z \in H$ and any Jacobi diagram $D \in \mathcal{A}^<(H)$, by
\[
\varsigma\!\big(\tfrac{1}{2} \, x \wedge y \wedge z ,D\big) = 
\tfrac{1}{2} \sum_i \left(\omega(x,c_i)\, 
D_i^{\Yrev}(y,z) \;+\; \omega(y,c_i)\, D_i^{\Yrev}(z,x)
\;+\; \omega(z,c_i)\, D_i^{\Yrev}(x,y)\right)
\]
where the sum ranges over all external vertices $i$ of $D$ with color $c_i \in H$, and the notation $D_i^{\Yrev}$ is as in \eqref{eq:Yrev}.  
(One verifies without difficulty that $\varsigma$ is well-defined.)

\begin{lemma} \label{lem:equivariance}
Let $\theta$ be an even symplectic expansion of $\pi$.
Then, $\tau$ takes values 
in $ \frac{1}{2} \Lambda^3 H\rtimes \Sp(H)$
and, for any $f\in \calM$ and $D \in \calA^<(P)$, we have
\begin{equation} \label{eq:equivariance}
\delta^<(f_* D) = f_* \delta^<(D) 
+ \varsigma( f_\circledast, f_*p_* D)
\end{equation}
where we have denoted $\tau(f)= ( f_\circledast,f_*)$.
\end{lemma}

\begin{proof}
Let $f\in \calM$: we view $f_\circledast \in \Lambda^3 H^\Q$ 
as an element of $\Hom(H^\Q,\frakL_2^\Q)$
via the isomorphism~\eqref{eq:symp_cond}. Then, by definition of $f_\circledast$, we have
$$
\big(\varrho^\theta(f) f_*^{-1}\big) (u) 
= u_1+u_2 + f_\circledast(u_1) + (\deg \geq 3)
$$
for any $u\in\frakL^\Q$ (whose degree $k$ part is denoted by $u_k$).
On the other hand, by definition of $\varrho^\theta(f)$, we have
$$
\varrho^\theta(f)\big(\ell^\theta(x)\big)= \ell^\theta(f(x)) 
= [f(x)] + \ell_2^\theta(f(x)) + (\deg \geq 3)
$$
for all $x\in \pi$. Hence, taking $u := f_* \ell^\theta(x)$ for any $x\in \pi$, we obtain
\begin{equation} \label{eq:ell_ell}
f_* \ell_2^\theta(x) + f_\circledast(f_*([x]))= \ell_2^\theta(f(x)) 
\ \in \frakL_2^\Q.    
\end{equation}
Since $\ell_2^\theta(\pi) \subset \frac{1}{2} \frakL_2$ by our assumption on $\theta$,
we obtain the first statement of the lemma.

We now prove \eqref{eq:equivariance} for any $f\in \calM$ and $D\in \calA_n^<(P)$. 
Using  \eqref{eq:delta} and \eqref{eq:delta_i_bis}, we get
\begin{eqnarray*}
\delta^<(f_* D)- f_* \delta^<(D)    &=&
\sum_i \Big((f_*D)_i^{\Yrev}(\dot e_i, \ddot e_i) -  f_*\big(D_i^{\Yrev}(\dot c_i, \ddot c_i)\big) \Big)\\
&& + \sum_i \frac{\fr(e_i)-\fr(c_i)}{2} f_*(D_i^U) 
+ \sum_C\Big( \frac{\delta_C(f_*D)}{2} - \frac{f_*\delta_C(D) }{2} \Big).
\end{eqnarray*}
The third sum ranges over $Y$-diagram components $C$ of $D$
(or, equivalently, of $f_*D$)
and it is zero because of the $\Sp(H)$-equivariance of $\fr$.
The second sum ranges over external vertices $i$ of $D$
(or, equivalently, of $f_*D$) with color $c_i\in P$.
We have set $e_i:=f_*(c_i) \in P$ 
and, by the $\Sp(H)$-equivariance of $\fr$,
the second sum  is zero too. We now look at the first sum:
here, for every external vertex $i$ of $D$,
we use as before Sweedler's notations:
$$
 \lambda(p_*(c_i))=[\dot c_i , \ddot c_i ] \in \mathfrak{L}_2 \otimes \Q/\Z ,
 \quad  \lambda(p_*(e_i))=[\dot e_i , \ddot e_i ]
\in \mathfrak{L}_2 \otimes \Q/\Z
$$
But, \eqref{eq:ell_ell} implies that 
$f_*([\dot c_i , \ddot c_i ])-[\dot e_i , \ddot e_i ]
= -f_\circledast(p_*(e_i))$. Thus the first sum is equal to 
$$
\sum_i \sum_j  (f_*D)_i^{\Yrev}(u_{i,j},v_{i,j})
\quad \hbox{where } 
f_\circledast(p_*(e_i)) = \sum_j [u_{i,j},v_{i,j}]
$$
which proves property \eqref{eq:equivariance}.
\end{proof}

Lemma \ref{lem:equivariance} implies
that an even symplectic expansion $\theta$ induces a $1$-cocycle
\begin{equation} \label{eq:tau_reduced}
\Sp(H) \simeq \calM/\calI \longrightarrow
\big(\tfrac{1}{2} \Lambda^3 H\big)\big/ \Lambda^3 H, \
\{ f \} \longmapsto \{f_\circledast\}.
\end{equation}
Thus, we get an action of  $\Sp(H)$
on the abelian group 
$\calA^<(P) \oplus \big(\calA^<(P) \otimes \Q/\Z\big)$ 
defined by
\begin{equation} \label{eq:Sp-action}
\{f\} \cdot (D,E) := 
\big(f_*D,  f_* E + \varsigma(f_\circledast,f_*p_*D)\big),
\end{equation}
for any $f\in \calM$,
$D \in \calA^<(P)$ and $E \in \calA^<(P)\otimes \Q/\Z$.

\begin{proposition} \label{prop:equivariance}
Let $\theta$ be an even symplectic expansion.
The homomorphism
$$
\xymatrix{
\calA^<(P) \ar[rr]^-{(\id,\delta^<)} &&
\calA^<(P) \, \oplus \,  \big(\calA^<(P) \otimes \Q/\Z\big)
}
$$
is $\Sp(H)$-equivariant with respect to
the action defined by \eqref{eq:Sp-action}.
\end{proposition}

\begin{proof}
This follows easily from \eqref{eq:equivariance}.
\end{proof}

We conclude this section 
with a direct  consequence of 
Proposition \ref{prop:equivariance}.
Recall that $\bqm\calA^<(H)\eqm$ denotes the natural image of 
$\calA^<(H)$ in $\calA^<(H)\otimes \Q$.
It is easily seen that \eqref{eq:Sp-action} induces an 
action of $\Sp(H)$ on 
$\bqm \calA^<(H) \eqm \oplus\,  
\big(\calA^<(H)\!\otimes\!\Q/\Z\big)$.

\begin{corollary} \label{cor:equivariance}
Let $\theta$ be an even symplectic expansion 
of the type \eqref{eq:even_theta}.
Then, for any $n\geq 1$,  the homomorphism
$$
\xymatrix{
{\displaystyle\frac{F_n\Z[\calIC]}{F_{n+1}\Z[\calIC]} }
\ar[rr]^-{(Z_n^<,\zeta^<_{n+1})}  
&& \hbox{``\!$\calA^<_n(H)$\!''}
\, \oplus \,  \big(\calA_{n+1}^<(H)\!\otimes\!\Q/\Z\big)
}
$$
is $\Sp(H)$-equivariant.
\end{corollary}

\begin{proof}
According to \eqref{eq:psi_Z}
and \eqref{eq:AGAA}, we have the commutative diagram
$$
\xymatrix{
\calA_n^<(P) \ar[r]^-\psi
\ar[rd]_-{(p_*,\delta^<)}
&\frac{F_n\Z[\calIC]}{F_{n+1}\Z[\calIC]} 
\ar[d]^-{(Z_n^<,\zeta^<_{n+1})}
\\
& \bqm\calA^<_n(H)\eqm
\, \oplus \,  \big(\calA_{n+1}^<(H)\!\otimes\!\Q/\Z\big),
}
$$
where the isomorphism 
$p_*$ between $ \calA^<(H) \otimes\Q/\Z$
and $ \calA^<(P) \otimes\Q/\Z$ is implicit.
Thus, the result follows
from Proposition \ref{prop:equivariance}
and the surjectivity of $\psi$.
\end{proof}

\begin{remark}
Proposition \ref{prop:derivation} for our map $\delta^{<}$ is the analogue of \cite[Prop.~3.3]{NSS22a} for the map $\delta$ defined there. 
By contrast, an analogue of Proposition~\ref{prop:equivariance} seems to be  difficult to obtain for $\delta$; 
this is one of the motivations 
for introducing our map $\delta^{<}$. \hfill $\blacksquare$
\end{remark}

\section{Order two torsion in odd degrees}
\label{sec:odd_degrees}

In this section, we consider the order $2$ torsion
in the Lie algebra of homology cylinders  $\Gr^Y\!\!\calIC$ 
arising from surgeries along symmetric graph claspers.

\subsection{The space of beaded rooted Jacobi diagrams}

A Jacobi diagram $D$
is \emph{beaded}  if it comes with a group homomorphism
$H_1(D) \to \Z_2$. Graphically, the homomorphism
is encoded by decorating some edges of $D$ with beads \bead 
representing the generator of $\Z_2$: the value of the homomorphism on a cycle is the mod-2 sum of the beads encountered along that cycle.
A Jacobi diagram $D$ is \emph{rooted} 
if one external vertex (the \emph{root})
is distinguished from the others (the \emph{leaves}):
graphically, the root is shown with a~$\star$ on figures.

\begin{example}
\hphantom{.}\\
\begin{minipage}{0.2\textwidth}
$D:=\begin{array}{c}
 \tikz[baseline=-2pt, scale=0.8]{
  \draw[thick, color=black] (0,0) circle (0.5);
  \draw[thick, color=black] (-0.5,0) -- (0.5,0);
\draw[thick, color=black] (0,0.5) -- (0,0.85);
\draw [thick] (-0.5*0.707,-0.5*0.707) to [out=-140,in=90] (-0.5,-0.85);
\draw [thick] (0.5*0.707,-0.5*0.707) to [out=-50,in=90] (0.5,-0.85);
  \draw[black, fill=black] (-0.5,0) circle (0.4ex);
  \draw[black, fill=black] (0.5,0) circle (0.4ex);
  \draw[black, fill=black] (0.5*0.707,-0.5*0.707) circle (0.4ex);
  \draw[black, fill=black] (-0.5*0.707,-0.5*0.707) circle (0.4ex);
  \draw[black, fill=black] (0,0.5) circle (0.4ex);
  \draw[black, fill=lightgray] (0,0) circle (0.6ex);
   \draw[black, fill=lightgray] (0,-0.5) circle (0.6ex);
    \draw[black, fill=lightgray] (0.5*0.707,0.5*0.707) circle (0.6ex);
   \draw[black, fill=lightgray] (-0.5*0.707,0.5*0.707) circle (0.6ex);
\node[above =1pt] at (0,0.75) {$\star$};
}
\end{array}$
\end{minipage}
\begin{minipage}{0.785\textwidth}
In this beaded rooted Jacobi diagram $D$,
the homomorphism $H_1(D) \to \Z_2$ is trivial 
on the lower loop, and non-trivial on the upper loop.
\hfill $\blacksquare$
\end{minipage}
\end{example}

Let $P^{(2)} := P\otimes \Z_2 = 
H_1(\operatorname{F}(U);\Z_2)$.
We consider the $\Z_2$-vector space
$$
\calB^{<} \big(P^{(2)}\big) = 
\frac{\Z \left\{\parbox{3.5in}{\centering 
\text{beaded rooted connected Jacobi diagrams with }
\text{$P^{(2)}$-colored external vertices 
and totally-ordered leaves}}\right\}}{
\text{AS, IHX, multilinearity, self-loop, BSTU-like,  group, fully-special}}.
$$ 
Here ``AS'', ``IHX'', ``multilinearity'' and ``self-loop'' 
are as before 
(with the requirement that no bead should appear in the pictures defining locally those relations);
the BSTU-like relation 
is the following variant of the STU-like relation: \\
$\begin{tikzpicture}[baseline=-0.9em, x=2em, y=2em]
\footnotesize
\draw[thick,dotted] (0,0.5) -- (0,0);
\draw [thick] (0,0) -- (0,-1);
\draw[thick,dotted] (1,0.5) -- (1,0);
\draw [thick] (1,0) -- (1,-1);
\node[below=1pt] at (0.5,-1) {$\cdots<x<y<\cdots$};
\end{tikzpicture}
-
\begin{tikzpicture}[baseline=-0.9em, x=2em, y=2em]
\footnotesize
\draw[thick,dotted] (0,0.5) -- (0,0);
\draw[thick,dotted] (1,0.5) -- (1,0);
\draw [thick] (0,0) to [out=-90,in=90] (1,-1);
\fill [white] (0.5,-0.5) circle [radius=0.2];
\draw [thick] (1,0) to [out=-90,in=90] (0,-1);
\node[below=1pt] at (0.5,-1) {$\cdots<y<x<\cdots$};
\end{tikzpicture}
= {\hbox{\footnotesize $\omega\big(p_*(x),p_*(y)\big)$\ }} 
\Bigl(\begin{tikzpicture}[baseline=-0.9em, x=2em, y=2em]
\draw[thick, dotted] (0,0.5) -- (0,0);
\draw[thick, dotted] (1,0.5) -- (1,0);
\draw [thick] (0,0) -- (0,-0.5);
\draw [thick] (1,0) -- (1,-0.5);
\draw [thick] (0,-0.5) to [out=-90,in=-90] (1,-0.5) ;
\node[below=1pt] at (-0.3,-0.8) {$\cdots$};
\node[below=1pt] at (1.4,-0.8) {$\cdots$};
\end{tikzpicture}
+\begin{tikzpicture}[baseline=-0.9em, x=2em, y=2em]
\draw[thick,dotted] (0,0.5) -- (0,0);
\draw[thick, dotted] (1,0.5) -- (1,0);
\draw [thick] (0,0) -- (0,-0.5);
\draw [thick] (1,0) -- (1,-0.5);
\draw [thick] (0,-0.5) to [out=-90,in=-90] (1,-0.5) ;
\node[below=1pt] at (-0.3,-0.8) {$\cdots$};
\node[below=1pt] at (1.4,-0.8) {$\cdots$};
\draw[black, fill=lightgray] (0.5,-0.8) circle (0.6ex);
\end{tikzpicture}\Bigr)$\\
for any $x,y\in P^{(2)}$;
``group'' is the set of relations telling that beads 
encode a group homomorphism in $\Z_2$:\\
$\begin{tikzpicture}[baseline=-0.25em, x=2em, y=2em]
\draw[thick, dotted] (0,0) -- (0.5,0);
\draw [thick] (0.5,0) -- (2,0);
\draw[thick, dotted] (2,0) -- (2.5,0);
\draw[black, fill=lightgray] (1,0) circle (0.6ex);
\draw[black, fill=lightgray] (1.5,0) circle (0.6ex);
\end{tikzpicture}
\quad=\quad
\begin{tikzpicture}[baseline=-0.25em, x=2em, y=2em]
\draw[thick, dotted] (0,0) -- (0.5,0);
\draw [thick] (0.5,0) -- (2,0);
\draw[thick, dotted] (2,0) -- (2.5,0);
\end{tikzpicture}$,\qquad
$\begin{tikzpicture}[baseline=-0.25em, x=2em, y=2em]
\draw[thick, dotted] (0,0) -- (0.5,0);
\draw [thick] (0.5,0) -- (1.5,0);
\node at (1.5,0) {\small $\bullet$};
\draw [thick] (2.1,0.6) -- (1.5,0) -- (2.1,-0.6);
\draw[thick, dotted] (2.1,0.6) -- (2.4,0.9);
\draw[thick, dotted] (2.1,-0.6) -- (2.4,-0.9);
\draw[black, fill=lightgray] (1,0) circle (0.6ex);
\end{tikzpicture}
=
\begin{tikzpicture}[baseline=-0.25em, x=2em, y=2em]
\draw[thick,dotted] (0,0) -- (0.5,0);
\draw [thick] (0.5,0) -- (1.5,0);
\node at (1.5,0) {\small $\bullet$};
\draw [thick] (2.1,0.6) -- (1.5,0) -- (2.1,-0.6);
\draw[thick, dotted] (2.1,0.6) -- (2.4,0.9);
\draw[thick, dotted] (2.1,-0.6) -- (2.4,-0.9);
\draw[black, fill=lightgray] (1.8,0.3) circle (0.6ex);
\draw[black, fill=lightgray] (1.8,-0.3) circle (0.6ex);
\end{tikzpicture}$,\\
$\begin{tikzpicture}[baseline=-0.9em, x=2em, y=2em]
\draw[thick, dotted] (0,0.5) -- (0,0);
\draw [thick] (0,0) -- (0,-1);
\node[below=1pt] at (0,-1) {\footnotesize $\cdots<x<\cdots$};
\draw[black, fill=lightgray] (0,-0.5) circle (0.6ex);
\end{tikzpicture}
=\begin{tikzpicture}[baseline=-0.9em, x=2em, y=2em]
\draw[thick, dotted] (0,0.5) -- (0,0);
\draw [thick] (0,0) -- (0,-1);
\node[below=1pt] at (0,-1) {\footnotesize $\cdots<x<\cdots$};
\end{tikzpicture}$ \ (for $x\in P^{(2)}$), \quad 
$\begin{tikzpicture}[baseline=-0.9em, x=2em, y=2em]
\draw[thick, dotted] (0,-1.5) -- (0,-0.5);
\draw [thick] (0,-1) -- (0,0);
\node[below=1pt] at (0,0.5) {\footnotesize $r \star$};
\draw[black, fill=lightgray] (0,-0.5) circle (0.6ex);
\end{tikzpicture}
\quad = \quad 
\begin{tikzpicture}[baseline=-0.9em, x=2em, y=2em]
\draw[thick, dotted] (0,-1.5) -- (0,-0.5);
\draw [thick] (0,-1) -- (0,0);
\node[below=1pt] at (0,0.5) {\footnotesize $r \star$};
\end{tikzpicture}$ \ (for $r\in P^{(2)}$);\\[0.2cm]
lastly, ``fully-special'' sets to zero any allowable Jacobi diagram
 that has its root or all its leaves 
colored by the special element $s\in P^{(2)}$.

\begin{remark} \label{rem:H_is_enough}
By the fully-special relation, we can assume that 
the root of an allowable Jacobi diagram is colored by $H^{(2)}$
(instead of $P^{(2)}$). \hfill $\blacksquare$
\end{remark}

\subsection{Another surgery map}
\label{subsec:asm}

The following  will be proved by clasper calculus.

\begin{theorem} \label{th:Psi}
Let $n\geq 0$.
There is an $\Sp(H)$-equivariant homomorphism
$$
\Psi: \calB^<_n\big(P^{(2)}\big) 
\longrightarrow 
\Tors_2\Big(\frac{Y_{2n+1} \calIC}{Y_{2n+2}}\Big).
$$
\end{theorem}

\begin{proof}
We start by describing the  procedure 
to construct order $2$ elements in the module
${Y_{2n+1} \calIC}/{Y_{2n+2}}$ by surgery
along symmetric graph claspers. 
Next, we will analyze how the output surgery  depends 
on the initial data of the procedure.

Let us assume the following initial data:
\begin{itemize} \label{eq:data}
    \item a connected rooted beaded Jacobi diagram $D$
    of degree $n$, with $1+\ell$ external vertices 
    numbered from $0$ to $\ell$, the root corresponding to $0$;
    \item $1+\ell$ framed oriented knots in $U=\Sigma \times [-1,+1]$ denoted by $K_0, K_1, \dots, K_\ell$.
\end{itemize}
The homomorphism $H_1(D)\to \Z_2$ 
defined by the beads
induces a  (possibly disconnected) 
$2$-fold cover $\widetilde{D}$ of $D$.
Thus, the deck transformation 
defines a symmetry of~$\widetilde{D}$,
and every vertex of $D$ corresponds to a pair 
of \emph{twin}
(i.e$.$ symmetric) vertices in $\widetilde{D}$.
We \emph{augment} the Jacobi diagram $\widetilde{D}$
to $\widetilde{D}_+$  by gluing
the twin root vertices of $\widetilde{D}$
to two of the external vertices of a $Y$-diagram;
besides, $\widetilde{D}_+$ is \emph{rooted} 
at the free external vertex of the $Y$-diagram.
(That $\widetilde{D}_+$ is well-defined follows
from the symmetry of $\widetilde{D}$.)

We consider now  the abstract surface $S(\widetilde{D}_+)$
associated to the Jacobi diagram~$\widetilde{D}_+$
following ``Step 1'' of the procedure described 
in \S \ref{subsec:surgery}.
Then, we embed $S(\widetilde{D}_+)$
into the interior of $U=\Sigma \times [-1,+1]$
in such a way that the annulus of $S(\widetilde{D}_+)$
corresponding to the root represents $K_0$ and,
for any $i\in \{1,\dots,\ell\}$,
the two twin annuli of $S(\widetilde{D}_+)$
corresponding to the $i$-th external vertex of  $D$ 
represent two parallel copies $K'_i$ and $K''_i$
of the framed oriented knot $K_i$;
furthermore, we assume that $U$ has been subdivided 
into $1+\ell$ horizontal ``slices'' in which we see
$K_0$, $\{K'_1,K''_1\}$, \dots, $\{K'_\ell,K''_\ell\}$
successively (from bottom to top). Thus, 
we have obtained a graph clasper
\[G(D;K_0, K_1, \dots, K_\ell) \]
in $U=\Sigma \times [-1,+1]$ of degree $2n+1$. 
Here is a schematic picture:
\begin{equation} \label{eq:G}
\begin{tikzpicture}
  \node[anchor=north west] (img) at (0,0) {\includegraphics[scale = 1]{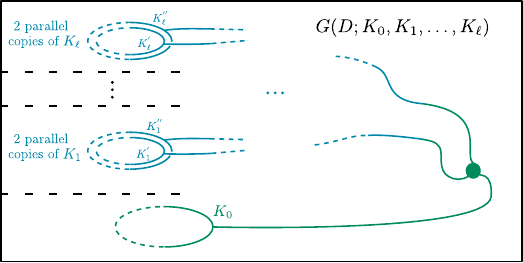}};
  \draw [decorate,decoration={brace,amplitude=10pt,mirror}]
     (-0.3,0) -- (-0.3,-4.6) node [midway,xshift=-1.2cm] {\shortstack{ $\Sigma \times [-1,1]$ \\ cut into \\ $1+l$ slices}};
\end{tikzpicture}
\end{equation}
Of course, $G(D;K_0, K_1, \dots, K_\ell)$ is ambiguously
defined, 
but the ``edge-sliding'' lemma in clasper calculus (see, for example \cite[Lemma A.1]{MM13})
implies that
$$
\Psi({D}; K_0, K_1, \dots, K_\ell)
:= \big\{U_G\big\} \
\in \frac{Y_{2n+1} \mathcal{IC}}{Y_{2n+2}}
\quad \hbox{with } G:= G(D;K_0, K_1, \dots, K_\ell)
$$
only depends on the combinatorics 
of the beaded rooted Jacobi diagram $D$
and on $K_0,K_1,\dots, K_\ell$. 
Furthermore, it is established below, in Lemma \ref{lem:psi_order_two}, that $\Psi({D}; K_0, K_1, \dots, K_\ell)$ 
is an element of order $2$.

We now prove that $\Psi$ induces a well-defined homomorphism on the module $\calB^<_n\big(P^{(2)}\big)$. The multilinearity relation with respect to a root is a standard fact in clasper calculus (via the “leaf-splitting” lemma, as formulated for instance in \cite[Lemma A.3]{MM13}); the multilinearity relation with respect to a leaf is proved below in Lemma \ref{lem:psi_splitting_leaf}. Observe that the multilinearity relation, combined with Lemma~\ref{lem:P}.(a), Habiro’s move 9, and Lemma \ref{lem:psi_order_two}, imply that $\Psi({D}; K_0, K_1, \dots, K_\ell)$ depends only on the types of $K_0, K_1, \dots, K_\ell$ in $P^{(2)}$. Equivalently, the datum of these framed oriented knots is the same as a coloring of the external vertices of $D$ by elements of $P^{(2)}$.

The group relations are clearly satisfied, since they merely encode the ambiguity arising in the description of a  homomorphism $H_1(D) \to \Z_2$ 
in terms of beads, which is equivalent to the specification of the $2$-fold cover $\widetilde{D}$ of $D$. The self-loop relation is also  satisfied, because surgery along a graph clasper of degree $r \geq 1$ that contains at least one looped edge does not alter the $Y_{r+1}$–equivalence class (as a consequence of \cite[Lemma 2.3]{GGP01}).

Next, the AS, IHX and BSTU-like relations are verified below in Lemmas \ref{lem:psi_AS}, \ref{lem:psi_IHX} and \ref{lem:BSTU}, respectively.  To complete the proof of the theorem, it remains to check the fully-special relation, for which we use the following fact: surgery along a graph clasper of degree $r \geq 2$ that has at least one special leaf does not change the $Y_{r+1}$–equivalence class (see, for instance, \cite[Lemma A.5]{MM13}). 
Hence, if $K_0$ is a $(-1)$–framed trivial knot, then we have 
$\Psi({D}; K_0, K_1, \dots, K_\ell)=0$ since the graph clasper $G(D;K_0, K_1, \dots, K_\ell)$ contains a special leaf. Furthermore, assume  that all $K_1, \dots, K_\ell$ are $(-1)$–framed trivial knots. Then, using the “leaf-splitting'' lemma together with Habiro’s move 2, we deduce from the above fact that
$$
\Psi({D}; K_0, K_1, \dots, K_\ell)=\psi(\overline{D}_+)
\in  \frac{Y_{2n+1} \mathcal{IC}}{Y_{2n+2}},
$$
where $\psi$ denotes here the ``usual'' surgery map of Theorem \ref{th:psi_connected}, and $\overline{D}_+$ is obtained from $\widetilde{D}_+$ by pairwise gluing  all its twin external vertices, and by coloring its root by $t(K_0)$. 
It follows from ``IHX'' and ``self-loop'' that 
$\overline{D}_+ =0 \in \calA^{<,c}_{2n+1}(P)$.
Consequently, $\Psi({D}; K_0, K_1, \dots, K_\ell)$ vanishes
in this case too.
\end{proof}

In the remainder of this subsection, we establish the auxiliary lemmas that we required for the proof of Theorem \ref{th:Psi}. Throughout, we continue to use the notation $D$ and $K_0, K_1, \dots, K_\ell$ introduced there for the input to the construction $\Psi(D; K_0, K_1, \dots, K_\ell)$.

\begin{lemma}[2-torsion] \label{lem:psi_order_two}
The element 
$\Psi({D}; K_0, K_1, \dots, K_\ell)$ has order $2$ 
in ${\displaystyle\frac{Y_{2n+1} \mathcal{IC}}{Y_{2n+2}}}$.
\end{lemma}

\begin{proof}
By the usual AS relation in clasper calculus, and with reference to the picture \eqref{eq:G} representing 
\(G(D; K_0, K_1, \dots, K_\ell)\), we observe that
\[
-\Psi(D; K_0, K_1, \dots, K_\ell)
\in \frac{Y_{2n+1} \mathcal{IC}}{Y_{2n+2}}
\]
is precisely the \(Y_{2n+2}\)-equivalence class of the surgery along the graph clasper
\begin{center}
\includegraphics[scale = 1.2]{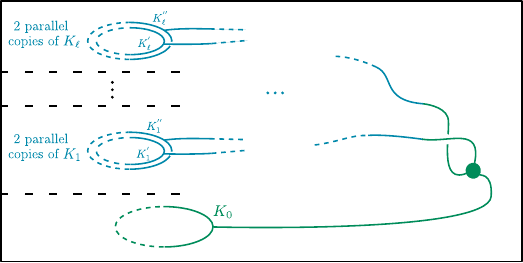}.
\end{center}
Besides,  the same arguments as those used to establish the well-definedness of 
\(\Psi(D; K_0, K_1, \dots, K_\ell)\), 
which relied on the symmetry of the Jacobi diagram  \(\widetilde{D}\), 
show that this graph clasper can also be used 
in the role of  $G(D; K_0, K_1, \dots, K_\ell)$. 
Consequently, \(\Psi(D; K_0, K_1, \dots, K_\ell)\) 
coincides with its additive inverse.
\end{proof}

\begin{lemma}[Multilinearity]  \label{lem:psi_splitting_leaf}
Let $i\in \{1,\dots, \ell\}$
and let $\check K_i$ be another choice of the framed oriented 
knot $K_i$ in $U$ (disjoint from $K_i$). Then, we have 
\begin{eqnarray*}
&&\Psi({D}; K_0, \dots ,K_i \# \check{K_i}, \dots, K_\ell)\\
&=& \Psi({D}; K_0, \dots, K_i, \dots, K_\ell) + \Psi({D}; K_0, \dots,  \check{K_i}, \dots, K_\ell).
\end{eqnarray*}
\end{lemma}

\begin{proof}
We can assume without loss of generality that $i=1$. 
By using the ``leaf-splitting'' lemma, we get 
\begin{align*}
    &\Psi({D}; K_0, K_1 \# \check{K_1}, \dots, K_\ell) - \Psi({D}; K_0, K_1, \dots, K_\ell) - \Psi({D}; K_0, \check{K_1}, \dots, K_\ell) \\ & \\ 
    =& \hspace{3mm}\raisebox{-8ex}{\includegraphics[scale = 0.6]{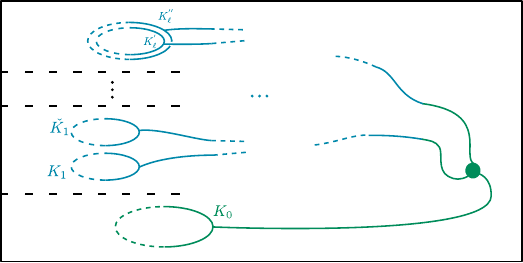}} \hspace{2mm} + \hspace{2mm}\raisebox{-8ex}{\includegraphics[scale = 0.6]{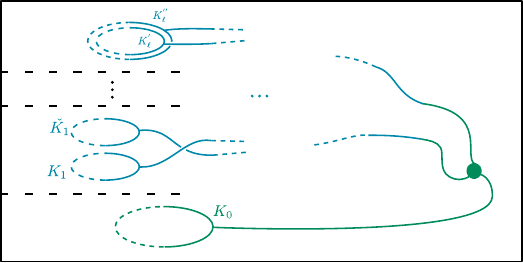}}.\\ 
\end{align*}
It follows from the symmetry of $\widetilde{D}$ 
that the above sum is zero: 
indeed, the symmetry 
swaps pairs of twin external vertices of $\widetilde{D}$,
and we conclude with the usual AS relation in clasper calculus.
\end{proof}

\begin{lemma}[AS]
\label{lem:psi_AS}
If $E$ is a beaded rooted Jacobi diagram differing
from $D$ by the orientation at an internal vertex, then
$$
\Psi({D}; K_0, K_1, \dots, K_\ell)
+ \Psi({E}; K_0, K_1, \dots, K_\ell) =0.
$$
\end{lemma}

\begin{proof}
A double application of the usual AS relation in clasper calculus
shows that 
$$
\Psi({E}; K_0, K_1, \dots, K_\ell) 
= (-1)^2\, \Psi({D}; K_0, K_1, \dots, K_\ell)
$$
and we conclude with Lemma \ref{lem:psi_order_two}.
\end{proof}

\begin{lemma}[IHX]
\label{lem:psi_IHX}
If $D_1$ and $D_2$ are beaded rooted Jacobi diagrams differing
from~$D$ as follows in a neighborhood of two adjacent internal vertices:
$$
D= \ltritree{}{}{}, \qquad  
D_1= \ltritreebis \quad D_2 =\revltritree,
$$
then we have 
$$
\Psi({D}; K_0, K_1, \dots, K_\ell) =
\Psi({D}_1; K_0, K_1, \dots, K_\ell)
+ \Psi({D}_2; K_0, K_1, \dots, K_\ell).
$$
\end{lemma}

\begin{proof}
Depicting the  claspers 
$G({D}_1; K_0, K_1, \dots, K_\ell) $
and $G({D}_2; K_0, K_1, \dots, K_\ell) $
along which we do surgery,
we  write $\Psi({D}_1; K_0, K_1, \dots, K_\ell) 
+ \Psi({D}_2; K_0, K_1, \dots, K_\ell)$ locally  as
\[
\ltritreebis~ \ltritreebis  + 
\revltritree ~ \revltritree.
\]
On the other hand, using twice the IHX relation in clasper calculus,
we see that  $\Psi({D}; K_0, K_1, \dots, K_\ell)$ writes locally as
\[
\ltritree{}{}{}~ \ltritree{}{}{}=\ltritreebis~ \ltritreebis  + 
\revltritree ~ \revltritree + 
\revltritree~\ltritreebis  +
\ltritreebis~ \revltritree.
\]
Once again, we can use the symmetry of $\widetilde{D}$ 
and the AS relation  in clasper calculus 
to deduce that the last two terms cancel out.
\end{proof}

\begin{lemma}[BSTU-like]
\label{lem:BSTU}
Let $i\in \{1,\dots, \ell-1\}$.
If $E$ (respectively, $E_{\bead}$) 
denotes the beaded rooted Jacobi diagram obtained from $D$ 
by gluing the pair of leaves $\{i,i+1\}$
(respectively, gluing $\{i,i+1\}$ 
and adding a bead at the gluing point),
then we have
{\small
\begin{eqnarray*}
&& \Psi({D}; K_0, K_1, \dots,K_i, K_{i+1}, \dots , K_\ell)
- \Psi({D}; K_0, K_1, \dots,K_{i+1}, K_i, \dots,   K_\ell)  \\
&=& \omega(k_i,k_{i+1})\, \big(
\Psi(E; K_0, \dots,\widehat{K_i},
\widehat{K_{i+1}}, \dots , K_\ell) +
\Psi(E_{\bead}; K_0,  \dots,\widehat{K_i},
\widehat{K_{i+1}}, \dots, K_\ell) \big)
\end{eqnarray*}}
where $k_r\in H$ denotes the homology class of $K_r$.
\end{lemma}

\begin{proof} 
The graph clasper 
$G:=G({D}; K_0, K_1, \dots,K_i, K_{i+1}, \dots , K_\ell)$ 
inside $U$ can be depicted   as\\[-0.5cm]
\begin{center}
  \includegraphics[scale = 1]{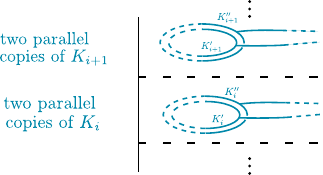} 
\end{center}
in a neighborhood of the $i$-th and $(i+1)$-st slices of $U$:
here, we have denoted by $K'_i,K''_i$ 
(respectively  $K'_{i+1},K''_{i+1}$) 
the two leaves of $G$
that are parallel copies of $K_i$ (respectively, of $K_{i+1}$) 
The graph clasper 
$\check G := G({D}; K_0, K_1, \dots, K_{i+1}, K_{i}, \dots , K_\ell)$
is depicted analogously, 
by interchanging the two pairs of parallel leaves.  
 For any \(r,s \in \{1,2\}\), we denote by \(G(r,s)\) (respectively, by \(\check G(r,s)\)) the graph clasper obtained from \(G\) (respectively, from \(\check G\)) by removing the leaves \(K_i^{(r)}\) and \(K_{i+1}^{(s)}\) and subsequently gluing together 
the resulting free ends of the incident edges.  
By doing some clasper calculus, in a manner entirely analogous 
to the proof of Lemma \ref{lem:generalized_STU-like} below, we obtain that
$$
U_{G}-U_{\check G}
=\omega(k_i,k_{i+1})\, 
\big( U_{G(2,1)}+U_{G(2,2)}+U_{\check G(1,1)}+ U_{\check G(1,2)}\big)
\in \frac{Y_{2n+1} \mathcal{IC}}{Y_{2n+2}}.
$$
Similarly, 
by applying  the usual STU-like relation in clasper calculus
(see Lemma~\ref{lem:generalized_STU-like} below), 
we have
$$
U_{G(2,1)}= U_{\check G(2,1)} + \omega(k_i,k_{i+1})\, U_{H}
\in \frac{Y_{2n+1} \mathcal{IC}}{Y_{2n+2}}
$$
and 
$$
U_{G(2,2)}= U_{\check G(2,2)} + \omega(k_i,k_{i+1})\, U_{J}
\in \frac{Y_{2n+1} \mathcal{IC}}{Y_{2n+2}}
$$
where $H$ (respectively, $J$) is obtained from $G$
by deleting  $K'_i,K''_i,K'_{i+1},K''_{i+1}$
and gluing the resulting extremities of edges as shown below:
$$
H = \raisebox{-12ex}{\includegraphics[scale = 1]{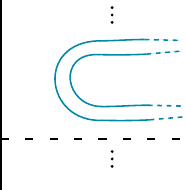}},
\qquad\qquad 
J = \hspace{2mm}\raisebox{-12ex}{\includegraphics[scale = 1]{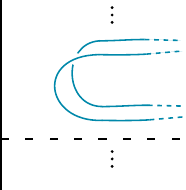}}
$$
But, by symmetry of $\widetilde{D}$ 
and the usual AS relation in clasper calculus, we have 
$$
U_{\check G(2,1)} = - U_{\check G(1,2)} 
\in \frac{Y_{2n+1} \mathcal{IC}}{Y_{2n+2}}
\quad \hbox{and} \quad 
U_{\check G(2,2)} = - U_{\check G(1,1)}
\in \frac{Y_{2n+1} \mathcal{IC}}{Y_{2n+2}}.
$$
Thus, we deduce from the above identities that
$$
U_{G}-U_{\check G} = 
\omega(k_i,k_{i+1})^2\, (U_H+U_J) 
= \omega(k_i,k_{i+1})\, (U_H+U_J) 
\in \frac{Y_{2n+1} \mathcal{IC}}{Y_{2n+2}}
$$
where the second equality follows from
Lemma \ref{lem:psi_order_two}, 
and we conclude by observing that the claspers
$H$ and $J$ can serve as 
$G(E; K_0, K_1, \dots,\widehat{K_i}, 
\widehat{K_{i+1}}, \dots , K_\ell)$
and 
$G(E_{\bead}; K_0, K_1, \dots,\widehat{K_i}, 
\widehat{K_{i+1}}, \dots , K_\ell)$.
\end{proof}

We conclude this subsection
by formulating a generalized version of the STU-like relation, which was employed in the proof of Lemma \ref{lem:BSTU} and will be utilized again in subsequent arguments.
We use here the notions of linking number 
$\operatorname{Lk}_\pm(-,-)$ 
for 2-component oriented links  in \(U = \Sigma \times [-1,+1]\)
as presented in \cite[\S B]{MM13}.

\begin{lemma}[Generalized STU-like] \label{lem:generalized_STU-like}
Let \(G\) be a connected graph clasper of degree~$k$
 in \(U\) such that  a slice 
\(\Sigma \times [t-\varepsilon, t+\varepsilon]\) of $U$ contains precisely two of its leaves, denoted \(K\) and~\(L\).  
Let \(G'\) be a graph clasper that coincides with \(G\) except for the relative positions of these two leaves: the leaf \(K'\) (corresponding to \(K\)) is now contained in the upper slice \(\Sigma \times [t,t+\varepsilon]\), whereas the leaf \(L'\) (corresponding to \(L\)) is now contained in the lower slice 
\(\Sigma \times [t-\varepsilon,t]\).
Then we have 
\begin{equation} \label{eq:generalized_STU-like}
U_{G}-U_{G'} = 
-\operatorname{Lk}_+(K,L)\cdot  U_H \in Y_k \calIC/Y_{k+1}
\end{equation}
where $H$ is the graph clasper obtained from $G$ 
by deleting $K$ and $L$, and gluing the free ends of the incident edges.
\end{lemma}

\begin{proof}
We consider a projection diagram of the oriented link $(K,L)$ 
onto the surface $\Sigma \times \{t-\epsilon\}$:
every mixed crossing in this diagram with $K$ underneath $L$
is an obstruction to move $K$ by isotopy
into a neighborhood of 
$\Sigma \times \{t+\epsilon\}$ in the knot exterior
\(\big(\Sigma \times [t-\varepsilon, t+\varepsilon]\big) \setminus L\).
Actually, this crossing can be changed 
at the price of a surgery along $H$,
as follows from the ``leaf/leaf'' version 
of the ``crossing-change'' lemma in clasper calculus
(see, for instance, \cite[Lemma 2.6.(2.c)]{Mei06}).
Thus, \eqref{eq:generalized_STU-like} is deduced 
from the diagrammatic formula
$$
\operatorname{Lk}_+(K,L)
=\sharp  \begin{array}{c}
\labellist
\small\hair 2pt
 \pinlabel {$K$} [br] at 1 19
 \pinlabel {$L$} [bl] at 19 20
\endlabellist
\includegraphics[scale=1.0]{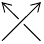}
\end{array}
-  \sharp  \begin{array}{c}
\labellist
\small\hair 2pt
 \pinlabel {$L$} [br] at 1 19
 \pinlabel {$K$} [bl] at 19 20
\endlabellist
\includegraphics[scale=1.0]{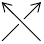}
\end{array}
$$
(see \cite[Lemma B.2.(2)]{MM13})
by a careful inspection of signs.
\end{proof}

\begin{remark} \label{rem:STU-like}
The usual STU-like relation is recovered as follows:
with the notations of Lemma \ref{lem:generalized_STU-like},
assume that $K\subset \Sigma \times [t-\varepsilon,t]$ and 
$L\subset \Sigma \times [t,t+\varepsilon]$.
Then we have $\operatorname{Lk}_-(K,L)$=0 and, so, we get 
$\omega([K],[L])=-\operatorname{Lk}_+(K,L)$
by \cite[Lemma B.2.(1)]{MM13}. 
Thus \eqref{eq:generalized_STU-like} writes
\begin{equation} \label{eq:STU-like}
   U_{G}-U_{G'} = \omega([K],[L])\cdot  U_H \in Y_k \calIC/Y_{k+1}.
\end{equation}

\up
\hfill $\blacksquare$
\end{remark}

\subsection{Relating the surgery maps}
\label{subsec:rsm}

In this subsection, we relate the surgery map 
$
\Psi: \calB^<_n\big(P^{(2)}\big) 
\to {Y_{2n+1}} \calIC/{Y_{2n+2}}
$
of Theorem \ref{th:Psi}
with the ``usual'' surgery map 
$\psi:\calA^{<,c}_{2n+1}(P) \to {Y_{2n+1}} \calIC/{Y_{2n+2}}$
of Theorem \ref{th:psi_connected}.

\begin{remark}
The case $n=0$ is easily handled using the definitions.
Indeed, we have in this case
$\calB^<_0\big(P^{(2)}\big) \simeq
P^{(2)} \otimes P^{(2)}$ 
(via the identification $x\!\star\textsf{---}\,y \mapsto x\otimes y$),
and, using Proposition \ref{prop:before}.(i),
we get the following  commutative diagram:
\begin{equation} \label{eq:E_0}
\xymatrix{
\quad \begin{tikzpicture}[thick,color=black,
baseline=-0.5em, x=0.6em, y=0.6em]
 \draw[thick] (0,0)--(2,0);
 \node[left=0pt] at (0.25,0) {\scriptsize $x\star$};
\node[right=0pt] at (1.75,0) {\scriptsize $y$};
\end{tikzpicture} \hspace{2mm} \ni \ar@{|->}[d]
& \calB^<_0\big(P^{(2)}\big) \ar[r]^-\Psi \ar[d] &
\Tors_2\Big({\displaystyle\frac{\calIC}{Y_2}}\Big) \\
\quad \begin{tikzpicture}[thick,color=black,
baseline=-0.5em, x=0.6em, y=0.6em]
\footnotesize
\node[left=1pt] at (-1,1) {\scriptsize $y$};
\node[right=1pt] at (1,1) {\scriptsize $y$};
\node[below=1pt] at (0,-1.5) {\scriptsize $x$};
\draw (-1,1) -- (0,0) -- (1,1);
\draw (0,0) -- (0,-1.5);
\draw [fill = black] (0,0) circle (.3ex);
\end{tikzpicture} \  \ni & 
\Tors_2\big(\calA_1(P)\big) \ar[ru]_-\psi &  
}  
\end{equation}   

\up \hfill $\blacksquare$
\end{remark}

Therefore, we assume in the sequel that  $n>0$. 
We shall define a homomorphism
$$
E:  \calB^<_n\big(P^{(2)}\big)
\longrightarrow \Tors_2\big(\calA^{<,c}_{2n+1}(P)\big)
$$
by assigning to every rooted beaded Jacobi diagram  $D$
the linear combination 
\begin{equation} \label{eq:E_def}
E(D):=\sum_{S} 
\Big(\prod_{j \in S}   \fr(c_j) \Big) \cdot \widetilde{D}_+^S.
\end{equation}
Here the sum is over the subsets $S$
of the set of leaves of $D$,
the color of a vertex $j\in S$ is denoted by $c_j\in P^{(2)}$,
the Jacobi diagram $\widetilde{D}_+$
is (as in the proof of Theorem~\ref{th:Psi})
the augmentation of the $2$-fold cover $\widetilde{D}$
by a $Y$-diagram,
and its $S$-modification  
$\widetilde{D}_+^S\in \calA^{<,c}(P)$ 
is defined as follows: 
the root of $\widetilde{D}_+$ 
is given the color of the root of~$D$
(after lifting it arbitrarily from $P^{(2)}$ to~$P$)
and is declared to be ``lower'' 
than all other vertices; 
for every leaf $i$ of $D$,
the corresponding twin vertices of $\widetilde{D}_+$
are either glued together if $i\in S$,
or they inherit both the color $c_i$
(lifting this arbitrarily from $P^{(2)}$ to $P$)
and the position of $i$ in the ordering of $D$ if $i\not\in S$.
Note that, for any~$S$,
the deck transformation of $\widetilde{D}$ extends  to an isomorphism between
$\widetilde{D}_+^S$ and the same Jacobi diagram but with the opposite orientation 
on the internal vertex of the extra $Y$-diagram: therefore, by the AS relation, 
$2\widetilde{D}_+^S =0 \in \calA^{<,c}(P)$:
it follows that the right-hand side of \eqref{eq:E_def} 
is well-defined,  and only depends on $D$, i.e$.$
it does not depend on the choices of the lifts from $\Z_2$ to $\Z$,
and from $P^{(2)}$ to $P$.

\begin{remark}
Recall that 
$\calA^{<,c}_{k}(P)\simeq \calA^{<,c}_{k}(H)$ via $p_*$ 
for all $k\geq 2$. Consequently, by identifying the target of the map $E$ with $\Tors_2\big(\calA^{<,c}_{2n+1}(H)\big)$, each summand $\widetilde{D}_+^S$ appearing in \eqref{eq:E_def} 
can be viewed as a Jacobi diagram colored by~$H^{(2)}$.\hfill $\blacksquare$
\end{remark}

\begin{example} \label{ex:E}
We give below the values $E(D)\in \calA^{<,c}_{7}(P)$
for two instances of $D\in \calB^<_3\big(P^{(2)}\big)$.
Each is written using the convention \eqref{eq:lifts}, 
and re-written as a single diagram 
using the notation~\eqref{eq:parallel_notation}:
\begin{align*}
E\Bigg( \begin{tikzpicture}[scale = 0.5,baseline = 0.1cm]
		\node (0) at (1, 0) {};
		\node (1) at (1, -0.75) {};
		\node (2) at (1, 1) {};
		\node (3) at (0.55, 0.775) {};
		\node (4) at (1.45, 0.775) {};
		\node (5) at (0, 1.25) {};
		\node (6) at (2, 1.25) {};
		\draw [thick] (0.center) to (1.center);
        \draw [thick] (0.center) to (1.center);
		\draw [thick] [in=180, out=-180, looseness=1.75] (2.center) to (0.center);
		\draw [thick] [bend left=90, looseness=1.75] (2.center) to (0.center);
		\draw [thick] (5.center) to (3.center);
		\draw [thick] (6.center) to (4.center);
        \draw [fill = black] (0.center) circle  (0.8ex);
        \draw [fill = black] (3.center) circle  (0.8ex);
        \draw [fill = black] (4.center) circle  (0.8ex);
        \node[below=0.5pt] at (1,-0.75) {\scriptsize $h \star$};
        \node[above] at (1,1.2) {\scriptsize $x$ \hspace{1mm} $<$ \hspace{1mm} $y$};
\end{tikzpicture}
 \Bigg) = &  
 \begin{tikzpicture}[scale = 0.4, baseline = 1cm]
	    \node  (0) at (0, 4.5) {};
		\node  (1) at (2.5, 1.55) {};
		\node  (2) at (3, 2.25) {};
		\node  (3) at (3, 2.25) {};
		\node  (4) at (2, 3) {};
		\node  (5) at (4, 3) {};
		\node  (7) at (2, 4) {};
		\node  (8) at (4, 4) {};
		\node  (9) at (1.525, 3.75) {};
		\node  (10) at (2.475, 3.75) {};
		\node  (11) at (1, 4.5) {};
		\node  (12) at (3.525, 3.75) {};
		\node  (13) at (4.475, 3.75) {};
		\node  (14) at (5, 4.5) {};
		\node  (15) at (3.5, 4.5) {};
		\node  (16) at (2.5, 4.5) {};
		\draw [thick, in=180, out=-90, looseness=1.50] (0.center) to (1.center);
		\draw [thick, in=-90, out=0] (1.center) to (2.center);
		\draw [thick] (4.center) to (3.center);
		\draw [thick] (3.center) to (5.center);
		\draw [thick, bend right=90, looseness=1.75] (7.center) to (4.center);
		\draw [thick, bend left=90, looseness=1.75] (7.center) to (4.center);
		\draw [thick, bend left=270, looseness=1.75] (8.center) to (5.center);
		\draw [thick, bend left=90, looseness=1.75] (8.center) to (5.center);
		\draw [thick] (11.center) to (9.center);
		\draw [thick] (13.center) to (14.center);
		\draw [thick] (12.center) to (16.center);
        \draw [white, fill = white] (3,4.15) circle  (1ex);
		\draw [thick] (10.center) to (15.center);
        \draw [fill = black] (3.center) circle  (0.8ex);
        \draw [fill = black] (4.center) circle  (0.8ex);
        \draw [fill = black] (5.center) circle  (0.8ex);
        \draw [fill = black] (9.center) circle  (0.8ex);
        \draw [fill = black] (10.center) circle  (0.8ex);
        \draw [fill = black] (13.center) circle  (0.8ex);
        \draw [fill = black] (12.center) circle  (0.8ex);
        \node[above] at (2.5,4.5) {\scriptsize $h < x < x< y< y$};
\end{tikzpicture} + \fr(x)  \begin{tikzpicture}[scale = 0.4, baseline = 1cm]
        \node  (0) at (0, 4.5) {};
		\node  (1) at (2.5, 1) {};
		\node  (2) at (3, 1.75) {};
		\node  (3) at (3, 1.75) {};
		\node  (4) at (2, 2.5) {};
		\node  (5) at (4, 2.5) {};
		\node  (6) at (2, 3.5) {};
		\node  (7) at (4, 3.5) {};
		\node  (8) at (1.525, 3.25) {};
		\node  (9) at (2.475, 3.25) {};
		\node  (10) at (1.25, 4) {};
		\node  (11) at (3.525, 3.25) {};
		\node  (12) at (4.475, 3.25) {};
		\node  (13) at (5, 4.5) {};
		\node  (14) at (3 .75, 4.5) {};
		\node  (15) at (2.25, 4) {};
		\node  (16) at (1.75, 4.5) {};
        \draw [thick,in=180, out=-90, looseness=1.50] (0.center) to (1.center);
		\draw [thick,in=-90, out=0] (1.center) to (2.center);
		\draw [thick] (4.center) to (3.center);
		\draw [thick] (3.center) to (5.center);
		\draw [thick ,bend right=90, looseness=1.75] (6.center) to (4.center);
		\draw [thick,bend left=90, looseness=1.75] (6.center) to (4.center);
		\draw [thick,bend left=270, looseness=1.75] (7.center) to (5.center);
		\draw [thick,bend left=90, looseness=1.75] (7.center) to (5.center);
		\draw [thick,in=180, out=-90] (10.center) to (8.center);
		\draw [thick,in=-90, out=45, looseness=1.25] (12.center) to (13.center);
		\draw [thick, in=-90, out=90] (11.center) to (15.center);
        \draw [white, fill = white] (2.775, 3.65) circle (1ex); 
		\draw [thick,in=-90, out=75] (9.center) to (14.center);
		\draw [thick,bend left=45] (10.center) to (16.center);
		\draw [thick,bend left=45] (16.center) to (15.center);
        \draw [fill = black] (3.center) circle  (0.8ex);
        \draw [fill = black] (4.center) circle  (0.8ex);
        \draw [fill = black] (5.center) circle  (0.8ex);
        \draw [fill = black] (9.center) circle  (0.8ex);
        \draw [fill = black] (8.center) circle  (0.8ex);
        \draw [fill = black] (11.center) circle  (0.8ex);
        \draw [fill = black] (12.center) circle  (0.8ex);
          \node[above] at (2.5,4.5) {\scriptsize $h \hspace{4.5mm}< \hspace{5 mm} y< y$};
        \end{tikzpicture} \\
        & + \fr(y)\begin{tikzpicture}[scale = 0.4, baseline = 1cm] 
        \node  (0) at (0, 4.5) {};
		\node  (1) at (2.5, 1) {};
		\node  (2) at (3, 1.75) {};
		\node  (3) at (3, 1.75) {};
		\node  (4) at (2, 2.5) {};
		\node  (5) at (4, 2.5) {};
		\node  (6) at (2, 3.5) {};
		\node  (7) at (4, 3.5) {};
		\node  (8) at (1.525, 3.25) {};
		\node  (9) at (2.475, 3.25) {};
		\node  (10) at (1.25, 4.5) {};
		\node  (11) at (3.525, 3.25) {};
		\node  (12) at (4.475, 3.25) {};
		\node  (13) at (4.75, 4) {};
		\node  (14) at (3.75, 4) {};
		\node  (15) at (2.5, 4.5) {};
		\node  (17) at (4.25, 4.5) {};
		\node  (18) at (3, 3.6) {};
        \draw [thick, in=180, out=-90, looseness=1.50] (0.center) to (1.center);
		\draw [thick,in=-90, out=0] (1.center) to (2.center);
		\draw [thick] (4.center) to (3.center);
		\draw [thick] (3.center) to (5.center);
		\draw [ thick,bend right=90, looseness=1.75] (6.center) to (4.center);
		\draw [ thick,bend left=90, looseness=1.75] (6.center) to (4.center);
		\draw [ thick,bend left=270, looseness=1.75] (7.center) to (5.center);
		\draw [ thick,bend left=90, looseness=1.75] (7.center) to (5.center);
		\draw [ thick,in=165, out=-90, looseness=0.75] (10.center) to (8.center);
		\draw [ thick,in=-90, out=45, looseness=1.25] (12.center) to (13.center);
		\draw [ thick,in=-90, out=135, looseness=1.25] (11.center) to (15.center);
        \draw [white, fill = white] (3, 3.5) circle (1.3ex); 
		\draw [ thick,in=-90, out=90, looseness=0.75] (9.center) to (14.center);
		\draw [ thick,in=180, out=90] (14.center) to (17.center);
		\draw [ thick,in=90, out=0] (17.center) to (13.center);
        \draw [fill = black] (3.center) circle  (0.8ex);
        \draw [fill = black] (4.center) circle  (0.8ex);
        \draw [fill = black] (5.center) circle  (0.8ex);
        \draw [fill = black] (9.center) circle  (0.8ex);
        \draw [fill = black] (8.center) circle  (0.8ex);
        \draw [fill = black] (11.center) circle  (0.8ex);
        \draw [fill = black] (12.center) circle  (0.8ex);
         \node[above] at (1.25 ,4.5) {\scriptsize $h < x < x$};
        \end{tikzpicture} + \fr(x)\fr(y)\begin{tikzpicture}[scale = 0.4, baseline = 1cm]
		\node  (0) at (0, 4.5) {};
		\node  (1) at (2.5, 1) {};
		\node  (2) at (3, 1.75) {};
		\node  (3) at (3, 1.75) {};
		\node  (4) at (2, 2.5) {};
		\node  (5) at (4, 2.5) {};
		\node  (6) at (2, 3.5) {};
		\node  (7) at (4, 3.5) {};
		\node  (8) at (1.525, 3.25) {};
		\node  (9) at (2.475, 3.25) {};
		\node  (10) at (1.25, 4) {};
		\node  (11) at (3.525, 3.25) {};
		\node  (12) at (4.475, 3.25) {};
		\node  (13) at (4.75, 4) {};
		\node  (14) at (3.75, 4) {};
		\node  (15) at (2.25, 4) {};
		\node  (16) at (1.75, 4.5) {};
		\node  (17) at (4.25, 4.5) {};
		\draw [thick,in=180, out=-90, looseness=1.50] (0.center) to (1.center);
		\draw [thick,in=-90, out=0] (1.center) to (2.center);
		\draw [thick] (4.center) to (3.center);
		\draw [thick] (3.center) to (5.center);
		\draw [thick,bend right=90, looseness=1.75] (6.center) to (4.center);
		\draw [thick,bend left=90, looseness=1.75] (6.center) to (4.center);
		\draw [thick,bend left=270, looseness=1.75] (7.center) to (5.center);
		\draw [thick,bend left=90, looseness=1.75] (7.center) to (5.center);
		\draw [thick,in=165, out=-90, looseness=0.75] (10.center) to (8.center);
		\draw [thick,in=-90, out=45, looseness=1.25] (12.center) to (13.center);
		\draw [thick,in=-90, out=90, looseness=0.75] (11.center) to (15.center);  
        \draw [white, fill = white] (3, 3.5) circle (1.3ex); 
		\draw [thick,in=-90, out=90, looseness=0.50] (9.center) to (14.center);
		\draw [thick,bend left=45] (10.center) to (16.center);
		\draw [thick,bend left=45] (16.center) to (15.center);
		\draw [thick,in=180, out=90] (14.center) to (17.center);
		\draw [thick,in=90, out=0] (17.center) to (13.center);
        \draw [fill = black] (3.center) circle  (0.8ex);
        \draw [fill = black] (4.center) circle  (0.8ex);
        \draw [fill = black] (5.center) circle  (0.8ex);
        \draw [fill = black] (9.center) circle  (0.8ex);
        \draw [fill = black] (8.center) circle  (0.8ex);
        \draw [fill = black] (11.center) circle  (0.8ex);
        \draw [fill = black] (12.center) circle  (0.8ex);
         \node[above] at (0,4.5) {\scriptsize $h$};
\end{tikzpicture} \\
= &  \begin{tikzpicture}[thick, scale = 0.4, baseline = 1cm]
	    \node (0) at (-0.50, 4.5) {};
		\node  (1) at (2.5, 1.25) {};
		\node  (2) at (3, 2) {};
		\node  (3) at (3, 2) {};
		\node  (4) at (2, 2.75) {};
		\node  (5) at (4, 2.75) {};
		\node  (7) at (2, 3.75) {};
		\node  (8) at (4, 3.75) {};
		\node  (9) at (1.525, 3.5) {};
		\node  (10) at (2.475, 3.5) {};
		\node  (11) at (1, 4.5) {};
		\node  (12) at (3.525, 3.5) {};
		\node  (13) at (4.475, 3.5) {};
		\node  (14) at (5, 4.5) {};
		\node  (15) at (3.75, 4.5) {};
		\node  (16) at (2.25, 4.5) {};
		\draw [in=180, out=-90, looseness=1.50] (0.center) to (1.center);
		\draw [in=-90, out=0] (1.center) to (2.center);
		\draw (4.center) to (3.center);
		\draw (3.center) to (5.center);
		\draw [bend right=90, looseness=1.75] (7.center) to (4.center);
		\draw [bend left=90, looseness=1.75] (7.center) to (4.center);
		\draw [bend left=270, looseness=1.75] (8.center) to (5.center);
		\draw [bend left=90, looseness=1.75] (8.center) to (5.center);
		\draw [in=-180, out=-90] (11.center) to (9.center);
		\draw [in=-90, out=0] (13.center) to (14.center);
		\draw [in=-90, out=90] (12.center) to (16.center);
        \draw [white, fill = white] (3, 3.975) circle (1.3ex);
		\draw [in=-90, out=90] (10.center) to (15.center);
        \draw [fill = black] (3.center) circle  (0.8ex);
        \draw [fill = black] (4.center) circle  (0.8ex);
        \draw [fill = black] (5.center) circle  (0.8ex);
        \draw [fill = black] (9.center) circle  (0.8ex);
        \draw [fill = black] (10.center) circle  (0.8ex);
        \draw [fill = black] (13.center) circle  (0.8ex);
        \draw [fill = black] (12.center) circle  (0.8ex);
        \node[above] at (3,4.5) {\scriptsize ${x} \parallel {x} \hspace{1mm} <\hspace{1mm} {y} \parallel {y}$};
        \node[above] at (-0.2,4.6) {\scriptsize $ h< $};
\end{tikzpicture} 
\end{align*}

\begin{align*}
E\Bigg(\begin{tikzpicture}[scale = 0.5,baseline = 0.1cm]
		\node (0) at (1, 0) {};
		\node (1) at (1, -0.75) {};
		\node (2) at (1, 1) {};
		\node (3) at (0.55, 0.775) {};
		\node (4) at (1.45, 0.775) {};
		\node (5) at (0, 1.25) {};
		\node (6) at (2, 1.25) {};
		\draw [thick] (0.center) to (1.center);
        \draw [thick] (0.center) to (1.center);
		\draw [thick] [in=180, out=-180, looseness=1.75] (2.center) to (0.center);
		\draw [thick] [bend left=90, looseness=1.75] (2.center) to (0.center);
		\draw [thick] (5.center) to (3.center);
		\draw [thick] (6.center) to (4.center);
        \draw [fill = black] (0.center) circle  (0.8ex);
        \draw [fill = black] (3.center) circle  (0.8ex);
        \draw [fill = black] (4.center) circle  (0.8ex);
        \node[below=0.5pt] at (1,-0.75) {\scriptsize $h \star$};
        \node[above] at (1,1.2) {\scriptsize $x$ \hspace{1mm} $<$ \hspace{1mm} $y$};
        \draw[black, fill=lightgray] (1.435,0.225) circle (0.95ex);
\end{tikzpicture}
 \Bigg) = &  
 \begin{tikzpicture}[scale = 0.4, baseline = 1cm]
        \node  (0) at (0, 4.5) {};
		\node  (1) at (2.5, 1) {};
		\node  (2) at (3, 1.75) {};
		\node  (3) at (3, 1.75) {};
		\node  (4) at (2, 2.5) {};
		\node  (5) at (4, 2.5) {};
		\node  (6) at (2, 3.5) {};
		\node  (7) at (4, 3.5) {};
		\node  (8) at (1.525, 3.25) {};
		\node  (9) at (2.475, 3.25) {};
		\node  (10) at (1.25, 4.5) {};
		\node  (11) at (3.525, 3.25) {};
		\node  (12) at (4.475, 3.25) {};
		\node  (13) at (5, 4.5) {};
		\node  (14) at (3.75, 4.5) {};
		\node  (15) at (2.5, 4.5) {};
		\node  (17) at (2.75, 2.75) {};
		\node  (18) at (4.175, 2.625) {};
		\draw [thick,in=180, out=-90, looseness=1.50] (0.center) to (1.center);
		\draw [thick,in=-90, out=0] (1.center) to (2.center);
		\draw [thick] (4.center) to (3.center);
		\draw [thick] (3.center) to (5.center);
		\draw [thick,bend right=90, looseness=1.75] (6.center) to (4.center);
		\draw [thick,in=165, out=-90, looseness=0.75] (10.center) to (8.center);
		\draw [thick,in=-90, out=30, looseness=0.75] (12.center) to (13.center);
		\draw [thick,in=-90, out=135, looseness=1.25] (11.center) to (15.center);
        \draw [white, fill = white] (3.025, 3.65) circle (1.3ex);
		\draw [thick,in=-90, out=45] (9.center) to (14.center);
		\draw [thick,in=105, out=0] (7.center) to (12.center);
		\draw [thick,in=105, out=0] (6.center) to (9.center);
		\draw [thick,bend right=15] (4.center) to (17.center);
		\draw [thick,in=-120, out=45] (17.center) to (12.center);
        \draw [white, fill = white] (2.75, 2.75) circle (1.3ex); 
		\draw [thick,in=135, out=-90] (9.center) to (17.center);
		\draw [thick,in=60, out=-45] (17.center) to (18.center);
		\draw [thick,bend left=15] (18.center) to (5.center);
         \draw [white, fill = white] (3.475, 2.975) circle (1.3ex);
        \draw [white, fill = white] (3.575, 2.675) circle (1.3ex);
        \draw [thick,bend left=270, looseness=1.75] (7.center) to (5.center);
        \node[above] at (2.5,4.5) {\scriptsize $h < x < x< y< y$};
        \draw [fill = black] (3.center) circle  (0.8ex);
        \draw [fill = black] (4.center) circle  (0.8ex);
        \draw [fill = black] (5.center) circle  (0.8ex);
        \draw [fill = black] (9.center) circle  (0.8ex);
        \draw [fill = black] (8.center) circle  (0.8ex);
        \draw [fill = black] (11.center) circle  (0.8ex);
        \draw [fill = black] (12.center) circle  (0.8ex);
\end{tikzpicture} + \fr(x) \begin{tikzpicture}[scale = 0.4, baseline = 1cm]
         \node  (0) at (0, 4.5) {};
		\node  (1) at (2.5, 1) {};
		\node  (2) at (3, 1.75) {};
		\node  (3) at (3, 1.75) {};
		\node  (4) at (2, 2.5) {};
		\node  (5) at (4, 2.5) {};
		\node  (6) at (2, 3.5) {};
		\node  (7) at (4, 3.5) {};
		\node  (8) at (1.525, 3.25) {};
		\node  (9) at (2.475, 3.25) {};
		\node  (10) at (1.25, 4) {};
		\node  (11) at (3.525, 3.25) {};
		\node  (12) at (4.475, 3.25) {};
		\node  (13) at (5, 4.5) {};
		\node  (14) at (3.75, 4.5) {};
		\node  (15) at (2.25, 4) {};
		\node  (19) at (2.75, 2.75) {};
		\node  (20) at (4.175, 2.625) {};
        \node (21) at (1.75, 4.5) {};
        \draw [thick,in=180, out=-90, looseness=1.50] (0.center) to (1.center);
		\draw [thick,in=-90, out=0] (1.center) to (2.center);
		\draw [thick] (4.center) to (3.center);
		\draw [thick] (3.center) to (5.center);
		\draw [thick,bend right=90, looseness=1.75] (6.center) to (4.center);
		\draw [thick,in=165, out=-90, looseness=0.75] (10.center) to (8.center);
		\draw [thick,in=-90, out=45, looseness=1.25] (12.center) to (13.center);
		\draw [thick,in=-90, out=135, looseness=1.25] (11.center) to (15.center);
         \draw [white, fill = white] (2.9, 3.525) circle (1.3ex); 
		\draw [thick,in=-90, out= 45, looseness=1.5] (9.center) to (14.center);
		\draw [thick,in=105, out=0] (7.center) to (12.center);
		\draw [thick,in=105, out=0] (6.center) to (9.center);
		\draw [thick,bend right=15] (4.center) to (19.center);
		\draw [thick,in=-120, out=45] (19.center) to (12.center);
        \draw [white, fill = white] (2.75,2.75) circle (1.3ex);
		\draw [thick, in=135, out=-90] (9.center) to (19.center);
		\draw [thick,in=60, out=-45] (19.center) to (20.center);
		\draw [thick,bend left=15] (20.center) to (5.center);
        \draw [white, fill = white] (3.475, 2.975) circle (1.3ex);
        \draw [white, fill = white] (3.575, 2.675) circle (1.3ex);
        \draw [thick,bend left=270, looseness=1.75] (7.center) to (5.center);
        \draw [thick, in=180, out=90] (10.center) to (21.center);
        \draw [thick, in=90, out=0] (21.center) to (15.center);
        \draw [fill = black] (3.center) circle  (0.8ex);
        \draw [fill = black] (4.center) circle  (0.8ex);
        \draw [fill = black] (5.center) circle  (0.8ex);
        \draw [fill = black] (9.center) circle  (0.8ex);
        \draw [fill = black] (8.center) circle  (0.8ex);
        \draw [fill = black] (11.center) circle (0.8ex);
        \draw [fill = black] (12.center) circle (0.8ex);
            \node[above] at (2.5,4.5) {\scriptsize $h \hspace{4.5mm}< \hspace{5 mm} y< y$};
        \end{tikzpicture} \\
        & + \fr(y)\begin{tikzpicture}[scale = 0.4, baseline = 1cm]
         \node  (0) at (0, 4.5) {};
		\node  (1) at (2.5, 1) {};
		\node  (2) at (3, 1.75) {};
		\node  (3) at (3, 1.75) {};
		\node  (4) at (2, 2.5) {};
		\node  (5) at (4, 2.5) {};
		\node  (6) at (2, 3.5) {};
		\node  (7) at (4, 3.5) {};
		\node  (8) at (1.525, 3.25) {};
		\node  (9) at (2.475, 3.25) {};
		\node  (10) at (1.25, 4.5) {};
		\node  (11) at (3.525, 3.25) {};
		\node  (12) at (4.475, 3.25) {};
		\node  (13) at (4.75, 4) {};
		\node  (14) at (3.75, 4) {};
		\node  (15) at (2.5, 4.5) {};
		\node  (17) at (4.25, 4.5) {};
		\node  (19) at (2.75, 2.75) {};
		\node  (20) at (4.175, 2.625) {};
        \draw [thick,in=180, out=-90, looseness=1.50] (0.center) to (1.center);
		\draw [thick,in=-90, out=0] (1.center) to (2.center);
		\draw [thick] (4.center) to (3.center);
		\draw [thick] (3.center) to (5.center);
		\draw [thick,bend right=90, looseness=1.75] (6.center) to (4.center);
		\draw [thick,in=165, out=-90, looseness=0.75] (10.center) to (8.center);
		\draw [thick,in=-90, out=45, looseness=1.25] (12.center) to (13.center);
		\draw [thick,in=-90, out=135, looseness=1.25] (11.center) to (15.center);
         \draw [white, fill = white] (3.075, 3.575) circle (1.3ex); 
		\draw [thick,in=-90, out=60, looseness=0.75] (9.center) to (14.center);
		\draw [thick,in=180, out=90] (14.center) to (17.center);
		\draw [thick,in=90, out=0] (17.center) to (13.center);
		\draw [thick,in=105, out=0] (7.center) to (12.center);
		\draw [thick,in=105, out=0] (6.center) to (9.center);
		\draw [thick,bend right=15] (4.center) to (19.center);
		\draw [thick,in=-120, out=45] (19.center) to (12.center);
        \draw [white, fill = white] (2.75,2.75) circle (1.3ex);
		\draw [thick, in=135, out=-90] (9.center) to (19.center);
		\draw [thick,in=60, out=-45] (19.center) to (20.center);
		\draw [thick,bend left=15] (20.center) to (5.center);
        \draw [white, fill = white] (3.475, 2.975) circle (1.3ex);
        \draw [white, fill = white] (3.575, 2.675) circle (1.3ex);
        \draw [thick,bend left=270, looseness=1.75] (7.center) to (5.center);
        \draw [fill = black] (3.center) circle  (0.8ex);
        \draw [fill = black] (4.center) circle  (0.8ex);
        \draw [fill = black] (5.center) circle  (0.8ex);
        \draw [fill = black] (9.center) circle  (0.8ex);
        \draw [fill = black] (8.center) circle  (0.8ex);
        \draw [fill = black] (11.center) circle  (0.8ex);
        \draw [fill = black] (12.center) circle  (0.8ex);
         \node[above] at (1.25 ,4.5) {\scriptsize $h < x < x$};
        \end{tikzpicture} + \fr(x)\fr(y)\begin{tikzpicture}[scale = 0.4, baseline = 1cm]
         \node  (0) at (0, 4.5) {};
		\node  (1) at (2.5, 1) {};
		\node  (2) at (3, 1.75) {};
		\node  (3) at (3, 1.75) {};
		\node  (4) at (2, 2.5) {};
		\node  (5) at (4, 2.5) {};
		\node  (6) at (2, 3.5) {};
		\node  (7) at (4, 3.5) {};
		\node  (8) at (1.525, 3.25) {};
		\node  (9) at (2.475, 3.25) {};
		\node  (10) at (1.25, 4) {};
		\node  (11) at (3.525, 3.25) {};
		\node  (12) at (4.475, 3.25) {};
		\node  (13) at (4.75, 4) {};
		\node  (14) at (3.75, 4) {};
		\node  (15) at (2.25, 4) {};
		\node  (17) at (4.25, 4.5) {};
		\node  (19) at (2.75, 2.75) {};
		\node  (20) at (4.175, 2.625) {};
        \node (21) at (1.75, 4.5) {};
        \draw [thick,in=180, out=-90, looseness=1.50] (0.center) to (1.center);
		\draw [thick,in=-90, out=0] (1.center) to (2.center);
		\draw [thick] (4.center) to (3.center);
		\draw [thick] (3.center) to (5.center);
		\draw [thick,bend right=90, looseness=1.75] (6.center) to (4.center);
		\draw [thick,in=165, out=-90, looseness=0.75] (10.center) to (8.center);
		\draw [thick,in=-90, out=45, looseness=1.25] (12.center) to (13.center);
		\draw [thick,in=-90, out=135, looseness=1.25] (11.center) to (15.center);
         \draw [white, fill = white] (2.9, 3.525) circle (1.3ex); 
		\draw [thick,in=-90, out=60, looseness=0.75] (9.center) to (14.center);
		\draw [thick,in=180, out=90] (14.center) to (17.center);
		\draw [thick,in=90, out=0] (17.center) to (13.center);
		\draw [thick,in=105, out=0] (7.center) to (12.center);
		\draw [thick,in=105, out=0] (6.center) to (9.center);
		\draw [thick,bend right=15] (4.center) to (19.center);
		\draw [thick,in=-120, out=45] (19.center) to (12.center);
        \draw [white, fill = white] (2.75,2.75) circle (1.3ex);
		\draw [thick, in=135, out=-90] (9.center) to (19.center);
		\draw [thick,in=60, out=-45] (19.center) to (20.center);
		\draw [thick,bend left=15] (20.center) to (5.center);
        \draw [white, fill = white] (3.475, 2.975) circle (1.3ex);
        \draw [white, fill = white] (3.575, 2.675) circle (1.3ex);
        \draw [thick,bend left=270, looseness=1.75] (7.center) to (5.center);
        \draw [thick, in=180, out=90] (10.center) to (21.center);
        \draw [thick, in=90, out=0] (21.center) to (15.center);
        \draw [fill = black] (3.center) circle  (0.8ex);
        \draw [fill = black] (4.center) circle  (0.8ex);
        \draw [fill = black] (5.center) circle  (0.8ex);
        \draw [fill = black] (9.center) circle  (0.8ex);
        \draw [fill = black] (8.center) circle  (0.8ex);
        \draw [fill = black] (11.center) circle (0.8ex);
        \draw [fill = black] (12.center) circle (0.8ex);
         \node[above] at (0,4.5) {\scriptsize $h$};
        \end{tikzpicture} \\
        & = 
 \begin{tikzpicture}[scale = 0.4, baseline = 1cm]
        \node  (0) at (-0.5, 4.5) {};
		\node  (1) at (2.5, 1) {};
		\node  (2) at (3, 1.75) {};
		\node  (3) at (3, 1.75) {};
		\node  (4) at (2, 2.5) {};
		\node  (5) at (4, 2.5) {};
		\node  (6) at (2, 3.5) {};
		\node  (7) at (4, 3.5) {};
		\node  (8) at (1.525, 3.25) {};
		\node  (9) at (2.475, 3.25) {};
		\node  (10) at (1.25, 4.5) {};
		\node  (11) at (3.525, 3.25) {};
		\node  (12) at (4.475, 3.25) {};
		\node  (13) at (5, 4.5) {};
		\node  (14) at (3.75, 4.5) {};
		\node  (15) at (2.5, 4.5) {};
		\node  (17) at (2.75, 2.75) {};
		\node  (18) at (4.175, 2.625) {};
		\draw [thick,in=180, out=-90, looseness=1.50] (0.center) to (1.center);
		\draw [thick,in=-90, out=0] (1.center) to (2.center);
		\draw [thick] (4.center) to (3.center);
		\draw [thick] (3.center) to (5.center);
		\draw [thick,bend right=90, looseness=1.75] (6.center) to (4.center);
		\draw [thick,in=165, out=-90, looseness=0.75] (10.center) to (8.center);
		\draw [thick,in=-90, out=30, looseness=0.75] (12.center) to (13.center);
		\draw [thick,in=-90, out=135, looseness=1.25] (11.center) to (15.center);
        \draw [white, fill = white] (3.025, 3.65) circle (1.3ex);
		\draw [thick,in=-90, out=45] (9.center) to (14.center);
		\draw [thick,in=105, out=0] (7.center) to (12.center);
		\draw [thick,in=105, out=0] (6.center) to (9.center);
		\draw [thick,bend right=15] (4.center) to (17.center);
		\draw [thick,in=-120, out=45] (17.center) to (12.center);
        \draw [white, fill = white] (2.75, 2.75) circle (1.3ex); 
		\draw [thick,in=135, out=-90] (9.center) to (17.center);
		\draw [thick,in=60, out=-45] (17.center) to (18.center);
		\draw [thick,bend left=15] (18.center) to (5.center);
         \draw [white, fill = white] (3.475, 2.975) circle (1.3ex);
        \draw [white, fill = white] (3.575, 2.675) circle (1.3ex);
        \draw [thick,bend left=270, looseness=1.75] (7.center) to (5.center);
       \draw [fill = black] (3.center) circle  (0.8ex);
        \draw [fill = black] (4.center) circle  (0.8ex);
        \draw [fill = black] (5.center) circle  (0.8ex);
        \draw [fill = black] (9.center) circle  (0.8ex);
        \draw [fill = black] (8.center) circle  (0.8ex);
        \draw [fill = black] (11.center) circle  (0.8ex);
        \draw [fill = black] (12.center) circle  (0.8ex);
        \node[above] at (3.1,4.5) {\scriptsize $ {x} \parallel {x} \hspace{0.3mm} <\hspace{0.3mm} {y} \parallel {y}$};
        \node[above] at (-0.1,4.6) {\scriptsize $ h \hspace{1mm}< $};
\end{tikzpicture}
\end{align*}

\up 
\hfill $\blacksquare$
\end{example}

\begin{proposition} \label{prop:Psi_psi}
Let $n>0$. The homomorphism $E$ 
is well-defined and $\Sp(H)$-equivariant.
Furthermore, 
it fits into the following commutative diagram:
\begin{equation} \label{eq:Psi_psi}
\xymatrix{
\calB^<_n\big(P^{(2)}\big) \ar[r]^-\Psi \ar[d]_-E &
\Tors_2\Big({\displaystyle\frac{Y_{2n+1} \calIC}{Y_{2n+2}}}\Big) \\ 
\Tors_2\big(\calA^{<,c}_{2n+1}(H)\big) \ar[ru]_-\psi &  
}
\end{equation}
\end{proposition}

\begin{proof}
To prove that the homomorphism $E$ is well-defined, 
we need to check that the quantity \eqref{eq:E_def} 
written for a generator $D$ of $\calB^<_n\big(P^{(2)}\big)$
is compatible with each of the defining relations of this module.

Clearly, $E$ is compatible with the group relations,
since they merely encode the ambiguity arising in the description of a group homomorphism $H_1(D) \to \Z_2$ in terms of beads, which is equivalent to the specification of the $2$-fold cover $\widetilde{D}$ of $D$.
The self-loop relation is also obviously respected,
by using the first observation in Remark \ref{rem:few}.
The homomorphism $E$ also respects the AS relation, 
because of the AS relation in $\calA^{<,c}(H)$ 
and the fact that \eqref{eq:E_def} is a $2$-torsion element. 
To prove that $E$ is compatible with the IHX relation,
we use the same combinatorial argument as for Lemma \ref{lem:psi_IHX}
(using ``IHX'' in $\calA^{<,c}(H)$ 
and the symmetry of $\widetilde{D}$).

To deal with the BSTU-like relation,
we consider two consecutive leaves in a Jacobi diagram
$D \in \calB^<_n\big(P^{(2)}\big)$ with colors $x,y\in P^{(2)}$:
$$
D= \begin{tikzpicture}[thick, baseline=-0.9em, x=2em, y=2em]
\footnotesize
\draw[dotted] (0,0.5) -- (0,0);
\draw (0,0) -- (0,-0.8);
\draw[dotted] (0.8,0.5) -- (0.8,0);
\draw (0.8,0) -- (0.8,-0.8);
\node[below=1pt] at (0.4,-0.8) {$\cdots<x<y<\cdots$};
\end{tikzpicture}
$$
The sum defining $E(D)$, indexed by the subsets $S$
of the set of leaves of $D$, 
splits into 4 sums, corresponding to whether or not
each of the two leaves under consideration belongs to $S$ or not.
Thus, each summand of $E(D)$ can be depicted locally 
as one of the following cases:
\begin{center}
\begin{tabular}{c|c}
\setlength{\tabcolsep}{10pt}
  \begin{tikzpicture}[thick, baseline=-0.9em, x=2em, y=2em]
\footnotesize
\draw[dotted] (-0.95,0.5) -- (-0.95,0);
\draw[dotted] (-0.65,0.5) -- (-0.65,0);
\draw (-0.95,0)  to [out=-90,in=90](-1.2,-1);
\draw (-0.65,0)  to [out=-90,in=90](-0.4,-1);
\draw[dotted] (0.95,0.5) -- (0.95,0);
\draw[dotted] (0.65,0.5) -- (0.65,0);
\draw (0.95,0)  to [out=-90,in=90](1.2,-1);
\draw (0.65,0)  to [out=-90,in=90](0.4,-1);
\node[below=1pt] at (0,-1) {$\cdots<{x}<{x}<{y}<{y}<\cdots$};
\end{tikzpicture}&
$\fr(x)\fr(y)$\begin{tikzpicture}[thick, x=2em, y=2em, baseline=(mathaxis)]
\footnotesize
\path[use as bounding box] (-1.6,0.6) rectangle (1.6,-1.4);
\coordinate (mathaxis) at (0,-0.25);
\draw[dotted] (-0.95,0.5) -- (-0.95,0);
\draw[dotted] (-0.65,0.5) -- (-0.65,0);
\draw (-0.95,0)  to [out=-90,in=90](-1.2,-0.8);
\draw (-0.65,0)  to [out=-90,in=90](-0.4,-0.8);
\draw[dotted] (0.95,0.5) -- (0.95,0);
\draw[dotted] (0.65,0.5) -- (0.65,0);
\draw (0.95,0)  to [out=-90,in=90](1.2,-0.8);
\draw (0.65,0)  to [out=-90,in=90](0.4,-0.8);
\draw (-1.2,-0.8) to[out=-90, in=-90, ,looseness = 1.2] (-0.4,-0.8);
\draw (1.2,-0.8) to[out=-90, in=-90,looseness = 1.2] (0.4,-0.8);
\end{tikzpicture}
  \\ \hline
$\fr(y)$ \begin{tikzpicture}[thick, x=2em, y=2em, baseline=(mathaxis)]
\footnotesize
\path[use as bounding box] (-1.6,0.6) rectangle (1.6,-1.4);
\coordinate (mathaxis) at (0,-0.5);
\footnotesize
\draw[dotted] (-0.95,0.5) -- (-0.95,0);
\draw[dotted] (-0.65,0.5) -- (-0.65,0);
\draw (-0.95,0)  to [out=-90,in=90](-1.2,-0.8);
\draw (-0.65,0)  to [out=-90,in=90](-0.4,-0.8);
\draw[dotted] (0.95,0.5) -- (0.95,0);
\draw[dotted] (0.65,0.5) -- (0.65,0);
\draw (0.95,0)  to [out=-90,in=90](1.2,-0.8);
\draw (0.65,0)  to [out=-90,in=90](0.4,-0.8);
\draw (1.2,-0.8) to[out=-90, in=-90] (0.4,-0.8);
\node[below=3pt] at (-0.8,-0.9) {$\cdots<{x}<{x}<\cdots$};
\end{tikzpicture}&
 $\fr(x)$ \begin{tikzpicture}[thick, x=2em, y=2em, baseline=(mathaxis)]
\footnotesize
\path[use as bounding box] (-1.6,0.6) rectangle (1.6,-1.4);
\coordinate (mathaxis) at (0,-0.5);
\footnotesize
\draw[dotted] (-0.95,0.5) -- (-0.95,0);
\draw[dotted] (-0.65,0.5) -- (-0.65,0);
\draw (-0.95,0)  to [out=-90,in=90](-1.2,-0.8);
\draw (-0.65,0)  to [out=-90,in=90](-0.4,-0.8);
\draw[dotted] (0.95,0.5) -- (0.95,0);
\draw[dotted] (0.65,0.5) -- (0.65,0);
\draw (0.95,0)  to [out=-90,in=90](1.2,-0.8);
\draw (0.65,0)  to [out=-90,in=90](0.4,-0.8);
\draw (-1.2,-0.8) to[out=-90, in=-90] (-0.4,-0.8);
\node[below=3pt] at (0.8,-0.9) {$\cdots<{y}<{y}<\cdots$};
\end{tikzpicture}
\end{tabular}
\end{center}
Therefore, we obtain
\begin{align*}
&E\left(
\begin{tikzpicture}[thick, baseline=-0.9em, x=2em, y=2em]
\footnotesize
\draw[dotted] (0,0.5) -- (0,0);
\draw (0,0) -- (0,-0.8);
\draw[dotted] (0.8,0.5) -- (0.8,0);
\draw (0.8,0) -- (0.8,-0.8);
\node[below=1pt] at (0.4,-0.8) {$\cdots<x<y<\cdots$};
\end{tikzpicture} -
\begin{tikzpicture}[thick,baseline=-0.9em, x=2em, y=2em]
\footnotesize
\draw[dotted] (0,0.5) -- (0,0);
\draw[dotted] (1,0.5) -- (1,0);
\draw (0,0) to [out=-90,in=90] (1,-0.8);
\fill [white] (0.5,-0.5) circle [radius=0.2];
\draw (1,0) to [out=-90,in=90] (0,-0.8);
\node[below=1pt] at (0.5,-0.8) {$\cdots<y<x<\cdots$};
\end{tikzpicture}\right)\\
&= \sum
\begin{tikzpicture}[thick,baseline=-0.9em, x=2em, y=2em]
\footnotesize
\draw[dotted] (-0.95,0.5) -- (-0.95,0);
\draw[dotted] (-0.65,0.5) -- (-0.65,0);
\draw (-0.95,0)  to [out=-90,in=90](-1.2,-0.8);
\draw (-0.65,0)  to [out=-90,in=90](-0.4,-0.8);
\draw[dotted] (0.95,0.5) -- (0.95,0);
\draw[dotted] (0.65,0.5) -- (0.65,0);
\draw (0.95,0)  to [out=-90,in=90](1.2,-0.8);
\draw (0.65,0)  to [out=-90,in=90](0.4,-0.8);
\node[below=1pt] at (0,-0.8) {$\cdots<{x}<{x}<{y}<{y}<\cdots$};
\end{tikzpicture}
- \sum
\begin{tikzpicture}[thick,baseline=-0.9em, x=2em, y=2em]
\footnotesize
\draw[dotted] (-0.95,0.8) -- (-0.95,0.3);
\draw[dotted] (-0.65,0.8) -- (-0.65,0.3);
\draw[dotted] (0.95,0.8) -- (0.95,0.3);
\draw[dotted] (0.65,0.8) -- (0.65,0.3);
\draw (-0.65,0.3)  to [out=-90,in=90, looseness = 0.8](1.2,-0.8);
\draw (-0.95,0.3)  to [out=-90,in=90](0.4,-0.8);
\fill [white] (0,-0.35) circle [radius=0.12];
\fill [white] (0,-0.18) circle [radius=0.12];
\fill [white] (0.28,-0.24) circle [radius=0.12];
\fill [white] (-0.28,-0.24) circle [radius=0.12];
\draw (0.65,0.3)  to [out=-90,in=90, looseness = 0.8](-1.2,-0.8);
\draw (0.95,0.3)  to [out=-90,in=90](-0.4,-0.8);
\node[below=1pt] at (0,-0.8) {$\cdots<{y}<{y}<{x}<{x}<\cdots$};
\end{tikzpicture}\\
&=\omega(p_*({x}),p_*({y}))^2\, \left( \sum
\begin{tikzpicture}[thick,baseline=-0.4em, x=2em, y=2em]
\footnotesize
\draw[dotted] (-0.95,0.5) -- (-0.95,0);
\draw[dotted] (-0.55,0.5) -- (-0.55,0);
\draw[dotted] (0.95,0.5) -- (0.95,0);
\draw[dotted] (0.55,0.5) -- (0.55,0);
\draw (-0.95,0)  to [out=-90,in=-90](0.55,0);
\fill [white] (0,-0.4) circle [radius=0.3];
\draw (-0.55,0)  to [out=-90,in=-90](0.95,0);
\node[below=1pt] at (-1,-0.6) {$\cdots$};
\node[below=1pt] at (1.1,-0.6) {$\cdots$};
\end{tikzpicture}
+ \sum
\begin{tikzpicture}[thick,baseline=-0.4em, x=2em, y=2em]
\footnotesize
\draw[dotted] (-0.95,0.5) -- (-0.95,0);
\draw[dotted] (-0.55,0.5) -- (-0.55,0);
\draw[dotted] (0.95,0.5) -- (0.95,0);
\draw[dotted] (0.65,0.5) -- (0.65,0);
\draw (-0.95,0)  to [out=-90,in=-90](0.95,0);
\draw (-0.55,0)  to [out=-90,in=-90](0.55,0);
\node[below=1pt] at (-1,-0.6) {$\cdots$};
\node[below=1pt] at (1.1,-0.6) {$\cdots$};
\end{tikzpicture}
\right)\\
&=\omega(p_*({x}),p_*({y}))\, E\left( 
\begin{tikzpicture}[thick,baseline=-0.9em, x=2em, y=2em]
\footnotesize
\draw[dotted] (0,0.5) -- (0,0);
\draw[dotted] (1,0.5) -- (1,0);
\draw (0,0) -- (0,-0.5);
\draw (1,0) -- (1,-0.5);
\draw (0,-0.5) to [out=-90,in=-90] (1,-0.5) ;
\node[below=1pt] at (-0.3,-0.7) {$\cdots$};
\node[below=1pt] at (1.4,-0.7) {$\cdots$};
\end{tikzpicture}
+  \begin{tikzpicture}[thick,baseline=-0.9em, x=2em, y=2em]
\draw[dotted] (0,0.5) -- (0,0);
\draw[dotted] (1,0.5) -- (1,0);
\draw (0,0) -- (0,-0.5);
\draw (1,0) -- (1,-0.5);
\draw (0,-0.5) to [out=-90,in=-90] (1,-0.5) ;
\node[below=1pt] at (-0.3,-0.7) {$\cdots$};
\node[below=1pt] at (1.4,-0.7) {$\cdots$};
 \draw[thin, black, fill=lightgray] (0.5,-0.8) circle (2pt);
\end{tikzpicture}\right).
\end{align*}
Here, the summations are taken over all subsets of the set of leaves 
of \(D\) that do not contain the two leaves under consideration; 
the second equality follows from the application of the STU-like relation (together with the AS relation) in \(\mathcal{A}^{<,c}(H)\), combined with an analysis of the symmetries of the corresponding Jacobi diagrams.

Next, we prove that  $E$ is compatible
with the multilinearity relation.
At the root of a Jacobi diagram $D$, 
this follows immediately from the definition of $E(D)$
using ``multilinearity'' in \(\mathcal{A}^{<,c}(H)\). 
At a  leaf of a Jacobi diagram $D(x+y)$ colored by $x+y\in P^{(2)}$, 
we can write locally $E(D(x+y))-E(D(x))-E(D(y))$
in the following way
using the notation \eqref{eq:parallel_notation}: 
\begin{eqnarray*}
&&\begin{tikzpicture}[thick,baseline=-1.5em, x=1em, y=1em]
\footnotesize
\draw[dotted] (0.7,-0.8)-- (0.7,0);
\draw[dotted] (-0.7,-0.8)-- (-0.7,0);
\draw (0.7,-2)-- (0.7,-0.8);
\draw (-0.7,-2)-- (-0.7,-0.8);
\node[below] at (0,-2) {$\cdots<x+y \parallel x+y<\cdots$};
\end{tikzpicture}
- \begin{tikzpicture}[thick,baseline=-1.5em, x=1em, y=1em]
\footnotesize
\draw[dotted] (0.7,-0.8)-- (0.7,0);
\draw[dotted] (-0.7,-0.8)-- (-0.7,0);
\draw (0.7,-2)-- (0.7,-0.8);
\draw (-0.7,-2)-- (-0.7,-0.8);
\node[below] at (0,-2) {$\cdots<x \parallel x<\cdots$};
\end{tikzpicture} -
\begin{tikzpicture}[thick,baseline=-1.5em, x=1em, y=1em]
\footnotesize
\draw[dotted] (0.7,-0.8)-- (0.7,0);
\draw[dotted] (-0.7,-0.8)-- (-0.7,0);
\draw (0.7,-2)-- (0.7,-0.8);
\draw (-0.7,-2)-- (-0.7,-0.8);
\node[below] at (0,-2) {$\cdots<y \parallel y<\cdots$};
\end{tikzpicture}\\
&=& 
\Big( \begin{tikzpicture}[thick,baseline=-1.5em, x=1em, y=1em]
\footnotesize
\draw[dotted] (0.7,-0.8)-- (0.7,0);
\draw[dotted] (-0.7,-0.8)-- (-0.7,0);
\draw (0.7,-2)-- (0.7,-0.8);
\draw (-0.7,-2)-- (-0.7,-0.8);
\node[below] at (0,-2) {$\cdots<x+y<x+y<\cdots$};
\end{tikzpicture}
\!\!\!\! + {\fr(x+y)}\quad
\begin{tikzpicture}[thick,baseline=-1.5em, x=1em, y=1em]
\footnotesize]
\draw[dotted] (0.7,-0.8)-- (0.7,0);
\draw[dotted] (-0.7,-0.8)-- (-0.7,0);
 \draw (0.7,-0.8) to [out=-90, in=-90, looseness=3] (-0.7,-0.8);
\end{tikzpicture} \ \Big) 
- \!\!\!\! \begin{tikzpicture}[thick,baseline=-1.5em, x=1em, y=1em]
\footnotesize
\draw[dotted] (0.7,-0.8)-- (0.7,0);
\draw[dotted] (-0.7,-0.8)-- (-0.7,0);
\draw (0.7,-2)-- (0.7,-0.8);
\draw (-0.7,-2)-- (-0.7,-0.8);
\node[below] at (0,-2) {$\cdots<x \parallel x<\cdots$};
\end{tikzpicture} - \!\!\!\!
\begin{tikzpicture}[thick,baseline=-1.5em, x=1em, y=1em]
\footnotesize
\draw[dotted] (0.7,-0.8)-- (0.7,0);
\draw[dotted] (-0.7,-0.8)-- (-0.7,0);
\draw (0.7,-2)-- (0.7,-0.8);
\draw (-0.7,-2)-- (-0.7,-0.8);
\node[below] at (0,-2) {$\cdots<y \parallel y<\cdots$};
\end{tikzpicture}\\
&=&   \begin{tikzpicture}[thick,baseline=-1.5em, x=1em, y=1em]
\footnotesize
\draw[dotted] (0.7,-0.8)-- (0.7,0);
\draw[dotted] (-0.7,-0.8)-- (-0.7,0);
\draw (0.7,-2)-- (0.7,-0.8);
\draw (-0.7,-2)-- (-0.7,-0.8);
\node[below] at (0,-2) {$\cdots<x < y<\cdots$};
\end{tikzpicture}+
\begin{tikzpicture}[thick,baseline=-1.5em, x=1em, y=1em]
\footnotesize
\draw[dotted] (0.7,-0.8)-- (0.7,0);
\draw[dotted] (-0.7,-0.8)-- (-0.7,0);
\draw (0.7,-2)-- (0.7,-0.8);
\draw (-0.7,-2)-- (-0.7,-0.8);
\node[below] at (0,-2) {$\cdots<y < x<\cdots$};
\end{tikzpicture} + {\omega(x,y)}\quad
\begin{tikzpicture}[thick,baseline=-1.5em, x=1em, y=1em]
\footnotesize]
\draw[dotted] (0.7,-0.8)-- (0.7,0);
\draw[dotted] (-0.7,-0.8)-- (-0.7,0);
 \draw (0.7,-0.8) to [out=-90, in=-90, looseness=3] (-0.7,-0.8);
\end{tikzpicture} \ = \ 0.
\end{eqnarray*}
Here the second equality is obtained 
by ``multilinearity'' in $\calA^{<,c}(H)$ using \eqref{eq:fr2};
as for  the third equality, 
it follows from ``STU-like'' 
and ``AS'' in $\calA^{<,c}(H)$
using the symmetry of the Jacobi diagrams.

We now prove that $E$ is compatible
with the fully-special relation.
If $D \in  \calB^<_n\big(P^{(2)}\big)$ is a Jacobi diagram
whose root is colored by $s$, then  $E(D)$ is obviously trivial 
since the diagram in each  summand of \eqref{eq:E_def} 
has a vertex colored by $0\in H$.
If  now $D \in  \calB^<_n\big(P^{(2)}\big)$ is a Jacobi diagram
with all leaves colored by $s$, 
then the formula \eqref{eq:E_def} for $E(D)$
reduces to a Jacobi diagram in 
$\calA^{<,c}_{2n+1}(H)$ with a single external vertex:
this is zero by ``IHX'' and ``self-loop''.
Thus, we have completed the proof of the well-definedness of $E$.

To prove the commutativity of \eqref{eq:Psi_psi},
let $D$ be a generator of $\calB^<_n\big(P^{(2)}\big)$.
Choose an oriented framed knot $K_0$ 
whose type in $P^{(2)}$ is the color of the root of $D$
and, for every $i\in \{1,\dots,\ell\}$, 
choose an oriented framed knot $K_i$ 
whose type in $P^{(2)}$ is the color of the $i$-th leaf of $D$.
Then, the commutativity of \eqref{eq:Psi_psi} says that
\begin{equation} \label{eq:Psi_E}
\psi(E(D)) = \{U_G\}
\in \frac{Y_{2n+1} \calIC}{Y_{2n+2}}
\end{equation}
where
$$
G:= G(D\hbox{\small \ uncolored};K_0, K_1, \dots, K_\ell)
$$
is a graph clasper  constructed as in the proof of Theorem \ref{th:Psi}.
Recall that $G$ is a topological realization 
of the Jacobi diagram $\widetilde{D}_+$
whose leaves are organized by slices: the $0$-th slice of $U$  contains $K_0$
and, for every $i\in \{1,\dots,\ell\}$, 
the $i$-th slice of $U$ contains two parallel copies $K'_i,K''_i$ of $K_i$.
We apply Lemma \ref{lem:generalized_STU-like} for each pair 
of leaves $(K'_i,K''_i)$ of $G$, observing that
$$
\operatorname{Lk}_\pm(K'_i,K''_i)
= \operatorname{Lk}(K'_i,K''_i) = \operatorname{Fr}(K_i);
$$
we deduce that $\{U_G\}$ is 
in ${Y_{2n+1} \calIC}/{Y_{2n+2}}$
the sum of $2^\ell$ elements, namely $\psi(E(D))$,
thus proving \eqref{eq:Psi_E}.
\end{proof}

\subsection{The map $\Delta$}

\label{subsec:lower_bound}

We assume here that $n>0$
(the special case $n=0$ being fully considered 
in \S \ref{sec:deg_1}). 
We define a homomorphism 
$$
\Delta:  \calB^<_n\big(P^{(2)}\big)
\longrightarrow \calA^{<,c}_{2n+2}(H) \otimes \Q/\Z
$$
by assigning to every rooted beaded Jacobi diagram  $D$
the linear combination 
\begin{equation} \label{eq:Delta_def}
\Delta(D):=\sum_{S}  \frac{1}{2} 
\Big(\prod_{j \in S}   \fr(c_j) \Big) 
\cdot \Big(\widetilde{D}_{++}^S + \sum_{i \not\in S } \widetilde{D}_{+}^{S,i} \Big).
\end{equation}
Here the sum is over the subsets $S$
of the set of leaves of $D$,
the inner sum is over the leaves $i$ of $D$ such that $i\not \in S$,
and the inner summands $\widetilde{D}_{+}^{S,i},\widetilde{D}_{++}^S$ 
are defined as follows:
\begin{itemize}
\item $\widetilde{D}_+$
is (as in the proof of Theorem \ref{th:Psi})
the augmentation of the $2$-fold cover $\widetilde{D}$
by a $Y$-diagram,
and its $(S,i)$-modification  
$$\widetilde{D}_+^{S,i}\in \calA^{<,c}_{2n+2}(P)\otimes \Z_2
\simeq \calA^{<,c}_{2n+2}(H)\otimes \Z_2 $$ 
is constructed as follows: 
the root of $\widetilde{D}_+$ 
is given the same color as the root of $D$
(lifting this from $P^{(2)}$ to $P$)
and is declared to be ``lower'' 
than all other vertices; 
for every leaf $v$ of $D$,
the corresponding twin vertices of $\widetilde{D}_+$
are either  glued to a $Y$-diagram 
whose free vertex is colored by (a lift of) $c_i$ 
and is ordered like $i$ in $D$  if $v=i$,
or they both inherit the color $c_v$ 
(lifting this from $P^{(2)}$ to $P$)
and the position of $v$ in $D$ if $v\not\in S\cup\{i\}$,
or they are glued together if $v \in S$;
\item  $\widetilde{D}_{++}$ is the augmentation of $\widetilde{D}$ 
obtained by gluing the twin root vertices of $\widetilde{D}$ to 
the ``top''  external vertices  of a $H$-diagram, 
and its $S$-modification
$$
\widetilde{D}_{++}^S\in \calA^{<,c}_{2n+2}(P)\otimes \Z_2 
\simeq \calA^{<,c}_{2n+2}(H)\otimes \Z_2  
$$ 
is constructed as follows:
the ``bottom'' vertices of the extra $H$-diagram are given 
the same color as the root of $D$ (lifting this from $P^{(2)}$ to $P$)
and  declared to be ``lower'' 
than all other vertices;
for every leaf $v$ of $D$, 
the corresponding twin vertices of $\widetilde{D}_{++}$
are either  glued together if $v\in S$,
or they both inherit the color $c_v$
(lifting this from $P^{(2)}$ to $P$)
and the position of $v$ in $D$ if $v\not\in S$.
\end{itemize}
Alternatively, we can write \eqref{eq:Delta_def} as 
\begin{equation} \label{eq:Delta_def_bis}
\Delta(D) :=   \sum_{S}  \frac{1}{2} 
\Big(\prod_{j \in S}   \fr(c_j) \Big) \cdot \widetilde{D}_{++}^S 
+    \sum_i \sum_{S \not\ni i}
\frac{1}{2} \Big(\prod_{j \in S}   \fr(c_j) \Big) \cdot  \widetilde{D}_{+}^{S,i} 
\end{equation}
where the second sum is indexed by leaves $i$ of $D$,
and its inner sum is over  subsets $S$
of the set of leaves of $D$ not containing $i$.

\begin{example}
Here are the values of $\Delta$
on the two Jacobi diagrams of  Example~\ref{ex:E} 
 using convention \eqref{eq:lifts}
 and notation \eqref{eq:parallel_notation}:
$$ \Delta(\begin{tikzpicture}[scale = 0.5,baseline = 0.1cm]
		\node (0) at (1, 0) {};
		\node (1) at (1, -0.75) {};
		\node (2) at (1, 1) {};
		\node (3) at (0.55, 0.775) {};
		\node (4) at (1.45, 0.775) {};
		\node (5) at (0, 1.25) {};
		\node (6) at (2, 1.25) {};
		\draw [thick] (0.center) to (1.center);
        \draw [thick] (0.center) to (1.center);
		\draw [thick] [in=180, out=-180, looseness=1.75] (2.center) to (0.center);
		\draw [thick] [bend left=90, looseness=1.75] (2.center) to (0.center);
		\draw [thick] (5.center) to (3.center);
		\draw [thick] (6.center) to (4.center);
        \draw [fill = black] (0.center) circle  (0.8ex);
        \draw [fill = black] (3.center) circle  (0.8ex);
        \draw [fill = black] (4.center) circle  (0.8ex);
        \node[below=0.5pt] at (1,-0.75) {\scriptsize $h \star$};
        \node[above] at (1,1.2) {\scriptsize $x$ \hspace{1mm} $<$ \hspace{1mm} $y$};
\end{tikzpicture}) =  \frac{1}{2}\begin{tikzpicture}[thick, scale = 0.4, baseline = 1.2cm]
	    \node (0) at (-0.6, 4.5) {};
	    \node (17) at (-1.85, 4.5) {};
		\node  (1) at (2, 1.35) {};
	    \node (18) at (3, 1) {};
		\node  (2) at (2.25, 2) {};
        \node (19) at (-1.50, 4.5) {};
		\node  (3) at (3.75, 2) {};
		\node  (4) at (2, 2.75) {};
		\node  (5) at (4, 2.75) {};
		\node  (7) at (2, 3.75) {};
		\node  (8) at (4, 3.75) {};
		\node  (9) at (1.525, 3.5) {};
		\node  (10) at (2.475, 3.5) {};
		\node  (11) at (1, 4.5) {};
		\node  (12) at (3.525, 3.5) {};
		\node  (13) at (4.475, 3.5) {};
		\node  (14) at (5, 4.5) {};
		\node  (15) at (3.75, 4.5) {};
		\node  (16) at (2.25, 4.5) {};
		\draw [in=180, out=-90, looseness=1.50] (0.center) to (1.center);		
        \draw [in=180, out=-90, looseness=1.50] (17.center) to (18.center);
		\draw [in=-90, out=0] (1.center) to (2.center);
        \draw [in=-90, out=0] (18.center) to (3.center);
        \draw (2.center) to (3.center);
		\draw (4.center) to (2.center);
		\draw (3.center) to (5.center);
		\draw [bend right=90, looseness=1.75] (7.center) to (4.center);
		\draw [bend left=90, looseness=1.75] (7.center) to (4.center);
		\draw [bend left=270, looseness=1.75] (8.center) to (5.center);
		\draw [bend left=90, looseness=1.75] (8.center) to (5.center);
		\draw [in=-180, out=-90] (11.center) to (9.center);
		\draw [in=-90, out=0] (13.center) to (14.center);
		\draw [in=-90, out=90] (12.center) to (16.center);
        \draw [white, fill = white] (3, 3.975) circle (1.3ex);
		\draw [in=-90, out=90] (10.center) to (15.center);
        \draw [fill = black] (3.center) circle  (0.8ex);
        \draw [fill = black] (4.center) circle  (0.8ex);
        \draw [fill = black] (2.center) circle  (0.8ex);
        \draw [fill = black] (5.center) circle  (0.8ex);
        \draw [fill = black] (9.center) circle  (0.8ex);
        \draw [fill = black] (10.center) circle  (0.8ex);
        \draw [fill = black] (13.center) circle  (0.8ex);
        \draw [fill = black] (12.center) circle  (0.8ex);
        \node[above] at (3,4.5) {\scriptsize ${x} \parallel {x} \hspace{1mm} <\hspace{1mm} {y} \parallel {y}$};
        \node[above] at (-0.9,4.6) {\scriptsize $h < h< $};
\end{tikzpicture} + \frac{1}{2} \begin{tikzpicture}[scale = 0.4, baseline = 1.2cm]
        \node  (0) at (0, 5) {};
		\node  (1) at (2.5, 1) {};
		\node  (2) at (3, 1.75) {};
		\node  (3) at (3, 1.75) {};
		\node  (4) at (2, 2.5) {};
		\node  (5) at (4, 2.5) {};
		\node  (6) at (2, 3.5) {};
		\node  (7) at (4, 3.5) {};
		\node  (8) at (1.525, 3.25) {};
		\node  (9) at (2.475, 3.25) {};
		\node  (10) at (1.25, 4) {};
		\node  (11) at (3.525, 3.25) {};
		\node  (12) at (4.475, 3.25) {};
		\node  (13) at (5, 5) {};
		\node  (14) at (3 .75, 5) {};
		\node  (15) at (2.25, 4) {};
		\node  (16) at (1.75, 4.5) {};
		\node  (17) at (1.75, 4.5) {};
		\node  (18) at (1.75, 5) {};
        \draw [thick] (17.center) to (18.center);
        \draw [thick,in=180, out=-90, looseness=1.50] (0.center) to (1.center);
		\draw [thick,in=-90, out=0] (1.center) to (2.center);
		\draw [thick] (4.center) to (3.center);
		\draw [thick] (3.center) to (5.center);
		\draw [thick ,bend right=90, looseness=1.75] (6.center) to (4.center);
		\draw [thick,bend left=90, looseness=1.75] (6.center) to (4.center);
		\draw [thick,bend left=270, looseness=1.75] (7.center) to (5.center);
		\draw [thick,bend left=90, looseness=1.75] (7.center) to (5.center);
		\draw [thick,in=180, out=-90] (10.center) to (8.center);
		\draw [thick,in=-90, out=45, looseness=1.25] (12.center) to (13.center);
		\draw [thick, in=-90, out=90] (11.center) to (15.center);
        \draw [white, fill = white] (2.675, 3.65) circle (1ex); 
		\draw [thick,in=-90, out=75] (9.center) to (14.center);
		\draw [thick,bend left=45] (10.center) to (16.center);
		\draw [thick,bend left=45] (16.center) to (15.center);
        \draw [fill = black] (3.center) circle  (0.8ex);
        \draw [fill = black] (4.center) circle  (0.8ex);
        \draw [fill = black] (5.center) circle  (0.8ex);
        \draw [fill = black] (9.center) circle  (0.8ex);
        \draw [fill = black] (8.center) circle  (0.8ex);
        \draw [fill = black] (11.center) circle  (0.8ex);
        \draw [fill = black] (12.center) circle  (0.8ex);
         \draw [fill = black] (17.center) circle  (0.8ex);
          \node[above] at (2.4,5) {\scriptsize $h  \hspace{1mm} < \hspace{1mm} x \hspace{1mm} < \hspace{2mm} y\parallel y$};
        \end{tikzpicture} + \frac{1}{2}
        \begin{tikzpicture}[scale = 0.4, baseline = 1.2cm] 
        \node  (0) at (0, 5) {};
		\node  (1) at (2.5, 1) {};
		\node  (2) at (3, 1.75) {};
		\node  (3) at (3, 1.75) {};
		\node  (4) at (2, 2.5) {};
		\node  (5) at (4, 2.5) {};
		\node  (6) at (2, 3.5) {};
		\node  (7) at (4, 3.5) {};
		\node  (8) at (1.525, 3.25) {};
		\node  (9) at (2.475, 3.25) {};
		\node  (10) at (1.25, 5) {};
		\node  (11) at (3.525, 3.25) {};
		\node  (12) at (4.475, 3.25) {};
		\node  (13) at (4.75, 4) {};
		\node  (14) at (3.75, 4) {};
		\node  (15) at (2.25, 5) {};
		\node  (17) at (4.25, 4.5) {};
		\node  (18) at (4.25, 4.5) {};
		\node  (19) at (4.25, 5) {};
        \draw [thick] (18.center) to (19.center);
        \draw [thick, in=180, out=-90, looseness=1.50] (0.center) to (1.center);
		\draw [thick,in=-90, out=0] (1.center) to (2.center);
		\draw [thick] (4.center) to (3.center);
		\draw [thick] (3.center) to (5.center);
		\draw [ thick,bend right=90, looseness=1.75] (6.center) to (4.center);
		\draw [ thick,bend left=90, looseness=1.75] (6.center) to (4.center);
		\draw [ thick,bend left=270, looseness=1.75] (7.center) to (5.center);
		\draw [ thick,bend left=90, looseness=1.75] (7.center) to (5.center);
		\draw [ thick,in=165, out=-90, looseness=0.75] (10.center) to (8.center);
		\draw [ thick,in=-90, out=45, looseness=1.25] (12.center) to (13.center);
		\draw [ thick,in=-90, out=135, looseness=1.25] (11.center) to (15.center);
        \draw [white, fill = white] (3.1, 3.55) circle (1.3ex); 
		\draw [ thick,in=-90, out=90, looseness=0.75] (9.center) to (14.center);
		\draw [ thick,in=180, out=90] (14.center) to (17.center);
		\draw [ thick,in=90, out=0] (17.center) to (13.center);
        \draw [fill = black] (3.center) circle  (0.8ex);
        \draw [fill = black] (4.center) circle  (0.8ex);
        \draw [fill = black] (5.center) circle  (0.8ex);
        \draw [fill = black] (9.center) circle  (0.8ex);
        \draw [fill = black] (8.center) circle  (0.8ex);
        \draw [fill = black] (11.center) circle  (0.8ex);
        \draw [fill = black] (12.center) circle  (0.8ex);
        \draw [fill = black] (18.center) circle  (0.8ex);
         \node[above] at (2 ,5) {\scriptsize $h < x \parallel x \hspace{1mm}< \hspace{1mm}y$};
        \end{tikzpicture}$$

\begin{align*}
\Delta\Bigg(\begin{tikzpicture}[scale = 0.5,baseline = 0.1cm]
		\node (0) at (1, 0) {};
		\node (1) at (1, -0.75) {};
		\node (2) at (1, 1) {};
		\node (3) at (0.55, 0.775) {};
		\node (4) at (1.45, 0.775) {};
		\node (5) at (0, 1.25) {};
		\node (6) at (2, 1.25) {};
		\draw [thick] (0.center) to (1.center);
        \draw [thick] (0.center) to (1.center);
		\draw [thick] [in=180, out=-180, looseness=1.75] (2.center) to (0.center);
		\draw [thick] [bend left=90, looseness=1.75] (2.center) to (0.center);
		\draw [thick] (5.center) to (3.center);
		\draw [thick] (6.center) to (4.center);
        \draw [fill = black] (0.center) circle  (0.8ex);
        \draw [fill = black] (3.center) circle  (0.8ex);
        \draw [fill = black] (4.center) circle  (0.8ex);
        \node[below=0.5pt] at (1,-0.75) {\scriptsize $h \star$};
        \node[above] at (1,1.2) {\scriptsize $x$ \hspace{1mm} $<$ \hspace{1mm} $y$};
        \draw[black, fill=lightgray] (1.435,0.225) circle (0.95ex);
\end{tikzpicture}
 \Bigg) = & \frac{1}{2} 
  \begin{tikzpicture}[scale = 0.4, baseline = 1cm]
        \node  (0) at (-0.5, 4.5) {};
	    \node (19) at (-1.85, 4.5) {};
        \node (20) at (3, 0.5) {};
		\node  (1) at (2, 1) {};
		\node  (2) at (2.25, 1.75) {};
		\node  (3) at (3.75, 1.75) {};
		\node  (4) at (2, 2.5) {};
		\node  (5) at (4, 2.5) {};
		\node  (6) at (2, 3.5) {};
		\node  (7) at (4, 3.5) {};
		\node  (8) at (1.525, 3.25) {};
		\node  (9) at (2.475, 3.25) {};
		\node  (10) at (1.25, 4.5) {};
		\node  (11) at (3.525, 3.25) {};
		\node  (12) at (4.475, 3.25) {};
		\node  (13) at (5, 4.5) {};
		\node  (14) at (3.75, 4.5) {};
		\node  (15) at (2.5, 4.5) {};
		\node  (17) at (2.75, 2.75) {};
		\node  (18) at (4.175, 2.625) {};
		\draw [thick,in=180, out=-90, looseness=1.50] (0.center) to (1.center);
		\draw [thick,in=180, out=-90, looseness=1.50] (19.center) to (20.center);
		\draw [thick,in=-90, out=0] (1.center) to (2.center);
        \draw [thick,in=-90, out=0] (20.center) to (3.center);
		\draw [thick] (2.center) to (4.center);
        \draw [thick] (2.center) to (3.center);
		\draw [thick] (3.center) to (5.center);
		\draw [thick,bend right=90, looseness=1.75] (6.center) to (4.center);
		\draw [thick,in=165, out=-90, looseness=0.75] (10.center) to (8.center);
		\draw [thick,in=-90, out=30, looseness=0.75] (12.center) to (13.center);
		\draw [thick,in=-90, out=135, looseness=1.25] (11.center) to (15.center);
        \draw [white, fill = white] (3.025, 3.65) circle (1.3ex);
		\draw [thick,in=-90, out=45] (9.center) to (14.center);
		\draw [thick,in=105, out=0] (7.center) to (12.center);
		\draw [thick,in=105, out=0] (6.center) to (9.center);
		\draw [thick,bend right=15] (4.center) to (17.center);
		\draw [thick,in=-120, out=45] (17.center) to (12.center);
        \draw [white, fill = white] (2.75, 2.75) circle (1.3ex); 
		\draw [thick,in=135, out=-90] (9.center) to (17.center);
		\draw [thick,in=60, out=-45] (17.center) to (18.center);
		\draw [thick,bend left=15] (18.center) to (5.center);
         \draw [white, fill = white] (3.475, 2.975) circle (1.3ex);
        \draw [white, fill = white] (3.575, 2.675) circle (1.3ex);
        \draw [thick,bend left=270, looseness=1.75] (7.center) to (5.center);     
       \draw [fill = black] (2.center) circle  (0.8ex);
       \draw [fill = black] (3.center) circle  (0.8ex);
        \draw [fill = black] (4.center) circle  (0.8ex);
        \draw [fill = black] (5.center) circle  (0.8ex);
        \draw [fill = black] (9.center) circle  (0.8ex);
        \draw [fill = black] (8.center) circle  (0.8ex);
        \draw [fill = black] (11.center) circle  (0.8ex);
        \draw [fill = black] (12.center) circle  (0.8ex);
        \node[above] at (3.1,4.5) {\scriptsize $ {x} \parallel {x} \hspace{0.3mm} <\hspace{0.3mm} {y} \parallel {y}$};
        \node[above] at (-0.9,4.6) {\scriptsize $h < h< $};
\end{tikzpicture} + \frac{1}{2} \begin{tikzpicture}[scale = 0.4, baseline = 1.2cm]
         \node  (0) at (0, 5) {};
		\node  (1) at (2.5, 1) {};
		\node  (2) at (3, 1.75) {};
		\node  (3) at (3, 1.75) {};
		\node  (4) at (2, 2.5) {};
		\node  (5) at (4, 2.5) {};
		\node  (6) at (2, 3.5) {};
		\node  (7) at (4, 3.5) {};
		\node  (8) at (1.525, 3.25) {};
		\node  (9) at (2.475, 3.25) {};
		\node  (10) at (1.25, 4) {};
		\node  (11) at (3.525, 3.25) {};
		\node  (12) at (4.475, 3.25) {};
		\node  (13) at (5, 5) {};
		\node  (14) at (3.75, 5) {};
		\node  (15) at (2.25, 4) {};
		\node  (19) at (2.75, 2.75) {};
		\node  (20) at (4.175, 2.625) {};
        \node (21) at (1.75, 4.5) {};
        \node (22) at (1.75,4.5) {};
        \node (23) at (1.75,5) {};
        \draw [thick] (22.center) to (23.center);
        \draw [thick,in=180, out=-90, looseness=1.50] (0.center) to (1.center);
		\draw [thick,in=-90, out=0] (1.center) to (2.center);
		\draw [thick] (4.center) to (3.center);
		\draw [thick] (3.center) to (5.center);
		\draw [thick,bend right=90, looseness=1.75] (6.center) to (4.center);
		\draw [thick,in=165, out=-90, looseness=0.75] (10.center) to (8.center);
		\draw [thick,in=-90, out=45, looseness=1.25] (12.center) to (13.center);
		\draw [thick,in=-90, out=135, looseness=1.25] (11.center) to (15.center);
         \draw [white, fill = white] (2.9, 3.525) circle (1.3ex); 
		\draw [thick,in=-90, out= 45, looseness=1.5] (9.center) to (14.center);
		\draw [thick,in=105, out=0] (7.center) to (12.center);
		\draw [thick,in=105, out=0] (6.center) to (9.center);
		\draw [thick,bend right=15] (4.center) to (19.center);
		\draw [thick,in=-120, out=45] (19.center) to (12.center);
        \draw [white, fill = white] (2.75,2.75) circle (1.3ex);
		\draw [thick, in=135, out=-90] (9.center) to (19.center);
		\draw [thick,in=60, out=-45] (19.center) to (20.center);
		\draw [thick,bend left=15] (20.center) to (5.center);
        \draw [white, fill = white] (3.475, 2.975) circle (1.3ex);
        \draw [white, fill = white] (3.575, 2.675) circle (1.3ex);
        \draw [thick,bend left=270, looseness=1.75] (7.center) to (5.center);
        \draw [thick, in=180, out=90] (10.center) to (21.center);
        \draw [thick, in=90, out=0] (21.center) to (15.center);
        \draw [fill = black] (3.center) circle  (0.8ex);
        \draw [fill = black] (4.center) circle  (0.8ex);
        \draw [fill = black] (5.center) circle  (0.8ex);
        \draw [fill = black] (9.center) circle  (0.8ex);
        \draw [fill = black] (8.center) circle  (0.8ex);
        \draw [fill = black] (11.center) circle (0.8ex);
        \draw [fill = black] (12.center) circle (0.8ex);
        \draw [fill = black] (22.center) circle (0.8ex);
            \node[above] at (2.5,5) {\scriptsize $h \hspace{1mm}< \hspace{1mm} x \hspace{1mm} < \hspace{1 mm} y\parallel y$};
        \end{tikzpicture} + \frac{1}{2}\begin{tikzpicture}[scale = 0.4, baseline = 1.2cm]
         \node  (0) at (0, 5) {};
		\node  (1) at (2.5, 1) {};
		\node  (2) at (3, 1.75) {};
		\node  (3) at (3, 1.75) {};
		\node  (4) at (2, 2.5) {};
		\node  (5) at (4, 2.5) {};
		\node  (6) at (2, 3.5) {};
		\node  (7) at (4, 3.5) {};
		\node  (8) at (1.525, 3.25) {};
		\node  (9) at (2.475, 3.25) {};
		\node  (10) at (1.25, 5) {};
		\node  (11) at (3.525, 3.25) {};
		\node  (12) at (4.475, 3.25) {};
		\node  (13) at (4.75, 4) {};
		\node  (14) at (3.75, 4) {};
		\node  (15) at (2.25, 5) {};
		\node  (17) at (4.25, 4.5) {};
		\node  (19) at (2.75, 2.75) {};
		\node  (20) at (4.175, 2.625) {};
		\node  (21) at (4.25, 4.5) {};
		\node  (22) at (4.25, 5) {};
        \draw [thick] (21.center) to (22.center);
        \draw [thick,in=180, out=-90, looseness=1.50] (0.center) to (1.center);
		\draw [thick,in=-90, out=0] (1.center) to (2.center);
		\draw [thick] (4.center) to (3.center);
		\draw [thick] (3.center) to (5.center);
		\draw [thick,bend right=90, looseness=1.75] (6.center) to (4.center);
		\draw [thick,in=165, out=-90, looseness=0.75] (10.center) to (8.center);
		\draw [thick,in=-90, out=45, looseness=1.25] (12.center) to (13.center);
		\draw [thick,in=-90, out=135, looseness=1.25] (11.center) to (15.center);
         \draw [white, fill = white] (3.075, 3.575) circle (1.3ex); 
		\draw [thick,in=-90, out=60, looseness=0.75] (9.center) to (14.center);
		\draw [thick,in=180, out=90] (14.center) to (17.center);
		\draw [thick,in=90, out=0] (17.center) to (13.center);
		\draw [thick,in=105, out=0] (7.center) to (12.center);
		\draw [thick,in=105, out=0] (6.center) to (9.center);
		\draw [thick,bend right=15] (4.center) to (19.center);
		\draw [thick,in=-120, out=45] (19.center) to (12.center);
        \draw [white, fill = white] (2.75,2.75) circle (1.3ex);
		\draw [thick, in=135, out=-90] (9.center) to (19.center);
		\draw [thick,in=60, out=-45] (19.center) to (20.center);
		\draw [thick,bend left=15] (20.center) to (5.center);
        \draw [white, fill = white] (3.475, 2.975) circle (1.3ex);
        \draw [white, fill = white] (3.575, 2.675) circle (1.3ex);
        \draw [thick,bend left=270, looseness=1.75] (7.center) to (5.center);
        \draw [fill = black] (3.center) circle  (0.8ex);
        \draw [fill = black] (4.center) circle  (0.8ex);
        \draw [fill = black] (5.center) circle  (0.8ex);
        \draw [fill = black] (9.center) circle  (0.8ex);
        \draw [fill = black] (8.center) circle  (0.8ex);
        \draw [fill = black] (11.center) circle  (0.8ex);
        \draw [fill = black] (12.center) circle  (0.8ex);
        \draw [fill = black] (21.center) circle (0.8ex);
         \node[above] at (2 ,5) {\scriptsize $h < x \parallel x\hspace{1mm} < \hspace{1mm}  y$};
        \end{tikzpicture}
\end{align*}
\hfill $\blacksquare$
\end{example}

We now compare this new map $\Delta$ with the map $\delta^<$ introduced in Proposition~\ref{prop:delta} (which, let us recall, depends on the choice of an even symplectic expansion $\theta$).

\begin{lemma} \label{lem:Delta}
Let $n>0$.
The homomorphism $\Delta$  is well-defined and $\Sp(H)$-equivariant.
Furthermore, 
it fits into the following commutative diagram:
$$
\xymatrix{
 \calB^<_n\big(P^{(2)}\big) 
 \ar[rd]_-\Delta \ar[r]^-E &
 \Tors_2\big(\calA^{<,c}_{2n+1}(H)\big) \ar[d]^-{\delta^<} \\
&  \Tors_2\big( \calA^{<,c}_{2n+2}(H) \otimes \Q/\Z\big)
}
$$
\end{lemma}

\begin{proof}
Let $D\in  \calB^<_n\big(P^{(2)}\big)$ 
be a  Jacobi diagram, and recall that $E(D)$ is given 
by the sum~\eqref{eq:E_def} indexed by subsets $S$
of the set of leaves of $D$.
To compute
$$
\delta^< (E(D)) = 
\delta^< \Big( \sum_S \Big(\prod_{j \in S}   \fr(c_j) \Big) 
\cdot \widetilde{D}_+^S \Big),
$$
we need to determine $\delta^<(\widetilde{D}_+^S)$ 
for every $S$.
Since $\widetilde{D}_+^S$ is connected  with degree $>2$,  
formula \eqref{eq:connected} gives
$$
\delta^<(\widetilde{D}_+^S)  
= \sum_i\big(\widetilde{D}_+^S\big)_i^{II}(c_i',c_i'')
$$
where the sum ranges over external vertices $i$ of $\widetilde{D}_+^S$
and we use the notations of \S \ref{subsec:delta^<}:
in particular, $c_i\in H$ is the color of $i$ and we have set 
$\vartheta(c_i)= c_i' \otimes c_i'' \in H^{\otimes 2} \otimes \Q/\Z$
in Sweedler's notation.
Actually, each such vertex $i$ can be of two types:
\begin{itemize}
\item \emph{If $i$ is the root of $\widetilde{D}_+^S$},
then  we deduce from (\ref{eq:vartheta})
--- which, in Sweedler's notation, writes
$\vartheta(c_i)= c_i' \otimes c_i'' = \frac{1}{2} c_i \otimes c_i 
+{[\dot c_i, \ddot c_i]}$ --- 
that
$$
\big(\widetilde{D}_+^S\big)_i^{II}(c_i',c_i'')= \frac{1}{2}\widetilde{D}^S_{++}.
$$
(Indeed, using the symmetry of $\widetilde{D}$ and the AS relation twice, we get
\begin{align*}
\quad \big(\widetilde{D}_+^S\big)_i^{II}(\dot c_i,\ddot c_i)
- \big(\widetilde{D}_+^S\big)_i^{II}(\ddot c_i,\dot c_i) &=\!\!\!\!
\begin{tikzpicture}[baseline=-0.5em, x=1.85em, y=1.85em,thick]
\draw (0.5,-0.2) to[out=0, in=210, out looseness=2](0.366,1/2);
\fill [white] (0,0.3) circle [radius=0.15];
\draw (-0.5,-0.2) to[out=180, in=-30, out looseness=2](-0.366,1/2);
\draw (0.5,-0.2) -- (-0.5,-0.2);
\draw[dotted] ({-0.5+sqrt(3)/2},1/2) -- ({-0.5+sqrt(3)/2*1.5},1.5/2);
\draw[dotted] ({0.5-sqrt(3)/2},1/2) -- ({0.5-sqrt(3)/2*1.5},1.5/2);
\draw (-0.4,-0.2)--(-0.4,-0.8);
\draw (0.4,-0.2)--(0.4,-0.8);
\draw [fill = black] (-0.4,-0.2) circle (.3ex);
\draw [fill = black] (0.4,-0.2) circle (.3ex);
\node at (0.05,-1.1) {\scriptsize $\cdots<\dot c_i <\ddot c_i<\cdots$};
\end{tikzpicture}
\!\!\!\! -\!\!\!\!
\begin{tikzpicture}[baseline=-0.5em, x=1.85em, y=1.85em,thick]
\draw (0.5,-0.2) to[out=0, in=210, out looseness=2](0.366,1/2);
\fill [white] (0,0.3) circle [radius=0.15];
\draw (-0.5,-0.2) to[out=180, in=-30, out looseness=2](-0.366,1/2);
\draw (0.5,-0.2) -- (-0.5,-0.2);
\draw[dotted] ({-0.5+sqrt(3)/2},1/2) -- ({-0.5+sqrt(3)/2*1.5},1.5/2);
\draw[dotted] ({0.5-sqrt(3)/2},1/2) -- ({0.5-sqrt(3)/2*1.5},1.5/2);
\draw (-0.4,-0.2)--(-0.4,-0.8);
\draw (0.4,-0.2)--(0.4,-0.8);
\draw [fill = black] (-0.4,-0.2) circle (.3ex);
\draw [fill = black] (0.4,-0.2) circle (.3ex);
\node at (0.05,-1.1) {\scriptsize$\cdots<\ddot c_i<\dot c_i<\cdots$};
\end{tikzpicture}\\
&= \!\!\!\!\begin{tikzpicture}[baseline=-0.5em, x=1.85em, y=1.85em,thick]
\draw (0.5,-0.2) to[out=0, in=210, out looseness=2](0.366,1/2);
\fill [white] (0,0.3) circle [radius=0.15];
\draw (-0.5,-0.2) to[out=180, in=-30, out looseness=2](-0.366,1/2);
\draw (0.5,-0.2) -- (-0.5,-0.2);
\draw[dotted] ({-0.5+sqrt(3)/2},1/2) -- ({-0.5+sqrt(3)/2*1.5},1.5/2);
\draw[dotted] ({0.5-sqrt(3)/2},1/2) -- ({0.5-sqrt(3)/2*1.5},1.5/2);
\draw (-0.4,-0.2)--(-0.4,-0.8);
\draw (0.4,-0.2)--(0.4,-0.8);
\draw [fill = black] (-0.4,-0.2) circle (.3ex);
\draw [fill = black] (0.4,-0.2) circle (.3ex);
\node at (0.05,-1.1) {\scriptsize$\cdots<\dot c_i <\ddot c_i<\cdots$};
\end{tikzpicture}
\!\!\!\!- \!\!\!\!
\begin{tikzpicture}[baseline=-0.5em, x=1.85em, y=1.85em,thick]
\draw (-0.5,-0.2) to[out=180, in=210, out looseness=2](0.366,1/2);
\fill [white] (0,0.3) circle [radius=0.15];
\draw (0.5,-0.2) to[out=0, in=-30, out looseness=2](-0.366,1/2);
\draw (0.5,-0.2) -- (-0.5,-0.2);
\draw[dotted] ({-0.5+sqrt(3)/2},1/2) -- ({-0.5+sqrt(3)/2*1.5},1.5/2);
\draw[dotted] ({0.5-sqrt(3)/2},1/2) -- ({0.5-sqrt(3)/2*1.5},1.5/2);
\draw (-0.4,-0.2)--(-0.4,-0.8);
\draw (0.4,-0.2)--(0.4,-0.8);
\draw [fill = black] (-0.4,-0.2) circle (.3ex);
\draw [fill = black] (0.4,-0.2) circle (.3ex);
\node at (0.05,-1.1) {\scriptsize$\cdots<\ddot c_i<\dot c_i<\cdots$};
\end{tikzpicture}\\
&= \!\!\!\! \begin{tikzpicture}[baseline=-0.5em, x=1.85em, y=1.85em,thick]
\draw (0.5,-0.2) to[out=0, in=210, out looseness=2](0.366,1/2);
\fill [white] (0,0.3) circle [radius=0.15];
\draw (-0.5,-0.2) to[out=180, in=-30, out looseness=2](-0.366,1/2);
\draw (0.5,-0.2) -- (-0.5,-0.2);
\draw[dotted] ({-0.5+sqrt(3)/2},1/2) -- ({-0.5+sqrt(3)/2*1.5},1.5/2);
\draw[dotted] ({0.5-sqrt(3)/2},1/2) -- ({0.5-sqrt(3)/2*1.5},1.5/2);
\draw (-0.4,-0.2)--(-0.4,-0.8);
\draw (0.4,-0.2)--(0.4,-0.8);
\draw [fill = black] (-0.4,-0.2) circle (.3ex);
\draw [fill = black] (0.4,-0.2) circle (.3ex);
\node at (0.05,-1.1) {\scriptsize$\cdots<\dot c_i <\ddot c_i<\cdots$};
\end{tikzpicture}
\!\!\!\!-\!\!\!\!
\begin{tikzpicture}[baseline=-0.5em, x=1.85em, y=1.85em,thick]
\draw (0.5,-0.2) to[out=0, in=210, out looseness=2](0.366,1/2);
\fill [white] (0,0.3) circle [radius=0.15];
\draw (-0.5,-0.2) to[out=180, in=-30, out looseness=2](-0.366,1/2);
\draw (0.5,-0.2) -- (-0.5,-0.2);
\draw[dotted] ({-0.5+sqrt(3)/2},1/2) -- ({-0.5+sqrt(3)/2*1.5},1.5/2);
\draw[dotted] ({0.5-sqrt(3)/2},1/2) -- ({0.5-sqrt(3)/2*1.5},1.5/2);
\draw (0.4,-0.2)--(-0.4,-0.8);
\fill [white] (0,-0.5) circle [radius=0.15];
\draw (-0.4,-0.2)--(0.4,-0.8);
\draw [fill = black] (-0.4,-0.2) circle (.3ex);
\draw [fill = black] (0.4,-0.2) circle (.3ex);
\node at (0.05,-1.1) {\scriptsize$\cdots<\ddot c_i< \dot c_i<\cdots$};
\end{tikzpicture}, 
\end{align*} 
which is zero by ``STU-like'' and \eqref{eq:bubble}.)
\item \emph{If $i$ is not the root of $\widetilde{D}_+^S$},
then $i$ has a twin $j$ close to it, with the same color $c:=c_i=c_j$:
$$
\begin{tikzpicture}[thick, baseline=-1.6em, x=1.5em, y=1.5em]
\footnotesize
\draw [dotted] (-1.8,0.5) -- (-1.8,0);
\draw [dotted] (-1.3,0.5) -- (-1.3,0);
\draw [dotted] (0.8,0.5) -- (0.8,0);
\draw [dotted] (0.3,0.5) -- (0.3,0);
\draw (-1.8,0) to [out=-90,in=-90] (0.3,0);
\fill [white] (-0.6,-0.55) circle [radius=0.3];
\draw (-1.3,0) to [out=-90,in=-90] (0.8,0);
\draw (-1.15,-0.6) -- (-1.15,-1.5);
\draw (0.15,-0.6) -- (0.15,-1.5);
\draw [fill = black] (-1.15,-0.6) circle (.3ex);
\draw [fill = black] (0.15,-0.6) circle (.3ex);
\node[below=1pt] at (-0.4,-1.5){$\cdots <c_i<c_j<\cdots$};
\end{tikzpicture};
$$
thus, we obtain
\begin{eqnarray*}
&  &
\big(\widetilde{D}_+^S\big)_i^{II}(c',c'') + \big(\widetilde{D}_+^S\big)_j^{II}(c',c'') \\
&=& \begin{tikzpicture}[thick, baseline=-1.6em, x=1.5em, y=1.5em]
\footnotesize
\draw [dotted] (-1.8,0.5) -- (-1.8,0);
\draw [dotted] (-1.1,0.5) -- (-1.1,0);
\draw [dotted] (0.8,0.5) -- (0.8,0);
\draw [dotted] (1.4,0.5) -- (1.4,0);
\draw (-1.8,0) to [out=-90,in=-90] (0.8,0);
\fill [white] (-0.1,-0.55) circle [radius=0.3];
\draw (-1.1,0) to [out=-90,in=-90] (1.4,0);
\draw (-1.7,-0.35) -- (-1.7,-1.5);
\draw (-0.6,-0.75) -- (-0.6,-1.5);
\draw (0.7,-0.65) -- (0.7,-1.5);
\draw [fill = black] (-1.7,-0.35) circle (.3ex);
\draw [fill = black] (-0.6,-0.75) circle (.3ex);
\draw [fill = black] (0.7,-0.65) circle (.3ex);
\node[below=1pt] at (-0.55,-1.5){$\cdots <c'<c''<c< \cdots$};
\end{tikzpicture} + \begin{tikzpicture}[thick, baseline=-1.6em, x=1.5em, y=1.5em]
\footnotesize
\begin{scope}[xscale = -1]
\draw [dotted] (-1.8,0.5) -- (-1.8,0);
\draw [dotted] (-1.1,0.5) -- (-1.1,0);
\draw [dotted] (0.8,0.5) -- (0.8,0);
\draw [dotted] (1.4,0.5) -- (1.4,0);
\draw (-1.1,0) to [out=-90,in=-90] (1.4,0);
\fill [white] (-0.1,-0.55) circle [radius=0.3];
\draw (-1.8,0) to [out=-90,in=-90] (0.8,0);
\draw (-1.7,-0.35) -- (-1.7,-1.5);
\draw (-0.6,-0.75) -- (-0.6,-1.5);
\draw (0.7,-0.65) -- (0.7,-1.5);
\draw [fill = black] (-1.7,-0.35) circle (.3ex);
\draw [fill = black] (-0.6,-0.75) circle (.3ex);
\draw [fill = black] (0.7,-0.65) circle (.3ex);
\end{scope}
\node[below=1pt] at (0.7,-1.5){$\cdots  <c<c'<c''< \cdots$};
\end{tikzpicture}.
\end{eqnarray*}
But, by the STU-like relation, 
the first term of the above sum is equal to
\[ \qquad
\omega(c',c)\!\!\begin{tikzpicture}[thick, baseline=-1.6em, x=1.5em, y=1.5em]
\footnotesize
\draw [dotted] (-1.8,0.5) -- (-1.8,0);
\draw [dotted] (-1.1,0.5) -- (-1.1,0);
\draw [dotted] (0.8,0.5) -- (0.8,0);
\draw [dotted] (1.4,0.5) -- (1.4,0);
\draw (-1.8,0) to [out=-90,in=-90] (0.8,0);
\fill [white] (-0.1,-0.55) circle [radius=0.3];
\draw (-1.1,0) to [out=-90,in=-90] (1.4,0);
\draw (-1.7,-0.35) -- (-1.7,-0.65);
\draw (-0.6,-0.75) -- (-0.6,-1.8);
\draw (0.7,-0.6) -- (0.7,-0.65);
\draw [fill = black] (-1.7,-0.35) circle (.3ex);
\draw [fill = black] (-0.6,-0.75) circle (.3ex);
\draw [fill = black] (0.7,-0.65) circle (.3ex);
\fill [white] (-0.6,-1.35) circle [radius=0.15];
\draw (-1.7,-0.65) to [out=-90,in=-90] (0.7,-0.65);
\node[below=1pt] at (-0.4,-1.8){$\cdots <c'' < \cdots$};
\end{tikzpicture} + \omega(c'',c)\!\! \begin{tikzpicture}[thick, baseline=-1.6em, x=1.5em, y=1.5em]
\footnotesize
\draw [dotted] (-1.8,0.5) -- (-1.8,0);
\draw [dotted] (-1.1,0.5) -- (-1.1,0);
\draw [dotted] (0.8,0.5) -- (0.8,0);
\draw [dotted] (1.4,0.5) -- (1.4,0);
\draw (-1.8,0) to [out=-90,in=-90] (0.8,0);
\fill [white] (-0.1,-0.55) circle [radius=0.3];
\draw (-1.1,0) to [out=-90,in=-90] (1.4,0);
\draw (-1.7,-0.35) -- (-1.7,-1.5);
\draw (-0.6,-0.75) -- (-0.6,-1.05);
\draw (0.7,-0.6) -- (0.7,-1.05);
\draw [fill = black] (-1.7,-0.35) circle (.3ex);
\draw [fill = black] (-0.6,-0.75) circle (.3ex);
\draw [fill = black] (0.7,-0.65) circle (.3ex);
\draw (-0.6,-1.05) to [out=-90,in=-90] (0.7,-1.05);
\node[below=1pt] at (-1.7,-1.5){$\cdots< c'< \cdots$};
\end{tikzpicture} +\!\!\! \begin{tikzpicture}[thick, baseline=-1.6em, x=1.5em, y=1.5em]
\footnotesize
\draw [dotted] (-1.8,0.5) -- (-1.8,0);
\draw [dotted] (-1.1,0.5) -- (-1.1,0);
\draw [dotted] (0.8,0.5) -- (0.8,0);
\draw [dotted] (1.4,0.5) -- (1.4,0);
\draw (-1.8,0) to [out=-90,in=-90] (0.8,0);
\fill [white] (-0.1,-0.55) circle [radius=0.3];
\draw (-1.1,0) to [out=-90,in=-90] (1.4,0);
\draw (-1.7,-0.35) -- (-1.7,-0.65);
\draw (-0.6,-0.75) -- (-0.6,-0.8);
\draw [fill = black] (-1.7,-0.35) circle (.3ex);
\draw [fill = black] (-0.6,-0.75) circle (.3ex);
\draw [fill = black] (0.7,-0.65) circle (.3ex);
\draw (-1.7,-0.65) to [out=-90,in=90] (-0.4,-1.7);
\draw (-0.6,-0.8) to [out=-90,in=90] (0.9,-1.7);
\fill [white] (-0.2,-1.1) circle [radius=0.23];
\fill [white] (-0.95,-1.2) circle [radius=0.2];
\draw (0.7,-0.65) to [out=-90,in=90] (-1.6,-1.7);
\node[below=1pt] at (-0.3,-1.7){$\cdots <c<c'<c''< \cdots$};
\end{tikzpicture},
\] 
and the symmetry of $\widetilde{D}$ implies 
that the very  last term  here is 
 \[\begin{tikzpicture}[thick, baseline=-1.6em, x=1.5em, y=1.5em]
\footnotesize
\draw [dotted] (-1.8,0.5) -- (-1.8,0);
\draw [dotted] (-1.1,0.5) -- (-1.1,0);
\draw [dotted] (0.8,0.5) -- (0.8,0);
\draw [dotted] (1.4,0.5) -- (1.4,0);
\draw (-1.8,0) to [out=-90,in=-90] (0.8,0);
\fill [white] (-0.1,-0.55) circle [radius=0.3];
\draw (-1.1,0) to [out=-90,in=-90] (1.4,0);
\draw (-1.7,-0.35) -- (-1.7,-0.65);
\draw (-0.6,-0.75) -- (-0.6,-0.8);
\draw [fill = black] (-1.7,-0.35) circle (.3ex);
\draw [fill = black] (-0.6,-0.75) circle (.3ex);
\draw [fill = black] (0.7,-0.65) circle (.3ex);
\draw (-1.7,-0.65) to [out=-90,in=90] (-0.4,-1.7);
\draw (-0.6,-0.8) to [out=-90,in=90] (0.9,-1.7);
\fill [white] (-0.2,-1.1) circle [radius=0.23];
\fill [white] (-0.95,-1.2) circle [radius=0.2];
\draw (0.7,-0.65) to [out=-90,in=90] (-1.6,-1.7);
\node[below=1pt] at (-0.3,-1.7){$\cdots <c<c'<c''< \cdots$};
\end{tikzpicture} = -\begin{tikzpicture}[thick, baseline=-1.6em, x=1.5em, y=1.5em]
\footnotesize
\begin{scope}[xscale = -1]
\draw [dotted] (-1.8,0.5) -- (-1.8,0);
\draw [dotted] (-1.1,0.5) -- (-1.1,0);
\draw [dotted] (0.8,0.5) -- (0.8,0);
\draw [dotted] (1.4,0.5) -- (1.4,0);
\draw (-1.1,0) to [out=-90,in=-90] (1.4,0);
\fill [white] (-0.1,-0.55) circle [radius=0.3];
\draw (-1.8,0) to [out=-90,in=-90] (0.8,0);
\draw (-1.7,-0.35) -- (-1.7,-1.5);
\draw (-0.6,-0.75) -- (-0.6,-1.5);
\draw (0.7,-0.65) -- (0.7,-1.5);
\draw [fill = black] (-1.7,-0.35) circle (.3ex);
\draw [fill = black] (-0.6,-0.75) circle (.3ex);
\draw [fill = black] (0.7,-0.65) circle (.3ex);
\end{scope}
\node[below=1pt] at (0.7,-1.5){$\cdots <c< c'<c''< \cdots$};
\end{tikzpicture}.\]
Hence, using Lemma \ref{lem:expansion}.(c.iii), we obtain
\begin{eqnarray*}
&& \big(\widetilde{D}_+^S\big)_i^{II}(c',c'') + \big(\widetilde{D}_+^S\big)_j^{II}(c',c'')  
\\ & = &   
\frac{1}{2} \begin{tikzpicture}[thick, baseline=-1.6em, x=1.5em, y=1.5em]
\footnotesize
\draw [dotted] (-1.8,0.5) -- (-1.8,0);
\draw [dotted] (-1.1,0.5) -- (-1.1,0);
\draw [dotted] (0.8,0.5) -- (0.8,0);
\draw [dotted] (1.4,0.5) -- (1.4,0);
\draw (-1.8,0) to [out=-90,in=-90] (0.8,0);
\fill [white] (-0.1,-0.55) circle [radius=0.3];
\draw (-1.1,0) to [out=-90,in=-90] (1.4,0);
\draw (-1.7,-0.35) -- (-1.7,-0.65);
\draw (-0.6,-0.75) -- (-0.6,-1.8);
\draw (0.7,-0.6) -- (0.7,-0.65);
\draw [fill = black] (-1.7,-0.35) circle (.3ex);
\draw [fill = black] (-0.6,-0.75) circle (.3ex);
\draw [fill = black] (0.7,-0.65) circle (.3ex);
\fill [white] (-0.6,-1.35) circle [radius=0.15];
\draw (-1.7,-0.65) to [out=-90,in=-90] (0.7,-0.65);
\node[below=1pt] at (-0.4,-1.8){$\cdots <c < \cdots$};
\end{tikzpicture} +  \frac{1}{2} \ \begin{tikzpicture}[thick, baseline=-1.6em, x=1.5em, y=1.5em]
\footnotesize
\draw [dotted] (-1.8,0.5) -- (-1.8,0);
\draw [dotted] (-1.1,0.5) -- (-1.1,0);
\draw [dotted] (0.8,0.5) -- (0.8,0);
\draw [dotted] (1.4,0.5) -- (1.4,0);
\draw (-1.8,0) to [out=-90,in=-90] (0.8,0);
\fill [white] (-0.1,-0.55) circle [radius=0.3];
\draw (-1.1,0) to [out=-90,in=-90] (1.4,0);
\draw (-1.7,-0.35) -- (-1.7,-1.5);
\draw (-0.6,-0.75) -- (-0.6,-1.05);
\draw (0.7,-0.6) -- (0.7,-1.05);
\draw [fill = black] (-1.7,-0.35) circle (.3ex);
\draw [fill = black] (-0.6,-0.75) circle (.3ex);
\draw [fill = black] (0.7,-0.65) circle (.3ex);
\draw (-0.6,-1.05) to [out=-90,in=-90] (0.7,-1.05);
\node[below=1pt] at (-1.7,-1.5){$\cdots< c< \cdots$};\end{tikzpicture} \\
&=&  \frac{1}{2} \ \begin{tikzpicture}[thick, baseline=-1.6em, x=1.5em, y=1.5em]
\footnotesize
\draw [dotted] (-1.8,0.5) -- (-1.8,0);
\draw [dotted] (-1.3,0.5) -- (-1.3,0);
\draw [dotted] (0.8,0.5) -- (0.8,0);
\draw [dotted] (0.3,0.5) -- (0.3,0);
\draw (-1.8,0) to [out=-90,in=-90] (0.3,0);
\fill [white] (-0.6,-0.55) circle [radius=0.3];
\draw (-1.3,0) to [out=-90,in=-90] (0.8,0);
\draw (-1.3,-0.55) -- (-1.3,-0.8);
\draw (0.1,-0.6) -- (0.1,-0.8);
\draw (-0.6,-1.2) -- (-0.6,-1.6);
\draw [fill = black] (0.1,-0.6) circle (.3ex);
\draw [fill = black] (-1.3,-0.55) circle (.3ex);
\draw [fill = black] (-0.6,-1.2) circle (.3ex);
\draw (-1.3,-0.8) to [out=-90,in=-90] (0.1,-0.8);
\node[below=1pt] at (-0.6,-1.6){$\cdots< c <\cdots$};
\end{tikzpicture} \ = \ 
 \frac{1}{2} \widetilde{D}_{+}^{S,i}.
\end{eqnarray*}
\end{itemize}
We conclude that $\delta^<( E(D))$ is equal to the right-hand side 
of \eqref{eq:Delta_def}, which shows that the homomorphism $\Delta$ is well-defined
and coincides with $\delta^< \circ E$.
\end{proof}

The principal result of this section is presented below:

\begin{theorem} \label{th:Delta_zeta}
For $n>0$, we have the commutative triangle
\begin{equation} \label{eq:Delta_zeta}
\xymatrix{
 \calB^<_n\big(P^{(2)}\big) 
 \ar[rrd]^-\Delta \ar[d]_-\Psi 
 &&  \\
\Tors_2\Big({\displaystyle\frac{Y_{2n+1} \calIC}{Y_{2n+2}}}\Big) \ar[rr]^-{\zeta^<_{2n+2}} 
&&  \Tors_2\big( \calA^{<,c}_{2n+2}(H) \otimes \Q/\Z\big).
}    
\end{equation}
\end{theorem}

\begin{proof}
By  Proposition \ref{prop:Psi_psi} 
and Theorem \ref{th:psi_delta_Z}, we have 
the commutative diagram
$$
\xymatrix{
 \calB^<_n\big(P^{(2)}\big) 
 \ar[rr]^-E \ar[d]_-\Psi 
 &&
 \Tors_2\big(\calA^{<,c}_{2n+1}(H)\big) 
 \ar[d]^-{\delta^<} \ar[lld]^-{\psi} \\
\Tors_2\Big({\displaystyle\frac{Y_{2n+1} \calIC}{Y_{2n+2}}}\Big) \ar[rr]^-{\zeta^<_{2n+2}} 
&&  \Tors_2\big( \calA^{<,c}_{2n+2}(H) \otimes \Q/\Z\big).
}
$$
Thus, we conclude with Lemma \ref{lem:Delta}.
\end{proof}

It follows from Theorem \ref{th:Delta_zeta}
that the restriction of the invariant $\zeta^<_{2n+2}$ 
to the image of $\Psi$ \emph{does not} depend on the choices 
inherent to the construction of the LMO functor.
Here is another direct consequence of Theorem \ref{th:Delta_zeta}:

\begin{corollary} \label{cor:lower_bound}
For $n>0$, the $\Sp(H)$-module 
$\Tors_2\big({Y_{2n+1}\calIC}/{Y_{2n+2}}\big)$
surjects onto the $\Sp(H)$-module  
$\Delta\big(\calB^<_n\big(P^{(2)}\big)\big)$.
\end{corollary}

\noindent
Due to the relative complexity of $\mathcal{B}^<_n\big(P^{(2)}\big)$,
it may be challenging 
to determine the exact nature 
of the $\Sp(H)$-module
$\Delta\big(\mathcal{B}^<_n\big(P^{(2)}\big)\big)$.
However, in Section~\ref{subsec:gr_lower_bound}, by foregoing the $\Sp(H)$-action, we will obtain a result that is more manageable than Corollary~\ref{cor:lower_bound}
to estimate the cardinality 
of $\Tors_2\big({Y_{2n+1}\calIC}/{Y_{2n+2}}\big)$.

\section{Loop filtrations and their associated graded modules}
\label{sec:loop}

In this section, we study loop filtrations on modules of Jacobi diagrams 
and we compute the map $\Delta$ of \S \ref{subsec:lower_bound} 
on the associated graded modules.

\subsection{The loop filtration on $\calA^{<}(H)$}
\label{subsec:loop}

The \emph{loop degree} of a Jacobi diagram $D$
is defined as $\ell(D):=-\chi(D)+1$:
for instance, if $D$ is connected,
$\ell(D)$ is the first Betti number of $D$.
The module $\calA^{<}(H)$ can be endowed
with the \emph{loop filtration}
$$
\calA^{<}(H) =  L_{-\infty} \calA^{<}(H) \supset
\cdots \supset L_{-1} \calA^{<}(H) \supset L_0 \calA^{<}(H) 
\supset L_1 \calA^{<}(H) \supset \cdots
$$
where, for every $k\geq 0$,
the submodule $L_k \calA^{<}(H)$
is generated by Jacobi diagrams of loop degree at least $k$.

To identify the associated graded $\Gr^L\!\calA^{<}(H)$
of this loop filtration, we consider the module 
$$
\qquad \calA(H) := \frac{\Z\left\{\parbox{2in}{\centering 
\text{allowable Jacobi diagrams   }
\text{with $H$-colored external vertices}}\right\}}{\text{AS, IHX, multilinearity, self-loop, slide}}\, ,
$$
where the relation ``slide'' 
asserts the nullity of any $Y$-diagram showing two vertices
with the same color in $H$.
This  module has the bigrading 
$\calA(H)=\calA_{*,\circ}(H)$
defined by the pair (internal degree, loop degree).

\begin{proposition} \label{prop:Gr^LA^<}
We have a canonical isomorphism
$\Gr^L_\circ\!\calA^{<}_* (H) \simeq \calA_{*,\circ }(H)$.
\end{proposition}

\begin{proof}
Let $r: \calA(H) \to \Gr^L\!\calA^{<} (H) $ 
be the homomorphism 
that maps every Jacobi diagram $D$ of loop degree $k$
to the  class of the same Jacobi diagram $D^<\in  \Gr^L_k \calA^{<}_* (H)$
with an arbitrary order on the external vertices.
We claim that $r$ is an isomorphism.

Indeed, 
let $\calA'(S)$ be the module defined as in \eqref{eq:A(S)}
but with the ``swap'' relation replaced by the ``slide'' relation
(which  asserts the nullity of any $Y$-diagram showing 
twice the same color  in $S$).
Similarly to \cite[Lemma~8.4]{CHM08} and Proposition \ref{prop:varphi},
we define an isomorphism
$$
\varphi: \calA'(S) \longrightarrow \calA^<(H), \
D \longmapsto (-1)^{\chi(D)} D^<,
$$
with inverse map given by 
$$
\varphi^{-1}(E):=(-1)^{\chi(E)} \cdot 
\left(\begin{array}{c}
\text{sum of all ways of connecting  } \\ 
\text{some $a_i$-colored vertices of $E$} \\
\text{ to some lower $b_i$-colored vertices}
\end{array}\right)_{
\!\!\!\hbox{\& forget the order}}
$$
for any Jacobi diagram $E$ whose vertices are exclusively
colored by $S \subset H$.
Clearly,  $\varphi$ and $\varphi^{-1}$
both preserve the loop filtrations,
so that the induced map 
$\Gr^L \varphi: \calA'(S) = \Gr^L\! \calA' (S) 
\to  \Gr^L\! \calA^{<}(H)$ is an isomorphism.
We clearly have $\Gr^L \varphi=r$
via the obvious isomorphism between  $\calA'(S)$ and 
$\calA(H)$ mapping any Jacobi diagram~$D$ to $(-1)^{\chi(D)} D$: 
hence $r$ is bijective.
\end{proof}

Besides, recall  the isomorphism
$X: \calA(H)\otimes \Q \to \calA^<(H)\otimes \Q$ 
denoted by $\chi$ in \cite{HM09} and defined by
\begin{equation} \label{eq:X}
X(D) := \frac{1}{e!} \left(\begin{array}{c}
\text{sum of all ways of ordering } \\ 
\text{the external vertices of $D$}
\end{array}\right)    
\end{equation}
for any Jacobi diagram $D$ with $e$ external vertices.
Thus, 
the filtered vector space $\calA^<(H)\otimes \Q$ identifies
to its associated graded $\Gr^L\!\calA^<(H)\otimes \Q$ through $X^{-1}$.

\subsection{Structure of $\calA_{*,0}^c(H)$ and $\calA_{*,1}^c(H)$}
\label{subsec:few_loops}

We review here some known results
on the algebraic structure of the tree part $\calA_{*,0}^c(H)$ 
and $1$-loop part $\calA_{*,1}^c(H)$ of the module  $\calA^c(H)=\calA_*^c(H)$.

In order to investigate  the module of tree Jacobi diagrams $\mathcal{A}_{*,0}^c(H)$, Levine introduces in \cite{Lev06} the free quasi-Lie algebra $\mathfrak{L}'(H)$ generated by $H$ (or, more generally, by an arbitrary free abelian group).
He proves that the canonical map $\frakL'_n(H)\to \frakL_n(H)$ 
is an isomorphism in any odd degree $n$, and satisfies
$$
\xymatrix @!0 @R=0.5cm  @C=2cm  {
0 \ar[r] & \frakL_k(H) \otimes \Z_2 \ar[r] &  \frakL'_{2k}(H) \ar[r] &
 \frakL_{2k}(H) \ar[r] & 0 \\
 & x \otimes 1 \ar@{|->}[r] & [x,x] &  &
}
$$
in any even degree $n=2k$.
The kernels $\mathsf{D}_n(H)$ and $\mathsf{D}'_n(H)$
of the  bracketing maps $H\otimes \frakL_{n+1}(H) \to \frakL_{n+2}(H)$
and $H\otimes \frakL'_{n+1}(H) \to \frakL'_{n+2}(H)$, respectively,
fit into the commutative triangle
\begin{equation} \label{eq:triangle_eta}
\xymatrix{
\calA^c_{n,0}(H) \ar[r]^-\eta \ar[rd]_-\eta & \mathsf{D}_n(H)\\
& \mathsf{D}'_n(H) \ar[u]
}
\end{equation}
for every $n>1$;
here, the vertical map is the canonical homomorphism, 
and $\eta$ assigns to any tree Jacobi diagram $T$ the sum
\begin{equation} \label{eq:eta_formula}
\eta(T) := \sum_v (\hbox{color of $v$})
\otimes (\hbox{Lie bracket defined by $T$ rooted at $v$})
\end{equation}
indexed by the external vertices $v$ of $T$. 
Then, Levine proves that the following two short sequences 
are exact for $k\geq 1$:
\begin{equation}\label{eq:Levine1}
\!\!\begin{tikzcd}[row sep=-0.4em]
0 \arrow[r] 
& H \otimes \frakL_{k+1}(H) \otimes \Z_2 \arrow[r] 
& \mathsf{D}'_{2k+1}(H) \arrow[r] 
& \mathsf{D}_{2k+1}(H) \arrow[r] 
& 0 \\
& h \otimes u \otimes 1 \arrow[r, mapsto] 
& \eta \Big(
\begin{tikzpicture}[thick,baseline=-0.5em,x=0.6em,y=0.6em]
\footnotesize
\node at (-1.6,1.4) {\scriptsize $u$};
\node at (1.6,1.4) {\scriptsize $u$};
\node[below=1pt] at (0,-1.4) {\scriptsize $h$};
\draw (-1,1) -- (0,0) -- (1,1);
\draw (0,0) -- (0,-1.5);
\fill (0,0) circle (.4ex);
\end{tikzpicture}
\Big)
\end{tikzcd}
\end{equation}

 \vspace{-0.3cm}
\begin{equation} \label{eq:Levine2}
\xymatrix  @!0 @R=2pc @C=6.6pc {
0 \ar[r] &  \mathsf{D}'_{2k}(H) \ar[r] & \mathsf{D}_{2k}(H) \ar[r] & \frakL_{k+1}(H)\otimes \Z_2 \ar[r] & 0.\\
&& \frac{1}{2} \eta(\begin{tikzpicture}[thick,color=black,
baseline=-0.25em, x=0.6em, y=0.6em]
 \draw[thick] (0,0)--(2,0);
 \node[left=0pt] at (0.25,0) {$u$};
\node[right=0pt] at (1.75,0) {$u$};
\end{tikzpicture}) \ar@{|->}[r] & u\otimes 1
}    
\end{equation}
Furthermore, Conant, Schneiderman and Teichner prove in \cite{CST12}
that the map $\eta: \calA^c_{n,0}(H) \to \mathsf{D}'_n(H)$ 
is an isomorphism for every $n>1$:
thus, the short exact sequences \eqref{eq:Levine1} and \eqref{eq:Levine2}
relate the module of tree Jacobi diagrams
to the free abelian group $\mathsf{D}(H)$.
In particular, it follows from \eqref{eq:Levine2} that we have an embedding
\begin{equation} \label{eq:Levine_embedding}
\operatorname{Lev}:\frakL_{k+1}(H)\otimes \Z_2 
\longrightarrow 
\frac{\frac{1}{2}\calA^c_{2k,0}(H)}{\calA^c_{2k,0}(H)}
\subset \calA^c_{2k,0}(H) \otimes \Q/\Z, \
u \otimes 1 \longmapsto 
\Big\{\frac{1}{2} u \textsf{-----}{} u\Big\}.
\end{equation}

\begin{remark}
In complete analogy with \eqref{eq:triangle_eta} for $n := 1$, we have the following commutative diagram:
$$
\Lambda^3 H \simeq \xymatrix{
\calA^c_{1,0}(H) \ar[r]^-\eta_-\simeq & \mathsf{D}_1(H)\\
\tilde\calA^c_{1,0}(H) \ar[r]_-\eta^-\simeq  \ar@{->>}[u]
& \mathsf{D}'_1(H); \ar[u]
}
$$
here, the module $\tilde\calA(H)$ is defined like $\calA(H)$ by omitting the slide relation, and the isomorphisms $\eta$ are given by the same expression as in \eqref{eq:triangle_eta}, 
which also writes as in \eqref{eq:expand}. 
So, the short exact sequence
\eqref{eq:Levine1}  holds in the case $k := 0$ too \cite{Lev06}.
\hfill $\blacksquare$
\end{remark}

We  now turn to the  submodule $\calA^c_{*,1}(H)$ of one-loop Jacobi diagrams.
Let $n\geq 2$ and let $\mathcal{D}_{2n}$ be the dihedral group 
(of order $2n$)
generated by $x,y$ with relations $x^n,y^2,(xy)^2$.
Let $\big(H^{\otimes n}\big)_{\mathcal{D}_{2n}}$ 
be the coinvariant quotient of $H^{\otimes n}$ under the action defined by
\begin{eqnarray*}
x\cdot (h_1 \otimes h_2 \otimes \cdots  \otimes h_{n})
&=& h_n \otimes h_1 \otimes \cdots \otimes  h_{n-1}, \\
y\cdot (h_1 \otimes h_2 \otimes \cdots  \otimes h_{n})
&=& (-1)^n\ h_n \otimes h_{n-1} \otimes \cdots \otimes h_1.
\end{eqnarray*}
According to \cite[Prop$.$ 5.1]{NSS22a},  we have an isomorphism:
\begin{equation} \label{eq:wheel}
\big(H^{\otimes n}\big)_{\mathcal{D}_{2n}} \longrightarrow \calA^c_{n,1}(H), \quad
h_1 \otimes h_2 \otimes \cdots  \otimes h_{n} 
\longmapsto \begin{tikzpicture}[scale = 1, baseline = 0cm, thick]
\draw (1,0) circle (0.25 cm);
\draw ({1+0.25*cos(30)},{0+0.25*sin(30)}) -- ({1+0.5*cos(30)},{0+0.5*sin(30)});
\draw ({1-0.25*cos(30)},{0+0.25*sin(30)}) -- ({1-0.5*cos(30)},{0+0.5*sin(30)});
\draw ({1-0.25*cos(60)},{0+0.25*sin(60)}) -- ({1-0.5*cos(60)},{0+0.5*sin(60)});
\node at ({1+0.25*cos(150)},{0+0.25*sin(150)}) {\scriptsize $\bullet$};
\node at ({1+0.25*cos(120)},{0+0.25*sin(120)}) {\scriptsize $\bullet$};
\node at ({1+0.25*cos(30)},{0+0.25*sin(30)}) {\scriptsize $\bullet$};
\node at ({1+0.7*cos(150)},{0+0.7*sin(150)}) {\scriptsize $h_1$};
\node at ({1+0.7*cos(120)},{0+0.7*sin(120)}) {\scriptsize $h_2$};
\node at ({1+0.7*cos(30)},{0+0.7*sin(30)}) {\scriptsize $h_n$};
\node at ({1+0.5*cos(90)},{0+0.5*sin(90)}) {$\cdot$};
\node at ({1+0.5*cos(75)},{0+0.5*sin(75)}) {$\cdot$};
\node at ({1+0.5*cos(60)},{0+0.5*sin(60)}) {$\cdot$};
\end{tikzpicture}.
\end{equation}

Observe that $\big(H^{\otimes n}\big)_{\mathcal{D}_{2n}}$ can be viewed as the module
of \emph{bracelets} (also called \emph{turnover necklaces})
with $n$ beads colored by $H$ (taking the sign $(-1)^n$ during a turnover): 
this allows for an explicit description of  $\calA^c_{n,1}(H)$.
Specifically, $\calA^c_{n,1}(H)$ is torsion-free for $n$ even, 
and there is an isomorphism
\begin{equation} \label{eq:torsion_wheel}
H^{\otimes k} \otimes \Z_2 \longrightarrow    
\operatorname{Tors}  \calA^c_{2k-1,1}(H), \quad
(h_1\otimes h_{2} \otimes \cdots \otimes h_{k}) \otimes 1 
\longmapsto \begin{tikzpicture}[scale = 1, baseline = 0cm, thick]
\draw (1,0) circle (0.25 cm);
\draw ({1+0.25*cos(200)},{0+0.25*sin(200)}) -- ({1+0.5*cos(200)},{0+0.5*sin(200)});
\node at ({1+0.7*cos(200)},{0+0.7*sin(200)}) {\scriptsize $h_1$};
\draw ({1+0.25*cos(170)},{0+0.25*sin(170)}) -- ({1+0.5*cos(170)},{0+0.5*sin(170)});
\node at ({1+0.7*cos(170)},{0+0.7*sin(170)}) 
{\scriptsize $h_2$};
\draw ({1+0.25*cos(90)},{0+0.25*sin(90)}) -- ({1+0.5*cos(90)},{0+0.5*sin(90)});
\node at ({1+0.7*cos(90)},{0+0.7*sin(90)}) {\scriptsize $h_k$};
\draw ({1+0.25*cos(-20)},{0+0.25*sin(-20)}) -- ({1+0.5*cos(-20)},{0+0.5*sin(-20)});
\node at ({1+0.7*cos(-20)},{0+0.7*sin(-20)}) 
{\scriptsize $h_1$};
\draw ({1+0.25*cos(10)},{0+0.25*sin(10)}) -- ({1+0.5*cos(10)},{0+0.5*sin(10)});
\node at ({1+0.7*cos(10)},{0+0.7*sin(10)}) 
{\scriptsize $h_2$};
\node at ({1+0.25*cos(10)},{0+0.25*sin(10)}) {\scriptsize $\bullet$};
\node at ({1+0.25*cos(200)},{0+0.25*sin(200)}) {\scriptsize $\bullet$};
\node at ({1+0.25*cos(170)},{0+0.25*sin(170)}) {\scriptsize $\bullet$};
\node at ({1+0.25*cos(-20)},{0+0.25*sin(-20)}) {\scriptsize $\bullet$};
\node at ({1+0.25*cos(90)},{0+0.25*sin(90)}) {\scriptsize $\bullet$};

\draw ({1+0.25*cos(170)},{0+0.25*sin(170)}) -- ({1+0.5*cos(170)},{0+0.5*sin(170)});
\node at ({1+0.45*cos(65)},{0+0.45*sin(65)}) {$\cdot$};
\node at ({1+0.45*cos(50)},{0+0.45*sin(50)}) {$\cdot$};
\node at ({1+0.45*cos(35)},{0+0.45*sin(35)}) {$\cdot$};
\node at ({1+0.45*cos(180-35)},{0+0.45*sin(180-35)}) {$\cdot$};
\node at ({1+0.45*cos(180-50)},{0+0.45*sin(180-50)}) {$\cdot$};
\node at ({1+0.45*cos(115)},{0+0.45*sin(115)}) {$\cdot$};
\end{tikzpicture}
\end{equation}
in any odd degree $n=2k-1$.
Furthermore, we have
$$
\operatorname{rk}\big(\calA^c_{n,1}(H)\big)=
\left\{ \begin{array}{ll}
B_n(2g)     & \hbox{if $n$ is even,} \\
B_n(2g)  - (2g)^k     & \hbox{if $n=2k-1$ is odd,} 
\end{array} \right.
$$
where, for every $n\geq 2$ and $c\geq 1$, 
the number of bracelets with $n$ beads on $c$ colors is denoted by $B_n(c)$,
and can be derived from P\'olya's enumeration formula.
(We refer to \cite[Prop$.$ 5.2]{NSS22a} for details.)

\subsection{The $s$-loop filtration}
\label{subsec:sloop}

In this subsection, we assume that $n>0$
and we study the behaviour of the map
$\Delta: \calB^<_n\big(P^{(2)}\big)
\to \calA^{<,c}_{2n+2}(H) \otimes \Q/\Z$
(defined in \S \ref{subsec:lower_bound})
with respect to the loop filtrations.

The \emph{$s$-loop degree} of a  beaded rooted Jacobi diagram
in $ \calB^<_n\big(P^{(2)}\big)$
is  twice its first Betti number 
plus the number of $s$-colored vertices. 
Then, for any $n\geq 1$, we obtain  the \emph{$s$-loop filtration} 
$$
\calB^<_n\big(P^{(2)}\big)
= L_0 \calB^<_n\big(P^{(2)}\big) 
\supset L_1 \calB^<_n\big(P^{(2)}\big) 
\supset L_2 \calB^<_n\big(P^{(2)}\big) \supset \cdots
$$
To describe the associated graded 
$\Gr^L\calB^{<}_n\big(P^{(2)}\big)=\Gr^L_\circ\calB^{<}_n\big(P^{(2)}\big)$
of the $s$-loop filtration,
we consider the following module:
$$
\qquad \calB_n\big(H^{(2)};s\big) := \frac{\Z\left\{\parbox{2.8in}{\centering 
\text{beaded rooted connected Jacobi diagrams}
\text{of internal degree $n$} 
\text{with external vertices
 colored by $H^{(2)}\cup \{s\}$}}\right\}}{\text{AS, IHX, multilinearity, self-loop, group, fully-special}}
$$ 
with  grading $\calB_n\big(H^{(2)};s\big)=\calB_{n,\circ}\big(H^{(2)};s\big)$
defined by the $s$-loop degree.

\begin{example} \label{ex:B_1}
The submodule $\calB_{1,i}\big(H^{(2)};s\big)$
is trivial in $s$-loop degree $i\geq 2$, since both diagrams\\[-0.5cm]
$$
 \begin{tikzpicture}[thick,color=black,
baseline=-0.5em, x=0.6em, y=0.6em]
\footnotesize
\node[left=1pt] at (-1,1) {\scriptsize $s$};
\node[right=1pt] at (1,1) {\scriptsize $s$};
\node[below=1pt] at (0,-1.5) {\scriptsize $h \star$};
\draw (-1,1) -- (0,0) -- (1,1);
\draw (0,0) -- (0,-1.5);
\draw [fill = black] (0,0) circle (.3ex);
\end{tikzpicture} \qquad \hbox{and} \qquad
\begin{tikzpicture}[scale = 0.4,baseline = 0.05cm]
		\node (0) at (1, 0) {};
		\node (1) at (1, -0.65) {};
		\node (2) at (1, 1) {};
        \draw [thick] (0.center) to (1.center);
		\draw [thick] [in=180, out=-180, looseness=1.75] (2.center) to (0.center);
		\draw [thick] [bend left=90, looseness=1.75] (2.center) to (0.center);
        \draw [fill = black] (0.center) circle  (0.8ex);
        \draw[black, fill=lightgray] (2.center) circle (1.0ex);
        \node[below=0.2pt] at (1.center) {\scriptsize $h \star$};
\end{tikzpicture}
$$
vanish by the fully-special relation.
So, the module $\calB_1\big(H^{(2)};s\big)$  decomposes as 
$$
\calB_1\big(H^{(2)};s\big) 
= \calB_{1,0} \big(H^{(2)};s\big) \oplus \calB_{1,1} \big(H^{(2)};s\big)
$$
with respect to the $s$-loop degree,
and we have canonical isomorphisms 
$$
H^{(2)} \otimes \frakL'_2\big(H^{(2)}\big) 
\stackrel{\simeq}{\longrightarrow} \calB_{1,0} \big(H^{(2)};s\big), \quad
h \otimes [x,y] \longmapsto
 \begin{tikzpicture}[thick,color=black,
baseline=-0.5em, x=0.6em, y=0.6em]
\footnotesize
\node[left=1pt] at (-1,1) {\scriptsize $x$};
\node[right=1pt] at (1,1) {\scriptsize $y$};
\node[below=1pt] at (0,-1.5) {\scriptsize $h \star$};
\draw (-1,1) -- (0,0) -- (1,1);
\draw (0,0) -- (0,-1.5);
\draw [fill = black] (0,0) circle (.3ex);
\end{tikzpicture},
$$
\vspace{-0.2cm}
$$
H^{(2)} \otimes H^{(2)} \stackrel{\simeq}{\longrightarrow}  \calB_{1,1} \big(H^{(2)};s\big)  ,\quad
h \otimes y \longmapsto
 \begin{tikzpicture}[thick,color=black,
baseline=-0.5em, x=0.6em, y=0.6em]
\footnotesize
\node[left=1pt] at (-1,1) {\scriptsize $s$};
\node[right=1pt] at (1,1) {\scriptsize $y$};
\node[below=1pt] at (0,-1.5) {\scriptsize $h \star$};
\draw (-1,1) -- (0,0) -- (1,1);
\draw (0,0) -- (0,-1.5);
\draw [fill = black] (0,0) circle (.3ex);
\end{tikzpicture}.
$$

\up \hfill $\blacksquare$
\end{example}

\begin{lemma}
We have a canonical isomorphism
$\Gr^L_\circ \calB^{<}_n \big(P^{(2)}\big) 
\simeq \calB_{n,\circ }\big(H^{(2)};s\big)$.
\end{lemma}

\begin{proof} 
Let $r: \calB_{n,\circ }\big(H^{(2)};s\big)
\to  \Gr^L_\circ \calB^{<}_n \big(P^{(2)}\big)$ 
be the homomorphism 
that maps every Jacobi diagram $D\in \calB_{n,k}\big(H^{(2)};s\big)$
to the  class of the same Jacobi diagram 
$D^<\in  \Gr^L_k \calB^{<}_n \big(P^{(2)}\big)$
with an arbitrary order on the leaves
and an arbitrary choice of lift to $P^{(2)}$
for every $H^{(2)}$-colored vertex.
It is easily checked that $r$ is well-defined,
and we claim it to be an isomorphism.

To prove this claim, we introduce the module
$$
\calB^{<} (S \cup\{s\} ) := 
\frac{\Z_2 \left\{\parbox{2.7in}{\centering 
\text{beaded rooted connected Jacobi diagrams}
\text{with external vertices colored by $S \cup\{s\}$}
\text{and with totally-ordered leaves}}\right\}}{
\text{AS, IHX, self-loop,  BSTU-like,  group, fully-special}}.
$$ 
The module $\calB(S \cup\{s\} )$ is defined in a similar  way, 
but  without requiring an order  on leaves 
(and with no  BSTU-like relation).
It follows from~\eqref{eq:primary} 
that the $\Z_2$-vector space  $P^{(2)}$ 
has the basis $\sigma(S)\cup\{s\}$.
Therefore, we can identify 
$\calB^{<}_n (S \cup\{s\} )$ and  $\calB_n(S \cup\{s\} )$ 
to $\calB^{<}_n \big(P^{(2)}\big)$ and $\calB_{n }\big(H^{(2)};s\big)$, 
respectively.

We now proceed in a way similar 
to \cite[Lemma~8.4]{CHM08} and Proposition~\ref{prop:varphi}.
Let $\varphi:\calB_n(S \cup\{s\} ) \to \calB^{<}_n (S \cup\{s\} )$
be the homomorphism assigning to any diagram~$D$
the same diagram $D^<$ where the leaves
are ordered so that  all the leaves colored 
by $\{a_1,\dots, a_g\}$
are lower than the leaves colored 
by $\{b_1,\dots, b_g\}$,
and all $s$-colored leaves are given an arbitrary position.
Furthermore, we define a homomorphism
$\rho: \calB^{<}_n (S \cup\{s\} ) \to  \calB_n(S \cup\{s\} )  $
by assigning to any diagram $E$ the linear combination
$$
\rho(E) := 
\left(\!\!\!\!\begin{array}{c}
\text{\small  sum of all ways of connecting  some  $a_i$-colored leaves  of $E$} \\
\text{\small  to some lower $b_i$-colored leaves, 
inserting $0$ or 1 bead} 
\end{array}\!\!\!\right)_{
\!\!\!\hbox{ \! \& \small forget the order}}.
$$
Clearly, $\rho \circ \varphi$ is 
the identity of $\calB_n(S \cup\{s\} )$, 
and $\varphi$ is surjective as a consequence of ``BSTU-like'':
hence $\varphi$ is an isomorphism with inverse $\rho$.
Since both of them preserve the $s$-loop filtrations,
the induced map
$$
\Gr^L \varphi: \calB_n(S \cup\{s\}) = \Gr^L  \calB_n(S \cup\{s\}) 
\longrightarrow  \Gr^L  \calB^{<}_n (S \cup\{s\} )
$$
is an isomorphism.
Via the above-mentioned identifications,
$\Gr^L \varphi$ corresponds to the map~$r$, which proves our claim.
\end{proof}

\subsection{The map $\Gr^L \!\Delta$}
\label{subsec:gr_lower_bound}

In this subsection, we consider the map $\Gr^L \!\Delta$
induced by the homomorphism $\Delta$ on the graded modules
associated to the loop filtrations. 

\begin{proposition} \label{prop:Gr_Delta}
Let~$n>0$.
The map $\Delta: \calB^<_n\big(P^{(2)}\big)
\to \calA^{<,c}_{2n+2}(H) \otimes \Q/\Z$
sends the $s$-loop filtration to the loop filtration,  
and the induced map at the graded level 
$$
\xymatrix @!0 @R=1cm @C=2cm  {
 \Gr_k^L  \calB^<_n\big(P^{(2)}\big) \ar@{=}[d]
\ar[rr]^-{\Gr^L \!\Delta}&&
\Gr_{k}^L \calA^{<,c}_{2n+2}(H) \otimes \Q/\Z  \ar@{=}[d] \\
 \calB_{n,k}\big(H^{(2)};s\big) &&  \calA^{c}_{2n+2,k}(H)  \otimes \Q/\Z
}
$$
maps any diagram 
$D\in \calB_{n,k}\big(H^{(2)};s\big)$ 
to half the  diagram 
${\widetilde{D}_{++}^s \in \calA^{c}_{2n+2,k}(H)  \otimes \Q/\Z}$.
Here, $\widetilde{D}_{++}^s$ is built from  the $2$-fold cover $\widetilde{D}$ of $D$
by, first, gluing the twin vertices arising from $s$-colored leaves
and, second,
by gluing its twin root vertices to the ``top'' external vertices of an $H$-diagram  
(assigning to the ``bottom'' external vertices of the latter the color of the root of $D$).
\end{proposition}

\begin{example} \label{ex:Gr_Delta}
Using Proposition \ref{prop:Gr_Delta} and Example \ref{ex:B_1},
we see that $\Gr^L \!\Delta$ is determined in degree $n=1$ 
by the following values:
$$
(\Gr^L \!\Delta)\Big(\begin{tikzpicture}[thick,color=black,
baseline=-0.5em, x=0.6em, y=0.6em]
\footnotesize
\node[left=1pt] at (-1,1) {\scriptsize $x$};
\node[right=1pt] at (1,1) {\scriptsize $y$};
\node[below=1pt] at (0,-1.5) {\scriptsize $h \star$};
\draw (-1,1) -- (0,0) -- (1,1);
\draw (0,0) -- (0,-1.5);
\draw [fill = black] (0,0) circle (.3ex);
\end{tikzpicture}\Big) =  \frac{1}{2}\begin{tikzpicture}[scale = 0.4, baseline = 1.2cm, thick]
		\node (0) at (0.5, 2) {};
		\node (1) at (0.5, 3) {};
		\node (2) at (3.5, 3) {};
		\node (3) at (3.5, 2) {};
		\node (4) at (3.5, 4) {};
		\node (5) at (0.5, 4) {};
		\node (6) at (-0.5, 5) {};
		\node (7) at (0.5, 4) {};
		\node (8) at (1.5, 5) {};
		\node (9) at (2.5, 5) {};
		\node (10) at (4.5, 5) {};
        \draw (0.center) to (7.center);
		\draw (1.center) to (2.center);
		\draw (4.center) to (3.center);
		\draw (6.center) to (7.center);
		\draw (7.center) to (8.center);
		\draw (9.center) to (4.center);
		\draw (4.center) to (10.center);
        \draw [black, fill = black] (1.center)  circle (0.8ex);
        \draw [black, fill = black] (2.center)  circle (0.8ex);
        \draw [black, fill = black] (4.center)  circle (0.8ex);
        \draw [black, fill = black] (5.center)  circle (0.8ex);
        \node[above = 1pt] at (6) {\scriptsize$x$};
        \node[above = 0pt] at (8) {\scriptsize$y$};
        \node[above = 1pt] at (9) {\scriptsize$x$};
        \node[above = 0pt] at (10) {\scriptsize$y$};
        \node[below] at (0) {\scriptsize$h$};
        \node[below] at (3) {\scriptsize$h$};
\end{tikzpicture},
\quad
(\Gr^L \!\Delta)\Big( \begin{tikzpicture}[thick,color=black,
baseline=-0.5em, x=0.6em, y=0.6em]
\footnotesize
\node[left=1pt] at (-1,1) {\scriptsize $s$};
\node[right=1pt] at (1,1) {\scriptsize $y$};
\node[below=1pt] at (0,-1.5) {\scriptsize $h \star$};
\draw (-1,1) -- (0,0) -- (1,1);
\draw (0,0) -- (0,-1.5);
\draw [fill = black] (0,0) circle (.3ex);
\end{tikzpicture}\Big)=  \frac{1}{2}
\begin{tikzpicture}[scale = 0.4, baseline = 1.2cm, thick]
		\node (0) at (0.5, 2) {};
		\node (1) at (0.5, 3) {};
		\node (2) at (3.5, 3) {};
		\node (3) at (3.5, 2) {};
		\node (4) at (3.5, 4) {};
		\node (5) at (0.5, 4) {};
		\node (6) at (-0.25, 4.75) {};
		\node (7) at (0.5, 4) {};
		\node (8) at (2.5, 6) {};
		\node (9) at (2.75, 4.75) {};
		\node (10) at (4.5, 6) {};
        \draw (0.center) to (7.center);
		\draw (1.center) to (2.center);
		\draw (4.center) to (3.center);
		\draw (6.center) to (7.center);
		\draw (7.center) to (8.center);
		\draw (9.center) to (4.center);
		\draw (4.center) to (10.center);
        \fill [white] (1.86,5.35) circle [radius=0.2];
		\draw [in=135, out=135, looseness=1.50] (6.center) to (9.center);
        \draw [black, fill = black] (1.center)  circle (0.8ex);
        \draw [black, fill = black] (2.center)  circle (0.8ex);
        \draw [black, fill = black] (4.center)  circle (0.8ex);
        \draw [black, fill = black] (5.center)  circle (0.8ex);
        \node[above = 0pt] at (8) {\scriptsize$y$};
        \node[above = 0pt] at (10) {\scriptsize$y$};
        \node[below] at (0) {\scriptsize$h$};
        \node[below] at (3) {\scriptsize$h$};
\end{tikzpicture}.
$$
(Recall that we are using the convention \eqref{eq:lifts},
so that $x,y,h\in H$ on the right-hand sides of the above identities denote
arbitrary lifts of  $x,y,h\in H^{(2)}$.)

Moreover, to further illustrate Proposition \ref{prop:Gr_Delta},
we present an example of the computation
of $\Gr^L \!\Delta$ in degree $n=3$:
$$(
\Gr^L \!\Delta) \Big(\begin{tikzpicture}[scale = 0.6,baseline = 0.1cm]
		\node (0) at (1, 0) {};
		\node (1) at (1, -0.75) {};
		\node (2) at (1, 1) {};
		\node (3) at (0.55, 0.775) {};
		\node (4) at (1.45, 0.775) {};
		\node (5) at (0, 1.25) {};
		\node (6) at (2, 1.25) {};
		\draw [thick] (0.center) to (1.center);
        \draw [thick] (0.center) to (1.center);
		\draw [thick] [in=180, out=-180, looseness=1.75] (2.center) to (0.center);
		\draw [thick] [bend left=90, looseness=1.75] (2.center) to (0.center);
		\draw [thick] (5.center) to (3.center);
		\draw [thick] (6.center) to (4.center);
        \draw [fill = black] (0.center) circle  (0.6ex);
        \draw [fill = black] (3.center) circle  (0.6ex);
        \draw [fill = black] (4.center) circle  (0.6ex);
        \node[below=0.5pt] at (1,-0.75) {\scriptsize $r \star$};
        \node[above] at (1,1.2) {\scriptsize $s$ \hspace{10mm} $x$};
        \draw[black, fill=lightgray] (1.435,0.225) circle (0.8ex);
\end{tikzpicture}\Big) =  \frac{1}{2} \begin{tikzpicture}[scale = 0.4, baseline = 1.2cm, thick]
        \node (0) at (0.5, 2) {};
		\node  (1) at (0.5, 3) {};
		\node  (2) at (3.5, 3) {};
		\node  (3) at (3.5, 2) {};
		\node  (4) at (3.5, 4) {};
		\node  (5) at (0.5, 4) {};
		\node  (7) at (0.5, 4) {};
		\node  (8) at (-0.25, 4.75) {};
		\node  (10) at (2.75, 4.75) {};
		\node  (11) at (4, 4.5) {};
		\node  (12) at (0, 5.5) {};
		\node  (13) at (1, 5.5) {};
		\node  (14) at (3, 5.5) {};
		\node  (15) at (4, 5.5) {};
		\node  (16) at (1.5, 6.5) {};
		\node  (17) at (4.5, 6.5) {};
		\node  (18) at (2.25, 4.25) {};
        \draw (0.center) to (7.center);
		\draw (1.center) to (2.center);
		\draw (4.center) to (3.center);
		\draw (7.center) to (8.center);
		\draw [bend left] (14.center) to (15.center);
		\draw (15.center) to (17.center);
		\draw [in=60, out=-30, looseness=0.75] (15.center) to (7.center);
		\draw [bend left=45] (8.center) to (12.center);
		\draw [bend left] (12.center) to (13.center);
		\draw [bend left] (11.center) to (4.center);
        \draw [in=75, out=-30, looseness=1.50] (18.center) to (11.center);
        \draw [in=90, out=120, looseness=0.75] (12.center) to (14.center);
        \fill [white] (1.75,4.70) circle [radius=0.2];
        \fill [white] (3,4.25) circle [radius=0.2];
        \fill [white] (2.75,4.85) circle [radius=0.2];
        \fill [white] (1.37,6.12) circle [radius=0.2];
		\draw [in=150, out=-45] (13.center) to (18.center);
		\draw [bend right=15] (13.center) to (16.center);
		\draw [bend left] (4.center) to (10.center);
        \draw [in=-150, out=105, looseness=0.75] (10.center) to (14.center);
        \draw [black, fill = black] (1.center)  circle (0.8ex);
        \draw [black, fill = black] (2.center)  circle (0.8ex);
        \draw [black, fill = black] (7.center)  circle (0.8ex);
        \draw [black, fill = black] (4.center)  circle (0.8ex);
        \draw [black, fill = black] (12.center)  circle (0.8ex);
        \draw [black, fill = black] (13.center)  circle (0.8ex);
        \draw [black, fill = black] (14.center)  circle (0.8ex);
        \draw [black, fill = black] (15.center)  circle (0.8ex);
        \node[above = 0pt] at (16) {\scriptsize$x$};
        \node[above = 0pt] at (17) {\scriptsize$x$};
        \node[below] at (0) {\scriptsize$r$};
        \node[below] at (3) {\scriptsize$r$};
\end{tikzpicture}
$$

\up\hfill $\blacksquare$
\end{example}

\begin{proof}[Proof of Proposition \ref{prop:Gr_Delta}]
Let $D$ be a generator of $\calB^<_n\big(P^{(2)}\big)$
with $s$-loop degree $k$.
Denote the first Betti number of $D$ by $b$,
and let $V_s$ be the set of external $s$-colored vertices of $D$.
Then, $V_s$ has cardinality $k-2b$ and we can assume 
that it does not contain the root (since, otherwise, $D$ would vanish 
in $\calB^<_n\big(P^{(2)}\big)$ by the ``fully-special'' relation).

The formula \eqref{eq:Delta_def_bis} for $\Delta(D)$  is given 
as a sum of two terms: say, $\Delta_1(D)$ and $\Delta_2(D)$.
The first sum $\Delta_1(D)$ is indexed by the subsets $S$ 
of the set of leaves of $D$.
If $S$ does not contain $V_s$, then its contribution to $\Delta_1(D)$
is trivial since 
$\widetilde{D}_{++}^S=0 \in \calA^{<,c}_{2n+2}(H) \otimes \Q/\Z$ 
by the fact that $p_*(s)=0$.
If $S$  strictly contains $V_s$, 
then we have
\begin{eqnarray*}
\ell(\widetilde{D}^S_{++}) \ = \
-\chi\big(\widetilde{D}_{++}^S\big)+1 &=&
-\big(2 \chi(D) -1 - \vert S \vert\big)+1\\
&=& -\big(2(1-b)-1-\vert S\vert  \big)+1
\ = \  2b +\vert S \vert > k.
\end{eqnarray*}
If we now have $S=V_s$, then the contribution of $S$ to $\Delta_1(D)$
is $\frac{1}{2} \widetilde{D}_{++}^S=\frac{1}{2} \widetilde{D}_{++}^s$.
Thus, to conclude  the proof, it remains
to justify that $\Delta_2(D)$ is in the loop-filtration module $L_{k+1}$
of $\calA^{<,c}_{2n+2}(H) \otimes \Q/\Z$.

Indeed, the second sum  $\Delta_2(D)$ is indexed by leaves $i$ of $D$,
and subsets $S$ of the set of leaves of $D$ not containing $i$.
If~$S$ does not contain $V_s$, then its contribution to $\Delta_2(D)$
is trivial since 
$\widetilde{D}_{+}^{S,i}=0 \in \calA^{<,c}_{2n+2}(H) \otimes \Q/\Z$.
If $S$ contains $V_s$, then a computation of Euler characteristics as above gives
$$
\ell(\widetilde{D}_{+}^{S,i}) = 2b+ \vert S \vert +1 > 2b+ \vert S \vert \geq k;
$$
so, the contribution of $S$ to $\Delta_2(D)$ in this case belongs 
to $L_{k+1 }\calA^{<,c}_{2n+2}(H) \otimes \Q/\Z$.
\end{proof}

\begin{remark} \label{rem:E_Levine}
By the same arguments as in the proof of Proposition \ref{prop:Gr_Delta},
we obtain that $E:  \calB^<_n\big(P^{(2)}\big)\to \calA^{<,c}_{2n+1}(H)$ is filtration-preserving.
So, for every $k\geq 0$, it induces a homomorphism
$\Gr^L\! E:   \calB_{n,k}\big(H^{(2)};s\big)
\to \calA^{c}_{2n+1,k}(H)$ with an explicit and simple formula.
For instance, 
the resulting homomorphism 
$$
\Gr^L\! E:  H \otimes \frakL'_{n+1}(H) \otimes \Z_2 \simeq \calB_{n,0}\big(H^{(2)};s\big)
\longrightarrow \calA^{c}_{2n+1,0}(H) \simeq \mathsf{D}'_{2n+1}(H)
$$
at $k:=0$  factorizes to $ H \otimes \frakL_{n+1}(H) \otimes \Z_2$,
and we recover the injection in Levine's short exact sequence \eqref{eq:Levine1}.
Thus, the map $E$ can be viewed as a far-reaching generalization
of Levine's construction. \hfill $\blacksquare$
\end{remark}

Forgetting the $\Sp(H)$-structure, 
we now obtain a lower bound on the size 
of the $2$-torsion subgroup of ${Y_{2n+1}\calIC}/{Y_{2n+2}}$.

\begin{proposition}     \label{prop:gr_lower_bound}
For $n>0$, we have
${\displaystyle \left\vert 
\Tors_2\!\Big(\frac{Y_{2n+1}\calIC}{Y_{2n+2}}\Big)\right\vert \geq 
\left\vert 
(\Gr^L \!\Delta) \big(\calB_n\big(H^{(2)};s\big)\big) \right\vert}$.
\end{proposition}

\begin{proof}
We endow $\Tors_2\big({Y_{2n+1}\calIC}/{Y_{2n+2}}\big)$
with the pull-back by $\zeta_{2n+2}^<$ of the loop filtration
on $\calA^{<,c}_{2n+2}(H) \otimes \Q/\Z$.
Since $\Delta$ is filtration-preserving (Proposition \ref{prop:Gr_Delta})
and satisfies $\Delta = \zeta^<_{2n+2} \circ \Psi$ (Theorem \ref{th:Delta_zeta}),
the map $\Psi$ is filtration-preserving too, and we get the identity
$$
\Gr^L \!\Delta = \big(\Gr^L\! \zeta^<_{2n+2}\big) \circ (\Gr \Psi)
$$
at the graded level. The claimed inequality follows from that.
\end{proof}

Next, the lower bound in Proposition \ref{prop:gr_lower_bound} 
can be expressed more explicitly as
$$
\left\vert 
\Tors_2\Big(\frac{Y_{2n+1}\calIC}{Y_{2n+2}}\Big)\right\vert \;\geq\; 
\sum_{k\geq 0} t_k(n),
\quad \text{where} \quad
t_k(n) := \big\vert 
(\Gr^L \!\Delta)\big(\calB_{n,k}\big(H^{(2)};s\big)\big) \big\vert.
$$
In the remainder of this subsection, 
we  determine the quantities $t_0(n)$ and $t_1(n)$.

\begin{proposition} \label{prop:t_0}
For $n>0$, we have ${\displaystyle t_0(n) =2^{\dim \frakL_{n+2}(H)}}$.
\end{proposition}

\noindent
Recall that the graded dimension of the free Lie algebra $\frakL(H)$
is known by a classical formula of Witt:
$$
\dim \frakL_k(H) = \frac{1}{k} \sum_{d \vert k} \mu(d) \, (2g)^{k/d},
\quad \hbox{for } k\geq 1,
$$
where $\mu:\mathbb{N}^* \to \{-1,0,+1\}$ is M\"obius' function.

\begin{proof}[Proof of Proposition \ref{prop:t_0}]
We identify $H \otimes \frakL'_{n+1}(H) \otimes \Z_2$ 
with $\calB_{n,0}\big(H^{(2)};s\big)$ by the map
$\big(h\otimes u \mapsto
{\displaystyle\mathop{\textsf{l}}_{h \star}^u}\,\big)$.
By Proposition \ref{prop:Gr_Delta},
$(\Gr^L \!\Delta): \calB_{n,0}\big(H^{(2)};s\big) \to
\calA^c_{2n+2,0}(H)\otimes \Q/\Z$ is given by the formula
$$
 (\Gr^L \!\Delta)(h \otimes u) 
 = \frac{1}{2}
\begin{tikzpicture}[thick,color=black,baseline=-0.2em, x=1.5em, y=1.5em]
\draw (-1/2,0) -- (1/2,0);
\draw (1/2,0)--({sqrt(3)/2},1);
\draw (-1/2,0)--({-sqrt(3)/2},1);
\draw (1/2,0)--({sqrt(3)/2},-1);
\draw (-1/2,0)--({-sqrt(3)/2},-1);
\node [right=1pt] at ({sqrt(3)/2},1) {$u$};
\node [right=1pt] at ({sqrt(3)/2},-1) {$h$};
\node [left=1pt] at (-{sqrt(3)/2},1) {$u$};
\node [left=1pt] at (-{sqrt(3)/2},-1) {$h$};
\draw [fill = black] (1/2,0) circle (.3ex);
\draw [fill = black] (-1/2,0) circle (.3ex);
\end{tikzpicture}.
$$
Hence, it factorizes by Levine's embedding 
\eqref{eq:Levine_embedding} of $\frakL_{n+2}(H)\otimes \Z_2$. 
Therefore, we obtain
$t_0(n) =
\big\vert \operatorname{Lev}(\frakL_{n+2}(H)\otimes \Z_2)\big\vert$. 
\end{proof}

\begin{proposition} \label{prop:t_1}
For $n>0$, we have
$${\displaystyle t_1(n)=\sqrt{2}^{(2g)^{n+1} + (2g)^{r+1}}},$$
where $r=r(n)\in \mathbb{N}$ is such that
$n+1=2^\ell (2r+1)$ for some  $\ell \in \mathbb{N}$.
\end{proposition}

Let $m\geq 1$.
To prove Proposition \ref{prop:t_1}, we shall consider 
the submodules $\calA^{c,s}_{2m,1}(H)'$ and  $\calA^{c,s}_{2m,1}(H)''$
of  $\calA^{c}_{2m,1}(H)$ that are
generated by the symmetric diagrams
$$
O(a_1,\dots, a_m,a_{m+1}, a_m, \dots, a_2) 
:= \begin{tikzpicture}[scale = 1, baseline = 0cm, thick]
\draw (1,0) circle (0.25 cm);
\draw ({1+0.25*cos(245)},{0+0.25*sin(245)}) -- ({1+0.5*cos(245)},{0+0.5*sin(245)});
\node at ({1+0.25*cos(245)},{0+0.25*sin(245)}) {\scriptsize $\bullet$};
\node at ({1+0.7*cos(245)},{0+0.7*sin(245)}) {\scriptsize $a_1$};
\draw ({1+0.25*cos(200)},{0+0.25*sin(200)}) -- ({1+0.5*cos(200)},{0+0.5*sin(200)});
\node at ({1+0.25*cos(200)},{0+0.25*sin(200)}) {\scriptsize $\bullet$};
\node at ({1+0.7*cos(200)},{0+0.7*sin(200)}) {\scriptsize $a_2$};
\draw ({1+0.25*cos(120)},{0+0.25*sin(120)}) -- ({1+0.5*cos(120)},{0+0.5*sin(120)});
\node at ({1+0.25*cos(120)},{0+0.25*sin(120)}) {\scriptsize $\bullet$};
\node at ({1+0.7*cos(120)},{0+0.7*sin(120)}) {\scriptsize $a_m$};
\draw ({1+0.25*cos(75)},{0+0.25*sin(75)}) -- ({1+0.5*cos(75)},{0+0.5*sin(75)});
\node at ({1+0.25*cos(75)},{0+0.25*sin(75)}) {\scriptsize $\bullet$};
\node at ({1+0.7*cos(75)},{0+0.7*sin(75)}) {\scriptsize $a_{m+1}$};
\draw ({1+0.25*cos(30)},{0+0.25*sin(30)}) -- ({1+0.5*cos(30)},{0+0.5*sin(30)});
\node at ({1+0.25*cos(30)},{0+0.25*sin(30)}) {\scriptsize $\bullet$};
\node at ({1+0.7*cos(30)},{0+0.7*sin(30)}) {\scriptsize $a_{m}$};
\draw ({1+0.25*cos(290)},{0+0.25*sin(290)}) -- ({1+0.5*cos(290)},{0+0.5*sin(290)});
\node at ({1+0.25*cos(290)},{0+0.25*sin(290)}) {\scriptsize $\bullet$};
\node at ({1+0.7*cos(290)},{0+0.7*sin(290)}) {\scriptsize $a_{2}$};
\node at ({1+0.5*cos(160)},{0+0.5*sin(160)}) {$\cdot$};
\node at ({1+0.5*cos(150)},{0+0.5*sin(150)}) {$\cdot$};
\node at ({1+0.5*cos(170)},{0+0.5*sin(170)}) {$\cdot$};
\node at ({1+0.5*cos(0)},{0+0.5*sin(-10)}) {$\cdot$};
\node at ({1+0.5*cos(-10)},{0+0.5*sin(-20)}) {$\cdot$};
\node at ({1+0.5*cos(-20)},{0+0.5*sin(-30)}) {$\cdot$};
\end{tikzpicture}
$$
and 
$$
O(a_1,a_2, \dots, a_m,a_{m}, \dots, a_2,a_1) 
:= \begin{tikzpicture}[scale = 1, baseline = 0cm, thick]
\draw (1,0) circle (0.25 cm);
\draw ({1+0.25*cos(292.5)},{0+0.25*sin(292.5)}) -- ({1+0.5*cos(292.5)},{0+0.5*sin(292.5)});
\node at ({1+0.25*cos(292.5)},{0+0.25*sin(292.5)}) {\scriptsize $\bullet$};
\node at ({1+0.7*cos(292.5)},{0+0.7*sin(292.5)}) {\scriptsize $a_1$};
\draw ({1+0.25*cos(247.5)},{0+0.25*sin(247.5)}) -- ({1+0.5*cos(247.5)},{0+0.5*sin(247.5)});
\node at ({1+0.25*cos(247.5)},{0+0.25*sin(247.5)}) {\scriptsize $\bullet$};
\node at ({1+0.7*cos(247.5)},{0+0.7*sin(247.5)}) {\scriptsize $a_1$};
\draw ({1+0.25*cos(337.5)},{0+0.25*sin(337.5)}) -- ({1+0.5*cos(337.5)},{0+0.5*sin(337.5)});
\node at ({1+0.25*cos(337.5)},{0+0.25*sin(337.5)}) {\scriptsize $\bullet$};
\node at ({1+0.7*cos(337.5)},{0+0.7*sin(337.5)}) {\scriptsize $a_2$};
\draw ({1+0.25*cos(202.5)},{0+0.25*sin(202.5)}) -- ({1+0.5*cos(202.5)},{0+0.5*sin(202.5)});
\node at ({1+0.25*cos(202.5)},{0+0.25*sin(202.5)}) {\scriptsize $\bullet$};
\node at ({1+0.7*cos(202.5)},{0+0.7*sin(202.5)}) {\scriptsize $a_2$};
\draw ({1+0.25*cos(67.5)},{0+0.25*sin(67.5)}) -- ({1+0.5*cos(67.5)},{0+0.5*sin(67.5)});
\node at ({1+0.25*cos(67.5)},{0+0.25*sin(67.5)}) {\scriptsize $\bullet$};
\node at ({1+0.7*cos(67.5)},{0+0.7*sin(67.5)}) {\scriptsize $a_m$};
\draw ({1+0.25*cos(112.5)},{0+0.25*sin(112.5)}) -- ({1+0.5*cos(112.5)},{0+0.5*sin(112.5)});
\node at ({1+0.25*cos(112.5)},{0+0.25*sin(112.5)}) {\scriptsize $\bullet$};
\node at ({1+0.7*cos(112.5)},{0+0.7*sin(112.5)}) {\scriptsize $a_m$};
\node at ({1+0.5*cos(147.5)},{0+0.5*sin(147.5)}) {$\cdot$};
\node at ({1+0.5*cos(157.5)},{0+0.5*sin(157.5)}) {$\cdot$};
\node at ({1+0.5*cos(167.5)},{0+0.5*sin(167.5)}) {$\cdot$};
\node at ({1+0.5*cos(12.5)},{0+0.5*sin(12.5)}) {$\cdot$};
\node at ({1+0.5*cos(22.5)},{0+0.5*sin(22.5)}) {$\cdot$};
\node at ({1+0.5*cos(32.5)},{0+0.5*sin(32.5)}) {$\cdot$};
\end{tikzpicture},
$$
respectively. 
Let also 
$$
\calA^{c,s}_{2m,1}(H) := \calA^{c,s}_{2m,1}(H)'+ \calA^{c,s}_{2m,1}(H)''.
$$
Recall from the proof of \cite[Prop.~5.2]{NSS22a} 
that the rank of $\calA^{c,s}_{2m,1}(H)$ can be computed as follows:
\begin{eqnarray} \label{eq:Acs}
\operatorname{rk}\, \mathcal{A}^{c,s}_{2m,1}(H)
&=& 2 B_{2m}(2g) - N_{2m}(2g) \\
\notag &=& \frac{1}{2m}
\sum_{r \in D_{4m} \setminus \mathbb{Z}_{2m}} \operatorname{Fix}(r)
\  = \   \frac{(2g)^{m+1}}{2}
+\frac{(2g)^m}{2}.
\end{eqnarray}
Here,  for $n\geq 2$ and $c\geq 1$, 
the number of bracelets (respectively, necklaces)
with $n$ beads on $c$ colors is denoted by $B_n(c)$
(respectively, by $N_n(c)$), and 
the first identity follows by estimating the number 
of symmetric bracelets  in terms of symmetric and non-symmetric necklaces;
the second identity is obtained 
by comparing  P\'olya's enumeration formula for the cyclic group
$\mathbb{Z}_{2m}$ acting on the set of $(2g)$-letter words of length $2m$,
with the corresponding formula for the dihedral group
$D_{4m}$ acting on the same set; 
the third identity is easily checked by determining the number 
of fixed points $\operatorname{Fix}(r)$
of any  reflection $r \in D_{4m} \setminus \mathbb{Z}_{2m}$.

We shall refine \eqref{eq:Acs} as  follows.

\begin{lemma} \label{lem:sym_diag}
Decompose $m=2^\ell (2r+1)$ with  $\ell,r \in \mathbb{N}$. Then
\begin{eqnarray*}
\operatorname{rk}\, \calA^{c,s}_{2m,1}(H)' &=& (2g)^{m+1}/2 + (2g)^{r+1}/2,  \\ \operatorname{rk}\, \calA^{c,s}_{2m,1}(H)'' &=& (2g)^{m}/2 + (2g)^{r+1}/2.  
\end{eqnarray*}
\end{lemma}

\begin{proof}
We shall use a few notions that have been introduced in \cite[\S 5]{NSS22b}. 
Recall that a \emph{necklace with arrow}
is a symmetric necklace (with $2m$ beads)  colored 
by the $(2g)$-element set $S=\{a_i,b_i \mid i=1,\dots,g\}$,
and endowed with an oriented axis 
showing its symmetry (the \emph{arrow}).
The set $\vec{N}$ of necklaces with arrows 
splits into two disjoint subsets $\vec{N}'$ and $\vec{N}''$,
according to the type of the arrow: 
either the arrow connects a bead to a bead, or
it connects the midpoint of an arc to the midpoint of an arc. 
Every element $x \in \vec{N}$ has an \emph{exponent of periodicity} $e(x) \in \mathbb{N}$, which is defined as follows:
the necklace with arrow obtained from $x$ by rotating its arrow
through the angle $\pi/2^k$ coincides with $x$ for $k = 0,1,\dots,e(x)-1$,
but it differs from $x$ for $k = e(x)$ and 
is then denoted by $\iota(x) \in \vec{N}$.
According to \cite[Lemma 5.4]{NSS22b}, $\iota$ is an involution of  $\vec{N}$
(without fixed point). We denote by 
$p:\vec{N} \to \vec{N}/\iota$ the canonical projection.

Let $\vec{N}_k$ be the subset of $\vec{N}$ 
that consists of necklaces with arrows having exponent of periodicity at least $k$;
observe that $e(x) \leq \ell +1$ for all $x\in \vec{N}$.
Hence we have a filtration 
$$
\vec{N} = \vec{N}_0 \supset  \vec{N}_1 \supset \cdots \supset  
\vec{N}_\ell  \supset \vec{N}_{\ell+1}  
\supset \vec{N}_{\ell+2}=\varnothing ,
$$
 which restricts to filtrations of $\vec{N}'$ and $\vec{N}''$.
The involution $\iota$ preserves each stratum $\vec{N}'_k \setminus  \vec{N}'_{k+1}$
(resp., $\vec{N}''_k \setminus  \vec{N}''_{k+1}$) of $\vec{N}'$ (resp., $\vec{N}''$) 
for every $k=0,\dots,\ell$. However, we have $\iota(\vec{N}'_{\ell+1})= \vec{N}''_{\ell+1}$.
Clearly, we have $\vert \vec{N}' \vert =(2g)^{m+1}$ and  $\vert \vec{N}'' \vert =(2g)^{m}$.
Moreover, we have $\vert \vec{N}'_{\ell +1} \vert = \vert \vec{N}''_{\ell +1} \vert =(2g)^{r+1}$.
Therefore, we have
\begin{eqnarray*}
\big\vert p(\vec{N}') \vert &=&
\frac{1}{2}\big( (2g)^{m+1} - (2g)^{r+1} \big) + (2g)^{r+1} = 
\frac{1}{2}\big( (2g)^{m+1} + (2g)^{r+1} \big),   \\
\big\vert p(\vec{N}'') \vert &=&
\frac{1}{2}\big( (2g)^{m} - (2g)^{r+1} \big) + (2g)^{r+1} = 
\frac{1}{2}\big( (2g)^{m} + (2g)^{r+1} \big).
\end{eqnarray*}
Besides, as observed in \cite[\S 5]{NSS22b}, 
the canonical map
$\Z\cdot(\vec{N}/\iota) \to \calA^{c,s}_{2m,1}(H) $ is an isomorphism:
indeed, it is obviously surjective onto a torsion-free abelian group,
and we have
$$
\operatorname{rk}\, \calA^{c,s}_{2m,1}(H) 
\stackrel{\eqref{eq:Acs}}{=}
\big\vert \vec{N}' \big\vert/2 + \big\vert \vec{N}'' \big\vert/2 =
\big\vert \vec{N} \big\vert/2 =
\big\vert \vec{N}/\iota  \big\vert.
$$
Hence, the restriction of this map
to $\Z\cdot p(\vec{N'})$ is an isomorphism 
onto  $\calA^{c,s}_{2m,1}(H)'$, from which we deduce the rank of 
$\calA^{c,s}_{2m,1}(H)'$. Similarly, we obtain the rank 
of $\calA^{c,s}_{2m,1}(H)''$.
\end{proof}

\begin{proof}[Proof of Proposition \ref{prop:t_1}]
The module $H^{(2)} \otimes \big(H^{(2)}\big)^{\otimes n}$ can be identified
with $\calB_{n,1}\big(H^{(2)};s\big)$ through the isomorphism
$$
u_0 \otimes (u_1 \otimes \cdots \otimes u_n)  \longmapsto
 \begin{tikzpicture}[thick,color=black,
baseline=-0.5em, x=1.2em, y=1.2em]
\footnotesize
\node[above=1pt] at (0,1.5) {\small $s$};
\node[left=1pt] at (-1,-0.6) {\small $u_1$};
\node[left=1pt] at (-1,+0.6) {\small $u_n$};
\node[left=1pt] at (-0.5,0.2) {\scriptsize  $\vdots$};
\node[below=1pt] at (0,-1.5) {\small $u_0 \star$};
\draw (0,-1.5) -- (0,1.5);
\draw (0,-0.6) -- (-1,-0.6);
\draw (0,0.6) -- (-1,0.6);
\draw [fill = black] (0,0.6) circle (.3ex);
\draw [fill = black] (0,-0.6) circle (.3ex);
\end{tikzpicture}.
$$  We also identify $\calA^c_{2n+2,1}(H)$ with 
$\big(H^{\otimes (2n+2)}\big)_{\mathcal{D}_{4n+4}}$
by \eqref{eq:wheel}.
Then, by Proposition \ref{prop:Gr_Delta},
the map
$(\Gr^L \!\Delta): \calB_{n,1}\big(H^{(2)};s\big) \to
\calA^c_{2n+2,1}(H)\otimes \Q/\Z$ is given by 
$$
(\Gr^L \!\Delta)\big(
u_0 \otimes (u_1 \otimes \cdots \otimes u_n) \big)
=
\frac{1}{2} \big\{
u_0 \otimes u_1 \otimes \cdots \otimes u_n 
 \otimes 
 u_n \otimes \cdots \otimes u_1 \otimes u_0 \big\} .
$$
It follows that 
$(\Gr^L \!\Delta) \big(\calB_{n,1}\big(H^{(2)};s\big)\big)$
coincides with the image of 
$\calA^{c,s}_{2n+2,1}(H)''\otimes \Z_2$
in $\calA^c_{2n+2,1}(H)\otimes \Q/\Z$. 
We conclude thanks to Lemma \ref{lem:sym_diag}.
\end{proof}

By combining Proposition \ref{prop:t_0} with Proposition \ref{prop:t_1}, we derive a lower bound on the size of $\Tors_2\big({Y_{2n+1}\calIC}/{Y_{2n+2}}\big)$.
It would be interesting, though considerably more challenging,
to refine this bound by explicitly computing $t_k(n)$ for every $k \geq 2$.

\section{The mod $1$ reduction of the symmetrized LMO homomorphism}
\label{sec:symLMO}

 In this section, we consider
the symmetrized version of the LMO homomorphism 
introduced in \cite{HM09}:
we analyze its ``modulo 1 reduction'' 
and prove some congruence relations for its values.

\subsection{The symmetrized LMO homomorphism $Z$}
\label{subsec:sym_LMO}

The \emph{symmetrized} version of the LMO homomorphism
is defined in \cite{HM09} as
$$ 
Z:= X^{-1} \circ Z^<:
\calIC \longrightarrow \calA(H) \otimes \Q,
$$
where $X: \calA(H)\otimes \Q \to  \calA^<(H)\otimes \Q$
is the isomorphism defined by \eqref{eq:X}.
Thus, according to Proposition \ref{prop:Gr^LA^<},
the passage from $Z^<$ to $Z$ consists in replacing
the target $\calA^<(H)\otimes \Q$ of the latter
by its associated graded with respect to the loop filtration.

Equivalently, we have  $Z= \kappa \circ \widetilde{Z}^Y$
where $\kappa$ is the isomorphism defined 
in terms of the basis $S=\{a_1,\dots,a_g,b_1\dots,b_g\}$ of $H$  by
\begin{equation} \label{eq:kappa}
\kappa(D) := (-1)^{\chi(D)}\cdot 
\left( \! \! \!  \begin{array}{c} \hbox{sum of all ways of $(\times 1/2)$-gluing \emph{some} 
$a_i$-colored}\\
\hbox{vertices of $D$ to \emph{some} $b_i$-colored vertices of $D$} \end{array} \! \! \! 
\right)
\end{equation}
for every Jacobi diagram $D$.
Indeed, we have the commutative diagram  
\begin{equation}
\label{eq:threespaces}
\begin{tikzcd}
   \calA (S) \otimes \Q \arrow[r, "\varphi^\Q" above, "\simeq" below] 
   \arrow[rd, bend right, "\kappa \ " below," \ \simeq" above] & 
   \calA^<(H)\otimes \Q \\
    &\calA(H) \otimes \Q \arrow[u,"X" right, "\simeq" left]
\end{tikzcd}
\end{equation}
where $\varphi^\Q$ denotes the rational version 
of the map $\varphi$ from Proposition~\ref{prop:varphi}.

\begin{remark} \label{rem:star_product}
Recall from \cite[\S 3.1]{HM09} that
$\calA(H)\otimes \Q$ has an $\Sp(H^\Q)$-equivariant
associative  product $\star$ which corresponds 
to the multiplication \eqref{eq:star_product} through $\kappa$.
Hence, $Z:(\calIC,\circ) \to \big(\calA(H)\otimes \Q,\star\big)$ 
is a monoid map.
\hfill $\blacksquare$
\end{remark}

While the version $Z^<$ of the LMO homomorphism is suitable
to the study of its variations under clasper surgery (see \S \ref{sec:Z^<}),
this version $Z$ is more convenient
for comparison with some classical invariants.
In particular, the (leading terms of the) tree reduction and one-loop reduction of $Z$
have the following topological interpretations:
\begin{itemize}
\item Recall from \S \ref{subsec:few_loops} 
that $\mathsf{D}_k(H)$ is defined as  the kernel 
of the bracketing homomorphism
${H\otimes \frakL_{k+1}(H)} \to \frakL_{k+2}(H)$
in the free Lie algebra $ \frakL(H)$.
Let also $\calIC[k]$ be the submonoid of $\calIC$ acting trivially 
on the $k$-th nilpotent quotient $\pi/\Gamma_{k+1} \pi$ of $\pi$;
moreover, let 
\begin{equation} \label{eq:tau_k}
\tau_k: \calIC[k] \longrightarrow \mathsf{D}_k(H)
\end{equation}
be the \emph{$k$-th Johnson homomorphism}, 
encoding the action on the  $(k+1)$-st nilpotent quotient $\pi/\Gamma_{k+2} \pi$ of $\pi$.
It follows from \cite[Th$.$ 8.19]{CHM08} that the following diagram is commutative:
\begin{equation} \label{eq:tau_Z}
\xymatrix{
\calIC[k] \ar[rr]^-{Z_k} \ar[d]_-{\tau_k} 
&& \calA_k(H)\otimes \Q \ar@{->>}[d] \\
\mathsf{D}_k(H) \ar[r] & 
\mathsf{D}_k(H)\otimes\Q  & 
\ar[l]_-\simeq^\eta \calA_{k,0}^c(H)\otimes \Q.
}
\end{equation}
Since $\mathsf{D}_k(H)$ is torsion-free,
we can also view the $k$-th Johnson homomorphism
as a map $\tau_k: \calIC[k] \to  \calA_{k,0}^c(H)\otimes \Q $:
thus, \eqref{eq:tau_Z} says that the map 
$Z_{k,0}:\calIC[k] \to  \calA_{k,0}^c(H)\otimes \Q$
coincides with $\tau_k$.
In particular, when restricted to the submonoid $Y_k \calIC$ of $\calIC[k]$,
\eqref{eq:tau_Z} shows that 
$Z_{k,0}:Y_k\calIC \to \bqm \calA_{k,0}^c(H)\eqm$
(which coincides with the reduction of 
$Z_k^<:Y_k\calIC \to \bqm \calA^{<,c}_k(H)\eqm$ 
modulo $\bqm L_1 \calA^{<,c}_k(H)\eqm$)
is the same as $\tau_k: Y_k\calIC \to  \bqm \calA_{k,0}^c(H)\eqm$.

\item Let $\widehat{\Q[\pi]}$ be the $I$-adic completion
of the group algebra $\Q[\pi]$, and let $\tilde\alpha:\calIC \to K_1\big(\widehat{\Q[\pi]}\big)$
be the non-commutative version of the Reidemeister--Turaev torsion 
that is introduced in  \cite{NSS23}. 
It is shown there 
that the composition $\tilde \alpha_{\leq k-1}$ of $\tilde \alpha$
with the canonical map 
$K_1\big(\widehat{\Q[\pi]}\big) \to K_1\big(\widehat{\Q[\pi]}/\widehat{I}^{k}\big)$
vanishes on $Y_k\calIC$; hence we have a homomorphism
$$
\qquad\qquad \xymatrix{
Y_k \calIC \ar[r]^-{\tilde\alpha_{\leq k}} \ar@/_1cm/[rr]^-{\tilde\alpha_k} & 
\Ker \Big( K_1\big(\widehat{\Q[\pi]}/\widehat{I}^{k+1}\big) 
\to K_1\big(\widehat{\Q[\pi]}/\widehat{I}^k\big) \Big)\simeq \!\!\!\! &
\!\!\!\! \!\!\!\!  \!\!\!\! 
\big(H^{\otimes k}\big)_{\Z_k} \!\!\otimes \Q.
}
$$
Let also $p_+: \big(H^{\otimes k}\big)_{\Z_k} \otimes \Q  
\to \calA^c_{k,1}(H)\otimes \Q$
be the canonical projection 
$\big(H^{\otimes k}\big)_{\Z_k} \otimes \Q 
\to  \big(H^{\otimes k}\big)_{D_{2k}} \otimes \Q$
composed with the isomorphism \eqref{eq:wheel}.
According to \cite[Th$.$ 1.3]{NSS23}, the following diagram is commutative:
\begin{equation} \label{eq:alpha_Z}
\xymatrix{
Y_k\calIC \ar[r]^-{-2Z_k} \ar[d]_-{\tilde{\alpha}_k} 
& \calA_k^c(H)\otimes \Q \ar@{->>}[d] \\
\big(H^{\otimes k}\big)_{\Z_k} \otimes \Q 
\ar[r]_-{p_+} & \calA_{k,1}^c(H)\otimes \Q.
}
\end{equation}
\end{itemize}

\subsection{The modulo $1$ reduction $\zeta$ of $Z$}

We consider now the following reduction of 
the symmetrized LMO homomorphism:
$$
\xymatrix{
\calIC \ar[r]^-Z \ar@/_2pc/@{-->}[rr]_-{{\zeta}}
& \calA(H) \otimes \Q \ar@{->>}[r] &
{\displaystyle\frac{\calA(H) \otimes \Q}{\calR(H)}}
}
\quad \hbox{ where } \calR(H) := X^{-1}\big(\bqm\calA^<(H)\eqm\big)
$$
Note that $\zeta$ corresponds to the mod $1$ reduction 
$\zeta^<$ of $Z^<$  (defined in \S \ref{subsec:LMO})
via an isomorphism induced by $X$.
Thus, with an abuse of terminology,
we shall call 
$\zeta$ the \emph{modulo 1 reduction of $Z$}.

\begin{remark} It follows from the definition that, 
for any $k\geq 1$, we have
\begin{equation} \label{eq:XKZ}
\calR_k(H)=
\kappa\big(\!\bqm\calA_k(S)\eqm\!\big)=
Z_k\big(F_k\Z[\calIC]\big).     
\end{equation}
But, we emphasize that this lattice
of the $\Q$-vector space $\calA_k(H)\otimes \Q$
\emph{does not} coincide with $\bqm\calA_k(H)\eqm$.
Thus, the target  of $\zeta$ \emph{is not} 
the same as $\calA(H)\otimes \Q/\Z$. 
\hfill $\blacksquare$
\end{remark}

The lattice $\calR(H)$
can be described  by generators as follows.
Let $\widetilde{\calA}(H)$ 
be the module  defined as $\calA^<(H)$ in \S \ref{subsec:Jacobi_diagrams}, 
but with the STU-like relation removed.
We consider the endomorphism 
$w$ of $\widetilde{\calA}(H)\otimes \Q$ 
defined  by 
\begin{equation} \label{eq:w}
w(D) :=
\left(\begin{array}{c}
\hbox{sum of all ways of $\omega$-contracting}\\
\hbox{an ordered pair of external vertices in $D$}
\end{array}\right)
\end{equation}
on every  Jacobi diagram $D$.
For instance, we have
\[
w\Big(\tikz[baseline=-4pt, scale=0.3]{
  \draw[fill=black] (-1.3,0.4) circle (0.8ex);
  \draw[fill=black] (1.3,0.4) circle (0.8ex);

  \draw[thick] (-1.3,0.4) -- (1.3,0.4);

  \draw[thick] (-1.3,0.4) -- (-2.1,-1);
  \node at (-2.1,-1.4) {$\scriptstyle x$};
  \node at (-1.4,-1.4) {$\scriptstyle <$};

  \draw[thick] (-1.3,0.4) -- (-0.5,-1);
  \node at (-0.7,-1.4) {$\scriptstyle y$};
  \node at (-0.0,-1.4) {$\scriptstyle <$};

  \draw[thick] (1.3,0.4) -- (0.7,-1);
  \node at (0.7,-1.4) {$\scriptstyle z$};
  \node at (1.4,-1.4) {$\scriptstyle <$};

  \draw[thick] (1.3,0.4) -- (2.1,-1);
  \node at (2.1,-1.4) {$\scriptstyle t$};
}\Big) = 
-\omega(x,z)\!\tikz[baseline=-4pt, scale=0.3]{
  \draw[thick] (0,0) circle (0.6);

  \draw[fill=black] (-0.6,0) circle (0.8ex);
  \draw[fill=black] (0.6,0) circle (0.8ex);

  \draw[thick] (-0.6,0) to[out=180, in=90] (-1.2,-1);
  \node at (-1.3,-1.4) {$\scriptstyle y$};
  
  \draw[thick] (0.6,0) to[out=0, in=90] (1.2,-1);
  \node at (1.2,-1.4) {$\scriptstyle t$};

  \node at (-0.15,-1.4) {$\scriptstyle <$};
}
+\omega(x,t)\!\tikz[baseline=-4pt, scale=0.3]{
  \draw[thick] (0,0) circle (0.6);

  \draw[fill=black] (-0.6,0) circle (0.8ex);
  \draw[fill=black] (0.6,0) circle (0.8ex);

  \draw[thick] (-0.6,0) to[out=180, in=90] (-1.2,-1);
  \node at (-1.3,-1.4) {$\scriptstyle y$};
  
  \draw[thick] (0.6,0) to[out=0, in=90] (1.2,-1);
  \node at (1.2,-1.4) {$\scriptstyle z$};

  \node at (-0.15,-1.4) {$\scriptstyle <$};
}
+\omega(y,z)\!\tikz[baseline=-4pt, scale=0.3]{
  \draw[thick] (0,0) circle (0.6);

  \draw[fill=black] (-0.6,0) circle (0.8ex);
  \draw[fill=black] (0.6,0) circle (0.8ex);

  \draw[thick] (-0.6,0) to[out=180, in=90] (-1.2,-1);
  \node at (-1.3,-1.4) {$\scriptstyle x$};
  
  \draw[thick] (0.6,0) to[out=0, in=90] (1.2,-1);
  \node at (1.2,-1.4) {$\scriptstyle t$};

  \node at (-0.15,-1.4) {$\scriptstyle <$};
}
-\omega(y,t)\!\tikz[baseline=-4pt, scale=0.3]{
  \draw[thick] (0,0) circle (0.6);

  \draw[fill=black] (-0.6,0) circle (0.8ex);
  \draw[fill=black] (0.6,0) circle (0.8ex);

  \draw[thick] (-0.6,0) to[out=180, in=90] (-1.2,-1);
  \node at (-1.3,-1.4) {$\scriptstyle x$};
  
  \draw[thick] (0.6,0) to[out=0, in=90] (1.2,-1);
  \node at (1.2,-1.4) {$\scriptstyle z$};

  \node at (-0.15,-1.4) {$\scriptstyle <$};
}.\]
There is the  projection
$f:\widetilde{\calA}(H)\otimes \Q 
\to {\calA}(H)\otimes \Q $
that forgets the total order on vertices,
and the projection 
$m:\widetilde{\calA}(H)\otimes \Q  \to {\calA}^<(H)\otimes \Q$ 
that takes the quotient modulo ``STU-like''.
It is shown in \cite[Prop$.$ 3.1]{HM09}
that $X^{-1}\circ m = f\circ \exp_\circ(w/2)$.
Thus, we immediately obtain the following: 

\begin{lemma} \label{lem:X_gen}
The module  $\calR_k(H)$
is generated in  $\calA_k(H)\otimes \Q$
by the $\Q$-linear combinations
$$
f\big(\exp_\circ(w/2)(D)\big) =f\Big(
D+ \frac{w(D)}{2}+ \frac{w(w(D))}{2^2\, 2!}  +\,\cdots \Big)
$$ 
where $D \in \bqm \widetilde{\calA}_k(H)\eqm$ 
ranges over a system of generators for
$\bqm {\calA}_k^<(H)\eqm$ under $m$.
\end{lemma}

\begin{example}
\label{ex:R2(H)}
In degree $k:=2$,  the module
$ \calR_k(H)$ is generated by 
\begin{align*}
X^{-1}\Big(
\tikz[baseline=-0.5ex, scale=0.5]{
  \draw[thick] (0,0) circle (0.5);
  \draw[thick] (-0.5,0) -- (0.5,0);
  \draw[fill=black] (-0.5,0) circle (0.5ex);
  \draw[fill=black] (0.5,0) circle (0.5ex);
}\Big)
=& \hspace{2mm}
\tikz[baseline=-0.5ex, scale=0.5]{
  \draw[thick] (0,0) circle (0.5);
  \draw[thick] (-0.5,0) -- (0.5,0);
  \draw[fill=black] (-0.5,0) circle (0.5ex);
  \draw[fill=black] (0.5,0) circle (0.5ex);
} \\
X^{-1}\Big(
\tikz[baseline=-1.8ex, scale=0.25]{
  \draw[thick] (0,0) circle (0.6);
  \draw[fill=black] (-0.6,0) circle (0.8ex);
  \draw[fill=black] (0.6,0) circle (0.8ex);
  \draw[thick] (-0.6,0) to[out=180, in=90] (-1.2,-1);
  \node at (-1.3,-1.4) {$\scriptstyle a$};  
  \draw[thick] (0.6,0) to [out=0, in=90] (1.2,-1);
  \node at (1.2,-1.4) {$ \scriptstyle b$};
  \node at (-0.15,-1.4) {$ \scriptstyle <$};
}
\Big) =& \tikz[baseline=-1.8ex, scale=0.3]{
  \draw[thick] (0,0) circle (0.6);
  \draw[fill=black] (-0.6,0) circle (0.8ex);
  \draw[fill=black] (0.6,0) circle (0.8ex);
  \draw[thick] (-0.6,0) to[out=180, in=90] (-1.2,-1);
  \node at (-1.3,-1.4) {$\scriptstyle a$};  
  \draw[thick] (0.6,0) to[out=0, in=90] (1.2,-1);
  \node at (1.2,-1.4) {$ \scriptstyle b$};
} + \frac{1}{2}\omega(a,b)\hspace{1mm}\tikz[baseline=-0.5ex, scale=0.5]{
  \draw[thick] (0,0) circle (0.5);
  \draw[thick] (-0.5,0) -- (0.5,0);
  \draw[fill=black] (-0.5,0) circle (0.5ex);
  \draw[fill=black] (0.5,0) circle (0.5ex);
} \qquad \hbox{(for all $a,b\in H$)}\\
X^{-1}\Big(\tikz[baseline=-1ex, scale=0.3]{
  \node at (-2,-1.45) {$\scriptstyle a$};
  \node at (-0.7,-1.45) {$ \scriptstyle b$};
  \node at (0.7,-1.45) {$ \scriptstyle c$};
  \node at (2,-1.45) {$ \scriptstyle d$};
  \node at (-1.35,-1.45) {$ \scriptstyle <$};
  \node at (0,-1.45) {$ \scriptstyle <$};
  \node at (1.35,-1.45) {$ \scriptstyle <$};
  \draw[fill=black] (0,0) circle (0.8ex);
  \draw[thick] (0,0) -- (-0.7,-1); 
  \draw[thick] (0,0) -- (0.7,-1);  
  \draw[fill=black] (0,1) circle (0.8ex);
  \draw[thick] (0,1) -- (-2,-1); 
  \draw[thick] (0,1) -- (2,-1);  
  \draw[thick] (0,0) -- (0,1);
}\Big) =& \tikz[baseline=-1.8ex, scale=0.3]{
  \draw[fill=black] (-1.3,0.4) circle (0.8ex);
  \draw[fill=black] (1.3,0.4) circle (0.8ex);
  \draw[thick] (-1.3,0.4) -- (1.3,0.4);
  \draw[thick] (-1.3,0.4) -- (-2.1,-1);
  \node at (-2.1,-1.4) {$ \scriptstyle b$};
  \draw[thick] (-1.3,0.4) -- (-0.5,-1);
  \node at (-0.7,-1.4) {$ \scriptstyle c$};
  \draw[thick] (1.3,0.4) -- (0.7,-1);
  \node at (0.7,-1.4) {$ \scriptstyle d$};
  \draw[thick] (1.3,0.4) -- (2.1,-1);
  \node at (2.1,-1.4) {$\scriptstyle a$};
} + \frac{1}{4}\big( \omega(a,b)\omega(c,d) - \omega(a,c)\omega(b,d)\big)\hspace{1mm}\tikz[baseline=-0.5ex, scale=0.5]{
  \draw[thick] (0,0) circle (0.5);
  \draw[thick] (-0.5,0) -- (0.5,0);
  \draw[fill=black] (-0.5,0) circle (0.5ex);
  \draw[fill=black] (0.5,0) circle (0.5ex);
}\\
&+ \frac{\omega(a,b)}{2} \tikz[baseline=-1.8ex, scale=0.25]{
  \draw[thick] (0,0) circle (0.6);
  \draw[fill=black] (-0.6,0) circle (0.8ex);
  \draw[fill=black] (0.6,0) circle (0.8ex);
  \draw[thick] (-0.6,0) to[out=180, in=90] (-1.2,-1);
  \node at (-1.3,-1.4) {$ \scriptstyle c$};  
  \draw[thick] (0.6,0) to[out=0, in=90] (1.2,-1);
  \node at (1.2,-1.4) {$ \scriptstyle d$};
} - \frac{\omega(a,c)}{2} \tikz[baseline=-1.8ex, scale=0.25]{
  \draw[thick] (0,0) circle (0.6);
  \draw[fill=black] (-0.6,0) circle (0.8ex);
  \draw[fill=black] (0.6,0) circle (0.8ex);
  \draw[thick] (-0.6,0) to[out=180, in=90] (-1.2,-1);
  \node at (-1.3,-1.4) {$ \scriptstyle b$};  
  \draw[thick] (0.6,0) to[out=0, in=90] (1.2,-1);
  \node at (1.2,-1.4) {$ \scriptstyle d$};
} \\
&+ \frac{\omega(c,d)}{2} \tikz[baseline=-1.8ex, scale=0.25]{
  \draw[thick] (0,0) circle (0.6);
  \draw[fill=black] (-0.6,0) circle (0.8ex);
  \draw[fill=black] (0.6,0) circle (0.8ex);
  \draw[thick] (-0.6,0) to[out=180, in=90] (-1.2,-1);
  \node at (-1.3,-1.4) {$\scriptstyle a$};  
  \draw[thick] (0.6,0) to[out=0, in=90] (1.2,-1);
  \node at (1.2,-1.4) {$ \scriptstyle b$};
} - \frac{\omega(b,d)}{2} \tikz[baseline=-1.8ex, scale=0.25]{
  \draw[thick] (0,0) circle (0.6);
  \draw[fill=black] (-0.6,0) circle (0.8ex);
  \draw[fill=black] (0.6,0) circle (0.8ex);
  \draw[thick] (-0.6,0) to[out=180, in=90] (-1.2,-1);
  \node at (-1.3,-1.4) {$\scriptstyle a$};  
  \draw[thick] (0.6,0) to[out=0, in=90] (1.2,-1);
  \node at (1.2,-1.4) {$ \scriptstyle c$};
} \quad \hbox{(for all $a,b,c,d\in H$)}. \quad \blacksquare
\end{align*}
\end{example}

The lattice  $\calR(H)$
can also be described  by defining relations.
To state this description, 
we consider the symmetric bilinear pairing 
$\langle-,-\rangle: H \times H \rightarrow \mathbb{Z}$
given in terms of the basis $S$
by $\left\langle a_i, a_j\right\rangle=\left\langle b_i, b_j\right\rangle=0$ and 
$\left\langle a_i, b_j\right\rangle= \delta_{i j}$.
Let $c$ be the endomorphism of $\calA(H) \otimes \Q$ 
transforming any Jacobi diagram $D$ to
\begin{equation} \label{eq:c}
c(D) :=
\left(\begin{array}{c}
\hbox{sum of all ways of $(\langle-,-\rangle/2)$-contracting}\\
\hbox{an unordered pair of external vertices in $D$}
\end{array}\right).
\end{equation}
For instance, we have
\[
c \Big(\tikz[baseline=-3pt, scale=0.35]{
  \draw[thick] (-0.6,0) -- (0.6,0);

  \draw[fill=black] (-0.6,0) circle (0.8ex);
  \draw[fill=black] (0.6,0) circle (0.8ex);

  \draw[thick] (0.6,0) -- ++(60:0.8);
  \draw[thick] (0.6,0) -- ++(-60:0.8);

  \draw[thick] (-0.6,0) -- ++(120:0.8);
  \draw[thick] (-0.6,0) -- ++(-120:0.8);

  \draw (0.6,0) -- ++(60:0.8)  node[above, xshift=1mm] {$\scriptstyle t$};
  \draw (0.6,0) -- ++(-60:0.8) node[below,xshift=1mm] {$\scriptstyle z$};
  \draw (-0.6,0) -- ++(120:0.8)  node[above, xshift=-1mm] {$\scriptstyle x$};
  \draw (-0.6,0) -- ++(-120:0.8) node[below, xshift=-1mm] {$\scriptstyle y$};
}\Big)
= \frac{1}{2}\langle x,t \rangle \! \tikz[baseline=-3pt, scale=0.35]{
  \draw[thick] (0,0) circle (0.6);
  \draw[fill=black] (-0.6,0) circle (0.8ex);
  \draw[fill=black] (0.6,0) circle (0.8ex);

  \draw[thick] (-0.6,0) to[out=180,in=90] (-1.2,-1);
  \node at (-1.3,-1.4) {$\scriptstyle y$};

  \draw[thick] (0.6,0) to[out=0,in=90] (1.2,-1);
  \node at (1.2,-1.4) {$\scriptstyle z$};
}
+ \frac{1}{2}\langle y,z \rangle \! \tikz[baseline=-3pt, scale=0.35]{
  \draw[thick] (0,0) circle (0.6);
  \draw[fill=black] (-0.6,0) circle (0.8ex);
  \draw[fill=black] (0.6,0) circle (0.8ex);

  \draw[thick] (-0.6,0) to[out=180,in=90] (-1.2,-1);
  \node at (-1.3,-1.4) {$\scriptstyle x$};

  \draw[thick] (0.6,0) to[out=0,in=90] (1.2,-1);
  \node at (1.2,-1.4) {$\scriptstyle t$};
}
-\frac{1}{2}\langle x,z \rangle \! \tikz[baseline=-3pt, scale=0.35]{
  \draw[thick] (0,0) circle (0.6);
  \draw[fill=black] (-0.6,0) circle (0.8ex);
  \draw[fill=black] (0.6,0) circle (0.8ex);

  \draw[thick] (-0.6,0) to[out=180,in=90] (-1.2,-1);
  \node at (-1.3,-1.4) {$\scriptstyle y$};

  \draw[thick] (0.6,0) to[out=0,in=90] (1.2,-1);
  \node at (1.2,-1.4) {$\scriptstyle t$};
}
-\frac{1}{2}\langle y,t \rangle \! \tikz[baseline=-3pt, scale=0.35]{
  \draw[thick] (0,0) circle (0.6);
  \draw[fill=black] (-0.6,0) circle (0.8ex);
  \draw[fill=black] (0.6,0) circle (0.8ex);

  \draw[thick] (-0.6,0) to[out=180,in=90] (-1.2,-1);
  \node at (-1.3,-1.4) {$\scriptstyle x$};

  \draw[thick] (0.6,0) to[out=0,in=90] (1.2,-1);
  \node at (1.2,-1.4) {$\scriptstyle z$};
} .
\]
Note that the map $c$ increases the loop degree by one.

\begin{lemma} \label{lem:X_def_rel}
An $x \in \calA_k(H)\otimes \Q$, 
which is decomposed as $x = \sum_{r\geq 1-k} x_{r}$ 
with $x_r \in \calA_{k,r}(H)\otimes \Q$,
belongs to the lattice $\calR_k(H)$ if, and only if,
it satisfies 
\begin{equation}
\label{condition:C}
 \quad x_{(1-k)+i} \equiv \sum_{j=1}^i(-1)^{j+1} \cdot c^{[j]}\left(x_{(1-k)+i-j}\right) \mod  ``\calA_{k,(1-k)+i}(H)"
\end{equation}
 for every $i\geq 0$.
(Here $c^{[j]}= c^j/j!$ denotes the $j$-th  divided power of $c$.)
\end{lemma}

\begin{remark} \label{rem:powers_2}
It follows from Lemma \ref{lem:X_def_rel}
that the loop-degree $r$ part $x_r$ of 
any $x\in \calR_k(H)$ necessarily belongs to 
$({1}/{2^r}) \bqm\calA_{k,r}(H)\eqm$. 
\hfill $\blacksquare$
\end{remark}

\begin{example}
Lemma \ref{lem:X_def_rel} specializes as follows for $k:=2$.
Recall that $\calA^c_2(H)$ is torsion-free,
so that  $\calA^c_2(H)= \bqm \calA^c_2(H) \eqm$.
Then, an $x=x_0+x_1+x_2\in  \calA^c_2(H)\otimes \Q$ 
belongs to $\calR_2(H)$
if, and only if, we have
$$
\left\{\begin{array}{ll}
x_0 \equiv 0 &\mod  \calA^c_{2,0}(H),\\
x_1 \equiv c(x_0)  &\mod \calA^c_{2,1}(H),\\
x_2 \equiv c(x_1) - {c^2(x_0)}/{2} & \mod \calA^c_{2,2}(H). 
\end{array}\right.
$$

\up
\hfill $\blacksquare$
\end{example}

\begin{proof}[Proof of Lemma \ref{lem:X_def_rel}]
According to \eqref{eq:XKZ}, we are asked to identify
$\kappa(``\calA_k(S)")$  
as a lattice in $\calA_k(H)\otimes \Q$.
We start by observing that, by the definition \eqref{eq:kappa}, we have 
\begin{equation} \label{eq:kappa_alt}
\kappa = (-1)^{\chi(-)}\cdot \exp(c) \circ \iota,
\end{equation}
where  we denote by 
$\iota: \calA(S)\otimes \Q \rightarrow \calA(H)\otimes \Q$ 
the obvious isomorphism defined by the basis $S$ of $H$.

Let $x = \sum_{r\geq 1-k} x_{r}$ 
in  $\calA_{k}(H)\otimes \Q$ satisfying \eqref{condition:C}.
Set $\tilde x := \sum_{r\geq 1-k} \tilde x_{r}$, with
\[
\tilde{x}_{1-k+i} := 
\iota^{-1}\Big( \sum_{j=0}^i (-1)^{j+k+i}c^{[j]}(x_{1-k+i-j})\Big), \
\hbox{for every $i\geq 0$}.
\] 
The condition \eqref{condition:C} on $x$ implies that
$\tilde x \in ``\calA_k(S)"$. Moreover, we have
\[\begin{aligned}
  \sum_{s=0}^i(-1)^{s+k} c^{[i-s]}\left(\iota(\tilde{x}_{1-k+s})\right) =& \sum_{s=0}^i c^{[i-s]}\Big(\sum_{j=0}^s(-1)^j c^{[j]}
  \left(x_{1-k+s-j}\right)\Big) \\
=& \sum_{s=0}^i \sum_{j=0}^s(-1)^j  \frac{(i+j-s)!}{(i-s)!j!} c^{[i+j-s]}\left(x_{1-k+s-j}\right) \\
=& \sum_{t=0}^i \sum_{j=0}^{i-t}(-1)^j  
\binom{j+t}{t}c^{[j+t]}\left(x_{1-k+i-t-j}\right)  \\
=& \sum_{t=0}^i \sum_{u=t}^i(-1)^{u-t} \binom{u}{t}c^{[u]}
\left(x_{1-k+i-u}\right)  \\
=& \sum_{u=0}^i \sum_{t=0}^u(-1)^{u-t} \frac{u!}{(u-t)!t!} c^{[u]}\left(x_{1-k+i-u}\right) \\
=& \sum_{u=0}^i(-1)^u (1+(-1))^u c^{[u]}\left(x_{1-k+i-u}\right) 
= x_{1-k+i}.
\end{aligned} 
\]
Using \eqref{eq:kappa_alt},
we deduce that $\kappa(\tilde{x})=x$: therefore, 
$x$ belongs to $\kappa(`` \calA_k(S)")$.

Conversely, let $\tilde{x} \in \calA_k(S) $ 
and  set $x:= \kappa(\tilde{x})$.
Decompose $\tilde{x} =\sum_{r\geq 1-k} \tilde{x}_r$ 
and  ${x} =\sum_{r\geq 1-k} {x}_r$
with respect to the loop degree.
For every $i\geq 0$, \eqref{eq:kappa_alt} implies 
\[
(-1)^kx_{1-k+i} = 
\sum_{s = 0}^{i}(-1)^s\cdot c^{[i-s]}\iota(\tilde{x}_{1-k+s})
\] 
or, equivalently,
$$
(-1)^{k+i} \iota(\tilde{x}_{1-k+i}) = x_{1-k+i} 
-\sum_{s=0}^{i-1} (-1)^{k+s} c^{[i-s]}\iota(\tilde{x}_{1-k+s}).
$$
Then, by performing an induction on $i \geq 0$
with computations analogous to those in the preceding paragraph, we obtain that
$$
(-1)^{k+i}\iota(\tilde{x}_{1-k+i}) 
= \sum_{j = 0}^{i}(-1)^j\cdot c^{[j]}(x_{1-k+i-j}).
$$
Since the left-hand side of this identity is in $``\calA(H)"$,
 $x$ satisfies \eqref{condition:C}. 
\end{proof}

\subsection{Trace homomorphisms}

We present an alternative formulation 
(in Proposition~\ref{prop:P2} below) 
of the characterization of $\calR_k(H)$ in $\calA_k(H)\otimes \Q$ 
given in Lemma~\ref{lem:X_def_rel}.
Since this will suffice for our subsequent applications, we restrict our attention to connected Jacobi diagrams.
Our  tool is a sequence of ``trace'' homomorphisms
$$
\Tr^{(r)}: \calP_k^{(r)}(H) \longrightarrow
\bqm\calA^c_{k,r+1}(H)\eqm/2^{r+1}\bqm\calA^c_{k,r+1}(H)\eqm
$$
indexed by integers $r\geq 0$.
The source of $\Tr^{(r)}$ is the submodule
$$
\calP^{(r)}_k(H) \subset
\bigoplus_{j=0}^{r} \bqm \calA^c_{k,j}(H) \eqm
$$
of $\bqm \calA^c_{k}(H) \eqm$
defined inductively by 
$\calP^{(0)}_k(H):= \bqm \calA^c_{k,0}(H) \eqm$ and 
$$
\calP^{(r+1)}_k(H):=  \calP^{(r)}_k(H) 
\times_{\bqm\calA^c_{k,r+1}(H)\eqm \otimes \Z_{2^{r+1}}}
\bqm\calA^c_{k,r+1}(H)\eqm
$$
where the fibered product is taken over $\Tr^{(r)}$
and the mod $2^{r+1}$ reduction.
Then, for any $y=\sum_{j=0}^r y_j \in \calP_k^{(r)}(H)$, we 
define $\Tr^{(r)}: \calP_k^{(r)}(H)  \to
\bqm\calA^c_{k,r+1}(H)\eqm \otimes \Z_{2^{r+1}}$~by
\begin{equation} \label{eq:def_Tr}
\Tr^{(r)}(y) := 
\Bigg\{ \sum_{j=0}^r (-1)^{j}\, (2c)^{[j+1]}(y_{r-j})\Bigg\}.
\end{equation}
In a more explicit way, we can equivalently set
$$
\Tr^{(r)}(y) :=  \Bigg\{ \sum_{j=0}^r (-1)^{j}\,
\left(\!\!\begin{array}{c}
\hbox{\small sum of all ways of $\langle-,-\rangle$-contracting}\\
\hbox{\small  $(j+1)$ pairs of external vertices of $y_{r-j}$}
\end{array}\!\!\right) \Bigg\} .
$$ 

The map $\Tr:=\Tr^{(0)}$ on
$\bqm \calA^c_{k,0}(H) \eqm=\calP^{(0)}_k(H)$
is easily computed as follows.
It coincides 
with the ``symmetric trace'' 
$\operatorname{Tr}^S$ considered in \cite{Fae23}
in degree $k:=2$.

\begin{lemma} \label{lem:Tr^(0)}
The  map  $\Tr: \bqm\calA^c_{k,0}(H)\eqm \to
\bqm\calA^c_{k,1}(H)\eqm\otimes \Z_2$ 
is given by 
\begin{equation} \label{eq:Tr_2}
 \Tr(f(\tilde y)):= \big\{f w(\tilde y) \big\}
\end{equation}
for all $\tilde y \in 
\bqm \widetilde{\calA}^c_{k,0}(H)\eqm$,
where the homomorphism $w$ is defined by \eqref{eq:w}
and the map~$f$ forgets the order of vertices.
So, the homomorphism $\Tr$ is $\Sp(H)$-equivariant.
\end{lemma}

\begin{proof}
The $\Sp(H)$-equivariance is obvious from 
\eqref{eq:Tr_2}, and \eqref{eq:Tr_2}
follows from the fact that
$\langle u,v\rangle \equiv \omega(u,v)$
mod 2 for all $u,v\in H$.
\end{proof}

\begin{remark}
The $\Sp(H)$-isomorphism types of the source and target
of $\Tr$ are deduced from \S \ref{subsec:few_loops}:
\begin{itemize}
\item If $k$ is even, then both $\calA^c_{k,0}(H)$ 
and $\calA^c_{k,1}(H)$ are torsion-free;
so, $\Tr$ is in degree $k$ a homomorphism 
$\mathsf{D}'_{k}(H) \to (H^{\otimes k})_{\mathcal{D}_{2k}}\otimes \Z_2$.
\item If $k=2n-1$ is odd, 
then both $\calA^c_{k,0}(H)$ and $\calA^c_{k,1}(H)$ have $2$-torsion, as described in \eqref{eq:Levine1} and \eqref{eq:torsion_wheel};
so, $\Tr$ is in degree $k$ a homomorphism 
$\mathsf{D}_{k}(H) \to 
\big((H^{\otimes k})_{\mathcal{D}_{2k}}\otimes \Z_2\big)
\big/ \big(H^{\otimes n}\otimes \Z_2\big)$.
\hfill $\blacksquare$
\end{itemize}
\end{remark}

Lemma \ref{lem:Tr^(0)} admits the following generalization to arbitrary loop degree $r \geq 0$, yielding an intrinsic formulation
of the homomorphism $\Tr^{(r)}$.

\begin{lemma}  \label{lem:Tr^(r)} 
The map
$\Tr^{(r)}: \calP_k^{(r)}(H)  \to
\bqm\calA^c_{k,r+1}(H)\eqm \otimes \Z_{2^{r+1}}$
is given  by
\begin{equation} \label{eq:Tr^(r)}
\Tr^{(r)}\Big(\sum_{i=0}^r y_i\Big) = 
\left\lbrace f\left(\sum_{i =0}^r 2^{i} 
\frac{w^{r+1-i}}{(r+1-i)!}(\widetilde{z}_i)\right) \right\rbrace
\end{equation}
where $\widetilde z_0\in \bqm \widetilde{\calA}^c_{k,0}(H)\eqm$,
 $\widetilde z_1\in \bqm \widetilde{\calA}^c_{k,1}(H)\eqm$,
\dots, $\widetilde z_r \in \bqm \widetilde{\calA}^c_{k,r}(H)\eqm$
are chosen recursively so that we have
\begin{equation} \label{eq:y_i}
y_i= f\!\left( \sum_{j =0}^{i} 2^j\,\frac{w^{i-j}}{(i-j)!}(\widetilde z_j) \right) \in {\calA}^c_{k,i}(H), 
\end{equation}
for all $i\in \{0,1,\dots, r\}$.
Consequently, the homomorphism $\Tr^{(r)}$ is $\Sp(H)$-equivariant
(and it is independent of the choice $S$ of a basis of $H$).
\end{lemma}

\begin{example}
Setting $\Tr':=\Tr^{(1)}$ and $\calP'_k(H):=\calP^{(1)}_k(H)$,
we deduce from Lemma \ref{lem:Tr^(r)} that the map
$\Tr': \calP_k'(H)  \to
\bqm\calA^c_{k,2}(H)\eqm\otimes \Z_4$
is given  by
\begin{equation} \label{eq:Tr_4}
\Tr'\big(f\big(\tilde y + w(\tilde y)+2 \tilde z\big)\big)
:= \big\{  f \big(w^2(\tilde y)/2  + 2w(\tilde z)\big)\big\} 
\end{equation}
for any $\tilde y \in \bqm \widetilde{\calA}_{k,0}(H)\eqm$
and  $\tilde z \in \bqm \widetilde{\calA}_{k,1}(H)\eqm$.
\hfill $\blacksquare$    
\end{example}

\begin{proof}[Proof of Lemma \ref{lem:Tr^(r)}]
We shall prove formula \eqref{eq:Tr^(r)} 
by an induction on $r\geq 0$,
the case $r=0$ being provided by Lemma \ref{lem:Tr^(0)}.

To prove the induction step,
we consider an arbitrary  $y=\sum_{i=0}^r y_i \in \calP_k^{(r)}(H)$.
Thus, $\sum_{i=0}^{r-1} y_i$ belongs to $\calP_k^{(r-1)}(H)$
and $\Tr^{(r-1)}\big(\sum_{i=0}^{r-1} y_i\big)$ is the class of
$y_r$ modulo $2^r \bqm\calA^c_{k,r}(H)\eqm$.
By the induction hypothesis, we   chose
$\widetilde z_i \in\bqm  \widetilde{\calA}^c_{k,i}(H)\eqm$
satisfying \eqref{eq:y_i} for every $i\in \{0,\dots, r-1\}$.
Besides, 
by the induction hypothesis on \eqref{eq:Tr^(r)}, we  have 
$$
f\left(\sum_{i =0}^{r-1} 2^{i} 
\frac{w^{r-i}}{(r-i)!}(\widetilde{z}_i)\right)
\equiv y_r \mod 2^r \bqm\calA^c_{k,r}(H)\eqm,
$$
so that we can find a $z_r \in \bqm\calA^c_{k,r}(H)\eqm$
satisfying 
$$
y_r = f\left(\sum_{i =0}^{r-1} 2^{i} 
\frac{w^{r-i}}{(r-i)!}(\widetilde{z}_i)\right) +2^r z_r.
$$
Then, choosing any lift 
$\widetilde z_r \in \bqm \widetilde{\calA}^c_{k,r}(H)\eqm$
of $z_r$, 
we obtain \eqref{eq:y_i} for all  $i\in \{0,\dots, r\}$.

In order to prove \eqref{eq:Tr^(r)}
for $y=\sum_{i=0}^r y_i \in \calP_k^{(r)}(H)$,
 we now compute
\begin{align*} \operatorname{Tr}^{(r)}(y) &
\by{eq:def_Tr} \Bigg\{ \sum_{j=0}^{r} (-1)^j(2c)^{[j+1]}(y_{r-j})\Bigg\} \\ 
&=  \Bigg\{ f\!\left( \sum_{j=0}^{r} (-1)^j\frac{(2\tilde c)^{j+1}}{(j+1)!}  \left( \sum_{i=0}^{r-j} 2^i\,\frac{w^{r-j-i}(\widetilde z_i)} {(r-j-i)!} \right) \right)\Bigg\}.
\end{align*}
Here we use the endomorphism  $\tilde c$ 
of  $\widetilde{\calA}(H)\otimes \Q$ 
that is defined similarly to $c$ by
$$
\tilde c(D) :=
\left(\begin{array}{c}
\hbox{sum of all ways of $(\langle-,-\rangle/2)$-contracting}\\
\hbox{an ordered pair of external vertices in $D$}
\end{array}\right);
$$
note that we have $f\circ \tilde c= c \circ f$.
Since $w$ and $\tilde c$ commute, we next  obtain
\begin{align*} \operatorname{Tr}^{(r)}(y)
&=  \Bigg\{ f\!\left( \sum_{i=0}^{r} \frac{2^i}{(r+1-i)!} \sum_{j=0}^{r-i} (-1)^j {r+1-i\choose j+1}(2\tilde{c})^{j+1}  w^{r-i-j}(\widetilde z_i) \right) \Bigg\} \\ 
&=  \Bigg\{  f\!\left( \sum_{i=0}^{r} \frac{2^i}{(r+1-i)!} \Big( w^{r+1-i} - (w-2\tilde c)^{r+1-i} \Big) (\widetilde z_i) \right)\Bigg\}.
\end{align*} 
Setting $\delta(x):=w(x)/2  - \tilde c(x) 
\in \bqm \widetilde{\calA}_{k}(H)\eqm$
for every $x\in \bqm \widetilde{\calA}_{k}(H)\eqm$,
we obtain
$$
 \operatorname{Tr}^{(r)}(y) 
 =  \Bigg\{  f\!\left( \sum_{i=0}^{r} \frac{2^i}{(r+1-i)!} \Big( w^{r+1-i} - (2\delta)^{r+1-i} \Big) (\widetilde z_i) \right)\Bigg\}.
$$
Similarly to $\tilde c$ and $w$,
the operation $\delta$ on Jacobi diagrams consists
in taking the sum  of all ways of contracting an ordered pair of vertices with a given bilinear form, 
namely the form $(\omega(-,-)-\langle-,-\rangle)/2$.
Since this bilinear form takes values in $\Z$, 
we observe that 
$$
\frac{\delta^{r+1-i}}{(r+1-i)!}(\widetilde z_i) 
\in \bqm\widetilde{\calA}_{k}^c(H)\eqm,
$$ 
and we conclude that  
$$
\operatorname{Tr}^{(r)}(y)
= \Bigg\{ f\!\left( \sum_{i=0}^{r} 2^i\frac{ w^{r+1-i} }{(r+1-i)!} (\widetilde z_i) \right) \Bigg\}
\in \bqm\calA^c_{k,r+1}(H)\eqm \otimes \Z_{2^{r+1}}.
$$

Finally, we verify the $\Sp(H)$-equivariance of $\operatorname{Tr}^{(r)}$ by considering an $u\in \Sp(H)$
acting on $y=\sum_{i=0}^r y_i \in \calP_k^{(r)}(H)$.
Choose some $\widetilde{z}_0, \dots, \widetilde{z}_r$
satisfying \eqref{eq:y_i}.
Letting $y'_i:= u\cdot y_i$ and 
$\widetilde{z}'_i:= u\cdot \widetilde{z}_i$, 
we deduce from the $\Sp(H)$-invariance of $\omega$
that $\widetilde{z}'_0, \dots, \widetilde{z}'_r$
satisfy \eqref{eq:y_i} with respect to $y'_0,\dots,y'_r$.
Hence, by the formula \eqref{eq:Tr^(r)}, we get
\begin{eqnarray*}
\Tr^{(r)}(u\cdot y) &=& 
\left\lbrace f\left(\sum_{i =0}^r 2^{i} 
\frac{w^{r+1-i}}{(r+1-i)!}(\widetilde{z}'_i)\right) \right\rbrace\\
&=& u \cdot \left\lbrace f\left(\sum_{i =0}^r 2^{i} 
\frac{w^{r+1-i}}{(r+1-i)!}(\widetilde{z}_i)\right) \right\rbrace
\ = \ u \cdot \Tr^{(r)}(y).
\end{eqnarray*}

\up
\end{proof}

We now  reformulate (in the connected case) 
the characterization of $\calR_k(H)$ in $\calA_k(H)\otimes \Q$
given by Lemma~\ref{lem:X_def_rel}.

\begin{proposition} \label{prop:P2}
An $x \in \calA^c_k(H)\otimes \Q$,
which is decomposed as $x= \sum_{j\geq 0} x_{j}$
with $x_j \in \calA^c_{k,j}(H)\otimes \Q$,
belongs to the lattice  $\calR_k(H)$ if, and only if, we have
$$
\sum_{j=0}^r 2^j x_j \in  \calP_k^{(r)}(H)
\quad \hbox{for every $r\geq 0$}.
$$
\end{proposition}

\begin{proof}
By Lemma \ref{lem:X_def_rel}, 
an $x = \sum_{j\geq 0} x_{j} \in \calA_k^c(H)\otimes \Q$
belongs to $\calR_k(H)$ if, and only if, we have
$$
x_{i} \equiv \sum_{j=1}^i(-1)^{j+1} \cdot 
c^{[j]}\left(x_{i-j}\right) 
\mod  ``\calA_{k,i}(H)"
\quad \hbox{for every $i\geq 0$}
$$
or, equivalently, 
$$
2^i x_{i} \equiv \sum_{j=0}^{i-1}(-1)^{j} \cdot (2c)^{[j+1]}\left(2^{i-1-j} x_{i-1-j}\right) 
\mod  \ 2^i\,``\calA_{k,i}(H)"
\quad \hbox{for every $i\geq 0$}.
$$
Note that, whenever $\sum_{j= 0}^{i-1} 2^jx_{j}$ belongs to $\calP_k^{(i-1)}(H)$,
the right-hand side of the above congruence is exactly 
$\Tr^{(i-1)}(\sum_{j=0}^{i-1} 2^j x_{j})$: 
so, this congruence is equivalent to the condition 
$\sum_{j=0}^{i} 2^j x_{j}\in \calP_k^{(i)}(H)$.
Thus, the proposition is proved by an induction on $r\geq 0$.
\end{proof}

\subsection{Congruence relations}

Assume that $k\geq 2$.
The description of $\calR_k(H)$
given by Proposition~\ref{prop:P2} implies
some congruence relations between the loop-degree pieces~$Z_{k,r}$ 
of the symmetrized LMO homomorphism.

\begin{proposition} \label{prop:ZZ}
For every $r\geq 0$,
we have a surjective  homomorphism
\begin{equation} \label{eq:2^jZkj}
\bigoplus_{j=0}^r 2^j Z_{k,j}: 
\frac{Y_k\calIC}{Y_{k+1}}
\longrightarrow \calP_k^{(r)}(H)    
\end{equation}
which, for $r\geq 1$, decomposes 
into the following diagram of $\Sp(H)$-modules:
\begin{equation} \label{eq:ZZ}
\begin{tikzcd}
Y_k\calIC / Y_{k+1}
\arrow[rr,"{2^{r}Z_{k,r}}"] 
\arrow[d,"{\bigoplus_{j=0}^{r-1} 2^j Z_{k,j}}" left] 
& &\bqm \calA^c_{k,r}(H)\eqm  \arrow[d,"\mod 2^{r}"] \\
\calP_k^{(r-1)}(H) \arrow[rr,"\Tr^{(r-1)}"]&&
\bqm \calA^c_{k,r}(H)\eqm/{2^r}\bqm\calA^c_{k,r}(H)\eqm. 
\end{tikzcd}   
\end{equation}
\end{proposition}

\begin{proof}
Let $M\in Y_k \calIC$. Then, the element
$x:=Z_k(M) \in \calA^c_k(H) \otimes \Q$ satisfies
$$
x = X^{-1} Z_k^<(M) 
\in X^{-1}(\bqm \calA^{<}_k(H) \eqm) = \calR_k(H).
$$
Hence, by Proposition \ref{prop:P2},
the sum $\sum_{j=0}^{r} 2^j x_j $
belongs to $\calP_k^{(r)}(H)$: this proves that
the homomorphism \eqref{eq:2^jZkj} is well-defined.
To prove its surjectivity,
let  $y=\sum_{j=0}^r y_j \in \calP_k^{(r)}(H)$:
we can find by induction  some elements
$$
y_{r+1} \in \bqm \calA^c_{k,r+1}(H) \eqm, \quad 
y_{r+2} \in \bqm \calA^c_{k,r+2}(H)\eqm, \quad \hbox{etc }\dots
$$
such that $\sum_{j= 0}^s y_j$ belongs 
to $\calP_k^{(s)}(H)$ for every $s\geq 0$.
It follows, by  Proposition~\ref{prop:P2}, that 
$z:=\sum_{j\geq 0} y_j/2^j\in \calA^c_k(H)\otimes \Q$ 
belongs to $\calR_k(H)$.
Recall that, as a consequence of~\eqref{eq:psi_Z},
the homomorphism 
$Z_k^<: Y_k\calIC \to \bqm \calA^{<,c}_k(H) \eqm$
is surjective. Thus,  there exists an
$N\in Y_k\calIC$ such that $Z_k(N) = z$.
In particular, we get
$$
\Big(\bigoplus_{j=0}^r 2^j Z_{k,j}\Big)(N)=
\sum_{j=0}^r 2^j z_j= y.
$$

The well-definedness of the map \eqref{eq:2^jZkj}
at the next level $(r+1)$ gives the commutativity
of the diagram \eqref{eq:ZZ}.
Since $Z^<$ is $\Sp(H)$-equivariant at the graded level,
so is $Z=X^{-1} Z^<$. Hence, the top arrow 
and the left arrow in  \eqref{eq:ZZ} 
are $\Sp(H)$-equivariant.
The  $\Sp(H)$-equivariance of $\Tr^{(r-1)}$
is given by Lemma \ref{lem:Tr^(r)}.
\end{proof}

It is interesting to 
specialize Proposition \ref{prop:ZZ}
to the loop-degree cases $r:=1$ and $r:=2$,
for which we know topological interpretations of $Z_{k,r-1}$.

\begin{corollary} \label{cor:0_1}
We have the following diagram of $\Sp(H)$-modules:
\begin{equation} \label{eq:0_1}
\begin{tikzcd}
Y_k\calIC / Y_{k+1}
\arrow[rr,"p_+ \circ \tilde{\alpha}_k"] 
\arrow[d,"\tau_k" left] 
& &\bqm \calA^c_{k,1}(H)\eqm  \arrow[d,"\mod 2"] \\
\bqm \calA^c_{k,0}(H)\eqm \arrow[rr,"\Tr"]&&
\bqm \calA^c_{k,1}(H)\eqm/2\bqm\calA^c_{k,1}(H)\eqm . 
\end{tikzcd}    
\end{equation}
Therefore, the map $\Tr$ vanishes on $\tau_k(\Gamma_k \calI)$.
\end{corollary}

\begin{proof}
The diagram \eqref{eq:0_1} is the specialization of 
\eqref{eq:ZZ} at $r:=1$, since we have
$$
Z_{k,0}= \tau_k
\quad \hbox{and} \quad
2Z_{k,1}=-p_+ \circ \tilde\alpha_k
$$
by \eqref{eq:tau_Z} and \eqref{eq:alpha_Z}, respectively.
The second statement follows 
from the commutativity of \eqref{eq:0_1},
and the fact that $\tilde \alpha(\calI)$ 
is contained in the natural image of $H$
in $K_1\big(\widehat{\Q[\pi]}\big)$,
as explained in \cite[Remark~3.2]{NSS23}.
\end{proof}

\begin{example} \label{ex:Tr}
In degree $k:=2$, Corollary \ref{cor:0_1} says that
$$
\Tr(\tau_2(M))= \big( \alpha(M) \!\!\!\mod 2\big) \ \in
\calA^c_{2,1}(H) \otimes \Z_2
$$
for every $M\in Y_2\calIC$.
Here $\alpha(M)\in S^2(H) \simeq \calA^c_{2,1}(H) $
is the quadratic part of 
the relative Alexander polynomial 
considered in \cite{MM13},
which is $\tilde{\alpha}_2=p_+\circ \tilde{\alpha}_2$
in the present notations. 
(See \cite[Th$.$ 4.7]{FMS26} in this connection.)
\hfill $\blacksquare$
\end{example}

\begin{corollary} \label{cor:0_1_2}
We have the following diagram of $\Sp(H)$-modules:
\begin{equation} \label{eq:0_1_2}
\begin{tikzcd}
Y_k\calIC / Y_{k+1}
\arrow[rr,"{4Z_{k,2}}"] 
\arrow[d,"{\left(\tau_{k},-p_+\tilde \alpha_k\right)}" left] 
& &\bqm \calA^c_{k,2}(H)\eqm  \arrow[d,"\mod 4"] \\
\calP_k'(H) \arrow[rr,"\Tr'"]&&
\bqm \calA^c_{k,2}(H)\eqm/4\bqm\calA^c_{k,2}(H)\eqm . 
\end{tikzcd}    
\end{equation}
\end{corollary} 

\begin{proof}
The diagram \eqref{eq:0_1_2} 
is the specialization of  \eqref{eq:ZZ} at $r:=2$.
\end{proof}

\begin{example} \label{ex:Tr'}
In degree $k:=2$, Corollary \ref{cor:0_1_2} says that
$$
\big(4d''(M)\!\!\!\mod 4\big) \cdot \thetagraph
= \Tr'\big(\tau_2(M),-\alpha(M)\big) 
\in \calA^c_{2,2}(H) \otimes \Z_4
$$
for every $M\in Y_2\calIC$.
Here $d''(M)\in (1/4)\Z$ is the invariant 
defined in \cite{MM13}
as the coefficient of the theta graph in $Z_{2,2}(M)$:
it is an  alternative to Morita's
``core of the Casson invariant''.
(See also \cite[Th$.$ 4.7]{FMS26} in this connection.)
\hfill $\blacksquare$
\end{example}

However, it should be recalled 
that there are currently no known topological 
interpretations of the loop-degree $r$ 
part $Z_{k,r}$  for $r\geq 2$.

\subsection{Back to the modulo $1$ reduction of $Z^<$}
\label{subsec:back}

We now focus on the module 
$Y_k \calIC/Y_{k+1}$ in odd degree $k=2n+1$ 
(with $n\geq 0$).
To simplify our notation, 
its $2$-torsion module is denoted by
$$
\calT:= \Tors_2\Big( \frac{Y_{2n+1} \calIC}{Y_{2n+2}}\Big).
$$
We endow $\calT$ with the filtration
\begin{equation} \label{eq:filt_T}
\calT = L_0\calT \supset L_1\calT \supset L_2\calT
\supset L_3\calT \supset \cdots
\end{equation}
that is the pull-back by  $\zeta^<_{2n+2}$
of the loop filtration 
on $\calA_{2n+2}^{<,c}(H)\otimes \Q/\Z$,
which has been defined  in  \S \ref{subsec:loop}.
Thus, for every $r\geq 0$, we obtain an injective map
\begin{equation} \label{eq:Gr_L}
\Gr^L_r \zeta^<_{2n+2}: 
\Gr^L_r \calT \longrightarrow 
\Gr^L_r \calA_{2n+2}^{<,c}(H)\otimes \Q/\Z
\simeq \calA^c_{2n+2,r}(H) \otimes \Q/\Z.
\end{equation}
The next two propositions compute
\eqref{eq:Gr_L} in terms of 
the loop-degree pieces of the symmetrized LMO homomorphism.

We start with the case $r:=0$, 
which deserves a special treatment.
Indeed, 
the corresponding homomorphism on $\calT=L_0\calT$
appears in \cite{CST16}
for the study of homology cylinders 
up to homology cobordisms.

\begin{proposition} \label{prop:R_2n+2}
There is a group homomorphism
$$
R_{2n+2}: \operatorname{Tors}_2(Y_{2n+1} \calIC /Y_{2n+2}) 
\longrightarrow \frakL_{n+2}\big(H^{(2)}\big), \
\{ M \} \longmapsto  \{\tau_{2n+2}(M)\},
$$
whose target is identified to 
$\mathsf{D}_{2n+2}(H)/\mathsf{D}'_{2n+2}(H)$
via the short exact sequence~\eqref{eq:Levine2}.
Furthermore, we have the following
commutative diagram of $\Sp(H)$-modules:
\begin{equation} \label{eq:R_2n+2}
\xymatrix{
\Tors_2(Y_{2n+1} \calIC /Y_{2n+2}) \ar[r]^-{R_{2n+2}}  
\ar[d]_-{\zeta_{2n+2}^<} 
& \frakL_{n+2}\big(H^{(2)}\big) 
\ar[d]^-{\operatorname{Lev}}  \\
 {\displaystyle\frac{L_0 \calA^{<,c}_{2n+2}(H) \otimes \Q/\Z }{L_1 \calA^{<,c}_{2n+2}(H) \otimes \Q/\Z }}
&\ar[l]^-\simeq  \calA^{c}_{2n+2,0}(H)  \otimes \Q/\Z 
}
\end{equation}
\end{proposition}
    
\begin{proof}
Let $\{M\}\in \Tors_2(Y_{2n+1} \calIC /Y_{2n+2})$.
Since $M^2$ is $Y_{2n+2}$-equivalent to the unit cylinder $U$, 
we have 
\begin{equation} \label{eq:k-th_tau}
\tau_{2n+1}(M)= \frac{1}{2} \tau_{2n+1}(M^2) =0 \in \calA^{c}_{2n+1,0}(H) \otimes \Q    
\end{equation}
and we are allowed to consider $\tau_{2n+2}(M)$.
Furthermore, for any $M'$ in the $Y_{2n+2}$-equivalence class of $M$,
there exists an $N\in Y_{2n+2}\calIC$ 
such that $M'$ is $Y_{2n+3}$-equivalent to $MN$: 
hence 
$$
\tau_{2n+2}(M')= \tau_{2n+2}(MN)= \tau_{2n+2}(M)+ \tau_{2n+2}(N) 
\equiv \tau_{2n+2}(M) \!\!\! \mod \mathsf{D}'_{2n+2}(H).
$$
This proves the well-definedness of the map $R_{2n+2}$.

We  now prove the commutativity of   \eqref{eq:R_2n+2}.
Observe that $Z(M^2)=Z(M)\star  Z(M)$ 
by Remark~\ref{rem:star_product};
since $Z(M)$ starts in degree ${2n+1}$
(with $Z_1(M)=0$ in any case), we deduce that
\begin{equation} \label{eq:k-th_Z}
Z_{2n+2,i}(M) = \frac{1}{2} Z_{2n+2,i}(M^2)
\quad \hbox{for every $i\geq 0$}.
\end{equation}
Thanks to  \eqref{eq:k-th_Z} for $i:=0$,
we conclude by applying \eqref{eq:tau_Z} to $M^2$.
\end{proof}

For the sequel, 
it is  convenient to set $R_{2n+2}^{(0)}:=
\operatorname{Lev} \circ R_{2n+2}$, 
viewing this homomorphism 
$$
R_{2n+2}^{(0)} : L_0\calT \longrightarrow
\frac{{2}^{-1} \calA^c_{2n+2,0}(H)}{\calA^c_{2n+2,0}(H)} 
$$
as the initial stage of the following sequence.

\begin{proposition} \label{prop:R^(r)}
For every $r\geq 0$,
there is a  homomorphism
$$
R^{(r)}_{2n+2}: L_r \calT \longrightarrow 
\frac{2^{r-1}\bqm \calA^c_{2n+2,r}(H)\eqm}{2^r 
\bqm \calA^c_{2n+2,r}(H) \eqm}    
$$
with kernel $L_{r+1} \calT$, 
which is defined by the formula
\begin{eqnarray}
\notag R^{(r)}_{2n+2}\big(\{M\}\big) &:=&
-\Tr^{(r-1)}\big(Z_{2n+2,0}(M),
 \dots, 2^{r-1}Z_{2n+2,r-1}(M) \big)\\
\label{eq:R^(r)} && +\big\{ 2^r Z_{2n+2,r}(M)\big\},    
\end{eqnarray}
Furthermore, it fits into 
the following commutative diagram of $\Sp(H)$-modules:
\begin{equation} \label{eq:R(r)_zeta}
\xymatrix{
L_r\calT \ar[r]^-{R^{(r)}_{2n+2}}  \ar[d]_-{\zeta_{2n+2}^<} 
&  {\displaystyle\frac{2^{r-1}\bqm \calA^c_{2n+2,r}(H)\eqm}{2^r \bqm \calA^c_{2n+2,r}(H) \eqm} }   
\ar[d]^-{\frac{1}{2^r}\cdot}  \\
 {\displaystyle\frac{L_r \calA^{<,c}_{2n+2}(H) \otimes \Q/\Z }{L_{r+1} \calA^{<,c}_{2n+2}(H) \otimes \Q/\Z }}
&\ar[l]^-\simeq  
\calA^{c}_{2n+2,r}(H)  \otimes \Q/\Z 
}
\end{equation}
\end{proposition}

\begin{proof}
We start by proving, thanks to an induction on $r\geq 0$,
that the map $R^{(r)}_{2n+2}$ is well-defined
with values in  the module
${\bqm \calA^c_{2n+2,r}(H)\eqm}/
{2^r \bqm \calA^c_{2n+2,r}(H) \eqm}$.
Let $M\in Y_{2n+1} \calIC$ 
whose $Y_{2n+2}$-class $\{M\}$ 
belongs to $L_r\calT$:
we have to verify that the quantity \eqref{eq:R^(r)}
makes sense in this module.

First of all, we prove that the first term 
in  \eqref{eq:R^(r)} is well-defined.
Since $\{M\}$ belongs to $L_{r-1}\calT$, 
we know from the induction hypothesis and the definition of the trace homomorphisms that
\begin{itemize}
\item[(i)] $2^j Z_{2n+2,j}(M)\in \bqm \calA^c_{2n+2,j}(H)\eqm$
for all $j=0,\dots r-1$,
\item[(ii)] $\sum_{j=0}^{r-2} 2^j Z_{2n+2,j}(M) 
\in \calP_{2n+2}^{(r-2)}(H)$
\end{itemize}
and, since we have $\{M\}\in  L_{r}\calT = \Ker R^{(r-1)}_{2n+2}$, 
we also know that $2^{r-1} Z_{2n+2,r-1}(M)$ represents
the class 
$$
\Tr^{(r-2)}\big(Z_{2n+2,0}(M),
 \dots, 2^{r-2}Z_{2n+2,r-2}(M) \big)
\in 
 \bqm \calA^c_{2n+2,r-1}(H)\eqm \otimes \Z_{2^{r-1}}.
$$
In other words, we have 
$\sum_{j=0}^{r-1} 2^j Z_{2n+2,j}(M)
\in  \calP_{2n+2}^{(r-1)}(H)$, so that the term
$$
\Tr^{(r-1)}\big(Z_{2n+2,0}(M)
, \dots, 2^{r-1}Z_{2n+2,r-1}(M) \big)
\in 
 \bqm \calA^c_{2n+2,r}(H)\eqm \otimes \Z_{2^{r}}
$$ 
in \eqref{eq:R^(r)} is well-defined.

We now check that the second term  
in \eqref{eq:R^(r)} is well-defined.
By applying Proposition \ref{prop:ZZ}
to $M^2 \in Y_{2n+2} \calIC$,
we obtain that 
\begin{eqnarray*}
&& \Tr^{(r-1)}\big(Z_{2n+2,0}(M^2),
 \dots, 2^{r-1}Z_{2n+2,r-1}(M^2) \big)\\
&\equiv & 2^r Z_{2n+2,r}(M^2) \mod 2^r \bqm \calA^{c}_{2n+2,r}(H) \eqm.
\end{eqnarray*}
But, \eqref{eq:k-th_Z} and the fact (i) above imply 
 that $2^jZ_{2n+2,j}(M^2)\in 2\bqm \calA^c_{2n+2,j}(H)\eqm$
for every $j=0,\dots, r-1$.
So we have 
$2^r Z_{2n+2,r}(M^2) \in 
2 \bqm \calA^{c}_{2n+2,r}(H) \eqm$,
i.e. $2^r Z_{2n+2,r}(M) 
\in \bqm \calA^{c}_{2n+2,r}(H) \eqm$.
Consequently, the congruence class 
$\big\{ 2^r Z_{2n+2,r}(M)\big\}$ in
$\bqm \calA^{c}_{2n+2,r}(H)\eqm \otimes \Z_{2^r}$
in \eqref{eq:R^(r)} is well-defined.

Furthermore, 
Proposition \ref{prop:ZZ} also implies that
the quantity \eqref{eq:R^(r)}
assigned to $M$ does not change
under $Y_{2n+2}$-surgery:
so  \eqref{eq:R^(r)}
only depends on $\{M\} \in L_r \calT$.
Therefore, the map $R^{(r)}_{2n+2}:L_r\calT \to
{\bqm \calA^c_{2n+2,r}(H)\eqm}\otimes \Z_{2^r}$ 
is well-defined.

We now prove the commutativity of 
\eqref{eq:R(r)_zeta}. Let $\{M\} \in L_r \calT$ as before.
Set $y_j := 2^jZ_{2n+2,j}(M)$ for all $j=0,\dots,r-1$.
By the above discussion, we know that 
$\sum_{j=0}^{r-1} y_j\in \calP_{2n+2}^{(r-1)}(H)$;
hence, we can find some
$y_j \in  \bqm \calA^c_{2n+2,j}(H)\eqm$
for every $j\geq r$ such that
 $\sum_{j=0}^{s} y_j\in \calP_{2n+2}^{(s)}(H)$
for every $s\geq 0$.
Then, Proposition \ref{prop:P2} implies 
 that $\sum_{j\geq 0} y_j/2^j$ 
belongs to $\calR_{2n+2}(H)$.
Therefore, 
denoting by $\big(L_k\calA^c(H)\big)_{k\geq 0}$ the loop filtration on $\calA^c(H)$,
we get
\begin{eqnarray*}
Z_{2n+2}(M) &\equiv & \sum_{j=0}^{r-1} y_j/2^j
+Z_{2n+2,r}(M)
\mod L_{r+1} \calA^c_{2n+2}(H)\otimes \Q \\
&\equiv & \frac{-y_r +2^rZ_{2n+2,r}(M)}{2^r}
\mod \big( L_{r+1} \calA^c_{2n+2}(H)\otimes \Q
+ \calR_{2n+2}(H) \big).
\end{eqnarray*}
It follows that \eqref{eq:R(r)_zeta}
commutes for $\{M\}  \in L_r \calT$ since
$\zeta_{2n+2}^<=\big(X\circ Z_{2n+2}\bmod 1\big)$ and $y_r\in  \bqm \calA^c_{2n+2,r}(H)\eqm$
has been chosen so that its class modulo
$2^r\bqm \calA^c_{2n+2,r}(H)\eqm$ represents 
$\Tr^{(r-1)}(y_0,\dots,y_{r-1})$.

The  commutativity of 
\eqref{eq:R(r)_zeta} implies that 
$R^{(r)}_{2n+2}$ is a homomorphism 
with kernel  $L_{r+1} \calT$.
Finally, since $L_{r}\calT$ is a $2$-torsion group,
$R^{(r)}_{2n+2}$ takes values in the $2$-torsion part
of 
$\bqm \calA^c_{2n+2,r}(H) \eqm\otimes \Z_{2r}$.
\end{proof}

Proposition \ref{prop:R^(r)} admits the following specializations for $r:=1$ and $r:=2$, thereby yielding (partial) topological interpretations of the homomorphism \eqref{eq:Gr_L} in low loop degrees.

\begin{corollary} \label{cor:R'_2n+2}
The  homomorphism $R'_{2n+2}:= R^{(1)}_{2n+2}$
is given by 
$$
R'_{2n+2}(\{M\}) = \Tr(\tau_{2n+2}(M)) +
\big\{ p_+\tilde\alpha_{2n+2}(M) \big\},
$$
and it fits into the following
commutative diagram of $\Sp(H)$-modules:
\begin{equation} \label{eq:R'_2n+2}
\xymatrix{
\Ker(R_{2n+2}) \ar[r]^-{R'_{2n+2}} 
\ar[d]_-{\zeta_{2n+2}^<} 
& \calA^c_{2n+2,1}(H)\otimes \Z_2 
\ar[d]^-{\frac{1}{2}\cdot}\\
 {\displaystyle\frac{L_1 \calA^{<,c}_{2n+2}(H) \otimes \Q/\Z }{L_2 \calA^{<,c}_{2n+2}(H) \otimes \Q/\Z }}
&\ar[l]^-\simeq  \calA^{c}_{2n+2,1}(H)  \otimes \Q/\Z 
}
\end{equation}
\end{corollary} 

\begin{corollary}  \label{cor:R''_2n+2}
The  homomorphism $R''_{2n+2}:= R^{(2)}_{2n+2}$
is given by     
$$
R''_{2n+2}\big(\{M\}\big) :=
\Tr'\big(-\tau_{2n+2}(M),
p_+\tilde\alpha_{2n+2}(M)\big)
+\big\{ 4 Z_{2n+2,2}(M)\big\},
$$
and it fits into the following
commutative diagram of $\Sp(H)$-modules:
\begin{equation} \label{eq:R''_2n+2}
\xymatrix{
\Ker(R'_{2n+2}) \ar[r]^-{R''_{2n+2}}  \ar[d]_-{\zeta_{2n+2}^<} 
& 2\bqm \calA^c_{2n+2,2}(H) \eqm 
\otimes \Z_4 \ar[d]^-{\frac{1}{4}\cdot }  \\
 {\displaystyle\frac{L_2 \calA^{<,c}_{2n+2}(H) \otimes \Q/\Z }{L_3 \calA^{<,c}_{2n+2}(H) \otimes \Q/\Z }}
&\ar[l]^-\simeq  \ \calA^{c}_{2n+2,2}(H)   \otimes \Q/\Z 
}
\end{equation}
\end{corollary}

\begin{remark}
Unlike Corollary \ref{cor:R'_2n+2},
Corollary \ref{cor:R''_2n+2} is limited by our lack of information (for arbitrary $n$)
on the structure of the module $\calA^{c}_{2n+2,2}(H)$
and on the topological meaning of $Z_{2n+2,2}$.
\hfill $\blacksquare$
\end{remark}

 We end this section by determining the variation  of $R_{2n+2}^{(r)}$ under appropriate  clasper surgery. 
 The statement is  a mere consequence of Theorem \ref{th:Delta_zeta}; it involves the map $\Delta$, introduced in \S \ref{subsec:lower_bound}, 
whose action on the associated graded of the loop filtration 
is given by  Proposition \ref{prop:Gr_Delta}.

\begin{corollary} 
For every $n>0$ and $r\geq 0$, 
we have the following commutative triangle of $\Sp(H)$-modules:
\begin{equation} \label{eq:Gr_Delta_zeta}
\xymatrix{
 \calB_{n,r}\big(H^{(2)};s\big)
 \ar[rrd]^-{\Gr^L \!\Delta} \ar[d]_-{\Gr \Psi} 
 &&  \\
L_{r} \calT / L_{r+1} \calT \ar[rr]^-{\frac{1}{2^{r}} R_{2n+2}^{(r)}} 
&& \calA^{c}_{2n+2,r}(H)   \otimes \Q/\Z 
}    
\end{equation}
\end{corollary}

\begin{proof}
By commutativity of the diagram \eqref{eq:Delta_zeta} 
and by definition 
of the filtration~$\big(L_k \calT\big)_{k \geq 0}$,
the map $\Psi:\calB^<_n\big(P^{(2)}\big) \to \calT$
is filtration-preserving because $\Delta$~is.
Then, we conclude with Proposition \ref{prop:R^(r)}.
\end{proof}

\section{The degree-one case} 
\label{sec:deg_1}

We illustrate our methodology for studying the torsion of 
$Y_{2n+1}\mathcal{IC} / Y_{2n+2}$ by considering firstly the special case $n = 0$. 
Although the $\Sp(H)$-module $\mathcal{IC} / Y_{2}$ was fully computed in \cite{MM03} 
in connection with Johnson’s computation of the abelianized Torelli group \cite{Joh85}, 
we will show how the LMO functor can be used 
as an alternative 
to the classical invariants that are considered 
in \cite{MM03}.

\subsection{Description of the $\Sp(H)$-module $\mathcal{IC} / Y_{2}$}

We  firstly determine the $\Sp(H)$-module 
$\Tors(\mathcal{IC} / Y_{2})$.
According to the slide relation,
the homomorphism from $P^{(2)} \otimes P^{(2)} =
\calB^<_0\big(P^{(2)}\big)$
to $\Tors\big(\calA_1(P)\big)$
that was considered in  \eqref{eq:E_0} is given by
$$
x \otimes y \longmapsto
\begin{tikzpicture}[thick,color=black,
baseline=-0.5em, x=0.6em, y=0.6em]
\footnotesize
\node[left=1pt] at (-1,1) {\scriptsize $y$};
\node[right=1pt] at (1,1) {\scriptsize $y$};
\node[below=1pt] at (0,-1.5) {\scriptsize $x$};
\draw (-1,1) -- (0,0) -- (1,1);
\draw (0,0) -- (0,-1.5);
\draw [fill = black] (0,0) circle (.3ex);
\end{tikzpicture} =
\begin{tikzpicture}[thick,color=black,
baseline=-0.5em, x=0.6em, y=0.6em]
\footnotesize
\node[left=1pt] at (-1,1) {\scriptsize $y$};
\node[right=1pt] at (1,1) {\scriptsize $s$};
\node[below=1pt] at (0,-1.5) {\scriptsize $x$};
\draw (-1,1) -- (0,0) -- (1,1);
\draw (0,0) -- (0,-1.5);
\draw [fill = black] (0,0) circle (.3ex);
\end{tikzpicture} =
\begin{tikzpicture}[thick,color=black,
baseline=-0.5em, x=0.6em, y=0.6em]
\footnotesize
\node[left=1pt] at (-1,1) {\scriptsize $y$};
\node[right=1pt] at (1,1) {\scriptsize $x$};
\node[below=1pt] at (0,-1.5) {\scriptsize $x$};
\draw (-1,1) -- (0,0) -- (1,1);
\draw (0,0) -- (0,-1.5);
\draw [fill = black] (0,0) circle (.3ex);
\end{tikzpicture}
\quad \hbox{(for all $x,y\in P^{(2)}$)}.
$$
(Recall that we are using here the convention \eqref{eq:lifts}.) 
Hence, it factorizes to a homomorphism $E: S^2 P^{(2)}\to \Tors\big(\calA_1(P)\big)$, 
and so does the surgery map $\Psi$.
Furthermore, $E$ (and so $\Psi$)
factorizes through $S^2 P^{(2)}/H^{(2)}$, 
where $H^{(2)}$ is embedded into $S^2 P^{(2)}$
via the map $\big(p_*(x)\mapsto x \cdot x - x \cdot s\big)$.
Hence, the commutative diagram \eqref{eq:E_0} can be rewritten as 
\begin{equation} \label{eq:E_0_bis}
\xymatrix{
\{x \cdot y\} \ \ni \quad & S^2 P^{(2)}/H^{(2)} \ar[r]^-\Psi \ar[d]_-E &
\Tors\Big({\displaystyle\frac{\calIC}{Y_2}}\Big), \\
\begin{tikzpicture}[thick,color=black,
baseline=-0.5em, x=0.6em, y=0.6em]
\footnotesize
\node[left=1pt] at (-1,1) {\scriptsize $x$};
\node[right=1pt] at (1,1) {\scriptsize $y$};
\node[below=1pt] at (0,-1.5) {\scriptsize $s$};
\draw (-1,1) -- (0,0) -- (1,1);
\draw (0,0) -- (0,-1.5);
\draw [fill = black] (0,0) circle (.3ex);
\end{tikzpicture} \ni \quad  & \Tors\big(\calA_1(P)\big) \ar[ru]_-\psi &  
}    
\end{equation}
and the surgery map $\Psi$ is given by \\[-0.3cm]
$$
\Psi(\{x\cdot  y\}) = \Bigg(\hbox{$U$ surgered along }
\begin{array}{c}\includegraphics[scale=0.7]{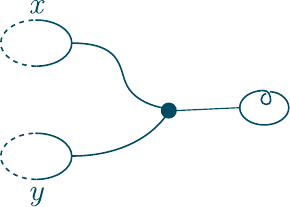}
\end{array}\Bigg) 
$$
where the $Y$-graph has one special leaf
and two arbitrary leaves with modulo 2 types 
$x,y \in P^{(2)}$.

We now define a homomorphism 
$\Delta: S^2 P^{(2)}/H^{(2)} \to \calA_2^{<,c}(H)\otimes \Q/\Z$
by associating to every $x\cdot y \in  S^2 P^{(2)}$ 
the linear combination
\begin{eqnarray}
\notag \Delta(x\cdot y) &:=& 
\frac{1}{2}\tikz[baseline=5pt]{
  \draw[thick, color=black] (0,0)--(0.5,0.5);
  \draw[thick, color=black] (0.5,0)--(1,0.5);
  \fill [white] (0.75,0.25) circle [radius=0.1];
  \draw[thick, color=black] (0.5,0.5)--(1,0);
  \draw[thick, color=black] (1,0.5)--(1.5,0);
  \draw[black, fill=black] (0.5,0.5) circle (0.25ex);
   \draw[black, fill=black] (1,0.5) circle (0.25ex);
  \draw[thick, color=black] (0.5,0.5) to[out=90, in=90] (1,0.5);
  \draw[color=black] (0.75,-0.25) node {\small $x< x  <y<y$}
}+\frac{1}{2} \tikz[baseline=-8pt, scale=0.5]{
  \draw[thick, color=black] (0,0) circle (0.5);

  \draw[black, fill=black] (-0.5,0) circle (0.5ex);
  \draw[black, fill=black] (0.5,0) circle (0.5ex);

  \draw[thick, color=black] (-0.5,0) to[out=180, in=90] (-1.2,-1);
  
  \draw[color=black] (-1.3,-1.4) node {\small $x$};
  
  \draw[thick, color=black] (0.5,0) to[out=0, in=90] (1.2,-1);
  \draw[color=black] (1.2,-1.4) node {\small $y$};
  \draw[color=black] (-0.15,-1.4) node {\small $< $}
} \\
\label{eq:Delta_0} 
&& + \frac{\fr({x})}{2} \tikz[baseline=-8pt, scale=0.5]{
  \draw[thick, color=black] (0,0) circle (0.5);

  \draw[black, fill=black] (-0.5,0) circle (0.5ex);
  \draw[black, fill=black] (0.5,0) circle (0.5ex);

  \draw[thick, color=black] (-0.5,0) to[out=180, in=90] (-1.2,-1);
  
  \draw[color=black] (-1.3,-1.4) node {\small $y$};
  
  \draw[thick, color=black] (0.5,0) to[out=0, in=90] (1.2,-1);
  \draw[color=black] (1.2,-1.4) node {\small $y$};
  \draw[color=black] (-0.15,-1.4) node {\small $< $}
} +\frac{\fr({y })}{2} \tikz[baseline=-8pt, scale=0.5]{
  \draw[thick, color=black] (0,0) circle (0.5);

  \draw[black, fill=black] (-0.5,0) circle (0.5ex);
  \draw[black, fill=black] (0.5,0) circle (0.5ex);

  \draw[thick, color=black] (-0.5,0) to[out=180, in=90] (-1.2,-1);
  
  \draw[color=black] (-1.3,-1.4) node {\small $x$};

  \draw[thick, color=black] (0.5,0) to[out=0, in=90] (1.2,-1);
  \draw[color=black] (1.2,-1.4) node {\small $x$};
  \draw[color=black] (-0.15,-1.4) node {\small $< $}
} + \frac{\fr({x})\cdot \fr({y})}{2}\, 
\tikz[baseline=-2pt, scale=0.5]{
  \draw[thick, color=black] (0,0) circle (0.5);

  \draw[thick, color=black] (-0.5,0) -- (0.5,0);

  \draw[black, fill=black] (-0.5,0) circle (0.5ex);
  \draw[black, fill=black] (0.5,0) circle (0.5ex);
}.
\end{eqnarray}
This is the analogue of the map $\Delta$ 
on $\calB^<_n\big(P^{(2)}\big)$ 
defined in \S \ref{subsec:lower_bound} for $n>0$;
indeed, it is related  to the homomorphism $\delta^<$
of Proposition \ref{prop:delta} as follows:

\begin{lemma} \label{lem:Delta_0}
The homomorphism $\Delta$ 
is well-defined and $\Sp(H)$-equivariant.
Furthermore, 
it fits into the following commutative diagram:
\begin{equation} \label{eq:triangle1}
\xymatrix{
 S^2 P^{(2)}/H^{(2)}
 \ar[rrd]_-\Delta \ar[rr]^-E &&
 \Tors\big(\calA_1(P) \big) \ar[d]^-{\delta^<} \\
&&  \calA^{<,c}_{2}(H) \otimes \Q/\Z
}
\end{equation}
\end{lemma}

\begin{proof}
Let $x,y \in P^{(2)}$.
Using the definition \eqref{eq:delta} 
of the map $\delta^<$,
the notation~\eqref{eq:parallel_notation},
and the facts that {$p_*(s)=0$}, $\fr(s)=1$,
we obtain 
\begin{eqnarray*}
\delta^< E(x\cdot y) &=& \delta^<\Big( 
\begin{tikzpicture}[thick,color=black,baseline=-1.0em, x=1em, y=1em]
\footnotesize
\draw (0,0) -- (1,-1);
\draw (0,0) -- (0,-1);
\draw (0,0) -- (-1,-1);
\draw [fill = black] (0,0) circle (.3ex);
\node[below=1pt] at (0,-1) {$x< y < s$};
\end{tikzpicture}
\Big)\\
&=& \frac{1}{2} \tikz[baseline=-8pt, scale=0.5]{
  \draw[thick, color=black] (0,0) circle (0.5);
  \draw[black, fill=black] (-0.5,0) circle (0.5ex);
  \draw[black, fill=black] (0.5,0) circle (0.5ex);
  \draw[thick, color=black] (-0.5,0) to[out=180, in=90] (-1.2,-1);
  \draw[color=black] (-1.3,-1.4) node {\small $x$};
  \draw[thick, color=black] (0.5,0) to[out=0, in=90] (1.2,-1);
  \draw[color=black] (1.2,-1.4) node {\small $y$};
  \draw[color=black] (-0.15,-1.4) node {\small $< $}
}+ 
\frac{1}{2}\begin{tikzpicture}[thick,baseline=-1.0em, x=2em, y=2em]
\hspace{-4mm}
\footnotesize
\node[below=1pt] at (0,-1){\quad \quad   $ x \parallel x \quad  < \qquad y \parallel y \qquad  < \quad s \parallel s $};
 \draw [fill = black] (0,0) circle (.3ex);
 \draw (0,0)-- (-2.6,-1);
\draw (0,0)-- (0,-1);
\draw (0,0)-- (2.6,-1);
\fill [white] (0.4,-0.12) circle [radius=0.1];
\fill [white] (0.0,-0.3) circle [radius=0.1];
\fill [white] (0.8,-0.3) circle [radius=0.1];
\draw [fill = black] (0.8,0) circle (.3ex);
\draw (0.8,0)-- (-1.8,-1);
\draw (0.8,0)-- (0.8,-1);
\draw (0.8,0)-- (3.4,-1);
\end{tikzpicture},
\end{eqnarray*}
which is the same as the right-hand side of \eqref{eq:Delta_0}.
This shows that the map $\Delta$ 
is well-defined and coincides
with $\delta^< E$. 
\end{proof}

We  now give a description of the $\Sp(H)$-module 
$\Tors\big({{ \calIC}/{Y_{2}}}\big)$.

\begin{theorem} \label{th:Tors(IC/Y_2)}
We have the commutative diagram
\begin{equation} \label{eq:triangle2}
\xymatrix{
 S^2 P^{(2)}/H^{(2)} \ar[d]_-\Psi
 \ar[rrd]^-\Delta  &&
 \\
\Tors\Big({\displaystyle\frac{ \calIC}{Y_{2}}}\Big) 
\ar[rr]_-{\zeta^<_{2}}  
&&   \calA^{<,c}_{2}(H) \otimes \Q/\Z,
}
\end{equation}
where $\Psi$ is an $\Sp(H)$-isomorphism, and 
$\zeta_2^<$ is an $\Sp(H)$-embedding.
\end{theorem}

\noindent
Note that, although it is derived from the LMO functor,
the map $\zeta_2^<$  on 
$\Tors\big({ \calIC}/{Y_{2}}\big)$
is independent of the choices inherent to the construction of the latter.

\begin{proof}[Proof of Theorem \ref{th:Tors(IC/Y_2)}]
According  to \eqref{eq:E_0_bis}
and Theorem \ref{th:psi_delta_Z}, we have
the commutative diagram
$$
\xymatrix{
 S^2 P^{(2)}/H^{(2)}
 \ar[rr]^-E \ar[d]_-\Psi 
 &&
 \Tors\big(\calA_1(P) \big) 
 \ar[d]^-{\delta^<} \ar[lld]_-{\psi} \\
\Tors\Big({\displaystyle\frac{ \calIC}{Y_{2}}}\Big) \ar[rr]^-{\zeta^<_{2}} 
&& \calA^{<,c}_{2}(H) \otimes \Q/\Z,
}
$$
hence \eqref{eq:triangle1} implies \eqref{eq:triangle2}.
Thus, to prove the theorem, it suffices to show
that $\Psi$ is surjective 
 and that $\Delta$ is injective.

We prove the surjectivity of $\Psi$.
Let $M\in \calIC$
such that $\{M\}\in \calIC/Y_2$ has finite order.
The surjectivity of $\psi:\calA_1(P) \to \calIC/Y_2$
implies  the existence of a $z \in \calA_1(P)$
such that $\psi(z)=\{M\}$. Furthermore,
the homomorphism 
$$
\xymatrix{
\calA_1(P) \ar[r]^-\psi & 
{\displaystyle \frac{\calIC}{Y_2}} \ar[r]^-{\iota}
& {\displaystyle \frac{F_1 \Z[\calIC]}{F_2 \Z[\calIC]}} 
\ar[r]^-{Z^<_1} & 
\calA^<_1(H) \otimes \Q \simeq \calA_1(P)\otimes \Q
}
$$
is, according to \eqref{eq:psi_Z}, 
the  rationalization map:
so, $z$ is a torsion element of $\calA_1(P)$.
But, the torsion submodule of $\calA_1(P)$ is easily derived 
from Proposition \ref{prop:varphi}
(or, alternatively, from \cite[Lemma 4.24]{MM03}):
it is generated by the elements
$$
 \begin{tikzpicture}[thick,color=black,
baseline=-0.5em, x=0.6em, y=0.6em]
\footnotesize
\node[left=1pt] at (-1,1) {\scriptsize $s$};
\node[right=1pt] at (1,1) {\scriptsize $s$};
\node[below=1pt] at (0,-1.5) {\scriptsize $s$};
\draw (-1,1) -- (0,0) -- (1,1);
\draw (0,0) -- (0,-1.5);
\draw [fill = black] (0,0) circle (.3ex);
\end{tikzpicture}
\quad \hbox{and} \quad
\begin{tikzpicture}[thick,color=black,
baseline=-0.5em, x=0.6em, y=0.6em]
\footnotesize
\node[left=1pt] at (-1,1) {\scriptsize $y$};
\node[right=1pt] at (1,1) {\scriptsize $y$};
\node[below=1pt] at (0,-1.5) {\scriptsize $x$};
\draw (-1,1) -- (0,0) -- (1,1);
\draw (0,0) -- (0,-1.5);
\draw [fill = black] (0,0) circle (.3ex);
\end{tikzpicture}
\ \hbox{for all $x,y\in 
\sigma(\{a_i,b_j \mid i,j=1,\dots,g\})$}.
$$
Clearly, all those elements are in the image of 
$E:S^2P^{(2)} \to \calA_1(P)$:
so $\{M\}$ is in the image 
of $\Psi:S^2 P^{(2)} \to \calIC/Y_2$.

To prove now the injectivity of $\Delta$, 
we consider on $\calS :=  {\displaystyle{S^2 P^{(2)}}/{H^{(2)}}}$
the filtration 
\begin{equation} \label{eq:filt_S}
 \calS  = L_0 \calS \ \supset \   L_1 \calS
\ \supset \  L_2\calS  \ \supset \ L_3 \calS =\{0\}    
\end{equation}
where $L_1 \calS$ is generated by the elements $\{s \cdot x \}$
for all $x\in P^{(2)}$,
and $ L_2\calS $ is generated by $\{s \cdot s \}$.
(This is the analogue of the $s$-loop filtration on
$\calB^<_n\big(P^{(2)}\big)$ 
defined in \S \ref{subsec:sloop} for $n>0$.)
Recall that $\calA := \calA^{<,c}_{2}(H) \otimes \Q/\Z$ 
is endowed with the loop filtration $(L_k \calA)_{k \geq 0}$,
and observe in \eqref{eq:Delta_0} that $\Delta$ is filtration-preserving.
We claim that $\Gr^L \!\Delta$ is injective:
since the filtration \eqref{eq:filt_S} is finite, 
this will imply that $\Delta$ is injective.

To prove our claim,  we consider the following sequence  of $\Sp(H)$-modules:
\begin{equation} \label{eq:ses}
\xymatrix{
0 \ar[r]
  & P^{(2)} \ar[r]^-u
  & {\displaystyle \frac{S^2 P^{(2)}}{H^{(2)}}} \ar[r]^-v
  & \Lambda^{2} H^{(2)} \ar[r]
  & 0.
}    
\end{equation}
Here, the maps $u$ and $v$ are given by  $u(x) := \{ s \cdot x\}$
and  $v(\{ x \cdot y\}):= p_*(x) \wedge p_*(y)$.
It follows from \eqref{eq:primary}
that we can work in the $\Z_2$-vector space $P^{(2)}$ with the basis
\begin{equation} \label{eq:basis_P(2)}
\{s\}\cup\{\sigma(a_i), \sigma(b_i) \mid i=1,\dots,g\};
\end{equation}
 it is easily checked in this way that \eqref{eq:ses} is exact.
Since the image of $u$ is $L_1\calS$, we deduce that 
$L_0\calS/L_1\calS \simeq \Lambda^2 H^{(2)}$
and that $L_1\calS/L_2 \calS\simeq P^{(2)}/(\Z_2s) \simeq H^{(2)}$.
Then, we conclude from \eqref{eq:Delta_0} 
that the map $\Gr^L \!\Delta$ is given by
$$
\Lambda^2 H^{(2)}\simeq \Gr_0^L \calS
\longrightarrow \Gr_0^L \calA \simeq \calA_{2,0}^{c} (H) \otimes \Q/\Z  ,  \quad 
x\wedge y \longmapsto \frac{1}{2}
\tikz[baseline=6pt, scale = 0.4]{
\draw [thick,color=black] (0,2-1.25)-- (2,2-1.25);
\draw [thick,color=black] (0,1.5)-- (0,0);
\draw [thick,color=black] (2,1.5)-- (2,0);
\draw [black, fill = black] (0,0.75) circle (.5ex);
\draw [black, fill = black] (2,0.75) circle (.5ex);
\draw[color=black] (0,3-1.25) node[yshift=2pt] {\small $x$};
\draw[color=black] (0,1.1-1.25) node[yshift=-2pt] {\small $y$};
\draw[color=black] (2,1.1-1.25) node[yshift=-2pt] {\small $x$};
\draw[color=black] (2,3-1.25) node[yshift=2pt] {\small $y$};}
$$
$$
 H^{(2)}\simeq \Gr_1^L \calS
\longrightarrow \Gr_1^L \calA \simeq \calA_{2,1}^{c} (H) \otimes \Q/\Z  ,  \quad 
x \longmapsto \frac{1}{2} \tikz[baseline=-8pt, scale=0.4]{
  \draw[thick, color=black] (0,0) circle (0.5);
  \draw[black, fill=black] (-0.5,0) circle (0.5ex);
  \draw[black, fill=black] (0.5,0) circle (0.5ex);
  \draw[thick, color=black] (-0.5,0) to[out=180, in=90] (-1.2,-1);
  \draw[color=black] (-1.3,-1.4) node {\small $x$};
  \draw[thick, color=black] (0.5,0) to[out=0, in=90] (1.2,-1);
  \draw[color=black] (1.2,-1.4) node {\small $x$};} 
$$
$$
 \Z_2 s \simeq \Gr_2^L \calS
\longrightarrow \Gr_2^L \calA \simeq \calA_{2,2}^{c} (H) \otimes \Q/\Z  ,  \quad 
s \longmapsto \frac{1}{2} \tikz[baseline=-2pt, scale=0.4]{
  \draw[thick, color=black] (0,0) circle (0.5);
  \draw[thick, color=black] (-0.5,0) -- (0.5,0);
  \draw[black, fill=black] (-0.5,0) circle (0.5ex);
  \draw[black, fill=black] (0.5,0) circle (0.5ex);}
$$
in degrees 0, 1 and 2, respectively;
in particular, note that $\Gr_0^L \Delta$ is Levine's embedding \eqref{eq:Levine_embedding}.
Consequently, $\Gr^L \!\Delta$ is injective.
\end{proof}

As a consequence, we obtain a characterization
of the $\Sp(H)$-module $\calIC/Y_2$ by means of the LMO functor.

\begin{corollary} \label{cor:IC/Y_2}
We have the commutative diagram of $\Sp(H)$-modules
\begin{equation} \label{eq:ladder_1}
\xymatrix{
0 \ar[r]
  & {\displaystyle {S^2 P^{(2)}}/{H^{(2)}}}  \ar[r]^-E \ar[d]_-\Psi
  &  \calA_1(P) \ar[r]^-{p_*} \ar[d]_-\psi
  & \calA_1^<(H)  \ar[r] \ar@{=}[d]
  & 0\\
0 \ar[r]
  & {\displaystyle \operatorname{Tors}(\calIC/Y_2)}  \ar[r]
  &  \calIC/Y_2  \ar[r]^-{Z_1^<}
  &  \bqm \calA_1^<(H) \eqm  \ar[r]
  & 0,
}    
\end{equation}
with exact rows and vertical isomorphisms.
Furthermore, the map
\begin{equation} \label{eq:big_map}
(Z_1^<,\zeta_2^<): {{ \calIC}/{Y_{2}}} 
\longrightarrow \calA_1^<(H)\, \oplus\, 
\big( \calA^{<,c}_{2}(H) \otimes \Q/\Z\big)
\end{equation}
is an $\Sp(H)$-embedding 
for the $\Sp(H)$-action defined by \eqref{eq:Sp-action}.
\end{corollary}

\noindent
Note that the component $\zeta_2^{<}$ of \eqref{eq:big_map}
on the entire group $\mathcal{IC}/Y_2$
\emph{does} depend on the choices 
inherent to the construction of the LMO functor.
On the contrary, according to \eqref{eq:tau_Z},
the component $Z_1^<$ of \eqref{eq:big_map}
coincides with the first Johnson homomorphism  $\tau_1$: 
hence, the invariant  $Z_1^<$ is intrinsic by nature.

\begin{remark}
That $\psi \colon \mathcal{A}_1(P) \to \mathcal{IC}/Y_2$ is an isomorphism
of $\Sp(H)$-modules
has already been established in~\cite{MM03}; 
see \S\ref{subsec:comparison} for a comparison of results. 
The isomorphism type of $\mathcal{IC}/Y_2$ 
as an abelian group was announced in \cite{Hab00}.
\hfill $\blacksquare$
\end{remark}

\begin{proof}[Proof of Corollary \ref{cor:IC/Y_2}]
The commutativity of the diagram follows from \eqref{eq:E_0_bis} 
and  \eqref{eq:psi_Z}.
Besides, it also  follows from \eqref{eq:psi_Z} that
$Z_1^<: \mathcal{IC}/Y_2 \to \calA_1^<(H)$  is rationally an isomorphism: 
hence the second row of \eqref{eq:ladder_1} is exact.

Since $\Psi$ is an isomorphism (by Theorem \ref{th:Tors(IC/Y_2)}),  
$E$ is injective. 
We also observed (in the proof of Theorem \ref{th:Tors(IC/Y_2)})
 that $E$ is surjective 
onto  $\Tors\big(\calA_1(P)\big)$.
Therefore, 
to obtain  the exactness of the first row of \eqref{eq:ladder_1},
it suffices to recall that  
$\calA_1(P)\otimes \Q \simeq \calA_1^<(H)\otimes \Q $ 
(see Proposition \ref{prop:before}),
so that the maximal torsion-free quotient of $\calA_1(P)$ 
is given by $p_*:\calA_1(P) \to \calA_1^<(H)$.

Since $Z_1^<: \mathcal{IC}/Y_2 \to \calA_1^<(H)$
is rationally an isomorphism, and  since 
$\zeta_2^<  : \Tors(\mathcal{IC}/Y_2) \to \calA^{<,c}_{2}(H) \otimes \Q/\Z$ 
is injective (by Theorem \ref{th:Tors(IC/Y_2)}),
the pair $(Z_1^<, \zeta_2^<)$ is injective on the whole of  $\mathcal{IC}/Y_2$.
Corollary \ref{cor:equivariance} gives its $\Sp(H)$-equivariance.
\end{proof}

\subsection{Extracting classical invariants from LMO}
\label{subsec:comparison}

For the sake of completeness,
we shall relate Theorem \ref{th:Tors(IC/Y_2)}
and Corollary \ref{cor:IC/Y_2}
to the description of $\calIC/Y_2$ found in \cite{MM03}.
The latter involves two classical invariants:
the first Johnson homomorphism $\tau_1$ mentioned at \eqref{eq:tau_k},
and  the \emph{Birman--Craggs homomorphism}
$$
\beta: \calIC \longrightarrow \operatorname{Cub}(\Omega)
$$
with values in the space of of cubic functions 
on the $\Z_2$-affine space 
$\Omega=\operatorname{Spin}(\Sigma)$.

Firstly, we relate the two approaches 
(``quantum'' vs ``classical'')
on the torsion parts.
Indeed, the main result of \cite{MM03} implies that
$\beta$ restricts to an isomorphism 
between $\Tors\big(\calIC/Y_2\big)$ 
and the space 
$\operatorname{Quad}(\Omega)$
of quadratic functions. 
To recover this fact with the LMO functor, 
recall from Remark \ref{rem:spin}
the map $e: P \to \operatorname{Aff}(\Omega)$, 
which allows for an identification 
$P^{(2)}\simeq \operatorname{Aff}(\Omega)$.

\begin{proposition} \label{prop:CvsQ}
The following diagram is commutative:
\begin{equation} \label{eq:diag_square}
\xymatrix{
\Tors\Big({\displaystyle\frac{ \calIC}{Y_{2}}}\Big) 
\ar[r]^-\beta \ar[d]_-{\zeta_2^<}& 
\operatorname{Quad}(\Omega) \\
\calA^{<,c}_{2}(H) \otimes \Q/\Z &  
S^2 P^{(2)}/H^{(2)} \ar[lu]_-\Psi
\ar[l]^-\Delta \ar[u]_-{\operatorname{mult}}
}
\end{equation}
where ``$\operatorname{mult}$''  is defined by multiplication of affine functions,
and the map $\beta$ is an isomorphism.
\end{proposition}

\begin{proof}
That the upper inner triangle is commutative follows
from the variation formula for the Rochlin invariant
under $Y$-graph surgery \cite[Lemma 3.16]{MM03}.
The lower inner triangle is  commutative 
by \eqref{eq:triangle2}. 

It is easily checked,
working with the basis \eqref{eq:basis_P(2)} of $P^{(2)}$, that ``mult'' is an isomorphism:
since $\Psi$ is an isomorphism (by Theorem~\ref{th:Tors(IC/Y_2)}),
$\beta$ is an isomorphism.
\end{proof}

\begin{remark}
By Proposition \ref{prop:CvsQ},
$\operatorname{Quad}(\Omega)$
embeds  into $\calA^{<,c}_{2}(H) \otimes \Q/\Z$ in an $\Sp(H)$-equivariant way.
Setting $\overline{x}:=e(\sigma(x))\in \operatorname{Aff}(\Omega)$
for any $x\in H$, 
and denoting by $\overline{1}:=e(s)\in \operatorname{Aff}(\Omega)$
the non-zero constant function, 
we deduce from  \eqref{eq:Delta_0} that the embedding is given by
\begin{eqnarray*}
\overline{x} \cdot \overline{y} &\longmapsto &
\frac{1}{2}\tikz[baseline=5pt]{
  \draw[thick, color=black] (0,0)--(0.5,0.5);
  \draw[thick, color=black] (0.5,0)--(1,0.5);
  \fill [white] (0.75,0.25) circle [radius=0.1];
  \draw[thick, color=black] (0.5,0.5)--(1,0);
  \draw[thick, color=black] (1,0.5)--(1.5,0);
  \draw[black, fill=black] (0.5,0.5) circle (0.25ex);
   \draw[black, fill=black] (1,0.5) circle (0.25ex);
  \draw[thick, color=black] (0.5,0.5) to[out=90, in=90] (1,0.5);
  \draw[color=black] (0.75,-0.25) node {\small $x< x  <y<y$}
}+\frac{1}{2} \tikz[baseline=-8pt, scale=0.5]{
  \draw[thick, color=black] (0,0) circle (0.5);

  \draw[black, fill=black] (-0.5,0) circle (0.5ex);
  \draw[black, fill=black] (0.5,0) circle (0.5ex);

  \draw[thick, color=black] (-0.5,0) to[out=180, in=90] (-1.2,-1);
  
  \draw[color=black] (-1.3,-1.4) node {\small $x$};
  
  \draw[thick, color=black] (0.5,0) to[out=0, in=90] (1.2,-1);
  \draw[color=black] (1.2,-1.4) node {\small $y$};
  \draw[color=black] (-0.15,-1.4) node {\small $< $}
},\\
\overline{x}  & \longmapsto&
 \frac{1}{2} \tikz[baseline=-8pt, scale=0.5]{
  \draw[thick, color=black] (0,0) circle (0.5);

  \draw[black, fill=black] (-0.5,0) circle (0.5ex);
  \draw[black, fill=black] (0.5,0) circle (0.5ex);

  \draw[thick, color=black] (-0.5,0) to[out=180, in=90] (-1.2,-1);
  
  \draw[color=black] (-1.3,-1.4) node {\small $x$};

  \draw[thick, color=black] (0.5,0) to[out=0, in=90] (1.2,-1);
  \draw[color=black] (1.2,-1.4) node {\small $x$};
  \draw[color=black] (-0.15,-1.4) node {\small $< $}
}, \qquad 
 \overline{1}  \ \longmapsto \   \frac{1}{2}\, 
\tikz[baseline=-2pt, scale=0.5]{
  \draw[thick, color=black] (0,0) circle (0.5);

  \draw[thick, color=black] (-0.5,0) -- (0.5,0);

  \draw[black, fill=black] (-0.5,0) circle (0.5ex);
  \draw[black, fill=black] (0.5,0) circle (0.5ex);
}.  
\end{eqnarray*} 

\up
\hfill $\blacksquare$
\end{remark}

\begin{remark} \label{rem:R_BC}
Diagram \eqref{eq:diag_square} exists in the category of filtered $\Sp(H)$-modules: indeed, 
$\calT := \operatorname{Tors}(\calIC/Y_2)$
 has the filtration \eqref{eq:filt_T}, 
$\calS := S^2 P^{(2)}/H^{(2)}$
 has the filtration \eqref{eq:filt_S}, and
$\calA := \calA_2^{<,c}(H) \otimes \Q/\Z$
 has the loop filtration.
Furthermore, by taking the complement to $2$, the degree-filtration on
$\calQ := \operatorname{Quad}(\Omega)$
becomes a decreasing filtration whose associated graded is isomorphic to
$\bigoplus_{k=0}^2 \Lambda^{2-k} H^{(2)}$
via formal differentiation.
Then, according to the results of~\S\ref{subsec:back}, 
$\Gr \zeta_2^<$ is given by the homomorphisms
\begin{align*}
R_2: \calT &\longrightarrow \frakL_2\big(H^{(2)}\big), \quad \{M\} \longmapsto \{\tau_2(M)\}, \\
R_2': L_1\calT &\longrightarrow S^2 H^{(2)}, \quad \{M\} \longmapsto \operatorname{Tr}(\tau_2(M)) + \{\alpha(M)\}, \\
R_2'': L_2\calT &\longrightarrow 2\Z_4, \quad \{M\} \longmapsto 
\operatorname{Tr}'(-\tau_2(M),\alpha(M)) + \{4d''(M)\},
\end{align*}
where the invariants $\alpha:\calIC \to S^2 H$ and $d'':\calIC[2] \to \frac{1}{8}\Z$ 
are introduced in \cite{MM13} (and  mentioned in Examples \ref{ex:Tr} \& \ref{ex:Tr'}). 
It follows from \eqref{eq:diag_square} 
that the homomorphisms $R_2$, $R'_2$ and $R''_2$ also describe $\Gr \beta$; 
this description generalizes \cite[Lemmas 3.18 \& 3.19]{MM13}.
\hfill $\blacksquare$ 
\end{remark}

We now consider the entire groups.
According to \cite{MM03}, 
the pair $(\tau_1,\beta)$
induces an isomorphism between
${{ \calIC}/{Y_{2}}}$
and the fibered product 
$\Lambda^3H \times_{\Lambda^3H^{(2)}} \operatorname{Cub}(\Omega) $,
where the homomorphism  $\Lambda^3 H \to \Lambda^3 H^{(2)}$
is the modulo $2$ reduction and the map
$\operatorname{Cub}(\Omega) \to\Lambda^3 H^{(2)}$
is the formal third differentiation.
We recover this with the LMO functor as follows.

\begin{proposition} \label{prop:deg_1}
The following diagram is commutative:
\begin{equation} \label{eq:diag_square_full}
\xymatrix{
{\displaystyle \frac{\calIC}{Y_2}} \ar[r]^-{(\tau_1,\beta)} 
\ar[d]_-{\big(Z_1^<,\zeta_2^<\big)} &  
\Lambda^3 H \times_{\Lambda^3 H^{(2)}} 
\operatorname{Cub}(\Omega) 
\\
\big(\calA^<_1(H)\otimes\Q\big) \oplus 
\big(\calA^{<,c}_{2}(H) \otimes \Q/\Z\big) & 
\calA_1(P) \ar[lu]^-\psi \ar[l]^-{(p_*, \delta^<)}
\ar[u]_-{(p_*,\operatorname{mult})}
}
\end{equation}
where ``mult'' is defined by multiplication 
of affine functions,
and the map $(\tau_1,\beta)$  is an isomorphism.
\end{proposition}

\begin{proof}
That the upper inner triangle is commutative follows
from the variation formula for $\tau_1$
under $Y$-graph surgery \cite[Lemma 4.22]{MM03}
(and the analogous formula for $\beta$).
The lower inner triangle is commutative 
by \eqref{eq:psi_Z} and \eqref{eq:AGAA}. 
It is proved, using the primary decomposition \eqref{eq:primary} of $P$, 
that $(p_*,\operatorname{mult})$ is an isomorphism 
(see \cite[Lemma 4.24]{MM03}).
Since $\psi$ is an isomorphism (by Corollary \ref{cor:IC/Y_2}), 
we deduce that $(\tau_1,\beta)$ is an isomorphism too.
\end{proof}

\begin{remark}
The commutative diagrams
\eqref{eq:diag_square} and \eqref{eq:diag_square_full} 
could also be deduced 
from the results of \cite[\S 5.2]{MM13}. \hfill $\blacksquare$
\end{remark}

\section{The degree-three case}

\label{sec:deg_3}

In this last section, we apply the results of  \S\ref{sec:odd_degrees}--\S\ref{sec:symLMO}
to the case of degree 3. Specifically, we give an intrinsic description of the quotient $Y_3 \calIC/Y_4$ as an $\Sp(H)$-module, and we derive several consequences.

\subsection{Description of the $\Sp(H)$-module $Y_3 \mathcal{IC} / Y_{4}$}

First of all, we make explicit the surgery map $\Psi$ 
introduced in \S \ref{subsec:asm},
the map $E$ introduced in \S \ref{subsec:rsm},
as well as the map $\Delta$ of \S \ref{subsec:lower_bound}.
These are the contents of the next two lemmas
(where we use the convention \eqref{eq:lifts}
to write formulas for those maps).

\begin{lemma} We have the commutative diagram
\begin{equation} \label{eq:E_psi's}
\xymatrix{
H^{(2)} \otimes \Lambda^2 P^{(2)}
\ar[r]^-\Psi \ar[d]_-E &
\Tors\Big({\displaystyle\frac{Y_{3} \calIC}{Y_{4}}}\Big), \\ 
\Tors\big(\calA^{<,c}_{3}(H)\big) \ar[ru]_-\psi &  
}
\end{equation}
where, for all $h\in H^{(2)}$  and for all  $x,y\in P^{(2)}$, we have
$$
E(h\otimes x\wedge y) =
\tikz[baseline = 1.75cm]{
\draw[thick, color = black](4.5,1.75) .. controls (4.5,2) .. (5,2);
\draw[thick, color = black](4.5,1.75) .. controls (4.5,1.5) .. (5,1.5);
\draw[thick, color = black](5,2) -- (5.4,2);
\draw[thick, color = black](5,1.5) -- (5.5,1.5);
\draw[thick, color = black](5,2) -- (5,2.25);
\draw[thick, color = black](5.5,1.5) -- (5.5,2.25);
\draw[thick, color = black](5.5,1.5) .. controls (6.5,1.5) .. (6.5,2.25);
\draw[thick, color = black](5.6,2) .. controls (6,2) .. (6,2.25);
\draw[thick, color = black](4.5,1.75) .. controls (4.25,1.75) .. (4.25,2.25);
\draw [black, fill = black] (4.5,1.75) circle (.4ex);
\draw [black, fill = black] (5,2) circle (.4ex);
\draw [black, fill = black] (5.5,1.5) circle (.4ex);
\draw[color=black] (4.25,2.45) node {\small $  h$};
\draw[color=black] (5,2.45) node {\small $  x$};
\draw[color=black] (5.5,2.45) node {\small $  x$};
\draw[color=black] (6,2.45) node {\small $  y$};
\draw[color=black] (6.5,2.45) node {\small $  y$};
\draw[color=black] (5.75,2.425) node {\scriptsize $<$};
\draw[color=black] (5.25,2.425) node {\scriptsize $\parallel$};
\draw[color=black] (6.25,2.425) node {\scriptsize $\parallel$};
\draw[color=black] (4.6,2.425) node {\scriptsize $<$};
}
=
 \forktree{  h}{  x}{  x}{  y}{  y} 
+\fr(y) \Odiag{  h}{  x}{  x} + \fr(x)\Odiag{h}{  y}{  y}  
$$
and\\[-0.5cm]
$$
\Psi(h\otimes x\wedge y) 
= \Bigg( \hbox{$U$ surgered along 
$\begin{array}{c}
\includegraphics[scale=0.85]{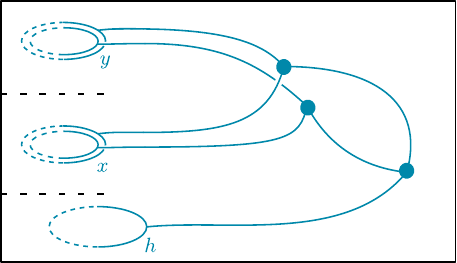}
\end{array}$} \Bigg).
$$
(Here the graph clasper has 3 layers of leaves:
there is one leaf realizing  $h\in H^{(2)}$ in the lower slice of $U$,  two parallel leaves realizing $x\in P^{(2)}$
in the middle slice, and two parallel leaves realizing $y\in P^{(2)}$ in the upper slice.)
\end{lemma}

\begin{proof}
Specializing Proposition \ref{prop:Psi_psi} 
to the case $n:=1$, we obtain the commutative diagram
\begin{equation} \label{eq:inter}
\xymatrix{
\calB^<_1\big(P^{(2)}\big) \ar[r]^-\Psi \ar[d]_-E &
\Tors\Big({\displaystyle\frac{Y_{3} \calIC}{Y_{4}}}\Big) \\ 
\Tors\big(\calA^{<,c}_{3}(H)\big) \ar[ru]_-\psi &  
}
\end{equation}
where (using Remark \ref{rem:H_is_enough}) 
the map $E$ is given by 
\begin{equation} \label{eq:Ehxy}
E\Big( \begin{tikzpicture}[baseline=-1ex, scale=0.2]
\draw[thick] (0,-0.5) -- (1,1);
\draw[thick] (0,-0.5) -- (-1,1);
\draw[thick] (0,-0.5) -- (0,-2);
\fill (0,-0.4) circle (10pt);
\node[below] at (0,-1.8) {$h \star$};
\node[above] at (0,1) {$x$ \scriptsize ~$<$ \normalsize $y$};
\end{tikzpicture} \Big) =
\tikz[baseline = 1.75cm]{
\draw[thick, color = black](4.5,1.75) .. controls (4.5,2) .. (5,2);
\draw[thick, color = black](4.5,1.75) .. controls (4.5,1.5) .. (5,1.5);
\draw[thick, color = black](5,2) -- (5.4,2);
\draw[thick, color = black](5,1.5) -- (5.5,1.5);
\draw[thick, color = black](5,2) -- (5,2.25);
\draw[thick, color = black](5.5,1.5) -- (5.5,2.25);
\draw[thick, color = black](5.5,1.5) .. controls (6.5,1.5) .. (6.5,2.25);
\draw[thick, color = black](5.6,2) .. controls (6,2) .. (6,2.25);
\draw[thick, color = black](4.5,1.75) .. controls (4.25,1.75) .. (4.25,2.25);
\draw [black, fill = black] (4.5,1.75) circle (.4ex);
\draw [black, fill = black] (5,2) circle (.4ex);
\draw [black, fill = black] (5.5,1.5) circle (.4ex);
\draw[color=black] (4.25,2.45) node {\small $  h$};
\draw[color=black] (5,2.45) node {\small $  x$};
\draw[color=black] (5.5,2.45) node {\small $  x$};
\draw[color=black] (6,2.45) node {\small $  y$};
\draw[color=black] (6.5,2.45) node {\small $  y$};
\draw[color=black] (5.75,2.425) node {\scriptsize $<$};
\draw[color=black] (5.25,2.425) node {\scriptsize $\parallel$};
\draw[color=black] (6.25,2.425) node {\scriptsize $\parallel$};
\draw[color=black] (4.6,2.425) node {\scriptsize $<$};
}  
\end{equation}
for all $h\in H^{(2)}$ and $x,y\in P^{(2)}$.
It easily seen that  \eqref{eq:Ehxy} is zero when $x=y$
(as~a consequence of ``IHX'' and ``AS'').
Hence, the map $E$ in \eqref{eq:inter} factorizes
through $H^{(2)} \otimes \Lambda^2 P^{(2)}$, 
and so does the surgery map $\Psi$.
\end{proof}

\begin{lemma}\label{lem:Delta_3}
We have the commutative diagram
$$
\xymatrix{
H^{(2)} \otimes \Lambda^2 P^{(2)}
 \ar[rd]_-\Delta \ar[r]^-E &
 \Tors\big(\calA^{<,c}_{3}(H)\big) \ar[d]^-{\delta^<} \\
&  \calA^{<,c}_{4}(H) \otimes \Q/\Z
}
$$
where the $\Sp(H)$-equivariant map $\Delta$ is given by 
\begin{equation} \label{eq:Delta_3}
\Delta(h \otimes (x \wedge y)) = \frac{1}{2}
\begin{tikzpicture}[scale = 0.4, baseline = 1.2cm]
        \node  (0) at (0.5, 4.5) {};
		\node  (1) at (1.75, 1.75) {};
		\node  (2) at (2.5, 2.25) {};
		\node  (3) at (3.5, 2.25) {};
		\node  (4) at (-0.5, 4.5) {};
		\node  (5) at (2, 1.25) {};
		\node  (6) at (2.5, 3) {};
		\node  (7) at (3.5, 3) {};
		\node  (8) at (1.5, 4.5) {};
		\node  (9) at (2.5, 4.5) {};
		\node  (10) at (3.5, 4.5) {};
		\node  (11) at (4.5, 4.5) {};
		\draw [thick, in=180, out=-90, looseness=1.25] (0.center) to (1.center);
		\draw [thick,in=-90, out=0] (1.center) to (2.center);
		\draw [thick,in=180, out=-90, looseness=1.50] (4.center) to (5.center);
		\draw [thick, in=-90, out=0, looseness=1.25] (5.center) to (3.center);
		\draw [thick] (6.center) to (2.center);
		\draw [thick] (7.center) to (3.center);
		\draw [thick] (2.center) to (3.center);
		\draw [thick] (6.center) to (8.center);
        \draw [thick] (9.center) to (7.center);
        \draw [white, fill = white] (3,3.75) circle  (1ex);
        \draw [thick] (6.center) to (10.center);
		\draw [thick] (7.center) to (11.center);
        \draw [black, fill = black] (2.center)  circle (0.8ex);
         \draw [black, fill = black] (3.center)  circle (0.8ex);
         \draw [black, fill = black] (6.center)  circle (0.8ex);
         \draw [black, fill = black] (7.center)  circle (0.8ex);
         \node[above] at (2.1,4.5) 
         {\resizebox{!}{.5\baselineskip}{$  h<  h<  x \parallel   x<   y\parallel   y$}};
\end{tikzpicture} + \frac{1}{2} \begin{tikzpicture}[scale = 0.4, baseline = 1.25cm]
		\node (0) at (0.5, 4.5) {};
		\node (1) at (1.75, 1.5) {};
		\node (2) at (2.5, 2.25) {};
		\node (3) at (2, 3) {};
		\node (4) at (3, 3) {};
		\node (5) at (1.5, 4.5) {};
		\node (6) at (2.5, 4.5) {};
		\node (7) at (3.5, 4.5) {};
		\node (8) at (1.5, 3.75) {};
		\draw [thick, in=180, out=-90] (0.center) to (1.center);
		\draw [thick, in=-90, out=0, looseness=1.25] (1.center) to (2.center);
		\draw [thick] (2.center) to (3.center);
		\draw [thick] (2.center) to (4.center);
		\draw [thick, in=-90, out=150, looseness=0.75] (3.center) to (8.center);
		\draw [thick] (8.center) to (5.center);
		\draw [thick, in=270, out=45] (4.center) to (7.center);
		\draw [thick, in=0, out=135] (4.center) to (8.center);
        \draw [white, fill = white] (2.375, 3.525) circle  (1ex); 
		\draw [thick, in=270, out=45] (3.center) to (6.center);
        \draw [black, fill = black] (2.center)  circle (0.8ex); \draw [black, fill = black] (3.center)  circle (0.8ex); \draw [black, fill = black] (4.center)  circle (0.8ex); \draw [black, fill = black] (8.center)  circle (0.8ex);
        \node[above] at (2,4.5) {\resizebox{!}{.5\baselineskip}{$ {h}< {x}<  {y} \parallel  {y}$}};
\end{tikzpicture} +  \frac{1}{2}
 \begin{tikzpicture}[scale = 0.4, baseline = 1.25cm]
		\node (0) at (0.5, 4.5) {};
		\node (1) at (1.75, 1.5) {};
		\node (2) at (2.5, 2.25) {};
		\node (3) at (2, 3) {};
		\node (4) at (3, 3) {};
		\node (5) at (1.5, 4.5) {};
		\node (6) at (2.5, 4.5) {};
		\node (7) at (3.5, 4.5) {};
		\node (8) at (1.5, 3.75) {};
		\node (9) at (3.5, 3.75) {};
		\draw [thick, in=180, out=-90] (0.center) to (1.center);
		\draw [thick, in=-90, out=0, looseness=1.25] (1.center) to (2.center);
		\draw [thick] (2.center) to (3.center);
		\draw [thick] (2.center) to (4.center);
		\draw [thick, in=-90, out=150, looseness=0.75] (3.center) to (8.center);
		\draw [thick](8.center) to (5.center);
		\draw [thick] (7.center) to (9.center);
		\draw [thick, in=-90, out=30, looseness=0.75] (4.center) to (9.center);
		\draw [thick, in=-90, out=150] (4.center) to (6.center);
        \draw [white, fill = white] (2.55, 3.55) circle  (1ex); 
		\draw [thick, in=180, out=60] (3.center) to (9.center);
        \draw [black, fill = black] (2.center)  circle (0.8ex); \draw [black, fill = black] (3.center)  circle (0.8ex); \draw [black, fill = black] (4.center)  circle (0.8ex); \draw [black, fill = black] (9.center)  circle (0.8ex);
        \node[above] at (2,4.5) {\resizebox{!}{.5\baselineskip}{$ {h}< {x} \parallel  {x}< {y}$}};
\end{tikzpicture},
\end{equation}
for all $h\in H^{(2)}$ and   $x,y\in P^{(2)}$.
\end{lemma}

\begin{proof}
This is obtained by specializing Lemma \ref{lem:Delta}
and formula \eqref{eq:Delta_def_bis} at $n:=1$.
\end{proof}

Furthermore, we will need the next two lemmas, 
whose proofs are postponed to the end of this section.

\begin{lemma} \label{lem:E_inj}
The map 
$E:H^{(2)} \otimes \Lambda^2 P^{(2)}
 \to \Tors \big(\calA^{<,c}_{3}(H)\big)$ 
 is an isomorphism of $\Sp(H)$-modules.
\end{lemma}
 
 \begin{lemma} \label{lem:cyclic}
 The linear map 
 $C :  \Lambda^3 P^{(2)} \rightarrow H^{(2)} \otimes \Lambda^2 P^{(2)}$  defined by 
 $$
C(x,y,z) := 
 p_*(x) \otimes (y \wedge z) + p_*(y) \otimes (z \wedge x) +p_*(z) \otimes (x \wedge y)
 $$ 
 is injective, 
 and  $\Psi: H^{(2)} \otimes \Lambda^2 P^{(2)} 
 \to \Tors(Y_3 \calIC/Y_4)$ vanishes on 
 $\Im(C)\simeq  \Lambda^3 P^{(2)}$.  
 \end{lemma}

\begin{remark}
The composite map $E\circ C:  \Lambda^3 P^{(2)} \to \calA^{<,c}_{3}(H)$
is a refinement of some constructions appearing
in  \cite[Def$.$ 40]{CST16} and \cite[Def$.$ 3.4]{NSS22a} for arbitrary degrees.
\hfill $\blacksquare$
\end{remark}

With the above lemmas, we are now able to prove the following description of the
$\Sp(H)$-module $\Tors\big(Y_3\mathcal{IC}/Y_4\big)$.

\begin{theorem} \label{th:Tors(Y_3IC/Y_4)}
We have the commutative diagram
\begin{equation} \label{eq:triangle3}
\xymatrix{
 \big(H^{(2)} \otimes \Lambda^2 P^{(2)}\big)\big/\Lambda^3 P^{(2)}  \ar[d]_-\Psi \ar[drr]^-\Delta
 &&
 \\
\Tors\Big({\displaystyle\frac{ Y_3\calIC}{Y_{4}}}\Big)
\ar[rr]_-{\zeta^<_{4}} 
&&   \calA^{<,c}_{4}(H) \otimes \Q/\Z 
}
\end{equation}
where $\Psi$ is an $\Sp(H)$-isomorphism,
and  $\zeta_4^< $  is an $\Sp(H)$-embedding.
\end{theorem}

\noindent
Note that, although it comes from the LMO functor,
the map $\zeta_4^<$  on 
$\Tors\big({ Y_3\calIC}/{Y_{4}}\big)$
is independent of the choices inherent to the construction of the latter.

\begin{proof}[Proof of Theorem \ref{th:Tors(Y_3IC/Y_4)}]
By Lemma \ref{lem:cyclic}, the homomorphism $\Psi$ 
factorizes through 
$\big(H^{(2)} \otimes \Lambda^2 P^{(2)}\big)\big/\Lambda^3 P^{(2)}$.
Thus, the diagram \eqref{eq:triangle3}
follows directly from Theorem~\ref{th:Delta_zeta} 
specialized to $n:=1$.

Since $\psi: \calA^{<,c}_{3}(H) \to Y_3\calIC/Y_4$ is surjective 
and is rationally an isomorphism,
its  restriction $\psi: \Tors\big( \calA^{<,c}_{3}(H) \big) 
\to \Tors\big(Y_3\calIC/Y_4\big)$  is surjective too.
Since $E$ is an isomorphism (by Lemma \ref{lem:E_inj}),
we deduce from \eqref{eq:E_psi's} that $\Psi$ is surjective.
Therefore, we are reduced to show that $\Delta$ is injective.

The $s$-loop filtration on
$\calB^<_n\big(P^{(2)}\big)$ 
defined in \S \ref{subsec:sloop}  for $n:=1$ 
induces a filtration
$$
\mathcal{S} = L_0 \mathcal{S} \supset  L_1 \mathcal{S} 
\supset L_2 \mathcal{S} =\{0\}
$$
on its quotient
$\mathcal{S}:= \big(H^{(2)} \otimes \Lambda^2 P^{(2)}\big)/\Lambda^3 P^{(2)}$: 
here $L_1\mathcal{S}$ denotes the submodule generated 
by the elements $\{h \otimes (s \wedge y)\}$
for all $h\in H^{(2)}$ and $y\in P^{(2)}$.
We also endow $\calA := \calA^{<,c}_{4}(H) \otimes \Q/\Z$ 
with the loop filtration $(L_k \calA)_{k \geq 0}$, and we know
from Proposition \ref{prop:Gr_Delta} that $\Delta:\calS \to \calA$ is filtration-preserving. 
To prove that $\Delta$ is injective,
it suffices to prove that $\Gr^L \!\Delta$ is injective.

For this purpose, 
we claim the exactness of following  sequence:
\begin{equation} \label{eq:ses_3_tor}
0 \longrightarrow S^2 H^{(2)}
\stackrel{w}{\longrightarrow}
\frac{H^{(2)} \otimes \Lambda^2 P^{(2)}}{\Lambda^3 P^{(2)}} 
\stackrel{z}{\longrightarrow}
\frakL_3\big(H^{(2)}\big) \longrightarrow 0.   
\end{equation}
Here, $w$ is defined by
$w(h \cdot p_*(y)) := \{ h \otimes( s \wedge y)\}$
and $z$ is defined by $z(\{h\otimes x \wedge y\}):=[h,[p_*(x),p_*(y)]]$,
for any $h\in H^{(2)}$ and $x,y\in P^{(2)}$.
The surjectivity of $z$ and the inclusion of $\Im(w)$ in $\Ker(z)$ are obvious.
Furthermore, we deduce from \eqref{eq:Delta_3} that
\begin{eqnarray}
 \label{eq:Delta_w} \big(\Delta(w(h\cdot y)) \!\! \mod L_2\calA \big)& =& 
 \frac{1}{2} \loopedfour{ {h}}{ {h}}{ {y}}{ {y}} 
\in  \Gr_1^L \calA \simeq \calA^{c}_{4,1}(H) \otimes \Q/\Z  \\
\notag &=& \frac{1}{2} \big\{  h   h   y   y\big\}
\in \big(H^{\otimes 4}\big)_{\mathfrak{D}_8} \otimes \Q/\Z
\end{eqnarray}
for any $h,y \in H^{(2)}$, using the identification \eqref{eq:wheel}.
It follows that $w$ is injective. Hence, using that $\big( {H^{(2)} \otimes \Lambda^2 H^{(2)}}\big)/\Lambda^3 H^{(2)} \simeq \frakL_3\big(H^{(2)}\big)$, we deduce that
\begin{eqnarray*}
\big\vert S^2 H^{(2)} \big\vert 
\ = \ \vert \Im(w) \vert  
& \leq & \vert \Ker(z) \vert     \\
&=&  \big\vert \big( H^{(2)} \otimes \Lambda^2 P^{(2)} \big)/{\Lambda^3 P^{(2)}} \big\vert
 - \big\vert \big( {H^{(2)} \otimes \Lambda^2 H^{(2)}}\big)/{\Lambda^3 H^{(2)}} \big\vert\\
 &=& \big\vert H^{(2)} \otimes \big( \Lambda^2 P^{(2)}/ \Lambda^2 H^{(2)} \big) \big\vert - 
 \big\vert \Lambda^3 P^{(2)} / \Lambda^3 H^{(2)} \big\vert \\
 &=&  \big\vert H^{(2)} \otimes H^{(2)}\big\vert 
 - \big\vert \Lambda^2 H^{(2)} \big\vert
 \ = \ \big\vert  S^2 H^{(2)} \big\vert;
\end{eqnarray*} 
hence we have $\Im(w)=\Ker(z)$, 
which proves the above claim.

Since the image of $w$ is $L_1\calS$, the sequence \eqref{eq:ses_3_tor}
shows that 
$L_0\calS/L_1\calS \simeq \frakL_3\big(H^{(2)}\big)$
and that $L_1\calS/L_2 \calS  \simeq S^2H^{(2)}$.
In addition, another application of \eqref{eq:Delta_3} shows that
\begin{equation} \label{eq:DzL}
\Delta(\{h\otimes x\wedge y\}) 
=\frac{1}{2}
\begin{array}{c}
\begin{tikzpicture}[scale = 0.4]
		\node (0)[label={[xshift=4pt]left:{$\scriptstyle   y$}}] at (0, 5) {};
		\node (1) at (1, 4.5) {};
		\node (2)[label={[xshift=4pt]left:{$\scriptstyle   x$}}] at (0, 4) {};
		\node (3) at (2, 4) {};
		\node (4)[label={[xshift=4pt]left:{$\scriptstyle   h$}}] at (0, 3) {};
		\node (5) at (3.5, 4) {};
		\node (6)[label={[xshift=-4pt]right:{$\scriptstyle   h$}}] at (5.5, 5) {};
		\node (7) at (4.5, 3.5) {};
		\node (8)[label={[xshift=-4pt]right:{$\scriptstyle   x$}}] at (5.5, 4) {};
		\node (9)[label={[xshift=-4pt]right:{$\scriptstyle   y$}}] at (5.5, 3) {};
		\draw [thick](0.center) to (1.center);
		\draw [thick] (2.center) to (1.center);
		\draw [thick] (4.center) to (3.center);
		\draw [thick] (1.center) to (3.center);
		\draw [thick] (3.center) to (5.center);
		\draw [thick] (5.center) to (6.center);
		\draw [thick] (5.center) to (7.center);
		\draw [thick] (7.center) to (8.center);
		\draw [thick] (7.center) to (9.center);
        \draw [black, fill = black] (1.center)  circle (0.8ex);
        \draw [black, fill = black] (3.center)  circle (0.8ex);
        \draw [black, fill = black] (5.center)  circle (0.8ex);
        \draw [black, fill = black] (7.center)  circle (0.8ex);
\end{tikzpicture}
\end{array} = \operatorname{Lev}([h,[x,y]]) 
\in  \calA^{c}_{4,0}(H)  \otimes \Q/\Z,
\end{equation}
where we use Levine's embedding \eqref{eq:Levine_embedding}.
So, to sum up, we have proved that $\Gr^L \!\Delta$ is given by 
$$
\frakL_3(H^{(2)})\simeq \Gr_0^L \calS
\longrightarrow \Gr_0^L \calA \simeq \calA_{4,0}^{c} (H) \otimes \Q/\Z  ,  \quad 
[h,[x,y]] \longmapsto \operatorname{Lev}([h,[x,y]]) 
$$
$$
S^2 H^{(2)}\simeq \Gr_1^L \calS
\longrightarrow \Gr_1^L \calA \simeq \calA_{4,1}^{c} (H) \otimes \Q/\Z  ,  \quad 
h \cdot y \longmapsto \frac{1}{2}   \loopedfour{ {h}}{ {h}}{ {y}}{ {y}} 
$$
in degrees 0 and $1$, respectively. 
We conclude that $\Gr^L \!\Delta$ is injective.
\end{proof}

As a consequence, we obtain a characterization
of the $\Sp(H)$-module $Y_3\calIC/Y_4$.

\begin{corollary} \label{cor:Y_3(IC)/Y_4}
We have the commutative diagram of $\Sp(H)$-modules
\begin{equation} \label{eq:ladder_3}
\xymatrix{
0 \ar[r]
  &  {\displaystyle\frac{H^{(2)} \otimes \Lambda^2 P^{(2)}}{\Lambda^3 P^{(2)}} } \ar[r]^-E \ar[d]_-\Psi
  & {\displaystyle \frac{\calA_3^{<,c}(H)}{E(\Lambda^3 P^{(2)})} }\ar[r] \ar[d]_-\psi
  & {\displaystyle \frac{\calA_3^{<,c}(H)}{\Tors \calA_3^{<,c}(H)}} \ar[r] \ar[d]
  & 0\\
0 \ar[r]
  & {\displaystyle \operatorname{Tors}(Y_3 \calIC/Y_4)}  \ar[r]
  &  Y_3\calIC/Y_4  \ar[r]^-{Z_3^<}
  & \bqm \calA_3^{<,c}(H)\eqm \ar[r]  & 0,
}    
\end{equation}
with exact rows and vertical isomorphisms.
Furthermore, the map
\begin{equation} \label{eq:big_map_3}
(Z_3^<,\zeta_4^<): {{ Y_3\calIC}/{Y_{4}}} 
\longrightarrow  \bqm \calA_3^{<,c}(H)\eqm \, \oplus\, 
\big( \calA^{<,c}_{4}(H) \otimes \Q/\Z\big)
\end{equation}
is an $\Sp(H)$-embedding 
for the $\Sp(H)$-action defined by \eqref{eq:Sp-action}.
\end{corollary}

\begin{proof}
The commutativity of the diagram follows from \eqref{eq:E_psi's}
and  \eqref{eq:psi_Z}.
The exactness of the first row is a consequence of Lemma \ref{lem:E_inj},
while the exactness of the second row follows from \eqref{eq:psi_Z}.
The third vertical arrow is induced by the canonical map 
$\calA_3^{<,c}(H) \to \calA_3^{<,c}(H) \otimes \Q$, so that it is obviously an isomorphism.
Since $\Psi$ is an isomorphism (by Theorem~\ref{th:Tors(Y_3IC/Y_4)}),  so is $\psi$. 

Since $Z_3^<: Y_3\mathcal{IC}/Y_4 \to \bqm\calA_3^{<,c}(H)\eqm$
is rationally an isomorphism,
and since $\zeta_4^<  : \Tors(Y_3\mathcal{IC}/Y_4) \to \calA^{<,c}_{4}(H) \otimes \Q/\Z$ 
is injective (by Theorem \ref{th:Tors(Y_3IC/Y_4)}),
the map \eqref{eq:big_map_3} is injective. 
It is also $\Sp(H)$-equivariant by  Corollary \ref{cor:equivariance}.
\end{proof}

\subsection{Extracting classical invariants from LMO}

\noindent
We  start by analyzing the component $Z_3^<$ of \eqref{eq:big_map_3}.
For this, we consider the following classical invariants:
\begin{itemize}
\item the third Johnson homomorphism $\tau_3$
with values in $\mathsf{D}_3(H) \simeq \bqm \calA^c_{3,0}(H) \eqm$;
\item the degree $3$ part $p_+ \tilde\alpha_3$ of the non-commutative 
Reidemeister--Turaev torsion, 
whose target $\bqm \calA^c_{3,1}(H) \eqm$ is identified 
with $\Lambda^3 H$ through the canonical isomorphism\\[-0.7cm]
$$ 
\tikz[baseline=8pt ,scale = 0.5]{
\draw [thick,color=black] (1,0.75) circle (0.5);
\draw [thick,color=black] (1,0.75+0.5)-- (1,0.75+1);
\draw [thick,color=black] (1+0.707*0.5,0.75 + 0.707*0.5)-- (1+0.707,0.75 + 0.707);
\draw [thick,color=black] (1-0.707*0.5,0.75 + 0.707*0.5)-- (1-0.707,0.75 + 0.707);
\draw [black, fill = black] (1,0.75+0.5) circle (.5ex);
\draw [black, fill = black] (1-0.707*0.5,0.75 + 0.707*0.5) circle (.5ex);
\draw [black, fill = black] (1+0.707*0.5,0.75 + 0.707*0.5) circle (.5ex);
\draw[color=black] (1-0.707*1.3,0.75 + 0.707*1.3) node {\small $x$};
\draw[color=black] (1,0.75+1.3) node {\small $y$};
\draw[color=black] (1+0.707*1.3,0.75 + 0.707*1.3) node {\small $z$};
} \longmapsto x\wedge y \wedge z.
$$
\end{itemize}
According to \eqref{eq:tau_Z} and \eqref{eq:alpha_Z},
they correspond  respectively to the $0$-loop part $Z_{3,0}$
and $1$-loop part $Z_{3,1}$ of $Z_3 = X^{-1} \circ Z_3^<$:
\begin{equation} \label{eq:Z_0_Z_1}
\xymatrix{
Y_3 \calIC /Y_4 \ar[r]^-{\tau_3}  \ar[dr]_-{Z_{3,0}} 
& \mathsf{D}_3(H)\\
& \bqm \calA^c_{3,0}(H) \eqm \ar[u]^-\simeq_-\eta
} \qquad
\xymatrix{
Y_3 \calIC /Y_4 \ar[r]^-{-p_+ \tilde\alpha_3}  \ar[dr]_-{2Z_{3,1}} 
& \Lambda^3 H\\
& \bqm \calA^c_{3,1}(H) \eqm \ar[u]^-\simeq
} 
\end{equation}

To express their mutual dependency, let $\Tr:\mathsf{D}_3(H) \to \Lambda^3 H^{(2)}$
be the map given by Lemma \ref{lem:Tr^(0)}
via the above mentioned isomorphisms: specifically, we have
\begin{eqnarray*}
\Tr\Big(\fivetree{v}{w}{x}{y}{z}\Big)
& =&\omega(v,z)\, w\wedge x \wedge y + \omega(v,y)\, w\wedge x\wedge z \\
&& +\,  \omega(w,z)\, v\wedge x \wedge y + \omega(w,y)\, v\wedge x\wedge z.     
\end{eqnarray*}
Let  also $\mathsf{D}_3(H) \times_{\Lambda^3H^{(2)}} \Lambda^3H$
be the fibered product  over $\Tr$ 
and the mod $2$ reduction  $\Lambda^3H \to \Lambda^3H^{(2)}$.

\begin{proposition} \label{prop:tf3}
We have the short exact sequence of $\Sp(H)$-modules
\[ 
\xymatrix{
0 \ar[r] & \Tors(Y_3 \calIC /Y_4)  \ar[r] &
Y_3 \calIC /Y_4 \ar[rr]^-{(\tau_3,\, p_+\circ\tilde{\alpha}_3)\ }&& 
\mathsf{D}_3(H) \times_{\Lambda^3H^{(2)}} \Lambda^3H \ar[r] & 0.
}
\]
\end{proposition}

\begin{proof}
This is deduced by specializing Proposition \ref{prop:ZZ}
at $r:=1$ and $k:=3$, as in  Corollary \ref{cor:0_1}.
\end{proof}

We now analyze the component $\zeta_4^<$ of \eqref{eq:big_map_3}.
Since the maximal torsion-free quotient of  $Y_3 \calIC /Y_4$ 
is characterized 
in Proposition \ref{prop:tf3}, we restrict ourselves 
to $\operatorname{Tors}(Y_3 \calIC /Y_4)$
and use the results of \S \ref{subsec:back}. Specifically,
we need the homomorphisms
$$
R_4: \operatorname{Tors}(Y_3 \calIC /Y_4) 
\longrightarrow \frakL_3\big(H^{(2)}\big), \
\{ M \} \longmapsto  \{\tau_4(M)\},
$$
and 
$$
R'_4: \Ker(R_4) \longrightarrow \calA^c_{4,1}(H)\otimes \Z_2,
 \ \{ M \} \longmapsto
 \Tr(\tau_4(M)) + \big\{ p_+\tilde\alpha_4(M) \big\}
$$
given by Proposition \ref{prop:R_2n+2}
and Corollary \ref{cor:R'_2n+2}, respectively.
Recall that $R_4=R_4^{(0)}$ and $R'_4=R_4^{(1)}$ 
determine, respectively,
the degree 0 part and degree 1 part of 
$$
\Gr^L \zeta^<_4: \Gr^L \operatorname{Tors}(Y_3 \calIC /Y_4)
\longrightarrow \Gr^L \calA_4^{<,c}(H) \otimes \Q/\Z
\simeq  \calA_4^{c}(H) \otimes \Q/\Z
$$
where the module $\operatorname{Tors}(Y_3 \calIC /Y_4)$ 
is endowed with the filtration \eqref{eq:filt_T}.

\begin{proposition} \label{prop:tors3}
We have the following diagram of $\Sp(H)$-modules
\begin{equation} \label{eq:ladder}
\begin{tikzcd}
0 \arrow[r] & S^2 H^{(2)} \arrow[r,"w"] \arrow[d,"\Psi\circ  w "]
& {\displaystyle \frac{H^{(2)} \otimes \Lambda^2 P^{(2)}}{\Lambda^3 P^{(2)}} } 
\arrow[r,"z"] \arrow[d,"\Psi"]& \frakL_3\big(H^{(2)}\big)\arrow[r]
\arrow[d,double, no head] 
& 0 \\
0 \arrow[r] & \Ker (R_4)\arrow[r]  & 
\Tors(Y_3 \calIC /Y_4)  \arrow[r,"R_4"]& 
\frakL_3\big(H^{(2)}\big) \arrow[r] & 0
\end{tikzcd}
\end{equation}
with exact rows and vertical isomorphisms,
where the maps $w$ and $z$ are given by 
$w(h \cdot p_*(y)) := \{ h \otimes( s \wedge y)\}$ and 
$z(\{h\otimes x \wedge y\}):=[h,[p_*(x),p_*(y)]]$, respectively.
Furthermore, we have the following commutative triangle:
\begin{equation} \label{eq:triangle}
\xymatrix{
h\cdot y \ \in   \ar@{|->}[rd]
&S^2 H^{(2)} \ar[rd] \ar[r]^-{\Psi \circ w}_-\simeq
& \Ker(R_4) \ar[d]^-{R'_4} \\
& \loopedfour{h}{h}{y}{y} \ \in
 & \calA^c_{4,1}(H) \otimes \Z_2
}    
\end{equation}
\end{proposition}

\begin{proof}
The first exact row in \eqref{eq:ladder} already appeared in \eqref{eq:ses_3_tor}, and $\Psi$ is an isomorphism by Theorem~\ref{th:Tors(Y_3IC/Y_4)}. Therefore, the first claim of the proposition reduces to the commutativity of the right-hand square. This follows from \eqref{eq:triangle3}, which shows that $\zeta_4^<\circ \Psi=\Delta$, from \eqref{eq:R_2n+2}, which connects $R_4$ with $\Gr_0^L \zeta_4^<$, and from \eqref{eq:DzL}, where $\Delta$ was computed at the tree level.

As for the commutativity of \eqref{eq:triangle}, this follows from \eqref{eq:R'_2n+2}, which relates $R'_4$ to $\Gr_1^L \zeta_4^<$, from diagram \eqref{eq:triangle3}, which gives $\zeta_4^<\circ \Psi=\Delta$, and from \eqref{eq:Delta_w}, which computes $\Delta \circ w$.
\end{proof}

We conclude this subsection with a characterization 
by classical invariants
of the  $Y_4$-triviality for homology cylinders.
Recall firstly from \cite[Th$.$ A]{MM13} 
that an $M\in \calIC$ is 
$Y_3$-equivalent to the unit cylinder $U$ 
if, and only if, the following conditions are satisfied:
\begin{itemize}
\item $\tau_1(M)=0\in \mathsf{D}_1(H)$
and, furthermore, $\tau_2(M)=0\in \mathsf{D}_2(H)$;
\item $\alpha(M)=0 \in S^2 H$, where
$\alpha$ is the quadratic part of 
the relative Alexander polynomial
(which is $\tilde{\alpha}_2=p_+ \tilde{\alpha}_2$  
in the present notations);
\item $\lambda(M)=0\in \Z$, where $\lambda$
denotes the variation of the Casson invariant
for a fixed Heegaard embedding of $\Sigma$ in $S^3$.
\end{itemize}
Then,  the following is deduced from
Proposition \ref{prop:tf3} and Proposition \ref{prop:tors3}.

\begin{corollary} \label{cor:Y_4}
An $M \in \calIC$ is 
$Y_4$-equivalent to the unit cylinder $U$ 
if, and only if, we have $\lambda(M)=0$
and the following conditions are successively satisfied:
\begin{itemize}
\item  $\tau_1(M)=0$, $\tau_2(M)=0$ and $\tau_3(M)=0$, 
i.e$.$ $M$ acts trivially on $\pi/\Gamma_5 \pi$;
\item  $\tilde \alpha_2(M)=0$
and $\tilde \alpha_3(M)=0$, 
i.e$.$ the $I^4$-reduction of $\tilde{\alpha}$ is trivial;
\item $R_4(M)=0$, 
i.e$.$ $\tau_4(M)$ is in 
$\mathsf{D}'_4(H)\subset \mathsf{D}_4(H)$;
\item $R'_4(M)=0$, i.e$.$
$\Tr(\tau_4(M)) = 
p_+\big(\tilde{\alpha}_4(M)\big) 
\in \calA_{4,1}^c(H)\otimes \Z_2 $.
\end{itemize}
\end{corollary}

\subsection{The Torelli Lie algebra in degree $3$}

Recall that the mapping cylinder construction
\eqref{eq:mcc} induces a Lie algebra map 
$\Gr \mathbf{c}: \Gr^\Gamma \calI \to \Gr^Y\!\!\calIC$.

\begin{theorem}
\label{th:onto}
The homomorphism $\operatorname{Gr}_3 \mathbf{c}$ 
restricted to the torsion subgroups
$$
\operatorname{Gr}_3 \mathbf{c}: 
\Tors(\Gamma_3 \calI /\Gamma_4 \calI)
\longrightarrow \Tors(Y_3 \calIC / Y_4)
$$
is surjective in genus $g\ge4$.
Therefore, 
the $\Sp(H)$-module $\Tors(\Gamma_3 \calI /\Gamma_4 \calI)$
surjects onto 
${\displaystyle \big({H^{(2)} \otimes \Lambda^2 P^{(2)}\big)}
/{\Lambda^3 P^{(2)}} }$.
\end{theorem}

The proof of Theorem \ref{th:onto} is based 
on the following three lemmas, where
we use the  notation $\widetilde{x}:=\sigma(x)\in P^{(2)}$
for every $x\in H$
and the basis \eqref{eq:basis_P(2)} of $P^{(2)}$.

\begin{lemma} \label{lem:Sp_gen}
The $\Sp(H)$-module $\big({H^{(2)} \otimes \Lambda^2 P^{(2)}}\big)/{\Lambda^3 P^{(2)}}$ is generated by the class of $a_{1} \otimes \widetilde{b}_{1} \wedge \widetilde{a}_{2}$, in genus $g\geq 3$.
\end{lemma}

\begin{proof}
We claim that the $\Sp(H)$-module 
${H^{(2)} \otimes \Lambda^2 P^{(2)}}$
is generated  by the elements 
$x:=a_{1} \otimes \widetilde{b}_{1} \wedge \widetilde{a}_{2}$ 
and $y:=a_{2} \otimes \widetilde{a}_{1} \wedge \widetilde{b}_{1}$.
Moreover, we observe that 
$$
y \equiv
x + b_1 \otimes \widetilde{a_1} \wedge \widetilde{a_2},
\mod \Lambda^3 P^{(2)},
$$ 
and that $b_1 \otimes \widetilde{a}_1 \wedge \widetilde{a}_2$
can be obtained  from $x$
via the symplectic transformation sending $(a_1,b_1)$ to $(-b_1,a_1)$.
(Here, and below, we define elements of $\Sp(H)$ 
by specifying only the elements of the basis $S$ that are modified;
besides, we use the description \eqref{eq:action_sigma}
of the $\Sp(H)$-action on $P^{(2)}$.)
We deduce that  the quotient $\big({H^{(2)} \otimes \Lambda^2 P^{(2)}}\big)/{\Lambda^3 P^{(2)}}$ 
is $\Sp(H)$-generated by the single element $x$.

We now prove the above claim.
As an abelian group, ${H^{(2)} \otimes \Lambda^2 P^{(2)}}$ is generated
by the elements $u\otimes \widetilde{v} \wedge \widetilde{w}$  
for all $u\in S$ and $v,w \in S \cup \{s\}$.
Firstly, using the symplectic transformation that exchanges $(a_i,b_i)$
with $(a_j,b_j)$ for every $i\neq j$,
and the symplectic transformation that maps $(a_i,b_i)$ to $(-b_i,a_i)$ for every $i$,
we easily see that the $\Sp(H)$-module  ${H^{(2)} \otimes \Lambda^2 P^{(2)}}$
is generated by the following elements:
$$
a_1\otimes \widetilde{a}_1 \wedge  \widetilde{b}_2, \
a_1\otimes \widetilde{a}_1 \wedge  \widetilde{b}_1,  \
x, \ y, \ a_1\otimes \widetilde{a}_2 \wedge  \widetilde{a}_3,  \
a_1\otimes s \wedge  \widetilde{b}_1,  \  a_1\otimes s \wedge  \widetilde{a}_2, \
a_1\otimes s \wedge  \widetilde{a}_1.
$$
Next, it remains to show that the above six elements
(other than $x$ and $y$) belong to
the $\Sp(H)$-submodule $V$ of ${H^{(2)} \otimes \Lambda^2 P^{(2)}}$ 
generated by $x$ and $y$:
\begin{itemize}
\item[(i)] The element of $\Sp(H)$ mapping $a_1$ to $a_1+b_1$   transforms 
$x$ to $x+b_1 \otimes \widetilde{b}_1 \wedge \widetilde{a}_2$;
hence we have $b_1 \otimes \widetilde{b}_1 \wedge \widetilde{a}_2 \in V$
and, so, $a_1\otimes \widetilde{a}_1 \wedge  \widetilde{b}_2 \in V$.
\item[(ii)] The element of $\Sp(H)$ mapping $(a_2,b_1)$ to $(a_1+a_2,b_1-b_2)$  transforms 
$y$ to $a_1 \otimes \widetilde{a}_1 \wedge  \widetilde{b}_1- a_1 \otimes \widetilde{a}_1 \wedge  \widetilde{b}_2+ y - a_2 \otimes \widetilde{a}_1 \wedge  \widetilde{b}_2$.
Clearly, $a_2 \otimes \widetilde{a}_1 \wedge  \widetilde{b}_2$ 
is in the $\Sp(H)$-orbit of  $x$, hence we deduce from (i) 
that $a_1 \otimes \widetilde{a}_1 \wedge  \widetilde{b}_1\in V$.
\item[(iii)] The element of $\Sp(H)$ mapping $(b_1,b_3)$ to $(b_1+a_3,b_3+a_1)$  transforms $x$ to
$x+ a_1 \otimes \widetilde{a}_3 \wedge \widetilde{a}_2 $; therefore,
we have $a_1\otimes \widetilde{a}_2 \wedge  \widetilde{a}_3 \in V$.
\item[(iv)]  The element of $\Sp(H)$ mapping $a_2$ to $a_2+b_2$  maps
$x$ to $x+ a_1 \otimes \widetilde{b}_1 \wedge  \widetilde{b}_2 +a_1 \otimes \widetilde{b}_1 \wedge s$.
Since $a_1 \otimes \widetilde{b}_1 \wedge  \widetilde{b}_2$ is in the $\Sp(H)$-orbit of $x$,
we deduce that $a_1\otimes s \wedge  \widetilde{b}_1\in V$.
\item[(v)]  The element of $\Sp(H)$ mapping $a_1$ to $a_1+b_1$   transforms 
$y$ to $y + a_2 \otimes s \wedge \widetilde{b}_1$; hence, we have 
$ a_2 \otimes s \wedge \widetilde{b}_1 \in V$
and, so, $ a_1\otimes s \wedge  \widetilde{a}_2\in V$.
\item[(vi)] The element of $\Sp(H)$ mapping $(a_2,b_1)$ to $(a_1+a_2,b_1-b_2)$  transforms 
 $ a_1\otimes s \wedge  \widetilde{a}_2 $ to  $ a_1\otimes s \wedge  \widetilde{a}_2 
 +   a_1\otimes s \wedge  \widetilde{a}_1 $ hence, by (v), 
$a_1\otimes s \wedge  \widetilde{a}_1$ is in  $V$.
\end{itemize}
This concludes the proof of the lemma.
\end{proof}

\begin{lemma} \label{lem:image_br}
In genus $g\geq 4$, the Jacobi diagram
$$
E\big(a_1 \otimes \widetilde{a_2}\wedge\widetilde{b_1}\big)
=  \forktree{a_1}{a_2}{a_2}{b_1}{b_1}
\in \Tors\big(\calA_3^{<,c}(H)\big)
$$
belongs  to the image of the bracketing map 
$B_3:\Lie_3(\Lambda^3H) \to \calA_3^{<,c}(H)$.
\end{lemma}

\begin{proof}
The Lie algebra isomorphism
$\varphi: \calA^c(S) \oplus \Z_2\cdot Y \to \calA^{<,c}(P)$
given in Remark~\ref{rem:bialgebra} maps
$$
T:=\fivetree{a_2}{b_1}{a_1}{b_1}{a_2}
\quad \hbox{to} \quad
\forktree{a_1}{a_2}{a_2}{b_1}{b_1}.
$$
Hence, it suffices to show that $T$ decomposes
into Lie brackets of length $3$ 
in the Lie subalgebra $\calA^c(S)$.
We compute 
$$
T_1:=\left[ \threetree{b_1}{a_2}{b_4} , \fourtree{a_1}{a_4}{a_2}{b_1} +  \fourtree{b_3}{a_2}{a_4}{a_3} \right] 
= \fivetree{a_2}{b_1}{a_1}{b_1}{a_2} + \fivetree{a_2}{b_4}{a_4}{b_1}{a_2} - \fivetree{a_2}{b_3}{a_3}{b_1}{a_2}
$$
(where we have used the IHX relation for the vanishing
of a one-loop diagram), and
$$
T_2:= \left[ \threetree{b_1}{a_1}{a_2} , 
\fourtree{b_4}{a_2}{b_1}{a_4}
-\fourtree{b_3}{a_2}{b_1}{a_3}\right] 
= - \fivetree{a_2}{b_4}{a_4}{b_1}{a_2} 
+ \fivetree{a_2}{b_3}{a_3}{b_1}{a_2}.
$$
Note that $T_1+T_2=T$.
Moreover, we have
$$
\left[\threetree{b_1}{a_2}{b_3}, \threetree{a_1}{a_3}{a_4}\right]    = \fourtree{b_1}{a_2}{a_4}{a_1} + \fourtree{a_2}{b_3}{a_3}{a_4} + \loopedtwo{a_2}{a_4},
$$
so that (using ``IHX'' again) we get
$$
T_1 = \left[\threetree{b_1}{a_2}{b_4},\left[\threetree{b_1}{a_2}{b_3}, \threetree{a_1}{a_3}{a_4}\right]   \right].
$$
Moreover, we have 
\begin{eqnarray*}
&&\left[ \threetree{b_3}{a_2}{a_4} , \threetree{a_3}{b_4}{b_1}\right] -\left[ \threetree{b_4}{a_2}{a_4} , \threetree{a_4}{b_4}{b_1}\right] \\
&=&\left( - \fourtree{b_3}{a_2}{b_1}{a_3} + \fourtree{a_2}{a_4}{b_4}{b_1}\right)
-\left( - \fourtree{b_4}{a_2}{b_1}{a_4} + \fourtree{a_2}{a_4}{b_4}{b_1}\right)\\
& = &
\fourtree{b_4}{a_2}{b_1}{a_4} - \fourtree{b_3}{a_2}{b_1}{a_3}
\end{eqnarray*}
so that we obtain 
$$
T_2= \left[ \threetree{b_1}{a_1}{a_2} ,
\left[ \threetree{b_3}{a_2}{a_4} , \threetree{a_3}{b_4}{b_1}\right] -\left[ \threetree{b_4}{a_2}{a_4} , \threetree{a_4}{b_4}{b_1}\right] \right].
$$

\up
\end{proof}

\begin{lemma} \label{lem:injectivity_tau_3}
In genus $g\geq 3$, 
$\tau_3\otimes \Q: (\Gamma_3 \calI /\Gamma_4 \calI)\otimes \Q
\to \mathsf{D}_3(H)\otimes \Q$ is injective.
\end{lemma}

\noindent
This result is well-known to specialists, but we could not locate a formulation that explicitly states the genus and applies to the particular surface $\Sigma$ in question (which has a single boundary component). For this reason, we provide a proof.

\begin{proof}[Proof of Lemma \ref{lem:injectivity_tau_3}]
Since we will make extensive use of Hain’s results, we shall largely adopt his notation.  
For any $g,n,r \ge 0$, denote by $\mathfrak{t}_{g,r}^n$ the \emph{Malcev Lie algebra} of the Torelli group $T_{g,r}^n$ associated with a compact oriented
surface of genus $g$ having $r$ boundary components and $n$ interior marked points.  
By \cite[Th$.$ 1.4]{Hai97}, the Lie algebra $\mathfrak{t}_{g,r}^n$ carries for $g\geq 3$ 
a $\Q$-mixed Hodge structure whose weight filtration agrees with its (completed) lower central series. Consequently, $\mathfrak{t}_{g,r}^n$ is isomorphic to the (degree-completion of) its associated graded Lie algebra $\Gr \mathfrak{t}_{g,r}^n$.  
Moreover, this graded Lie algebra is canonically identified with the associated graded $(\Gr^\Gamma T_{g,r}^n)\otimes \Q$ of the lower central series of $T_{g,r}^n$, after tensoring with $\Q$.

In our current setting, we focus on 
$T_{g,1}:=T_{g,1}^0$ (which corresponds to $\calI$ in our notation)
and on
$\mathfrak{t}_{g,1}:= \mathfrak{t}_{g,1}^0$
(whose associated graded object is 
$(\Gr^\Gamma \calI)\otimes \Q$
in our notation).
We also work with the \emph{Johnson filtration}
$$
T_{g,1} = T_{g,1}[1] \supset \cdots 
\supset T_{g,1}[k] \supset T_{g,1}[k+1] \supset \cdots 
$$
which is defined as the sequence of kernels of the actions of $T_{g,1}$
on the successive nilpotent quotients of the fundamental group
$\pi_1(\Sigma_{g,1},\star)$,
where the base point $\star$ lies in $\partial\Sigma_{g,1}$.
Its associated graded (with rational coefficients) is denoted by $\mathfrak{m}_{g,1}$.
Thus, the lemma requires us to show that the canonical map
$\Gr \mathfrak{t}_{g,1}\to \mathfrak{m}_{g,1}$
is injective in degree 3 for $g\geq 3$.

For this purpose, we also introduce the \emph{Johnson filtration} on $T_g^1$, defined in the same way as for $T_{g,1}$, by looking at the action on $\pi_1(\Sigma_g,\star)$, where $\star$ denotes the marked point of $\Sigma_g^1$. 
The \emph{Johnson filtration} on $T_g$ is then defined as the image of the latter under the canonical projection $T_g^1 \to T_g$. 
Denote by $\mathfrak{m}_g^1$ and $\mathfrak{m}_g$, respectively, the associated graded objects (with rational coefficients) of these two filtrations.

Following Morita, Sakasai and Suzuki \cite[Prop$.$~3.1]{MSS20}, we claim that the maps $\Gr \mathfrak{t}_{g,1}\to \mathfrak{m}_{g,1}$ and $\Gr \mathfrak{t}_{g}\to \mathfrak{m}_{g}$ have kernels that are naturally isomorphic. Their proof relies on the commutative diagram
\begin{equation} \label{eq:MSS}
\xymatrix{
0 \ar[r] & \Gr \mathfrak{p}_{g,1} \ar@{=}[d] \ar[r]
&  \Gr \mathfrak{t}_{g,1} \ar[r] \ar[d]& 
\Gr \mathfrak{t}_{g} \ar[d] \ar[r] & 0 \\
0 \ar[r] & \Gr \mathfrak{p}_{g,1} \ar[r]
&   \mathfrak{m}_{g,1} \ar[r] & 
 \mathfrak{m}_{g}  \ar[r] & 0.
}
\end{equation}
Here, $\mathfrak{p}_{g,1}$ denotes the Malcev Lie algebra of the fundamental group of the unit tangent bundle of $\Sigma_g$. This Lie algebra carries a mixed Hodge structure whose weight filtration is given by its lower central series \cite[Prop$.$ 4.4]{Hai97}. The top row of \eqref{eq:MSS} is obtained by applying the functor $\Gr(-)$ to the short exact sequence from \cite[Prop$.$ 3.6]{Hai97} in the category of Lie algebras endowed with mixed Hodge structures (as carried out at the bottom of \cite[p$.$~628]{Hai97}); in particular, the top row is exact.
The exactness of the bottom row of \eqref{eq:MSS} requires additional explanation. By \cite[Prop$.$ 2]{AK87}, there is a short exact sequence
$$
\xymatrix{
0\ar[r]&\Gr\mathfrak{p}_g^1\ar[r]&\frakm_g^1\ar[r]
&\frakm_g \ar[r]&0,
}
$$
where $\mathfrak{p}_{g}^1$ is the Malcev Lie algebra of the fundamental group of $\Sigma_g$. Moreover, in answer to a question posed by Asada and Nakamura \cite{AN95}, Hain \cite[\S 14.4]{Hai97} established the exactness of
$$
\xymatrix{
0\ar[r]& \Q(1) \ar[r]&\frakm_{g,1}\ar[r]
&\frakm_g^1 \ar[r]&0.
}
$$
All these statements are valid under the assumption $g\geq 3$. 
Thus, with this hypothesis, we have established the exactness of the bottom row of \eqref{eq:MSS}, 
which completes the proof of the preceding claim.

To conclude the proof of the lemma, it remains 
to justify that the canonical map
$\Gr \mathfrak{t}_{g}\to \mathfrak{m}_{g}$
is injective in degree 3 for $g\geq 3$.
This is implied by the last statement 
of \cite[Cor. 7.6]{Hai15} for $g\geq 4$,
combined with \cite[Prop$.$ 9.3]{Hai15} for $g=3$.
\end{proof}

\begin{remark}
Although this result is not required here, we note the recent preprint \cite{NW26}, in which Naef and Willwacher show the injectivity of
$
\tau_k\otimes \Q:\; (\Gamma_k \calI /\Gamma_{k+1} \calI)\otimes \Q 
\to \mathsf{D}_k(H)\otimes \Q
$
for any $k\geq 3$ provided that $g\geq 3k$.
\hfill $\blacksquare$
\end{remark}

\begin{proof}[Proof of Theorem \ref{th:onto}]
By Lemma \ref{lem:image_br}, there exists an element $\ell \in \Lie_3(\Lambda^3H)$ such that
$
B_3(\ell) = E\big(a_1 \otimes \widetilde{a_2}\wedge\widetilde{b_1}\big).
$
Proposition \ref{prop:bracketing} gives the commutative diagram
\begin{equation} \label{eq:carre}
\xymatrix{
\Lie_3\big(\Lambda^3 H\big)\ar[d]_{B_3}  \ar@{->>}[r]^-{J_3} &
\Gamma_3 \calI/\Gamma_{4} \calI \ar[d]^-{\Gr_3\!\mathbf{c}}\\
\calA_3^{<,c}(H) \ar@{->>}[r]_-{\psi_3} &
Y_3 \calIC/ Y_{4},
}    
\end{equation}
from which we deduce that
$$
(\Gr_3\!\mathbf{c})(J_3(\ell)) =
\psi\big(E\big(a_1 \otimes \widetilde{a_2}\wedge\widetilde{b_1}\big)\big) = 
\Psi\big(a_1 \otimes \widetilde{a_2}\wedge\widetilde{b_1}\big).
$$
So, since $\Psi$ is an $\Sp(H)$-isomorphism (by Theorem \ref{th:Tors(Y_3IC/Y_4)}), 
we deduce from Lemma~\ref{lem:Sp_gen} 
that the $\Sp(H)$-module 
$\Tors\big(Y_3\calIC/Y_4\big)$ is generated 
by $(\Gr_3\!\mathbf{c})(J_3(\ell))$.

Moreover, the homomorphism
$$
(\Gr_3 \mathbf{c})\otimes \Q: 
(\Gamma_3 \calI /\Gamma_4 \calI) \otimes \Q
\longrightarrow (Y_3\calIC/Y_4) \otimes \Q
$$
is injective, because the map $\tau_3\otimes \Q$ from Lemma \ref{lem:injectivity_tau_3} factors through 
$(\Gr_3 \mathbf{c})\otimes \Q$. Consequently, $J_3(\ell)$ lies in the torsion subgroup of $\Gamma_3 \calI /\Gamma_4 \calI$.
We conclude that the map
$\Gr_3 \mathbf{c}:\Tors(\Gamma_3 \calI /\Gamma_4 \calI)
\to \Tors(Y_3 \calIC / Y_4)$
is surjective. The second assertion then follows from Theorem \ref{th:Tors(Y_3IC/Y_4)}.
\end{proof}

Theorem \ref{th:onto} does not address whether 
$\operatorname{Gr}_3 \mathbf{c}: 
\Tors(\Gamma_3 \calI /\Gamma_4 \calI)
\to \Tors(Y_3 \calIC / Y_4)$
is injective.
(As shown in \cite{MM03} and \cite{FMS26}, 
the map $\operatorname{Gr}_k \mathbf{c}$ is injective for $k=1$ and $k=2$, respectively.)
We conclude with a characterization of this injectivity.

\begin{proposition} \label{prop:injectivity}
In genus $g\geq 4$, 
the following assertions are equivalent:
\begin{enumerate}[(i)]
\item $\Gamma_3 \calI / \Gamma_4 \calI$ 
embeds into $Y_3 \calIC / Y_4$ 
by $\Gr_3 \mathbf{c}$.
\item $\Tors(\Gamma_3 \calI / \Gamma_4 \calI)$ is isomorphic to $\Tors(Y_3 \calIC / Y_4)$
by $\Gr_3 \mathbf{c}$.
\item We have 
$\Gamma_4 \calI = \big\{ f \in \Gamma_3 \calI :
\tau_3(f)=0, \tau_4(f) \in \mathsf{D}'_4(H), \Tr(\tau_4(f))= 0 \big\}.$
\item $B_3^{-1}(\Ker(\psi_3)) \subset \Ker(J_3).$
\end{enumerate}
\end{proposition}

\begin{proof}
According to Theorem \ref{th:onto},  
statement (ii) is equivalent to
\begin{quote}
\emph{(ii')  $\Tors(\Gamma_3 \calI / \Gamma_4 \calI)$ 
embeds into $\Tors(Y_3 \calIC / Y_4)$}
via $\Gr_3 \mathbf{c}$.
\end{quote}
In the proof of Theorem \ref{th:onto}, we noted that 
$(\Gr_3 \mathbf{c})\otimes \Q$ is injective; hence (ii') is also
equivalent to (i).  
Moreover, the commutative diagram \eqref{eq:carre} implies that (i) is
equivalent to (iv).  
Finally, since $\tilde \alpha(\calI)$ lies in the natural image of $H$
inside $K_1\big(\widehat{\Q[\pi]}\big)$, 
Corollary \ref{cor:Y_4} yields that (i) 
is equivalent to (iii) as well.
\end{proof}

\subsection{Proof of some lemmas}

\begin{proof}[Proof of Lemma \ref{lem:E_inj}]
By \cite[Lemma 8.4]{CHM08} 
(see also Proposition \ref{prop:varphi}), 
we have the following  isomorphisms of abelian groups:
$$
\calA^{<,c}_3(H) \simeq \calA^c_3(S) \simeq  
\calA^c_3(H) =
\calA^c_{3,0}(H) \oplus  \calA^c_{3,1}(H)
$$
Besides, it follows from \eqref{eq:Levine1} that
$$
\Tors \calA^c_{3,0}(H) 
\simeq H^{(2)} \otimes \Lambda^2 H^{(2)},
$$
and it follows from \eqref{eq:torsion_wheel} that
$$
\Tors \calA^c_{3,1}(H) 
\simeq H^{(2)} \otimes H^{(2)}.
$$
Therefore, the finite abelian groups
$\Tors \calA^{<,c}_{3}(H) $ and 
$ H^{(2)} \otimes \Lambda^2 P^{(2)}$ have the same cardinality.
Thus, it is enough to prove the injectivity of $E$. 

To prove that $E$ is injective, we start 
with the commutative diagram
$$
\begin{tikzcd}[column sep=large, row sep=large]
H^{(2)} \otimes \Lambda^2 P^{(2)}
  \arrow[r, "E"]
  \arrow[d, "\mathrm{id} \otimes \Lambda^2 p_*"]
& \mathcal{A}_3^{<,c}(H) 
  \arrow[d, two heads]  & \\
H^{(2)} \otimes \Lambda^2 H^{(2)}
\arrow[r]
&    {\displaystyle\frac{L_0 \calA^{<,c}_3(H)}{L_1 \calA^{<,c}_{3}(H)}}
\simeq \calA^{c}_{3,0}(H)   
\end{tikzcd}
$$
where, as observed in Remark \ref{rem:E_Levine}, 
the bottom arrow coincides with the injection in Levine's short exact sequence \eqref{eq:Levine1}.
We deduce that the kernel 
of $E$ is included in 
$$
\Ker\big(\id \otimes \Lambda^2 p_*\big)
= H^{(2)} \otimes \Ker\big(\Lambda^2 p_*\big)
= \Im(k)
$$
where  $k$ is the homomorphism
$$
k: H^{(2)} \otimes H^{(2)} \longrightarrow
H^{(2)} \otimes \Lambda^2 P^{(2)}, \quad
h \otimes y \longmapsto h \otimes \big(s \wedge \sigma(y)\big).
$$
But, according to \eqref{eq:Ehxy}
and still using the convention \eqref{eq:lifts}, we have
$$
\big(\, E(k(h\otimes y))\!\!\! \mod L_2 \,\big)
= \tikz[baseline=8pt ,scale = 0.5]{
\draw [thick,color=black] (1,0.75) circle (0.5);
\draw [thick,color=black] (1,0.75+0.5)-- (1,0.75+1);
\draw [thick,color=black] (1+0.707*0.5,0.75 + 0.707*0.5)-- (1+0.707,0.75 + 0.707);
\draw [thick,color=black] (1-0.707*0.5,0.75 + 0.707*0.5)-- (1-0.707,0.75 + 0.707);
\draw [black, fill = black] (1,0.75+0.5) circle (.5ex);
\draw [black, fill = black] (1-0.707*0.5,0.75 + 0.707*0.5) circle (.5ex);
\draw [black, fill = black] (1+0.707*0.5,0.75 + 0.707*0.5) circle (.5ex);
\draw[color=black] (1-0.707*1.3,0.75 + 0.707*1.3) node {\small $h$};
\draw[color=black] (1,0.75+1.3) node {\small $y$};
\draw[color=black] (1+0.707*1.3,0.75 + 0.707*1.3) node {\small $y$};
} \in 
{\displaystyle \frac{L_1 \mathcal{A}_3^{<,c}(H)}{L_2 \mathcal{A}_3^{<,c}(H) }}
\simeq \mathcal{A}_{3,1}^{c}(H).
$$
Thus, using the isomorphism \eqref{eq:wheel}, 
we obtain that the homomorphism
$$
E\circ k : H^{(2)} \otimes H^{(2)}  
\longrightarrow  \mathcal{A}_{3,1}^{c}(H) \simeq 
\big(H^{\otimes 3}\big)_{\mathfrak{D}_6}, \quad
h\otimes y \longmapsto \{ h  y  y\}
$$ 
is injective. We conclude that $E$ is injective.
\end{proof}

\begin{proof}[Proof of Lemma \ref{lem:cyclic}]
Let $T$ be the $\Z_2$-vector subspace of $P^{(2)}$
generated by the  elements of $\sigma(S)$ 
for a choice of symplectic basis
 $S=(a_1,\dots,a_g,b_1,\dots,b_g)$ of~$H$. 
 Then, we have $P^{(2)}=T\oplus \Z_2\, s$ 
and $T \simeq H^{(2)}$ via $p_*$.
Thus, we have the decompositions 
$$
\Lambda^3 P^{(2)} = \Lambda^3 T\oplus (s \wedge \Lambda^2T)
\quad \hbox{and} \quad 
\Lambda^2 P^{(2)}  = \Lambda^2 T \oplus (s \wedge T).
$$
Moreover, the map $C$ is the direct sum of the homomorphism
$$
\Lambda^3 T \longrightarrow H^{(2)} \otimes  \Lambda^2 T 
\simeq T \otimes  \Lambda^2 T, \ 
x \wedge y \wedge z \longmapsto x \otimes y \wedge z+y \otimes z \wedge x +z \otimes x \wedge y
$$
with the homomorphism
$$
\Lambda^2 T\simeq \big( s\wedge \Lambda^2 T \big) \longrightarrow H^{(2)} \otimes (s\wedge T)
\simeq T \otimes T, \ 
y\wedge z \longmapsto y\otimes z + z \otimes y.
$$
Since those two components are injective, the map $C$ is injective.

We now prove that $\Psi(C(x\wedge y \wedge z))=0$
for any $x,y,z\in P^{(2)}$ using clasper calculus.  From the proof of \cite[Lemma 6.11 (2)]{NSS22a}, we get the following  identity in the abelian group $Y_3 \mathcal{IC}/Y_4$
(denoted here multiplicatively):
\tikzset{
  symdash/.style={
    dash pattern=on 2pt off 2pt,
    dash phase=1pt
  }
}
$$\begin{tikzpicture}[scale = 0.4, baseline = 1cm, thick]

\node (3)  at (-2.5,2) {};
\node (5)  at (-1.5,2) {};
\node (6)  at (0,2.25) {};
\node (8)  at (0.8,2) {};
\node (9)  at (-0.75,3.5) {};
\node (10) at (-0.75,3) {};

\node (13) at (-2.3,2.25) {};

\node (14) at (1,2.25) {};

\draw[in=90,out=0] (10.center) to (6.center);
\draw (9.center) -- (10.center);


\draw[symdash] (1.5,2.125) ++(90:0.7)
  arc[start angle=90,end angle=-90,radius=0.7];
\draw (1.5,2.125) ++(90:0.7)
  arc[start angle=90,end angle=270,radius=0.7];
  
\draw[white,fill=white] (0.8,2.25) circle (0.5ex);

\draw (13.center) -- (14.center);

\draw[white,fill=white] (-1.5,2.25) circle (1ex);

\draw[in=180,out=90] (5.center) to (10.center);

\draw (-3,2.125) ++(90:0.5)
  arc[start angle=90,end angle=-90,radius=0.5];
\draw[symdash] (-3,2.125) ++(90:0.5)
  arc[start angle=90,end angle=270,radius=0.5];

\draw (-3,2.125) ++(90:0.7)
  arc[start angle=90,end angle=-90,radius=0.7];
\draw[symdash] (-3,2.125) ++(90:0.7)
  arc[start angle=90,end angle=270,radius=0.7];


\draw (1.5,2.125) ++(90:0.5)
  arc[start angle=90,end angle=270,radius=0.5];
\draw[symdash] (1.5,2.125) ++(-90:0.5)
  arc[start angle=-90,end angle=90,radius=0.5];

\draw (-0.75,4) ++(180:0.5)
  arc[start angle=180,end angle=360,radius=0.5];
\draw[symdash] (-0.75,4) ++(0:0.5)
  arc[start angle=0,end angle=180,radius=0.5];

\draw[white,fill=white] (-2.3,2) circle (0.5ex);
\draw (3.center) -- (8.center);

\node at (-3,2.125)   {\scriptsize $x$};
\node at (1.5,2.125)  {\scriptsize $z$};
\node at (-0.75,4)    {\scriptsize $y$};

\end{tikzpicture}
= \begin{tikzpicture}[scale = 0.4, baseline = 1cm, thick]
		\node  (0) at (1, 2) {};
		\node  (1) at (2, 2) {};
		\node  (2) at (3, 2) {};
		\node  (3) at (4, 2) {};
		\node  (4) at (2, 3) {};
		\node  (5) at (3, 3) {};		
        \draw (0.center) to (3.center);
		\draw (4.center) to (1.center);
		\draw (5.center) to (2.center);
        \draw (0.5,2) ++(90:0.5)
      arc[start angle=90,end angle=-90,radius=0.5];
        \draw[symdash] (0.5,2) ++(90:0.5)
      arc[start angle=90,end angle=270,radius=0.5];
        \draw (3,3.5) ++(180:0.5)
      arc[start angle=180,end angle=360,radius=0.5];
        \draw[symdash] (3,3.5) ++(0:0.5)
      arc[start angle=0,end angle=180,radius=0.5];
        \draw (4.5,2) ++(90:0.5)
      arc[start angle=90,end angle=270,radius=0.5];
        \draw[symdash] (4.5,2) ++(-90:0.5)
      arc[start angle=-90,end angle=90,radius=0.5];
        \node at (0.5,2) {\scriptsize $x$};
        \node at (4.5,2) {\scriptsize $z$};
        \node at (3,3.5) {\scriptsize $y$};
		\node  (6) at (1.75, 3.25) {};
		\node  (7) at (2, 3.5) {};
		\node  (8) at (2.25, 3.25) {};
		\node  (9) at (1.75, 3.75) {};
		\node  (10) at (2, 4) {};
		\node  (11) at (2.25, 3.75) {};
        \draw [bend right=45] (4.center) to (8.center);
		\draw [bend left=45] (4.center) to (6.center);
		\draw [bend right=45] (10.center) to (9.center);
		\draw [bend left=45] (10.center) to (11.center);
		\draw [in=150, out=-90, looseness=0.75] (9.center) to (7.center);
		\draw [in=90, out=-45] (7.center) to (8.center);
        \draw [white, fill = white] (2,3.5) circle  (1ex);
        \draw [in=30, out=-90, looseness=0.75] (11.center) to (7.center);
		\draw [in=90, out=-135, looseness=0.75] (7.center) to (6.center);
\end{tikzpicture} \circ \Bigg(\begin{tikzpicture}[scale = 0.4, baseline = 1cm, thick]
		\node  (0) at (1, 2) {};
		\node  (1) at (2, 2) {};
		\node  (2) at (3, 2) {};
		\node  (3) at (4, 2) {};
		\node  (4) at (2, 3) {};
		\node  (5) at (3, 3) {};		
        \draw (0.center) to (3.center);
		\draw (4.center) to (1.center);
		\draw (5.center) to (2.center);
        \draw (0.5,2) ++(90:0.5)
      arc[start angle=90,end angle=-90,radius=0.5];
        \draw[symdash] (0.5,2) ++(90:0.5)
      arc[start angle=90,end angle=270,radius=0.5];
        \draw (2,3.5) ++(180:0.5)
      arc[start angle=180,end angle=360,radius=0.5];
        \draw[symdash] (2,3.5) ++(0:0.5)
      arc[start angle=0,end angle=180,radius=0.5];
        \draw (4.5,2) ++(90:0.5)
      arc[start angle=90,end angle=270,radius=0.5];
        \draw[symdash] (4.5,2) ++(-90:0.5)
      arc[start angle=-90,end angle=90,radius=0.5];
        \node at (0.5,2) {\scriptsize $x$};
        \node at (4.5,2) {\scriptsize $z$};
        \node at (2,3.5) {\scriptsize $y$};
		\node (6) at (2.75, 3.25) {};
		\node (7) at (3, 3.5) {};
		\node (8) at (3.25, 3.25) {};
		\node (9) at (2.75, 3.75) {};
		\node (10) at (3, 4) {};
		\node (11) at (3.25, 3.75) {};
        \draw [bend right=45] (5.center) to (8.center);
		\draw [bend left=45] (5.center) to (6.center);
		\draw [bend right=45] (10.center) to (9.center);
		\draw [bend left=45] (10.center) to (11.center);
		\draw [in=150, out=-90, looseness=0.75] (9.center) to (7.center);
		\draw [in=90, out=-45] (7.center) to (8.center);
        \draw [white, fill = white] (3,3.5) circle  (1ex);
        \draw [in=30, out=-90, looseness=0.75] (11.center) to (7.center);
		\draw [in=90, out=-135, looseness=0.75] (7.center) to (6.center);
\end{tikzpicture}\Bigg)^{-1} 
$$
Note that the left-hand side corresponds through $\Psi$ to the term $p_*(y) \otimes (z \wedge x)$ in $C(x,y,z)$. 
We get the other two terms similarly 
by turning around $x,y,z$ cyclically:
thus, $\Psi(C(x,y,z))$ is written as a product of six terms, 
which turn out to cancel by pairs. 
For example, the first term of the right-hand side 
of the above identity cancels with the term 
$$
\Bigg(\begin{tikzpicture}[scale = 0.4, baseline = 0.8cm, thick]
		\node  (0) at (1, 2) {};
		\node  (1) at (2, 2) {};
		\node  (2) at (3, 2) {};
		\node  (3) at (4, 2) {};
		\node  (4) at (2, 3-3) {};
		\node  (5) at (3, 3) {};
        \node  (20) at (2, 1) {};
        \draw (0.center) to (3.center);
		\draw (20.center) to (1.center);
		\draw (5.center) to (2.center);
\draw (0.5,2) ++(90:0.5)
      arc[start angle=90,end angle=-90,radius=0.5];
\draw[symdash] (0.5,2) ++(90:0.5)
      arc[start angle=90,end angle=270,radius=0.5];
\draw (3,3.5) ++(180:0.5)
      arc[start angle=180,end angle=360,radius=0.5];
\draw[symdash] (3,3.5) ++(0:0.5)
      arc[start angle=0,end angle=180,radius=0.5];
\draw (4.5,2) ++(90:0.5)
      arc[start angle=90,end angle=270,radius=0.5];
\draw[symdash] (4.5,2) ++(-90:0.5)
      arc[start angle=-90,end angle=90,radius=0.5];
        \node at (0.5,2) {\scriptsize $x$};
        \node at (4.5,2) {\scriptsize $z$};
        \node at (3,3.5) {\scriptsize $y$};
		\node  (6) at (1.75, 3.25-3) {};
		\node  (7) at (2, 3.5-3) {};
		\node  (8) at (2.25, 3.25-3) {};
		\node  (9) at (1.75, 3.75-3) {};
		\node  (10) at (2, 4-3) {};
		\node  (11) at (2.25, 3.75-3) {};
        \draw [bend right=45] (4.center) to (8.center);
		\draw [bend left=45] (4.center) to (6.center);
		\draw [bend right=45] (10.center) to (9.center);
		\draw [bend left=45] (10.center) to (11.center);
		\draw [in=150, out=-90, looseness=0.75] (9.center) to (7.center);
		\draw [in=90, out=-45] (7.center) to (8.center);
        \draw [white, fill = white] (2,0.5) circle  (1ex);
        \draw [in=30, out=-90, looseness=0.75] (11.center) to (7.center);
		\draw [in=90, out=-135, looseness=0.75] (7.center) to (6.center);
\end{tikzpicture}\Bigg)^{-1}$$ arising from $p_*(z) \otimes (x \wedge y)$; 
indeed, we have  the local move 
$$
\begin{tikzpicture}[scale = 0.4, baseline = 0.7cm, thick]
		\node  (0) at (1, 2) {};
		\node  (1) at (2, 2) {};
		\node  (2) at (3, 2) {};
		\node  (4) at (2, 3-3) {};
        \node  (20) at (2, 1) {};
        \draw (0.center) to (2.center);
		\draw (20.center) to (1.center);
		\node  (6) at (1.75, 3.25-3) {};
		\node  (7) at (2, 3.5-3) {};
		\node  (8) at (2.25, 3.25-3) {};
		\node  (9) at (1.75, 3.75-3) {};
		\node  (10) at (2, 4-3) {};
		\node  (11) at (2.25, 3.75-3) {};
        \draw [bend right=45] (4.center) to (8.center);
		\draw [bend left=45] (4.center) to (6.center);
		\draw [bend right=45] (10.center) to (9.center);
		\draw [bend left=45] (10.center) to (11.center);
		\draw [in=150, out=-90, looseness=0.75] (9.center) to (7.center);
		\draw [in=90, out=-45] (7.center) to (8.center);
        \draw [white, fill = white] (2,0.5) circle  (1ex);
        \draw [in=30, out=-90, looseness=0.75] (11.center) to (7.center);
		\draw [in=90, out=-135, looseness=0.75] (7.center) to (6.center);
\end{tikzpicture} =  \begin{tikzpicture}[scale = 0.4, baseline = 0.7cm, thick]
		\node  (0) at (1, 2) {};
		\node  (1) at (2, 2) {};
		\node  (2) at (3, 2) {};
		\node  (4) at (2, 3) {};
		\draw (4.center) to (1.center);
        \draw (0.center) to (2.center);
		\node  (6) at (1.75, 3.25) {};
		\node  (7) at (2, 3.5) {};
		\node  (8) at (2.25, 3.25) {};
		\node  (9) at (1.75, 3.75) {};
		\node  (10) at (2, 4) {};
		\node  (11) at (2.25, 3.75) {};
        \draw [bend right=45] (4.center) to (8.center);
		\draw [bend left=45] (4.center) to (6.center);
		\draw [bend right=45] (10.center) to (9.center);
		\draw [bend left=45] (10.center) to (11.center);
		\draw [in=150, out=-90, looseness=0.75] (9.center) to (7.center);
		\draw [in=90, out=-45] (7.center) to (8.center);
        \draw [white, fill = white] (2,3.5) circle  (1ex);
        \draw [in=30, out=-90, looseness=0.75] (11.center) to (7.center);
		\draw [in=90, out=-135, looseness=0.75] (7.center) to (6.center);
\end{tikzpicture}
$$ 
which follows 
from  a double application of \cite[Lemma 2.3]{Mei06}.
\end{proof}

\bibliographystyle{alpha}
\bibliography{Sp_Gr}

\end{document}